\documentclass[11pt,twoside]{article}
\usepackage{simplewick}
\usepackage{times}
\usepackage{amsmath,amssymb,amsthm}
\usepackage{color}
\usepackage{hep}
\usepackage{mathrsfs}
\usepackage{leftidx}
\usepackage{graphicx}
\usepackage{float}
\usepackage{enumerate}
\usepackage{titletoc}
\usepackage{epstopdf}
\usepackage{ulem}
\usepackage[toc,page,title,titletoc,header]{appendix}

\allowdisplaybreaks
\def \no{\nonumber}

\def\p{\partial}
\def\ve{\varepsilon}
\def\f{\frac}
\def\na{\nabla}

\def\al{\alpha}
\def\t{\tilde}

\def\vp{\varphi}
\def\O{\Omega}
\def\th{\theta}
\def\g{\gamma}
\def\G{\Gamma}

\def\dl{\delta}
\def\p{\partial}
\def\ve{\varepsilon}
\def\f{\frac}
\def\k{\kappa}

\def\na{\nabla}

\def\al{\alpha}

\def\t{\tilde}
\def\o{\omega}
\def\O{\Omega}
\def\vp{\varphi}
\def\th{\theta}
\def\Th{\Theta}
\def\g{\gamma}
\def\G{\Gamma}

\def\dl{\delta}
\def\b{\beta}

\def\ds{\displaystyle}
\allowdisplaybreaks

\begin{document}
	\footskip=0pt
	\footnotesep=2pt
	\let\oldsection\section
	\renewcommand\section{\setcounter{equation}{0}\oldsection}
	\renewcommand\thesection{\arabic{section}}
	\renewcommand\theequation{\thesection.\arabic{equation}}
	\newtheorem{claim}{\noindent Claim}[section]
	\newtheorem{theorem}{\noindent Theorem}[section]
	\newtheorem{lemma}{\noindent Lemma}[section]
	\newtheorem{proposition}{\noindent Proposition}[section]
	\newtheorem{definition}{\noindent Definition}[section]
	\newtheorem{remark}{\noindent Remark}[section]
	\newtheorem{corollary}{\noindent Corollary}[section]
	\newtheorem{example}{\noindent Example}[section]

\title{On global smooth solutions to the
2D isentropic and irrotational Chaplygin gases with short pulse data}
\author{Ding
Bingbing$^{1,*}$, \quad Xin Zhouping$^{2,*}$, \quad Yin
Huicheng$^{1,}$\footnote{Ding Bingbing (13851929236@163.com, bbding@njnu.edu.cn) and Yin Huicheng (huicheng@nju.edu.cn, 05407@njnu.edu.cn) were
supported by the NSFC (No.12331007, No.12071223). Xin Zhouping(zpxin@ims.cuhk.edu.hk)
is partially supported by the Zheng Ge Ru Foundation, Hong Kong RGC Earmarked Research Grants
CUHK-14301421, CUHK-14300819, CUHK-14302819, CUHK-14300917, Basic and Applied Basic Research Foundations of Guangdong Province 2020131515310002, and the Key Projects of National Nature Science Foundation of China (No.12131010, No.11931013).}\vspace{0.5cm}\\
\small 1.  School of Mathematical Sciences and Mathematical Institute, Nanjing Normal University,\\
\small  Nanjing, 210023, China.
\\
\vspace{0.5cm}
\small 2. Institute of Mathematical Sciences, The Chinese University of Hong Kong, Shatin, NT, Hong Kong.}
\date{}
\maketitle

\centerline {\bf Abstract} \vskip 0.3 true cm

This paper establishes the global existence of smooth solutions to the 2D isentropic and irrotational Euler equations for Chaplygin gases with a general class of short pulse initial data, which, in particular, resolves in this special case, the Majda's conjecture on the non-formation of shock waves of solutions from smooth initial data for multi-dimensional nonlinear symmetric systems which are totally linearly degenerate. Comparing to the 4D case,
the major difficulties in this paper are caused by
the slower time decay and the largeness of the solutions to the 2D quasilinear wave equation, some new auxiliary energies and multipliers are introduced to overcome these difficulties.

\vskip 0.2 true cm
{\bf Keywords:} 2D quasilinear wave equation, Chaplygin gases, short pulse initial data,
null condition,

\qquad\qquad \quad  inverse foliation density, Goursat problem

\vskip 0.2 true cm {\bf Mathematical Subject Classification:} 35L05, 35L72

\vskip 0.4 true cm

\tableofcontents

\section{Introduction}\label{in}
\subsection{The formulation of the problem and main results}\label{2}

Consider a general 2D quasilinear wave equation
\begin{equation}\label{quasi-X0}
\sum_{\al,\beta=0}^2g^{\al\beta}(\p\phi)\p_{\al\beta}^2\phi=0.
\end{equation}
The general short pulse initial data for \eqref{quasi-X0} are given as
\begin{align}\label{i0}
\phi|_{t=1}=\delta^{2-\varepsilon_0}\phi_0(\f{r-1}{\delta},\omega),\
\p_t\phi|_{t=1}=\delta^{1-\varepsilon_0}\phi_1(\f{r-1}{\delta},\omega),
\end{align}
where $\dl>0$ is a small constant, $0<\ve_0<1$ is a fixed constant, $\o=\f{x}r\in\mathbb S^1$, and $(\phi_0,\phi_1)(s,\o)$ are smooth functions defined in $\mathbb R\times\mathbb S^1$ with compact support in the variable $s$. For small $\p\phi$, set
\begin{equation}\label{H-01X}
\ds g^{\al\beta}(\p\phi)=m^{\al\beta}+\sum_{\gamma=0}^{2} g^{\al\beta,\gamma}\p_\gamma\phi+\sum_{\gamma,\nu=0}^2h^{\al\beta,\gamma\nu}\p_\gamma\phi\p_\nu\phi+O(|\p\phi|^3),
\end{equation}
where $m^{00}=-1$, $m^{ii}=1$ for $1\le i\le 2$, $m^{\al\beta}=0$ for $\al\neq\beta$,
$g^{\al\beta,\gamma}$ and $h^{\al\beta,\gamma\nu}$ are constants. Then the first and second null conditions are defined as that for $\xi_0=-1$ and $(\xi_1,\xi_2)\in\mathbb S^1$,
\begin{equation}\label{Euler-0}
\begin{array}{l}
\ds\sum_{\al,\beta,\gamma=0}^2g^{\al\beta,\gamma}\xi_\al\xi_\beta\xi_\gamma\equiv0,\\
\ds\sum_{\al,\beta,\gamma,\nu=0}^2 h^{\al\beta,\gamma\nu}\xi_\al\xi_\beta\xi_\gamma\xi_\nu\equiv0,
\end{array}
\end{equation}
respectively. It is emphasized that \eqref{Euler-0} is indispensable for the global existence of smooth solutions to \eqref{quasi-X0}
even for the Cauchy problem with small initial data. Indeed, if \eqref{quasi-X0} is equipped with the initial data
\begin{equation}\label{Y1-0X}
\phi(1,x)=\dl \vp_0(x),\ \p_t\phi(1,x)=\dl \vp_1(x),
\end{equation}
where $(\vp_0, \vp_1)(x)\in C_0^{\infty}(\mathbb R^2)$, then
it follows from \cite{A} that
\eqref{quasi-X0} has a global solution $\phi$ when \eqref{Euler-0} holds.
 Otherwise, when the  first null condition or the second null condition is violated,
 the smooth solution $\phi$ can blow up in finite time as long as $(\vp_0, \vp_1)\not\equiv 0$ (see \cite{A1}, \cite{A2},
\cite{CM}, \cite{0-Speck} and \cite{S2}).

We are interested in studying global solutions to \eqref{quasi-X0} with the short pulse data \eqref{i0} under suitable structural conditions, which turns out to be an important but difficult problem. To this end, we will focus in this paper on
the 2D compressible isentropic
irrotational Euler equations for Chaplygin gases
since the resulting second order quasilinear wave equation fulfills \eqref{Euler-0}.
The 2D compressible isentropic Euler system for Chaplygin gases is
\begin{equation}\label{Euler}
\left\{
\begin{aligned}
&\p_t\rho+div (\rho v)=0,\\
&\p_t(\rho v)+div (\rho v \otimes v)+\nabla p=0,\\
\end{aligned}
\right.
\end{equation}
where $\nabla=(\p_1,\p_2)=(\p_{x^1},\p_{x^2})$, $(t,x)\in [1,+\infty)\times\mathbb{R}^2$, $v=(v_1,v_2)$, $\rho$, $p$
stand for the velocity, density, pressure
respectively, and the equation of state is given by
\begin{equation}\label{pr}
p(\rho)=P_0-\frac{B}{\rho}
\end{equation}
with $P_0$ and $B$ being positive constants.
Supplement \eqref{Euler} with the irrotational initial data
\[
\big(\rho,v)(1,x)=(\bar\rho+\rho_0(x),v_0(x)\big)=(\bar\rho+\rho_0(x),v_1^0(x),v_2^0(x)\big)
\]
such that $\bar\rho+\rho_0(x)>0$ and $\text{rot}\ v_0(x)=\p_2v_1^0-\p_1v_2^0\equiv0$ with $\bar\rho$ being
a positive constant which can be normalized so that $c(\bar\rho)=1$. Then the irrotationality, $\text{rot}\ v(t,x)=(\p_2v_1-\p_1v_2)(t,x)\equiv0$, holds true as long as the solution remains smooth. Hence there
exists a potential function $\phi$ such that $v=\nabla\phi$, and the Bernoulli's law, $\p_t\phi+\f12|\nabla\phi|^2+h(\rho)=0$,
holds with the enthalpy $h(\rho)$ satisfying $h'(\rho)=\f{c^2(\rho)}\rho$ and $h(\bar\rho)=0$. In this case, for smooth
irrotational flows, \eqref{Euler} is equivalent to
\begin{equation}\label{1.9}
\begin{split}
&\ds\sum_{\al,\beta=0}^2g^{\al\beta}(\p \phi)\p_{\al\beta}^2\phi=-\p_t^2\phi+\triangle\phi+g^{\al\beta,\g}\p_\g\phi\p_{\al\beta}^2\phi
+h^{\al\beta,\nu\g}\p_\nu\phi\p_{\g}\phi\p_{\al\beta}^2\phi\\
\equiv& -\p_t^2\phi+\triangle\phi-2\sum_{i=1}^2\p_i\phi\p_t\p_i\phi
+2\p_t\phi\triangle\phi-\sum_{i,j=1}^2\p_i\phi\p_j\phi\p_{ij}^2\phi
+|\nabla\phi|^2\triangle\phi=0
\end{split}
\end{equation}
with $x^0=t$, $\p_0=\p_t$ and $\Delta=\p_1^2+\p_2^2$, where $g^{\al\beta,\g}$ and $h^{\al\beta,\nu\g}$ are corresponding constants. It is easy to check that \eqref{1.9}
satisfies both the null conditions in \eqref{Euler-0}. Our main goal in this paper is to study the global smooth solution to \eqref{1.9} with the short pulse initial data \eqref{i0}.

Note that in order to guarantee the strict hyperbolicity of \eqref{quasi-X0} or \eqref{1.9},
the smallness of $\p\phi$ should be required in general. On the other hand, as pointed out in \cite{Ding4},
even for the 2D linear wave equation, the solution may have large first order derivatives  if the initial first or
second order derivatives are large.
 Therefore, for the
short pulse initial data \eqref{i0}, due to the largeness of $(\p_x^2\phi, \p_{xt}^2\phi)|_{t=1}$
with small $\dl>0$, then it is difficult to keep the smallness of $\p\phi$ to \eqref{quasi-X0} or \eqref{1.9} in general.
To maintain the smallness of $\p\phi$, motivated by \cite{MPY} and \cite{Ding4}, one can impose such an extra outgoing constraint
condition
\begin{equation}\label{i1}
(\p_t+\p_r)\phi|_{t=1}= O(\delta^{2-\varepsilon_0}).
\end{equation}
In addition, from the forms of \eqref{i0} and \eqref{i1}, it naturally holds
\begin{align}\label{Y-0}
(\p_t+\p_r)^k\O^j\p^q\phi|_{t=1}=O(\delta^{2-\varepsilon_0-|q|}),\quad 0\leq k\leq 1,
\end{align}
where $\O=x^1\p_2-x^2\p_1$.
Here we point out that \eqref{i1}-\eqref{Y-0} are much weaker than the corresponding restrictions
in \cite{Ding4} for $0\le k\le 3$ and in \cite{MPY} for large $k$.
By suitable choices of $(\phi_0,\phi_1)$, \eqref{i1} can be easily fulfilled
(see Remark \ref{R1.1} below).
It follows from \eqref{i0} and \eqref{i1}-\eqref{Y-0} that
the short pulse data are some extensions of a class of large symmetric data, for
which the smallness restrictions are imposed on angular directions and along the ``good" direction tangent
to outgoing light conic surface $\{t=r\}$, but the largeness is kept at least for the second order ``bad" directional derivatives
$\p_t-\p_r$. This provides a powerful framework to study the blowup or the global existence of
smooth solutions to multi-dimensional hyperbolic systems or second order quasilinear wave equations
with short pulse data, see \cite{C3,Ding4,Ding6,K-R,Lu1,MPY,MY,Wang} and the references therein.

Note that although the short pulse
initial data $(\phi,\p_t\phi)|_{t=1}$ in \eqref{i0} do not have the uniform boundedness (smallness)
in $H^{\f{11}{4}+}(\mathbb R^2)\times H^{\f{7}{4}+}(\mathbb R^2)$
-norm (independent of $\delta$)
required for the well-posedness of regular solutions for 2D quasilinear wave equations
in \cite{A1, A, A2, H,Smith} (see Remark \ref{R1.1}), yet the conditions \eqref{i0} and \eqref{i1}
imply the suitably stronger smallness of the directional
derivative $\p_t+\p_r$ of $\phi$. This is essential for our global existence
of smooth solutions to the Cauchy problem \eqref{1.9} with \eqref{i0}.

The main result of this paper can be stated as follows:

\begin{theorem}\label{main}
Let $\varepsilon_0\in(0,\f12)$ hold and
$(\phi_0, \phi_1)(s,\o)\in C^{\infty}(\mathbb R\times\mathbb S^1)$ be any fixed smooth functions with compact
supports in $(-1,0)$ for the variable $s$.	Under the assumption \eqref{i1}, there exists a
suitably small positive constant $\delta_0$ such that for all $\delta\in(0,\delta_0]$, the Cauchy problem,
\eqref{1.9} with \eqref{i0}, admits a global smooth solution
	$\phi\in C^\infty([1,+\infty)\times\mathbb R^2)$.
	Furthermore, it holds that for all time $t\ge 1$ and $x\in\mathbb R^2$,
	\begin{equation}\label{np}
	|\p\phi(t,x)|\le C\delta^{1-\varepsilon_0}t^{-1/2},\ |\phi(t,x)|\lesssim\dl^{3/2-\ve_0}t^{1/2},
	\end{equation}
	where $C>0$ is a constant independent of $\dl$ and $\ve_0$.
\end{theorem}

Some comments are in order.

\begin{remark}\label{R1.0}
	Theorem \ref{main} generalizes our previous global existence of smooth solution to general 4D quasilinear wave equation (satisfying the first null condition) with a class of short pulse initial data \cite{Ding4} to the 2D compressible isentropic irrotational Euler
equations for Chaplygin gases.
\end{remark}

\begin{remark}\label{R1.0-1}
	The results in Theorem \ref{main} are motivated by the corresponding ones for the 2D homogeneous wave equation. Indeed, consider the Cauchy problem for the 2D wave equation $\Box\phi=0$ with the short pulse initial data \eqref{i0}. By using the explicit solution formula, one can check easily that $\p\phi$  may be large in general due to the largeness of $\p_x^2\phi|_{t=1}$ and $\p_{xt}^2\phi|_{t=1}$. However, as will be shown in Appendix, the smallness of $\p\phi$ with $|\p\phi(t,x)|\lesssim\dl^{1-\ve_0}t^{-1/2}$ for $t\geq 1$ can be obtained provided that the condition \eqref{i1} is satisfied.
\end{remark}

\begin{remark}\label{R1.1}
The initial data \eqref{i0} with \eqref{i1} are essentially the ``short pulse data"  introduced by D. Christodoulou
in \cite{C3}, where it is shown that the formation of black holes in vacuum spacetime is due to the condensation of the
gravitational waves for the 3D Einstein equations in general relativity (see also \cite{K-R}). Note that
properties \eqref{i1}-\eqref{Y-0} hold true for any given smooth function $\phi_0$ and the choice of $\phi_1(\f{r-1}\delta,\o)=-\p_s\phi_0(\f{r-1}\delta,\o)+\delta \psi(\f{r-1}{\delta}, \o)$
with any smooth fixed function $\psi(s, \o)$ and $supp_{s}\psi\subset (-1,0)$.
In addition, a large class of short pulse initial data with the property \eqref{i1}-\eqref{Y-0} with $0\le k\le 3$
can be found for general second order quasilinear wave equations (see Section 2 of \cite{Ding4}).
It is noted that for short pulse data \eqref{i0} and sufficiently small $\delta>0$, although both $||\phi||_{L^\infty}$
and $||\p\phi||_{L^\infty}$ are small at $t=1$, yet the initial data are still regarded as ``large" in the sense
that $|\p^2\phi|$ may be large, and meanwhile,
\begin{equation}\label{p1}
||\phi(1,\cdot)||_{H^s(\mathbb R^2)}=O(\delta^{\f{5-\varepsilon_0}2-s})\rightarrow+\infty\ \text{as}\ \delta\rightarrow 0^+\ \text{for}\ s\geq\f{11}4.
\end{equation}
\end{remark}

\begin{remark}\label{R1.2}
	Theorem \ref{main} implies in particular the uniform (independent of $\delta$) local in time well-posedness
for the Cauchy problem \eqref{1.9} with \eqref{i0} and  \eqref{i1}, which does not follow from the known results \cite{H,J2,Smith}.
Indeed, for the Cauchy problem of general 2D quasilinear wave equations with smooth coefficients,
	\begin{equation}\label{X-0}
	\left\{
	\begin{aligned}
		&\ds\sum_{\al,\beta=0}^2g^{\al\beta}(w,\p w)\p_{\al\beta}^2w=0,\\
		&(w(1,x), \p_tw(1,x))=(w_0(x), w_1(x))\in (H^s(\mathbb R^2),H^{s-1}(\mathbb R^2)),\\
	\end{aligned}
	\right.
    \end{equation}
    the local in time well-posedness of solution $w\in C([1,T],H^s(\mathbb R^2))\cap C^1([1,T],H^{s-1}(\mathbb R^2))$
    with $s>\f{11}4$ has been established in \cite{Smith}. However, such a theory cannot be applied to \eqref{1.9}
    with \eqref{i0}
    to yield the local well-posedness of smooth solution with time interval independent of $\delta$ due to \eqref{p1}.
\end{remark}

\begin{remark}
In \cite{Majda}, A. Majda proposed the following {\bf conjecture}: For any given $n$-dimensional nonlinear symmetric hyperbolic system with totally linearly
degenerate structures, a smooth solution, say in $C([0,T],H_{ul}^s(\mathbb R^n))\cap C^1([0,T],H_{ul}^{s-1}(\mathbb R^n))$, $T>0$,
$s>\f n2+1$, will exist globally in time in general when the solution runs out of the domain of definition of the
Cauchy problem. In particular, the shock wave formation does not occur for any smooth initial data.

This conjecture is physically plausible and resolving it would provide insights into the nonlinear nature of the condition that requires linear degeneracy of each characteristic field. It would also help one understand how shock wave formation arises in quasilinear hyperbolic systems.

It is worth noting that our Theorem \ref{main} solves Majda's conjecture for \eqref{Euler} with irrotational short pulse initial data. However, the conjecture remains open for general small data of the form $(\rho,v)(1,x)=(\bar\rho+\varepsilon\rho_0(x),\varepsilon v_0(x))$ with $\varepsilon>0$ being small, unless certain symmetries are assumed. For more information, refer to \cite{Ding,Godin07,Hou,Hou-Yin} and the references therein.
\end{remark}

\begin{remark} Note that for the short pulse data \eqref{i0}, \eqref{i1} ensures the smallness of $|\nabla\phi|$ for small $\delta>0$ so that \eqref{1.9} is hyperbolic. This is different from the case in \cite{MPY}, where a global smooth solution was established for the 3D semi-linear wave system with certain short pulse data. In \cite{MPY}, the short pulse data are required to satisfy $|(\p_t+\p_r)^k\O^q\p^m\phi|_{t=1}\lesssim\delta^{1/2-m}$ ($\forall k\leq N_0$ for large $N_0$) instead of \eqref{Y-0}.
Note that in general, as pointed out in \cite{Ding4}, the validity of \eqref{Y-0} cannot be expected for large $k$ since $\p_{t,x}^q\phi|_{t=1}=O(\delta^{2-\varepsilon_0/2-|q|})$ holds and \eqref{Y-0} is over-determined for large $k\geq 3$. Additionally, it has been noted that due to the largeness of the integer $N_0$, the problem in \cite{MPY} can be transformed into a small data one within the outermost outgoing cone. Recently, the authors in \cite{Wang} proved the global existence of the relativistic membrane equation using a similar approach to \cite{MPY} with short pulse initial data. It is worth mentioning that the relativistic membrane equation exhibits a special divergence structure and doesn't contain the quadratic terms $g^{\al\beta,\gamma}\p_\g\phi\p^2_{\al\beta}\phi$ whose time decay rate and the power of $\dl$ are worse than that for the cubic terms $h^{\al\beta,\gamma\nu}\p_\g\phi\p_\nu\phi\p^2_{\al\beta}\phi$. However, in this paper, we have to overcome the difficulties of large values of $\phi$ throughout the entire time-space $\mathbb R_+\times\mathbb R^2$ and the appearance of $g^{\al\beta,\g}\p_\g\phi\p_{\al\beta}^2\phi$, thus it seems difficult to modify the methods
 in \cite{MPY} and \cite{Wang} to obtain the global solution for \eqref{quasi-X0} or \eqref{1.9} with more general lower order short pulse initial data \eqref{i0}.
\end{remark}

\begin{remark}
There have been extensive recent works on global well-posedness or finite time blowup of smooth solutions to various nonlinear multi-dimensional wave equations with short pulse initial data. In particular, the authors studied the 3D problem described by equation $-(1+3G''(0)(\p_t\phi)^2)\p_t^2\phi+\Delta \phi =0$ in \cite{MY}. They found shock formation in finite time due to the genuinely nonlinear structure of the equation, which leads to the compression of incoming characteristic conic surfaces. This work serves as another main motivation for this work. Additionally, systematic results for general quasilinear wave equations of the form $-(1+(\p_t\phi)^p)\p_t^2\phi+\Delta \phi =0$ ($p\in\mathbb N$) with short pulse initial data have been obtained in \cite{Ding6,Lu1}, which show that there exists a critical power $p_c=\f{1}{1-\ve_0}$
such that when $p>p_c$, the global solution $\phi$ exists; when $1\le p\le p_c$,
the outgoing shock can be formed before the time $t=2$. On the other hand, by the similar analyses
in this paper together with \cite{Ding6,Lu1}, for the 2D nonlinear
 equation $-(1+(\p_t\phi)^p)\p_t^2\phi+\Delta \phi =0$ ($p\in\mathbb N$ and $p\ge 2$) with short pulse initial data,
 one can derive that  when $p>\t p_c=1+\f{1}{1-\ve_0}$, the global solution $\phi$ exists, otherwise, the solution
 $\phi$ blows up in finite time for $2\le p\le \t p_c$.
\end{remark}

\begin{remark}
Note that for the short pulse data \eqref{i0}, the corresponding initial data $(\rho,v)$ for \eqref{Euler} are small
perturbations of the non-vacuum uniform state $(\bar\rho,0)$. For general large initial data, one cannot expect the
global existence of smooth solution to \eqref{Euler} in general. Indeed, even in 1D case, such a global existence of
smooth solution fails for large data as shown by the following example. Consider the following Cauchy problem of
the 1D compressible isentropic Euler equation for Chaplygin gases
\begin{equation}\label{U-0}
\left\{
\begin{aligned}
&\p_t\rho+\p_x(\rho v)=0,\\
&\p_t(\rho v)+\p_x(\rho v^2+p(\rho))=0,\\
&(\rho,v)(0,x)=\big(1,v_0(x)\big),
\end{aligned}
\right.
\end{equation}
where $p(\rho)=2-\f{1}{\rho}$, $v_0(x)\in C_0^{\infty}(-1,1)$ and $v_0(x)\not\equiv 0$.
In terms of Lagrange coordinate $(s, m)$:
$$s=t, \quad m=\int_{\eta(t,0)}^x\rho(t,y)dy,$$
where $\eta(t,y)$ is the particular path through $y$, \eqref{U-0} becomes
\begin{equation*}
\left\{
\begin{aligned}
&\p_s\rho+\rho^2\p_m v=0,\\
&\p_s v+\p_m p(\rho)=0,\\
&\rho(0,m)=1,\quad v(0,m)=v_0(m).
\end{aligned}
\right.
\end{equation*}
Then the special volume $V=\f1\rho$ solves the following problem
\begin{equation*}
\left\{
\begin{aligned}
&\Box V=(\p_s^2-\p^2_m)V=0,\\
&(V,\p_sV)(0,m)=(1,v'_0(m)),
\end{aligned}
\right.
\end{equation*}
whose unique solution is
\begin{equation}\label{V}
V(s,m)=1+\f12\big(v_0(m+s)-v_0(m-s)\big).
\end{equation}
It follows from \eqref{V} that one can choose $v_0$ such that $v_0(m_2)-v_0(m_1)=-2$ for $m_1,m_2\in(-1,1)$ with $m_1<0<m_2$. Hence for such initial data, there exists a $s^*\leq\f{m_2-m_1}2$ such that
\begin{equation*}
\text{$0<V(s,m)$ for $s<s^*$}, \min V(s^*,m)=0,
\end{equation*}
which implies $\max\rho(s,m)\rightarrow\infty$ as $s\rightarrow s_-^*$, so concentration occurs at $s^*$. Note that even for such
 an example, the Majda's conjecture in \cite{Majda} still holds true since the singularity is the density concentration, not formation of shocks.
 For more results on finite time blow up of smooth solutions to $n$-dimensional quasilinear wave equations with general large
 data (not the short pulse data), we refer to \cite{Sid} and the references therein.
\end{remark}

\begin{remark}\label{1.10}
We now make some brief comments on the analysis of the proof of Theorem \ref{main}, the details are given in the next subsection. Although as for the 4D case in \cite{Ding4}, the main analysis consists of three parts: local estimates, global estimates near and inside the outermost outgoing cone respectively, yet the slower time decay in 2D makes the global estimates extremely  difficult and much more involved for \eqref{1.9} besides the difficulties due to largeness of the solution compared with the case of 4D in \cite{Ding4}. It turns out that both the first and second null conditions are necessary for the global existence results here, which is contrast to 4D case where only the first null condition is required.
The key point in this paper is how to overcome the difficulties resulting from the slow time decay and largeness of the solution simultaneously, and thus new strategies applicable to \eqref{1.9} (including new auxiliary energies) are introduced in this paper.
\end{remark}

\begin{remark}\label{1.10-00}
For the general 2D quasilinear wave equation \eqref{quasi-X0} satisfying the first and second null conditions and the short pulse initial data \eqref{i0} satisfying \eqref{i1}-\eqref{Y-0}, it is expected that the analogous results in Theorem \ref{main} can be proved by combining the ideas developed here for \eqref{1.9} and the framework and estimates for the general 4D quasilinear wave equations in \cite{Ding4}. This will involve non-trivial and lengthy calculations and will be left for future.
\end{remark}

\subsection{Sketch for the proof of Theorem \ref{main}}\label{6}

Let $C_0=\{(t,x):t\geq 1+2\delta,t=r\}$
be the outermost outgoing conic surface; $A_{2\delta}=\{(t,x):t\geq 1+2\delta,0\leq t-r\leq 2\delta\}$ be a domain
containing $C_0$; and $B_{2\delta}=\{(t,x):t\geq 1+\delta,t-r\geq 2\delta\}$ be a conic domain inside $C_0$ with the
lateral boundary given by $\tilde C_{2\delta}=\{(t,x):t\geq 1+2\delta,t-r=2\delta\}$.
At first, as in \cite{Ding4}, one can show the local existence of smooth solution $\phi$ to
\eqref{1.9} with \eqref{i0} and \eqref{i1}-\eqref{Y-0} for $t\in [1, 1+2\delta]$ and derive some basic properties for
the time $t=1+2\delta$ on the domains  $r\in[1-2\delta, 1+2\delta]$ and $r\in [1-3\delta, 1+\delta]$, respectively.
Then the main part of the analysis
consists of two parts: the global estimates of
the solution in $A_{2\delta}$, and the global estimates in $B_{2\delta}$.

\subsubsection{Global estimates of the solution in $A_{2\dl}$}\label{6-2}

As in \cite{Ding4}, motivated by the strategy of D. Christodoulou in \cite{C2},
we start to construct the solution $\phi$ of \eqref{1.9} in $A_{2\dl}$.
However, due to the slower time decay of solutions to the 2-D wave equation,
we need to  introduce some
new auxiliary energies and carry out very elaborate analysis on the related nonlinear forms
by utilizing the distinguished characters of the resulting new
2-D quasilinear wave equations satisfying the first and second null conditions.

As in \cite{C2} or \cite{J}, for a given smooth solution $\phi$ to \eqref{1.9},
one can study the corresponding eikonal equation $\ds\sum_{\al,\beta=0}^2g^{\al\beta}(\p
\phi)\p_\al{u}\p_\beta{u}=0$ with the initial data $u(1+2\delta,x)=1+2\delta-r$, and the inverse foliation density $\mu=
-(\sum_{\al,\beta=0}^2g^{\al\beta}\p_\al u\p_\beta t)^{-1}$.
Under the suitable bootstrap assumptions on the smallness and time decay rate of $\p\phi$ (see $(\star)$ in Section \ref{BA})
and with the help of the null conditions for \eqref{1.9},
$\mu$ satisfies  $\mathring L\mu=O(\delta^{1-\varepsilon_0}t^{-3/2}\mu)$, where  $\mathring L=-\mu \ds\sum_{\al,\beta=0}^2g^{\al\beta}\p_\al u\p_\beta$ is a vector field approximating $\p_t+\p_r$.
Thus $\mu\sim 1$ can be derived.
The positivity of $\mu$ means that the outgoing characteristic conic surfaces never intersect as long as the smooth solution $\phi$
exists. Set $\varphi=(\varphi_0, \varphi_1, \varphi_2):=\p\phi=(\p_0\phi, \p_1\phi, \p_2\phi)$. Then it follows from  \eqref{1.9} that
\begin{equation}\label{Y-1}
\mu\Box_{g}\varphi_\gamma=F_\gamma(\varphi, \p\varphi),\quad\gamma=0,1,2,
\end{equation}
where $g=g_{\al\beta}(\vp)dx^\al dx^\beta$ is the Lorentzian metric,
$(g_{\al\beta}(\vp))$ is the inverse matrix of $(g^{\al\beta}(\vp))$, $\Box_g=\f{1}{\sqrt{|\det g|}}\ds\sum_{\al,\beta=0}^2\p_\al(\sqrt{|\det g|}g^{\al\beta}\p_\beta)$, and $F_\gamma$ are smooth functions
in their arguments. To study the quasilinear wave system \eqref{Y-1}, we first focus on its linearization
\begin{equation}\label{Y-2}
\mu\Box_g\Psi=\Phi.
\end{equation}
As in \cite{Ding4},
it is crucial to derive the global time decay rate of $\Psi$ where $\Psi=\Psi_k=Z^k\varphi$ will be chosen in \eqref{Y-2},
here $Z$ stands for one of some first order vector fields. In this case, by computing
the commutator $[\mu\Box_{g}, Z^k]$, there will appear the quantities containing the $(k+2)$-th order derivatives of $\varphi$
in the expression $\Phi$ of \eqref{Y-2}. For examples, $\slashed\nabla Z^{k-1}\chi$ and $\slashed\nabla^2Z^{k-1}\mu$ will occur in $\Phi$,
where $\chi=g(\mathscr D_{X} \mathring L, X)$ is the second fundamental form
with $\mathscr D$ being the Levi-Civita connection of $g$ and $X=\f{\p}{\p\vartheta}$
($\vartheta$ is the extended local coordinate of $\th\in\mathbb S^1$ which is described by
$\mathring L\vartheta=0$ and $\vartheta|_{t=1+2\dl}=\th$). Following the analogous procedures
in Section 3-Section 11 of \cite{Ding4} and by a much involved analysis for the 2D problem \eqref{1.9}
with \eqref{i0} and \eqref{i1}-\eqref{Y-0}, we can eventually obtain $|\varphi|\lesssim\delta^{1-\varepsilon_0} t^{-1/2}$
and further close the basic bootstrap assumptions. Note that compared with the corresponding  treatments for the
global 4D problem in \cite{Ding4} and the 3D
problem before the shock formation in  \cite{C2,J}, it is more
difficult and complicated to derive the  global
weighted energy estimates for the 2D system \eqref{Y-2} with the suitable time decay rates due to the lower
space dimensions.
Indeed, to obtain the global weighted energy of $\Psi$, one can proceed as usual to
compute $\int_{D^{{\mathfrak t}, u}}\mu\Box_g\Psi(V\Psi)$ through appropriate choices of some first order vector field $V$,
where $D^{{\mathfrak t}, u}=\{(\mathfrak t', u',\vartheta): 1\leq \mathfrak t'<\mathfrak t, 0\leq u'\leq u, 0\le \vartheta\le 2\pi\}$.
We will choose the vector field $V$ as, respectively,
\begin{align*}
&J_1=-\varrho^{2m}\ds\sum_{\al,\kappa,\beta=0}^2g^{\al\kappa}Q_{\kappa\beta}\mathring L^\beta\p_\al,\ J_2=-\ds\sum_{\al,\kappa,\beta=0}^2g^{\al\kappa}Q_{\kappa\beta}\mathring{\underline L}^\beta\p_\al,\\
&J_3=\ds\sum_{\al=0}^2\big(\f12\varrho^{2m-1}\Psi\mathscr D^\al\Psi-\f14\Psi^2\mathscr D^\al(\varrho^{2m-1})\big)\p_\al,
\end{align*}
where $m\in (\f12, \f34)$ is a fixed constant,  $\varrho=\mathfrak t-u$, $\mathring{\underline L}=-\mu(\mathring L
+2\ds\sum_{\nu=0}^2 g^{\nu 0}\p_\nu)$, and $Q$ is
the {\it{energy-momentum tensor field}} of $\Psi$ given as
\begin{equation*}
\begin{split}
Q_{\al\beta}&=(\p_\al\Psi)(\p_\beta\Psi)-\f12 g_{\al\beta}\ds\sum_{\nu,\chi=0}^2g^{\nu\chi}(\p_\nu\Psi)(\p_\chi\Psi).
\end{split}
\end{equation*}

It should be pointed out that the weight $\varrho$ acts as the time ${\mathfrak t}$ in $J_1$ and $J_3$ because of
$\rho\sim {\mathfrak t}$, and
the requirement of $m>\f12$ in $J_1$ is due to the slow time decay rate of solutions to the 2D wave equation
such that some related integrals are
convergent (see \eqref{D22} in Section \ref{ert}, where $\int_{t_0}^{\mathfrak t}\tau^{-1/2-m}d\tau$ will be
uniformly bounded for ${\mathfrak t}$). However,
once $m>\f12$ is chosen, then $\int_{D^{{\mathfrak t}, u}}\mu\Box_g\Psi(V\Psi)$ with $V=J_1$ and $J_2$ will contain the non-negative integral $(m-\f12)\int_{D^{{\mathfrak t},u}}\mu\varrho^{2m-1}|\slashed d\Psi|^2$ which cannot be controlled by
the corresponding energies and fluxes on the left hand side of the resulting inequality. To overcome this crucial difficulty,
different from that in \cite{Ding4}, we will introduce
a new vector field $J_3$ (which is inspired by \cite[(10.18c)]{J}). Thanks to the special structure of $J_3$ and by some technical manipulations,
we can eventually obtain a new term $(m-1)\int_{D^{{\mathfrak t},u}}\mu\varrho^{2m-1}|\slashed d\Psi|^2$
instead of $(m-\f12)\int_{D^{{\mathfrak t},u}}\mu\varrho^{2m-1}|\slashed d\Psi|^2$  on the right hand side of the related energy
inequality, which is non-positive for $m<1$ and thus can be absorbed during the energy estimates. Meanwhile, in the estimate of
$\int_{D^{{\mathfrak t}, u}}\mu\Box_g\Psi(J_3\Psi)$, it is necessary to restrict $m<\f34$ since the resulting integral $\int_{t_0}^{\mathfrak t}\tau^{2m-5/2}d\tau$ is required to be
uniformly bounded (see \eqref{EF1}).

\subsubsection{Global estimates of the solution in $B_{2\delta}$}\label{7}

Inspired by the proof in \cite{Ding4}, the local existence result of equation \eqref{1.9} with \eqref{i0} and \eqref{i1}-\eqref{Y-0} (see Theorem \ref{Th2.1}) implies that for $r\in [1-3\delta, 1+\delta]$,
$\p^{\al}\phi$ on the time $\{t=1+2\delta\}$ admit the smallness of the higher order $\delta^{2-|\al|-\varepsilon_0}$.
In addition, note that the outgoing characteristic cones of $\eqref{1.9}$ starting from $\{t=1+2\delta, 1-2\delta\le r\le 1+\delta\}$
are almost straight and contain $\t C_{2\delta}$. By these properties, we can prove that on $\t C_{2\delta}$, the solution $\phi$ and its derivatives satisfy $|\p^\al\phi|\lesssim\delta^{2-\varepsilon_0}t^{-1/2}$ with the better smallness $O(\delta^{2-\varepsilon_0})$.
Based on such ``good" smallness of $\phi$ on $\t C_{2\delta}$, we will study the
global Goursat problem of \eqref{1.9} in the conic domain $B_{2\dl}$.
To this end, we intend to establish the global weighted spacetime energy estimates for the solution $\phi$
in $B_{2\dl}$ and make use of the modified Klainerman-Sobolev inequality to get the space-time decay rates for $\p^{\al}\phi$,
 which are different from that  in \cite{Ding4} only for the time decay rate.
 In addition, since the spacetime decay is slow in 2D, the energies in \cite{Ding4} are not enough, we then make use of the ghost weight $W=e^{2(1+t-r)^{-1/2}}$ introduced in \cite{A} to establish weighted spacetime energy estimates with the bootstrap assumptions. Based on this, we can obtain
the large space-time behaviors of
the solution $\phi$ up to the forth
order derivatives (see Proposition \ref{P4.1}), for examples, $\p\p^k\phi=O(\delta^{3/2-\varepsilon_0-k})t^{-1/2}$
 and $\p\p^k\phi=O(\delta^{3/2-\varepsilon_0-k})t^{-3/2}(1+t-r)$ with $k=0,1,2,3$ (which imply that $\p^3\phi$ may be large in $B_{2\dl}$).

It is worth mentioning that the slow space-time decay of solutions to the 2D wave equation make it difficult to close the bootstrap assumptions directly since the corresponding higher order
 energies of $\phi$ may grow in time as $t^{2\iota}$ (here $\iota>0$ is some positive constant, see Theorem \ref{T4.1}).
  To overcome this difficulty, we turn to studying the nonlinear equation for the error $\dot\phi=\phi-\phi_a$,
  where $\phi_a$ is the solution to the 2D free wave equation $\square \phi_a=0$ with the initial data $(\phi(1,x), \p_t\phi(1,x))$,
  and obtain the uniformly controllable weighted spacetime energy estimates for $\phi$ by a delicate analysis. In this case, the bootstrap energy
  assumptions of $\phi$ are closed, and then the global estimates of $\phi$
with $|\p\phi|\le C\delta^{3/2-\varepsilon_0}t^{-1/2}$ inside $B_{2\dl}$ are established.

\subsection{Notations}\label{X-1}

Through the whole paper, unless stated otherwise, Greek indices $\{\al, \beta, \cdots\}$, corresponding to the
spacetime coordinates, are chosen in $\{0, 1, 2\}$; Latin indices $\{i, j, k, \cdots\}$, corresponding to the spatial
coordinates, are $\{1, 2\}$; and the Einstein summation convention will be used.
In addition, the convention $f\lesssim g$ means that there exists a generic positive constant $C$ such that $f\le Cg$.

Since \eqref{1.9} is a nonlinear wave equation,
it is natural to introduce the inverse spacetime metric $(g^{\al\beta})$ as follows:
\begin{equation}\label{g-1}
g^{00}=-1,\quad g^{i0}=g^{0i}=-\p_i\phi,\quad g^{ij}=c\delta_{ij}-\p_i\phi\p_j\phi,
\end{equation}
while $(g_{\al\beta})$ represents the corresponding metric:
\begin{equation}\label{g}
g_{00}=-1+c^{-1}|\nabla\phi|^2,\quad g_{i0}=g_{0i}=-c^{-1}\p_i\phi,\quad g_{ij}=c^{-1}\delta_{ij},
\end{equation}
where $c=1+2\p_t\phi+|\nabla\phi|^2$,  and $\delta_{ij}$ is the usual Kronecker symbol.

Finally, the following notations will be used throughout the paper:
\begin{align*}
&\o^i:=\o_i=\f{x^i}{r},\quad i=1,2,\ \o=(\o_1, \o_2),\ t_0=1+2\dl,\\
&L=\p_t+\p_r,\ \underline L=\p_t-\p_r,\ \O=\epsilon_i^jx^i\p_j,\\
&S=t\p_t+r\p_r=\f{t-r}{2}\underline L+\f{t+r}{2}L,\\
&H_i=t\p_i+x^i\p_t=\o^i\big(\f{r-t}{2}\underline L+\f{t+r}{2}L\big)+\f{t\o^i_\perp}{r}\O,\\
&\Sigma_t=\{(t',x):t'=t,x\in\mathbb R^2\},
\end{align*}
where $\epsilon_{1}^2=1$, $\epsilon_{2}^1=-1$, $\epsilon_i^i=0$ and $\o_\perp=(-\o^2, \o^1)$.

\section{Some preliminaries}\label{p-0}

\subsection{Local existence of the smooth solution $\phi$}\label{LE}

In this subsection, we establish the local existence of the smooth solution $\phi$ to \eqref{1.9}
with \eqref{i0} and \eqref{i1}-\eqref{Y-0} for $1\leq t\leq t_0$. Meanwhile, some
key estimates of $\phi(1+2\dl, x)$ on some special spatial domains are derived.
\begin{theorem}\label{Th2.1}
Under the assumptions \eqref{i1} and \eqref{Y-0} on $(\phi_0,\phi_1)$, when $\delta>0$ is suitably small, the
Cauchy problem, \eqref{1.9} with \eqref{i0}, admits a local smooth solution $\phi\in C^\infty([1, t_0]\times\mathbb R^2)$.
Moreover, for $a\in\mathbb N_0$,
$b\in \mathbb N_0$, $q\in\mathbb N_0^3$ and $k\in \mathbb N_0$, it holds that
\begin{enumerate}[(i)]
\item
\begin{align}
&|L^a\p^q\O^k\phi(t_0,x)|\lesssim\delta^{2-|q|-\varepsilon_0},\quad r \in[1-2\delta, 1+2\delta],\label{local1-2}\\
&|\underline L^a\p^q\O^k\phi(t_0,x)|\lesssim\delta^{2-|q|-\varepsilon_0},\quad r\in [1-3\delta, 1+\delta].\label{local2-2}
\end{align}
\item
\begin{equation}\label{local3-2}
|\p^q\O^k\phi(t_0,x)|\lesssim\dl^{3-\ve_0-|q|},
\quad r\in [1-3\delta, 1+\delta].
\end{equation}
\item
\begin{equation}\label{local3-3}
|\underline L^aL^b\O^k\phi(t_0,x)|\lesssim\delta^{2-\varepsilon_0},\quad r\in[1-2\delta,1+\delta].
\end{equation}
\end{enumerate}

\end{theorem}

\begin{proof}
Although the proof is rather analogous to that of Theorem 3.1 in \cite{Ding4}, yet due to
the different structures between the general 4D quasilinear wave
equations satisfying the first null condition in \cite{Ding4} and the 2D quasilinear equation \eqref{1.9} fulfilling
both the first and second null conditions, we still give the details
for the reader's convenience.

Denote by $Z_g$ any fixed vector field in $\{S, H_i, i=1,2\}$. Suppose that for $1\leq t\leq t_0$
and $N_0\in\mathbb N_0$ with $N_0\ge 6$,
\begin{equation}\label{ba}
|\p^q\O^{k}Z_g^{a_0}\phi|\leq\delta^{3/2-|q|}\quad (|q|+k+a_0\leq N_0,\quad a_0\leq 1).
\end{equation}
Define the following energy for \eqref{1.9} and for $n\in\mathbb N_0$,
$$
M_n(t)=\sum_{|q|+k+a_0\leq n}\delta^{2|q|}\|\p\p^q\O^k Z_g^{a_0}\phi(t,\cdot)\|_{L^2(\mathbb R^2)}^2.
$$

Let $w=\delta^{|q|}\p^q\O^k Z_g^{a_0}\phi\quad(|q|+k+a_0\leq 2N_0-2, a_0\leq 1)$. It follows from \eqref{g-1}, \eqref{1.9}
and integration by parts that
\begin{equation}\label{Y-4}
\begin{split}
&\int_1^t\int_{\Sigma_\tau}(\p_twg^{\al\beta}\p_{\al\beta}^2w)(\tau,x) dxd\tau\\
=&\f12\int_{\Sigma_{t}}\big(-(\p_tw)^2-g^{ij}\p_iw\p_jw\big)(t,x)dx-\f12\int_{\Sigma_{1}}\big(-(\p_tw)^2-g^{ij}\p_iw\p_jw\big)(1,x)dx\\
&+\int_1^t\int_{\Sigma_\tau}\big(-\p_ig^{0i}(\p_tw)^2-(\p_ig^{ij})\p_jw\p_tw+\f12(\p_tg^{ij})\p_iw\p_jw\big)(\tau,x) dxd\tau
\end{split}
\end{equation}
with
\begin{equation}\label{w}
\begin{split}
g^{\al\beta}&\p_{\al\beta}^2w=\delta^{|q|}\sum_{\tiny\begin{array}{c}
|q_1|+|q_2|\leq|q|,\ k_1+k_2\leq k,\ a_1+a_2\leq a_0,\\ |q_2|+k_2+a_2<|q|+k+a_0\end{array}}(\p\p^{q_1}\O^{k_1}Z_g^{a_1}\phi)(\p^2\p^{q_2}\O^{k_2}Z_g^{a_2}\phi)\\
&+\delta^{|q|}\sum_{\tiny\begin{array}{c}
|q_1|+|q_2|+|q_3|\leq|q|,\\ k_1+k_2+k_3\leq k,\\
\ a_1+a_2+a_3\leq a_0,\\ |q_3|+k_3+a_3<|q|+k+a_0\end{array}}(\p\p^{q_1}\O^{k_1}Z_g^{a_1}\phi)(\p\p^{q_2}\O^{k_2}Z_g^{a_2}\phi)(\p^2\p^{q_3}\O^{k_3}Z_g^{a_3}\phi),
\end{split}
\end{equation}
where the unnecessary constant coefficients in \eqref{w} have been neglected.

It follows from the assumption \eqref{ba} and \eqref{Y-4} that
\begin{equation}\label{len}
\begin{split}
&\int_{\Sigma_t}\big((\p_tw)^2+|\nabla w|^2\big)(t,x)dx\\
\lesssim &\int_{\Sigma_1}\big((\p_tw)^2+|\nabla w|^2\big)(1,x)dx+\int_1^t\int_{\Sigma_\tau}\delta^{-1/2}\big((\p_tw)^2+|\nabla w|^2\big)(\tau,x)dxd\tau\\
&+\int_1^t\int_{\Sigma_\tau}|(\p_twg^{\al\beta}\p_{\al\beta}^2w)|(\tau,x) dxd\tau.
\end{split}
\end{equation}
Using the bootstrap assumption \eqref{ba} to estimate \eqref{w} and substituting
the resulting estimates into \eqref{len}, one can get from the Gronwall's inequality that for $1\leq t\leq t_0$,
$$
M_{2N_0-3}(t)\lesssim M_{2N_0-3}(1)e^{\delta^{-1/2}(t-1)}\lesssim\delta^{3-2\varepsilon_0}.
$$

Next, it follows from the following Sobolev's imbedding theorem on the circle $\mathbb S^1_r$
(with center at the origin and radius $r$):
$$
|w(t,x)|\lesssim \f{1}{\sqrt{r}}\|\O^{\leq 1}w\|_{L^2(\mathbb S^1_r)},
$$
together with $r\sim 1$ for $t\in [1,t_0]$ and $(t,x)\in \textrm{supp}\ w$ that
\begin{equation}\label{le}
|\p^q\O^k Z_g^{a_0}\phi(t,x)|\lesssim \|\O^{\leq 1}\p^q\O^k Z_g^{a_0}\phi\|_{L^2(\mathbb S^1_r)}\lesssim\delta^{1/2}\|\p\O^{\leq 1}\p^q\O^k Z_g^{a_0}\phi\|_{L^2(\Sigma_{t})}\lesssim\delta^{2-|q|-\varepsilon_0},\quad
\end{equation}
when $|q|+k+a_0\leq N_0$ and $a_0\leq 1$, where $N_0\geq 6$ has been used. Therefore, \eqref{ba} can be closed
for suitably small $\delta>0$ and $\varepsilon_0<\f12$.

This, together with $\ds L=(t+r)^{-1}(S+\o^iH_i)$, yields
\begin{equation}\label{Lle}
|L^{a_0}\p^q\O^k\phi(t,x)|\lesssim|Z_g^{a_0}\p^q\O^k\phi(t,x)|\lesssim\delta^{2-|q|-\varepsilon_0}
\end{equation}
with $|q|+k+a_0\leq N_0$ and $a_0\leq 1$.

\begin{figure}[htbp]
	\centering
	\includegraphics[scale=0.48]{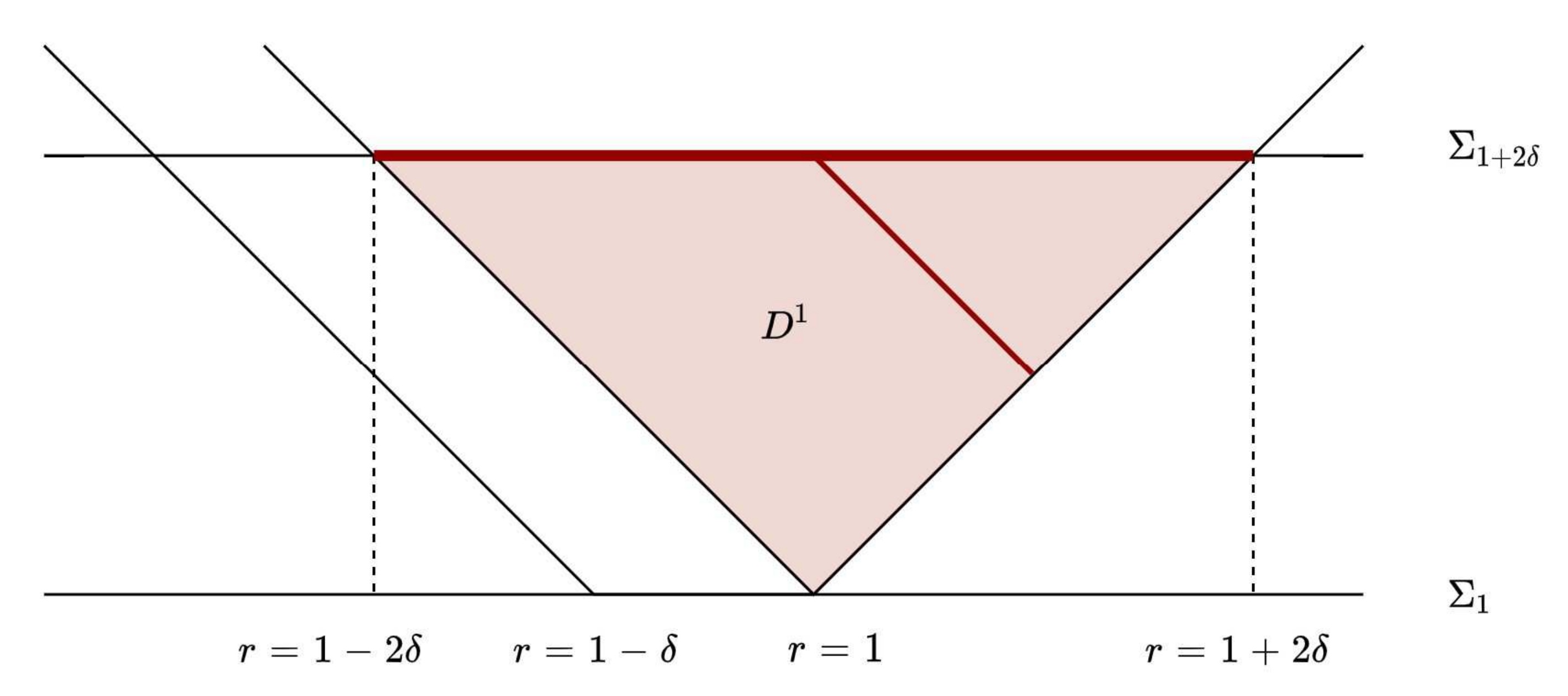}
	\caption{Space-time domain $D^1=\{(t,r): 1\le t\le t_0, 2-t\le r\le t\}$}\label{pic:p1}
\end{figure}

Now we improve the $L^\infty$ estimate of $\phi(t_0,x)$ on some special domains. Rewrite \eqref{1.9} as
\begin{equation}\label{me}
L\underline L\phi=\f{1}{2r}L\phi-\f{1}{2r}\underline L\phi+\f{1}{r^2}\Omega^2\phi+g^{\al\beta,\g}\p_\g\phi\p_{\al\beta}^2\phi+h^{\al\beta,\nu\g}\p_\nu\phi\p_\g\phi\p_{\al\beta}^2\phi.
\end{equation}

Acting the operator $L$ on both sides of \eqref{me} yields an expression of $\underline LL^2\phi$ by $\underline LL=L\underline L$ and direct computations.
It can be checked easily that the worst term in the expression of $\underline LL^2\phi$
is $L(g^{\al\beta,\g}\p_\g\phi\p_{\al\beta}^2\phi)$. Then one can use \eqref{Lle} to get $|\underline L L^2\phi|\lesssim\delta^{1-2\varepsilon_0}$ since \eqref{1.9} satisfies both null conditions. Using this together with
the vanishing property of $\phi$ on $C_0$,
one can integrate $\underline LL^2\phi$ along integral curves of $\underline L$ to show that for $(t,x)\in D^1$
(see Figure \ref{pic:p1}),
\begin{equation}\label{D-1}
|L^2\phi(t,x)|\lesssim\delta^{2-2\varepsilon_0}.
\end{equation}
Similarly, it holds that  for $(t,x)\in D^1$,
\begin{equation}\label{aux}
|L^2\p^q\O^k\phi(t,x)|\lesssim\delta^{2-|q|-2\varepsilon_0}\quad \text{for $|q|+k\leq N_0-4$}.
\end{equation}
Substituting \eqref{Lle} and \eqref{aux} into the expression of $\underline LL^2\phi$ again, and noting that the worst term
in the expression of $\underline LL^2\phi$ becomes $-\f{1}{2r}L\underline L\phi$, one then further gets $|\underline LL^2\phi|\lesssim\delta^{1-\varepsilon_0}$
for $(t,x)\in D^1$ by \eqref{Lle}. Hence the following improved smallness estimate holds:
\[
|L^2\phi|_{D^1}\lesssim\delta^{2-\varepsilon_0}\quad \text{for $1-2\delta\le r\le 1+2\delta$}.
\]
Similar arguments show that
$$|L^2\p^q\O^k\phi|_{D^1}\lesssim\delta^{2-|q|-\varepsilon_0},\quad\text{for}\  |q|+k\leq N_0-5\ \text{and}\ 1-2\delta\leq r\leq 1+2\delta.$$
{Analogously}, an induction argument yields that for $r\in [1-2\delta,1+2\delta]$,
 \begin{equation}\label{oi}
 |LL^a\p^q\O^k\phi|_{D^1}\lesssim\delta^{2-|q|-\varepsilon_0},\quad 3a+|q|+k\leq N_0-2.
 \end{equation}

 \begin{figure}[htbp]
 	\centering
 	\includegraphics[scale=0.42]{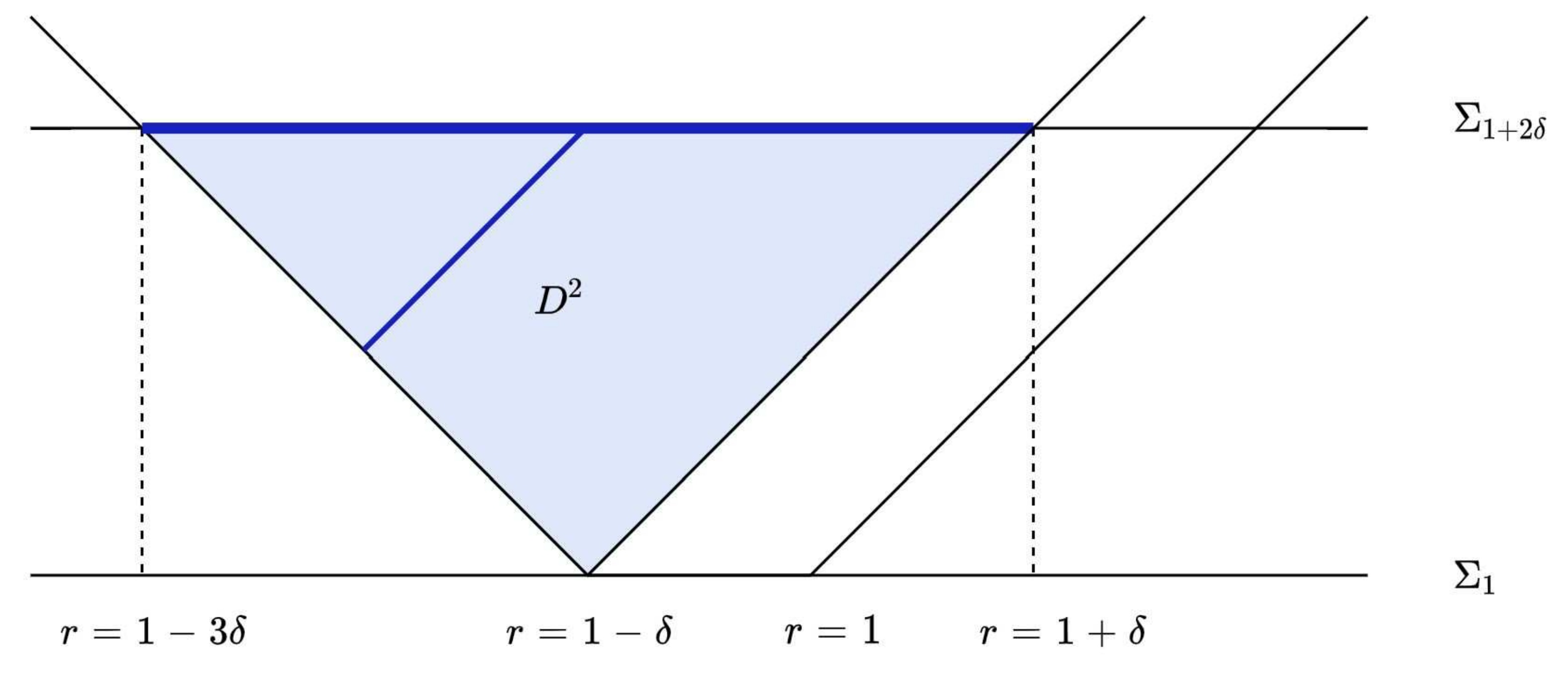}
 	\caption{Space domain for $1-3\delta\le r\le 1+\delta$ on $\Sigma_{1+2\dl}$}\label{pic:p2}
 \end{figure}

Similarly, by the expression of $L\underline L^a\p^q\O^\kappa\phi$, integrating
along integral curves of $L$ yields that for $r\in [1-3\delta,1+\delta]$ (see Figure \ref{pic:p2} below),
 \begin{equation}\label{ii}
 |\underline L^a\p^q\O^k\phi|_{D^2}\lesssim\delta^{2-|q|-\varepsilon_0},\quad 2a+|q|+k\leq N_0-1.
 \end{equation}
 On the other hand, it follows from \eqref{Lle} that
 \begin{equation}\label{Op}
 |\O^k\phi|_{D^2}\lesssim\dl^{3-\ve_0},\ k\leq N_0-1
 \end{equation}
 after integrating $L\O^k\phi$ along integral curves of $L$ in $D^2$.

Therefore, it follows from
 $\p_t=\f12(L+\underline L)$, $\p_i=\f{\o^i}{2}(L-\underline L)+\f{\o^i_{\perp}}{r}\O$,
\eqref{Lle} and \eqref{ii}-\eqref{Op} that for $|q|+k\leq N_0-3$
 and $r\in [1-3\delta, 1+\delta]$,
 \begin{equation}\label{iid}
 |\p^q\O^k\phi(t_0,x)|\lesssim\delta^{3-|q|-\varepsilon_0}.
 \end{equation}

 Finally, we show \eqref{local3-3}. Note that \eqref{iid} implies that on the
 surface $\Sigma_{t_0}$ with $r\in[1-2\delta,1+\delta]$, $|\underline L^{\leq 1}\O^k\phi|\lesssim\delta^{2-\varepsilon_0}$
 and $|L^{\leq 1}\O^k\phi|\lesssim\delta^{2-\varepsilon_0}$ for $k\leq N_0-5$, and
 hence $|L\underline L\O^k\phi|\lesssim\delta^{2-\varepsilon_0}$ by \eqref{me}.
 Furthermore, we claim that
 \begin{equation}\label{LLO}
 |\underline L^aL^b\O^k\phi|\lesssim\delta^{2-\varepsilon_0},\quad\text{for}\ 3a+3b+k\leq N_0-1,
 \end{equation}
which can be proved by induction. Indeed, assume that \eqref{LLO} holds for $a+b\leq n_0$
with $n_0\in \mathbb N_0$ satisfying $3n_0+k\leq N_0-1$. One needs to verify the estimate
in \eqref{LLO} for $3(n_0+1)+k\leq N_0$ and $a+b=n_0$. If $a\ge 1$, by \eqref{me} and the induction assumption, one can get  that
 	\begin{equation*}
 	\begin{split}
 	&|\underline L^aL^{b+1}\O^k\phi|=|\underline L^{a-1}L^b\O^k(L\underline L\phi)|\\
 	\lesssim&\delta^{2-\varepsilon_0}+\delta^{2-\varepsilon_0}|\underline L^{a+1}L^{b}\O^{\leq k}\phi|+\delta^{2-\varepsilon_0}|\underline L^{a-1}L^{b+2}\O^{\leq k}\phi|.
 	\end{split}
 	\end{equation*}
 This together with an induction argument yields
 	\begin{equation}\label{L1}
 	|\underline L^aL^{b+1}\O^k\phi|\lesssim\delta^{2-\varepsilon_0}+\delta^{2-\varepsilon_0}|\underline L^{a+1}L^{b}\O^{\leq k}\phi|.
 	\end{equation}
 	
 If $b\geq 1$, similar to the proof of \eqref{L1}, one has
 	\begin{equation}\label{Lb}
 	|\underline L^{a+1}L^{b}\O^k\phi|\lesssim\delta^{2-\varepsilon_0}+\delta^{2-\varepsilon_0}|\underline L^aL^{b+1}\O^{\leq k}\phi|.
 	\end{equation}

 Combining \eqref{L1} with \eqref{Lb} yields
 \[
 |\underline L^aL^{b+1}\O^k\phi|+|\underline L^{a+1}L^{b}\O^k\phi|\lesssim\delta^{2-\varepsilon_0},
 \]
 which implies \eqref{local3-3}.  Therefore, the proof of Theorem \ref{Th2.1} is finished.
\end{proof}

\subsection{The related geometry and definitions}\label{p}

In this subsection, we list some related geometric facts and definitions, which will be
utilized as basic tools later on. It is assumed that a smooth solution $\phi$ to \eqref{1.9} is given.
The {\it ``optical function"} corresponding to \eqref{1.9} can be defined as in \cite{C2}
(see also Definition 3.4 of \cite{J}).

\begin{definition} For a given smooth solution $\phi$ to the equation \eqref{1.9},
a $C^1$ function ${u}(t,x)$  is called an {\it{optical function}}
if ${u}$ satisfies the eikonal equation
\begin{equation}\label{Y-5}
g^{\al\beta}\p_\al{u}\p_\beta{u}=0.
\end{equation}
\end{definition}
Choose the initial data ${u}(t_0,x)=1+2\delta-r$ and impose the condition $\p_tu>0$
for \eqref{Y-5}.
For a given optical function, the {\it{inverse foliation density $\mu$}} of the
outgoing cones is defined as
\begin{equation}\label{demu}
\mu:=-\f{1}{g^{\al\beta}\p_\al u\p_\beta t}\quad (=-\f{1}{g^{\al 0}\p_\al u}).
\end{equation}
$\mu\ge C>0$ will be shown as long as the smooth solution
$\phi$ to \eqref{1.9} exists. We adopt most of terminologies and definitions in \cite{C2}
(see also \cite{J}).

Note that
$$
\tilde{L}=-g^{\al\beta}\p_\al u\p_\beta
$$
is a tangent vector field for the outgoing light cone $\{ u=C\}$. In addition, $\tilde{L}t=\mu^{-1}$.
Then it is natural to rescale $\tilde{ L}$ as
$$
\mathring{L}=\mu\tilde{L},
$$
which approximates $L=\p_t+\p_r$. To obtain an approximate vector  field of the incoming light cone,
one sets $\tilde T=-g^{\nu 0}\p_\nu-\mathring L$, which is near $-\p_r$ for $t=t_0$. Then in order
to define a null frame, one can set
$$
T=\mu\tilde T,\quad\mathring{\underline L}=\mu\mathring L+2T,
$$
where $\mathring{L}$ and $\mathring{\underline L}$ are two vector fields in the null frame.
Finally, the third vector field
$X$ in the null frame can be constructed by using $\mathring L$. Extending the local coordinate $\theta$ on $\mathbb{S}^1$ as
\begin{equation*}
\left\{
\begin{aligned}
&\mathring L\vartheta=0,\\
&\vartheta|_{t=t_0}=\th.
\end{aligned}
\right.
\end{equation*}
Subsequently, let $X=\f{\p}{\p\vartheta}$. Then $X$ is a tangent vector on $S_{\mathfrak t,u}$. Rewrite $X=X^\al\p_\al$.
Then $X^0=0$ holds due to $\f{\p t}{\p\vartheta}=0$.

\begin{lemma}\label{nullframe}
$\{\mathring L, \mathring{\underline L}, X\}$ constitutes a null frame with respect to the metric $(g_{\al\beta})$, and admits the following identities:
\begin{align}
&g(\mathring L,\mathring L)=g(\mathring{\underline L}, \mathring{\underline L})=g(\mathring L, X)=g(\mathring{\underline L}, X)=0,\\
&g(\mathring L, \mathring{\underline L})=-2\mu.
\end{align}
In addition,
\begin{align}
&\mathring L t=1,\quad\mathring L u=0,\\
&\mathring{\underline L}t=\mu,\quad \mathring{\underline L} u=2.
\end{align}
And
\begin{align}\label{LTTT}
g(\mathring L, T)=-\mu,\quad g(T,T)=\mu^2,
\end{align}
\begin{align}\label{T}
Tt=0,\quad T u=1.
\end{align}
\end{lemma}

As in \cite{J,MY}, one can perform the change of coordinates:
$(t, x^1, x^2)\longrightarrow (\mathfrak t, u, \vartheta)$ near $C_0$ with
\begin{equation}\label{H0-7}
(\mathfrak{t}, u, \vartheta)=(t, u(t,x), \vartheta(t,x)).
\end{equation}
In the new coordinate $(\mathfrak{t}, u, \vartheta)$, the following subsets are introduced (see Figure \ref{pic:p3} below):

\begin{definition} Set
	\begin{align*}
	&\Sigma_{\mathfrak t}^{u}:=\{({\mathfrak t}',u',\vartheta): {\mathfrak t}'=\mathfrak t,0\leq u'\leq  u\},\quad u\in [0, 4\delta],\\
	&C_{u}:=\{({\mathfrak t}',u',\vartheta): {\mathfrak t}'\geq t_0, u'=u\},\\
	&C_{u}^{\mathfrak t}:=\{({\mathfrak t}',u',\vartheta): t_0\leq {\mathfrak t}'\leq {\mathfrak t}, u'=u\},\\
	&S_{{\mathfrak t}, u}:=\Sigma_{\mathfrak t}\cap C_{u},\\
	&D^{{\mathfrak t}, u}:=\{({\mathfrak t}', u',\vartheta): t_0\leq {\mathfrak t}'<{\mathfrak t}, 0\leq u'\leq u\}.
	\end{align*}
\end{definition}

\begin{figure}[htbp]
	\centering
	\includegraphics[scale=0.45]{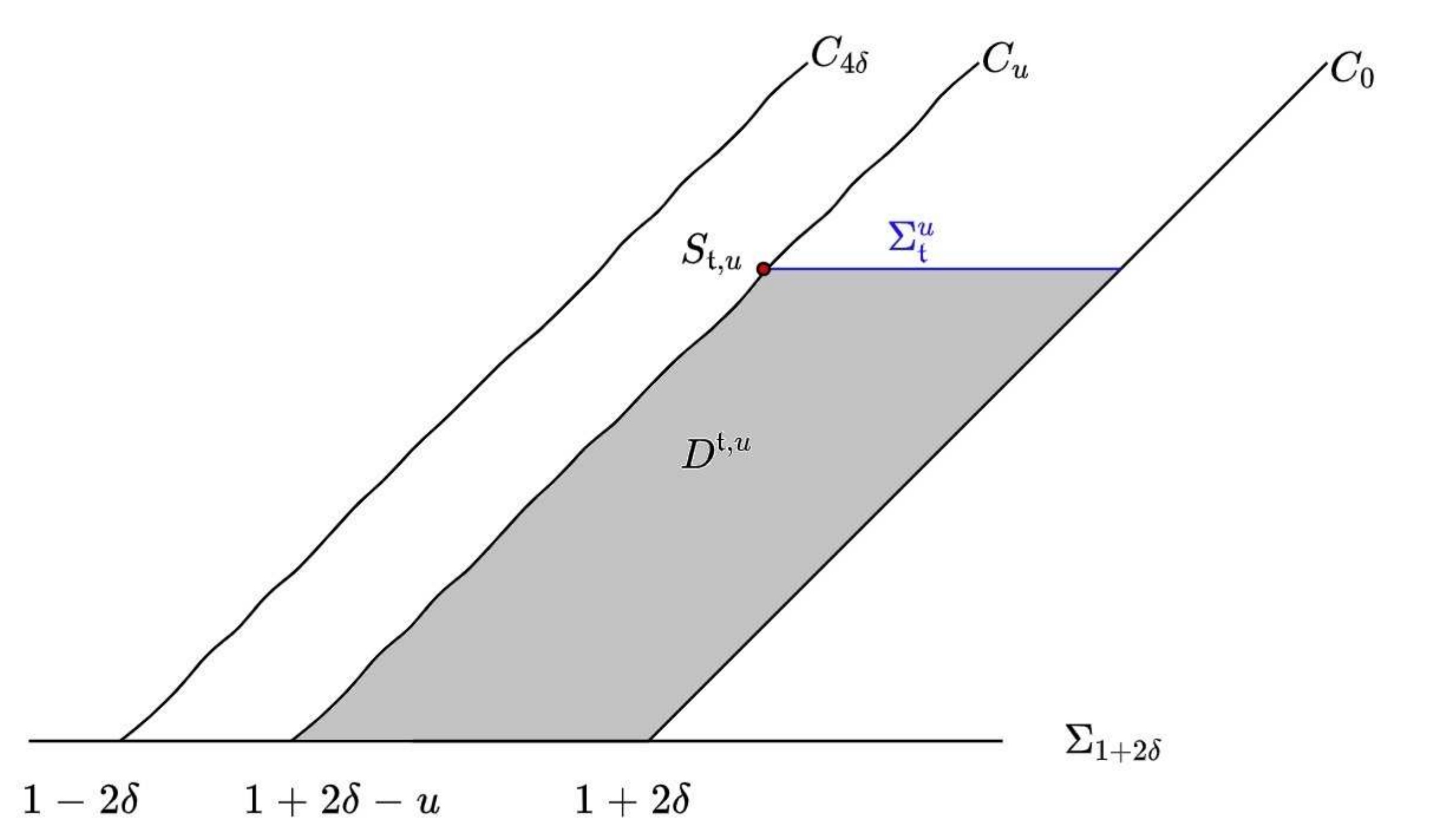}
	\caption{The indications of some domains}\label{pic:p3}
\end{figure}

Next, the following geometric notations which will be used frequently.
\begin{definition} For the metric $g$ on the spacetime,
	\begin{itemize}
		\item $\underline g=(g_{ij})$ is defined as the induced metric of $g$ on $\Sigma_{\mathfrak t}$, i.e.,
$\underline g(U,V)=g(U,V)$ for any tangent vectors $U$ and $V$ of $\Sigma_{\mathfrak t}$;
		\item $\slashed \Pi_\al^\beta:=\delta_\al^\beta-\delta_\al^0\mathring L^\beta+\mathring L_\al\tilde{T}^\beta$ is the
		projection tensor field on $S_{{\mathfrak t},u}$ of type $(1,1)$, where $\delta_\al^\beta$ is the Kronecker delta;
		\item For any $(m,n)$-type
		spacetime tensor field
		$\xi$, $\slashed\xi=\slashed\Pi\xi$ is the tensor field on $S_{{\mathfrak t},u}$  with components $$\slashed\xi^{\al_1\cdots\al_m}_{\beta_1\cdots\beta_n}:=(\slashed\Pi\xi)^{\al_1\cdots\al_m}_{\beta_1\cdots\beta_n}
		=\slashed\Pi_{\beta_1}^{\beta_1'}\cdots\slashed\Pi_{\beta_n}^{\beta_n'}
\slashed\Pi_{\al_1'}^{\al_1}\cdots\slashed\Pi_{\al_m'}^{\al_m}\xi^{\al_1'\cdots\al_m'}_{\beta_1'\cdots\beta_n'}.$$
		In particular, $\slashed g=(\slashed g_{\al\beta})$ is the induced metric of $g$ on $S_{{\mathfrak t},u}$;
		\item $(\slashed g^{XX})$ is defined as the inverse of $\slashed g_{XX}$ with $\slashed g_{XX}=g(X, X)$;
		\item $\mathscr D$ and $\slashed\nabla$ denote the Levi-Civita connections of $g$ and $\slashed g$, respectively;
		\item $\Box_g:=g^{\al\beta}\mathscr{D}^2_{\al\beta}$, $\slashed\triangle:=\slashed g^{XX}\slashed\nabla^2_{X}$;
		\item $\mathcal L_V\xi$ is the Lie derivative of $\xi$ with respect to $V$
and $\slashed{\mathcal L}_V\xi:=\slashed\Pi(\mathcal L_V\xi)$ for any tensor field $\xi$ and vector $V$;
		\item For any $(m,n)$-type spacetime tensor field $\xi$,
		$$
		|\xi|^2:=g_{\al_1\al_1'}\cdots g_{\al_m\al_m'}g^{\beta_1\beta_1'}\cdots g^{\beta_n\beta_n'}\xi_{\beta_1\cdots\beta_n}^{\al_1\cdots\al_m}\xi_{\beta_1'\cdots\beta_n'}^{\al_1'\cdots\al_m'};
		$$
		\item $\text{div}U:=\mathscr D_\al U^\al$ for any vectorfield $U$; $\slashed{\text{div}}Y:=\slashed{\nabla}_X Y^X$ and $\slashed{\text{div}}\kappa:=\slashed\nabla^X\kappa_X$ are angular divergence for any vectorfield $Y$ and one form $\kappa$ on $S_{{\mathfrak t},u}$.
	\end{itemize}
\end{definition}

Under the frame $\{\mathring L, \mathring{\underline L}, X\}$, {\it{the second fundamental forms}} $\chi$
and $\tilde\theta$ can be defined as
\begin{equation}\label{chith}
\chi_{XX}=g(\mathscr D_X \mathring L, X),\quad\tilde\theta_{XX}=g(\mathscr D_X\tilde T,X).
\end{equation}
Define {\it one-form tensors} $\zeta$ and $\xi$ as
\begin{equation}\label{zetaeta}
\zeta_X=g(\mathscr D_X \mathring L,\tilde T),\quad\xi_X=-g(\mathscr D_X T, \mathring L).
\end{equation}
Then $\mu\zeta_X=-X\mu+\xi_X$.
For any vector field $V$, denote its associate {\it{deformation tensor}} by
\begin{equation}\label{dt}
\leftidx{^{(V)}}\pi_{\al\beta}=g(\mathscr{D}_{\al}V,\p_{\beta})+g(\mathscr{D}_{\beta}V,\p_{\al}).
\end{equation}

On the initial hypersurface $\Sigma_{t_0}^{4\delta}$, one has that $\tilde T^i=\ds-\f{x^i}{r}+O(\delta^{1-\varepsilon_0})$,
$\mathring L^0=1$, $\mathring L^i=\ds\f{x^i}{r}+O(\delta^{2-\varepsilon_0})$ and
$\chi_{XX}=\ds\f1r\slashed g_{XX}+O(\delta^{1-\varepsilon_0})$.
Note that on $\Sigma_{t_0}$, $r$ is just $t_0-u$. For $\mathfrak t\geq t_0$, we
define the ``{\it{error vectors}}" with the components being
\begin{equation}\label{errorv}
\begin{split}
&\check{L}^0:=0,\ \check{L}^i:=\mathring L^i-\f{x^i}{\varrho},\ \check{T}^i:=\tilde T^i+\f{x^i}{\varrho},\\
&\check{\chi}_{XX}:=\chi_{XX}-\f{1}{\varrho}\slashed g_{XX},
\end{split}
\end{equation}
where $\varrho:={\mathfrak t}-u$.

Note that $\vartheta$ is the coordinate on $S_{{\mathfrak t}, u}$.
Then in the new coordinate system $({\mathfrak t}, u, \vartheta)$,
one has $\mathring L=\f{\p}{\p \mathfrak t}$. In addition, it follows from \eqref{T} that $T=\f{\p}{\p u}-\eta^X X$
for some smooth function $\eta^X$. And moreover, a similar analysis as for Lemma 3.66 of \cite{J} gives that

\begin{lemma}
In  domain $D^{{\mathfrak t}, u}$, the Jacobian determinant of map $({\mathfrak t}, u, \vartheta)\rightarrow
(x^0, x^1, x^2)$ is
\begin{equation}\label{Jacobian}
\det\f{\p(x^0, x^1, x^2)}{\p({\mathfrak t}, u, \vartheta)}=\mu(\det\underline g)^{-1/2}\sqrt{\slashed g_{XX}}.
\end{equation}
\end{lemma}

\begin{remark}\label{3.1}
\eqref{Jacobian} implies that for regular metrics $\underline g$ and $\slashed g$, i.e., $\det\underline g>0$
and $\slashed g_{XX}>0$, the transformation of coordinates between $({\mathfrak t}, u, \vartheta)$ and $(x^0, x^1, x^2)$
makes sense as long as $\mu>0$.
\end{remark}

On $\mathbb{S}^1$, the standard rotation vector field $\Omega=\epsilon_i^jx^i\p_j$ stands
as a tangential derivative.
In order to project $\Omega$ on $S_{{\mathfrak t}, u}$, as in (3.39b) of \cite{J},
one can denote by
\begin{equation*}
R:=\slashed\Pi\Omega,\quad \slashed d:=\slashed\Pi d
\end{equation*}
the rotation vectorfield and differential of $S_{{\mathfrak t}, u}$, respectively. Then
\begin{equation}\label{R}
\begin{split}
R&=({\slashed\Pi}\Omega)^i\p_i=({\slashed\Pi}^i_j\Omega^j)\p_i=(\delta_j^k-g_{ja}\tilde T^a\tilde T^k){\Omega}^j\p_k
=\Omega-g_{ja}\tilde{T}^a{\Omega}^j\tilde{T}.
\end{split}
\end{equation}
Set
\begin{equation}\label{omega}
\upsilon:=g_{ja}\tilde{T}^a{\Omega}^j=g_{ij}\check T^i\O^j.
\end{equation}
Then one has
$$R=\Omega-\upsilon\tilde{T}.$$

For domains with $\mu>0$, the following integrations and norms
will be utilized repeatedly in subsequent sections.

\begin{definition}
For any continuous function $f$, set
\begin{align*}
&\int_{S_{{\mathfrak t}, u}}f:=\int_{S_{{\mathfrak t}, u}}fd\nu_{\slashed g}:=\int_{\mathbb{S}^1}f(\mathfrak t, u,\vartheta)
\sqrt{\slashed g_{XX}(\mathfrak t, u,\vartheta)}d\vartheta,\qquad\|f\|_{L^2(S_{{\mathfrak t}, u})}^2:=\int_{S_{\mathfrak t, u}}|f|^2,\\
&\int_{ C^{\mathfrak t}_{ u}}f:=\int_{t_0}^{\mathfrak t}\int_{S_{\tau, u}}f(\tau, u,\vartheta) d\nu_{\slashed g}d\tau,\qquad\|f\|_{L^2( C^{\mathfrak t}_{ u})}^2:=\int_{ C^{\mathfrak t}_{ u}}|f|^2,\\
&\int_{\Sigma_{\mathfrak t}^{ u}}f:=\int_0^{ u}\int_{S_{{\mathfrak t}, u'}}f({\mathfrak t}, u',\vartheta) d\nu_{\slashed g}d u',\qquad\|f\|_{L^2(\Sigma_{\mathfrak t}^{u})}^2:=\int_{\Sigma_{\mathfrak t}^{ u}}|f|^2,\\
&\int_{D^{{\mathfrak t}, u}}f:=\int_{t_0}^{\mathfrak t}\int_0^{ u}\int_{S_{\tau, u'}}f(\tau, u',\vartheta)d\nu_{\slashed g}d u'd\tau,\quad\|f\|_{L^2(D^{{\mathfrak t}, u})}^2:=\int_{D^{{\mathfrak t}, u}}|f|^2.
\end{align*}
\end{definition}

For reader's convenience, the notation of contractions is recalled as follows:
\begin{definition}\label{def}
If $\Theta$ is a $(0,2)$-type spacetime tensor, $\kappa$ is a one-form, $U$ and $V$ are vector fields, the contraction
of $\Theta$ with respect to $U$ and $V$ is then defined as
$$
\Theta_{UV}:=\Theta_{\al\beta}U^{\al}V^{\beta},
$$
and the contraction of $\kappa$ with respect to $U$ is
$$
\kappa_U:=\kappa_{\al}U^{\al}.
$$
If $\xi$ is a $(0,2)$-type tensor on $S_{\mathfrak t,u}$, then the trace of $\xi$ is defined as
	\[
	\text{tr}\xi:=\slashed g^{XX}\xi_{XX}.
	\]
\end{definition}

\subsection{Basic equalities in the null frames}\label{be}

In this subsection, some basic equalities in the frame $\{\mathring L, \mathring{\underline L}, X\}$
or $\{T, \mathring{\underline L}, X\}$ will be given. Set
\begin{equation*}
\begin{split}
&\ds G_{\al\beta}^\gamma:=\f{\p g_{\al\beta}}{\p\varphi_\gamma}\quad\text{and}\quad \ds G_{\al\beta}^{\gamma\nu}:=\f{\p G_{\al\beta}^{\gamma}}{\p\varphi_\nu}.
\end{split}
\end{equation*}
For any vector fields $\ds U=U^\al\f{\p}{\p x^{\al}}$ and $\ds V=V^\al\f{\p}{\p x^{\al}}$, as in Definition \ref{def}, one can define $G_{UV}^\gamma:=G_{\al\beta}^\gamma U^\al V^\beta$ and
$G_{UV}^{\gamma\nu}:=G_{\al\beta}^{\gamma\nu}U^\al V^\beta$. Direct computations yield
\begin{equation}\label{G}
\begin{split}
&G_{\mathring L\mathring L}^\gamma=-2c^{-1}\mathring L^\gamma,\quad G_{\mathring L\tilde{T}}^0=2c^{-1},\quad G_{\mathring L\tilde{T}}^a=2c^{-1}\varphi_a-c^{-1}\tilde T^a,\\
&G_{\tilde{T}\tilde{T}}^0=-2c^{-1}, \quad G_{\tilde{T}\tilde{T}}^a=-2c^{-1}\varphi_a,\quad G_{\mathring LX}^0=0,\quad G_{\mathring LX}^a=-c^{-1}\slashed d_Xx^a,\\
& G_{\tilde{T}X}^\gamma=0,\quad G_{XX}^0=-2c^{-1}\slashed g_{XX},\quad G_{XX}^a=-2c^{-1}\varphi_a\slashed g_{XX}
\end{split}
\end{equation}
and
\begin{equation}\label{G1}
\begin{split}
&G_{\mathring L\mathring L}^{00}=8c^{-2},\quad G_{\mathring L\mathring{L}}^{0a}=4c^{-2}(\varphi_a+\mathring L^a),\quad G_{\mathring{L}\mathring L}^{ab}=4c^{-2}(\varphi_a\mathring L^b+\varphi_b\mathring L^a),\\
&G_{XX}^{00}=8c^{-2}\slashed g_{XX},\quad G_{XX}^{0a}=8c^{-2}\varphi_a\slashed g_{XX},\quad G_{XX}^{ab}=2c^{-2}(4\varphi_a\varphi_b-c\delta_{ab})\slashed g_{XX},\\
&G_{\mathring LX}^{00}=0,\quad G_{\mathring LX}^{0a}=2c^{-2}\slashed d_Xx^a,\quad G_{\mathring LX}^{ab}=2c^{-2}(\varphi_a\slashed d_Xx^b+\varphi_a\slashed d_Xx^a),\\
&G_{\mathring L\tilde T}^{00}=-8c^{-2},\quad G_{\mathring L\tilde T}^{0a}=2c^{-2}(\tilde T^a-4\varphi_a),\quad G_{\mathring L\tilde T}^{ab}=2c^{-2}(c\delta_{ab}-4\varphi_a\varphi_b+\varphi_a\tilde T^b+\varphi_b\tilde T^a),\\
&G_{\tilde T\tilde T}^{00}=8c^{-2},\quad G_{\tilde T\tilde T}^{0a}=8c^{-2}\varphi_a,\quad G_{\tilde T\tilde T}^{ab}=2c^{-2}(4\varphi_a\varphi_b-c\delta_{ab}),\quad G_{\tilde TX}^{\al\beta}=0,
\end{split}
\end{equation}
where $\slashed d_Xf=X^i{\slashed d}_if$ and ${\slashed d}_if={\slashed\Pi}_i^{\al}(d_{\al}f)={\slashed\Pi}_i^{\al}(\p_{\al}f)$
for any smooth function $f$.

It follows from \eqref{G} that the following lemma holds.

\begin{lemma}\label{mu}
 $\mu$ satisfies
 \begin{equation}\label{lmu}
 \begin{split}
 \mathring L\mu=c^{-1}\mu\big(\mathring L^\al\mathring L\varphi_\al-\mathring Lc\big).
 \end{split}
 \end{equation}
\end{lemma}

\begin{proof}
The proof is exactly same as those in \cite[Lemma 4.3]{Ding4}, one can also take the explicit expression \eqref{G}
into \cite[(4.18)]{Ding4} to get \eqref{lmu}.
\end{proof}

\begin{remark}\label{4.1}
As in \cite{Ding4}, the importance of the expression \eqref{lmu} should be emphasized here. Due to both the null  conditions and the special
structures of equation \eqref{1.9},
$\mathring L\mu$ is just a combination of $\mathring L\varphi_\al$ (note that $\mathring L c=2\mathring L\varphi_0
+2\ds\sum_{i=1}^2\varphi_i\mathring L\varphi_i$). Based on this and the smallness and some suitable time decay rate
of $\mathring L\varphi_\al$, we will be able to show that $\mu\geq C$ for some positive constant $C$ (see \eqref{phimu}
in Section \ref{BA}). This is in contrast to the cases in \cite{C2, LS, MY, J}, where $\mathring L\mu$
contains the factor $T\varphi_\al$ which leads to $\mu\rightarrow 0+$ in finite time (e.g. \cite[(2.36)]{MY}).
\end{remark}

Note that the deformation tensor defined in \eqref{dt}
will occur in the subsequent energy estimates. It is necessary to check the components
of $\leftidx{^{(V)}}\pi$ in the null frame $\{\mathring L,\underline{\mathring L},X\}$.

Let $\leftidx{^{(V)}}{\slashed\pi}_{UX}=\leftidx{^{(V)}}{\pi}_{UX}$ for $U\in\{\mathring L,\underline{\mathring L},X\}$.
Following the computations in \cite[Proposition 7.7]{J} or \cite[(4.21)-(4.23)]{Ding4}, one has

\begin{enumerate}[(1)]
\item for $V=T$,
\begin{equation}\label{Lpi}
\begin{split}
&\leftidx{^{(T)}}\pi_{\mathring L\mathring L}=0,\quad \leftidx{^{(T)}}\pi_{T\tilde T}=2T\mu,\quad
\leftidx{^{(T)}}\pi_{\mathring LT}=-T\mu,\quad
\leftidx{^{(T)}}{\slashed\pi}_{TX}=0,\\
&\leftidx{^{(T)}}{\slashed\pi}_{\mathring LX}=-\slashed d_X\mu-2c^{-1}\mu\tilde T^a\slashed d_X\varphi_a,\quad \leftidx{^{(T)}}{\slashed\pi}_{XX}=2\mu\tilde\theta_{XX};
\end{split}
\end{equation}

\item for $V=\mathring L$,
\begin{equation}\label{uLpi}
\begin{split}
&\leftidx{^{(\mathring L)}}\pi_{\mathring L\mathring L}=0,\quad \leftidx{^{(\mathring L)}}\pi_{T\tilde T}=2\mathring L\mu,\quad \leftidx{^{(\mathring L)}}\pi_{\mathring LT}=-\mathring L\mu,\quad \leftidx{^{(\mathring L)}}{\slashed\pi}_{\mathring LX}=0,\\
&\leftidx{^{(\mathring L)}}{\slashed\pi}_{TX}=\slashed d_X\mu+2c^{-1}\mu\tilde T^a\slashed d_X\varphi_a,\quad \leftidx{^{(\mathring L)}}{\slashed\pi}_{XX}=2\chi_{XX};
\end{split}
\end{equation}
\item for $V=R$,
\begin{equation}\label{Rpi}
\begin{split}
&\leftidx{^{(R)}}\pi_{\mathring L\mathring L}=0,\quad \leftidx{^{(R)}}\pi_{T\tilde T}=2R\mu,\quad \leftidx{^{(R)}}\pi_{\mathring LT}=-R\mu,\\
&\leftidx{^{(R)}}{\slashed\pi}_{\mathring LX}=-{R}^X\check{\chi}_{XX}+g_{aj}\epsilon_{i}^j\check L^i\slashed d_Xx^a+\upsilon c^{-1}\tilde T^a\slashed d_X\varphi_a-\f12c^{-1}R^X\slashed g_{XX}(\mathring Lc)-\f12c^{-1}\upsilon\slashed d_Xc\\
&\leftidx{^{(R)}}{\slashed\pi}_{TX}=\mu{R}^X\check{\chi}_{XX}+\upsilon\slashed d_X\mu+\mu g_{aj}\epsilon_{i}^j\check T^i\slashed d_Xx^a+\f12c^{-1}\mu R^X\slashed g_{XX}(\mathring Lc)-c^{-1}\mu (R\varphi_a)\slashed d_Xx^a\\
&\qquad\quad\quad+\f12c^{-1}\mu\upsilon\slashed d_Xc,\\
&\leftidx{^{(R)}}{\slashed\pi}_{XX}=2\upsilon\chi_{XX}+c^{-1}\upsilon(\mathring Lc)\slashed g_{XX}-2c^{-1}\upsilon(\slashed d_Xx^a)\slashed  d_X\varphi_a-c^{-1}(Rc)\slashed g_{XX}.
\end{split}
\end{equation}
\end{enumerate}

As seen in \eqref{lmu} and  \eqref{Lpi}-\eqref{Rpi},
the components of $\mathring L$ and $\t T$ appear frequently. In view of $\tilde T^i=\varphi_i-\mathring L^i$,
one can find the equations for $\mathring L^i$ and $\check L^i$ under the null frame $\{T,\mathring L, X\}$  as follows.

\begin{lemma}\label{H-4}
It holds that
\begin{align}
&\mathring L\mathring L^i=-c^{-1}\mathring L^\al(\mathring L\varphi_\al)\tilde T^i,\label{LL}\\
&\mathring L\big(\varrho\check L^i\big)=\varrho\mathring L\mathring L^i
=-c^{-1}\varrho\mathring L^\al(\mathring L\varphi_\al)\tilde T^i,\label{LeL}\\
&\slashed d_X\mathring L^i=(\textrm{tr}\chi){\slashed d}_Xx^i+\f12c^{-1}(\mathring Lc)\slashed d_Xx^i-\f12c^{-1}(\slashed d_Xc)\tilde{T}^i+c^{-1}\tilde{T}^a(\slashed d_X\varphi_a)\tilde{T}^i,\label{dL}\\
&\slashed d_X\check{L}^i=(\textrm{tr}\check\chi){\slashed d}_Xx^i+\f12c^{-1}(\mathring Lc)\slashed d_Xx^i-\f12c^{-1}(\slashed d_Xc)\tilde{T}^i+c^{-1}\tilde{T}^a(\slashed d_X\varphi_a)\tilde{T}^i,\label{deL}
\end{align}
where $\slashed d^X f={\slashed g}^{XX}{\slashed d}_Xf={\slashed g}^{XX}(Xf)$ for any smooth function $f$.
\end{lemma}

\begin{proof}
This follows directly from \eqref{G} as for the proof of \cite[Proposition 4.7]{J}.
\end{proof}

In addition, it follows from \eqref{chith}-\eqref{zetaeta}, \eqref{G} and Lemma \ref{H-4}
that
\begin{align}
&\zeta_X=c^{-1}\tilde T^a(\slashed d_X\varphi_a),\qquad\xi_X=c^{-1}\mu\tilde T^a(\slashed d_X\varphi_a)+\slashed d_X\mu,\label{zeta}\\
&\tilde\theta_{XX}=-\f12c^{-1}\mu^{-1}(Tc)\slashed g_{XX}-\f12c^{-1}(\mathring Lc)\slashed g_{XX}+c^{-1}(\slashed d_X\varphi_a)\slashed d_Xx^a-\chi_{XX},\label{theta}
\end{align}
which have been given in \cite[Lemma 5.1]{J} in terms of $G_{XY}$.
Then \eqref{zeta} implies that
\begin{equation}\label{TL}
\begin{split}
&T\mathring{L}^i=-c^{-1}(\tilde{T}^a\mathring L\varphi_a)T^i+(\slashed d_X\mu)(\slashed d^Xx^i)
+c^{-1}\mu(\tilde{T}^a\slashed d_X\varphi_a)\slashed d^Xx^i+\f12c^{-1}\mu(\slashed d^Xc)(\slashed d_Xx^i).
\end{split}
\end{equation}

For later analysis, one also needs the following connection coefficients in the new frames, which are given
in Lemmas 5.1 and 5.3 of \cite{J}.

\begin{lemma}\label{cd}
The covariant derivatives in the frame $\{T, \mathring L, X\}$ are
\begin{equation}\label{cdf}
\begin{split}
&\mathscr D_{\mathring L}\mathring L=(\mu^{-1}\mathring L\mu)\mathring L,\quad\mathscr D_{T}\mathring L=-\mathring L\mu\mathring L+\xi^XX,\quad\mathscr D_X\mathring L=-\zeta_X\mathring L+{\textrm{tr}\chi}X,\\
&\mathscr D_{\mathring L}T=-\mathring L\mu\mathring L-\mu\zeta^XX,\quad\mathscr D_TT=\mu\mathring L\mu\mathring L+(\mu^{-1}T\mu+\mathring L\mu)T-\mu(\slashed d^X\mu) X,\\
&\mathscr D_XT=\mu\zeta_X\mathring L+\mu^{-1}\xi_XT+(\mu\textrm{tr}\tilde\theta)X,\\
&\mathscr D_XX=\slashed{\nabla}_XX+(\tilde\theta_{XX}+\chi_{XX})\mathring L+\mu^{-1}\chi_{XX}T.
\end{split}
\end{equation}
\end{lemma}

\begin{lemma}
The covariant derivatives in the frame $\{\mathring{\underline L}, \mathring L, X\}$ are
\begin{equation}\label{LuL}
\begin{split}
&\mathscr D_{\mathring {\underline L}}{\mathring L}=-\mathring L\mu\mathring L+2\xi^XX,\quad\mathscr D_{\mathring L}\mathring{\underline L}=-2\mu\zeta^XX,\\
&\mathscr D_{\mathring{\underline L}}\mathring{\underline L}=(\mu^{-1}\mathring{\underline L}\mu+\mathring
 L\mu)\mathring{\underline L}-(2\mu \slashed d^X\mu) X.
\end{split}
 \end{equation}
\end{lemma}

Next we derive the equation for $\varphi_\gamma$ under the action of the covariant wave operator $\Box_g$.
With the metric $g$ given in \eqref{g}, one has
\begin{equation}\label{box}
\begin{split}
\Box_g\varphi_\gamma=&\f{1}{\sqrt{|\text{det}g|}}\p_\al\big(\sqrt{|\text{det}g|}g^{\al\beta}\p_\beta\varphi_\gamma\big)\\
=&-c^{-1}g^{\al\beta}\p_\al c(\p_\beta\varphi_\gamma)+g^{\al\beta}\p_{\al\beta}^2\varphi_\gamma\\
&-\sum_i\p_i\varphi_i(\p_t+\sum_j\varphi_j\p_j)\varphi_\gamma+\sum_i\p_i\varphi_\gamma(\p_t+\sum_j\varphi_j\p_j)\varphi_i.
\end{split}
\end{equation}
Differentiating \eqref{1.9} with respect to the variable $x^\gamma$ yields
\begin{equation}\label{gphi}
\begin{split}
g^{\al\beta}\p_{\al\beta}^2\varphi_\gamma=2\sum_i\p_i\varphi_\gamma(\p_t+\sum_j\varphi_j\p_j)\varphi_i
-2\sum_i\p_i\varphi_i(\p_t+\sum_j\varphi_j\p_j)\varphi_\gamma.
\end{split}
\end{equation}
Substituting \eqref{gphi} into \eqref{box} leads to
\begin{equation}\label{boxg}
\begin{split}
\Box_g\varphi_\gamma=&-c^{-1}g^{\al\beta}(\p_\al c)(\p_\beta\varphi_\gamma)\\
&-3\sum_i\p_i\varphi_i(\p_t+\sum_j\varphi_j\p_j)\varphi_\gamma+3\sum_i\p_i\varphi_\gamma(\p_t+\sum_j\varphi_j\p_j)\varphi_i.
\end{split}
\end{equation}
Note that
\begin{align}
&g^{\al\beta}=-\mathring L^\al\mathring L^\beta-\tilde T^\al\mathring L^\beta-\mathring L^\al\tilde T^\beta+(\slashed d_Xx^\al)(\slashed d^Xx^\beta),\label{gab}\\
&\p_i=c^{-1}\mu^{-1}\tilde T^i T+c^{-1}(\slashed d^Xx^i)X,\\
&\p_t+\sum_j\varphi_j\p_j=\mathring L+\mu^{-1}T.
\end{align}
Then in the frame $\{T,\mathring L,X\}$, \eqref{boxg} can be rewritten as
\begin{equation}\label{ge}
\begin{split}
\mu\Box_g \varphi_\gamma=&c^{-1}\big\{\f12\mu\mathring Lc+Tc-3\tilde T^iT\varphi_i-\f32\mu\tilde T^i\mathring L\varphi_i-\f32\mu\slashed d^Xx^i(\slashed d_X\varphi_i)\big\}\mathring L\varphi_\gamma\\
&+\f12c^{-1}\big\{\mathring Lc+3\tilde T^i\mathring L\varphi_i-3\slashed d^Xx^i(\slashed d_X\varphi_i)\big\}\mathring{\underline L}\varphi_\gamma\\
&+c^{-1}\big\{-\mu\slashed d^Xc+3\mu(\mathring L\varphi_i)\slashed d^Xx^i+3(T\varphi_i)\slashed d^Xx^i\big\}\slashed d_X\varphi_\gamma,
\end{split}
\end{equation}
where $\gamma=0,1,2.$

For later use, some identities involving commutators are given in the following lemmas,  which are proved in
Lemmas 4.10, 8.9 and 8.11 of \cite{J} for the 3D case.

\begin{lemma}\label{com}
In the frame $\{\mathring L,T,R\}$, it holds that
\begin{equation}\label{c}
\begin{split}
[\mathring L, R]&={\leftidx{^{(R)}}{\slashed\pi}_{\mathring L}}^X X,\ \ [\mathring L, T]={\leftidx{^{(T)}}{\slashed\pi}_{\mathring L}}^X X,\ \  [T,R]={\leftidx{^{(R)}}{\slashed\pi}_{T}}^X X,
\end{split}
\end{equation}
where ${\leftidx{^{(R)}}{\slashed\pi}_{\mathring L}}^X={\slashed g}^{XX} {\leftidx{^{(R)}}{\slashed\pi}_{\mathring L X}}$
and ${\leftidx{^{(R)}}{\slashed\pi}_{\mathring L X}}={\leftidx{^{(R)}}{\pi}_{\mathring L X}}=
{\leftidx{^{(R)}}{\pi}_{\al\beta}{\mathring L}^{\al} X^{\beta}}$.
\end{lemma}

\begin{lemma}\label{commute}
For any vector field $Z\in\{\mathring L,T,R\}$,
\begin{enumerate}
 \item
if $f$ is a smooth function, then
 \[
 \big([\slashed\nabla^2,\slashed{\mathcal L}_Z]f\big)_{XX}=\f12{\slashed\nabla}_X\big(\textrm{tr}\leftidx{^{(Z)}}{\slashed\pi}\big)\slashed d_Xf;
 \]

\item if $\Theta$ is a one-form on $S_{\mathfrak t, u}$, then
\[
([\slashed\nabla_X,\slashed{\mathcal L}_Z]\Theta)_X=\f12{\slashed\nabla}_X\big(\textrm{tr}\leftidx{^{(Z)}}{\slashed\pi}\big)\Theta_X;
\]
\item if  $\Theta$ is a $(0,2)$-type tensor on $S_{\mathfrak t, u}$, then
\begin{align*}
([\slashed\nabla_X,\slashed{\mathcal L}_Z]\Theta)_{XX}&={\slashed\nabla}_X\big(\textrm{tr}\leftidx{^{(Z)}}{\slashed\pi}\big)\Theta_{XX},\\
([\slashed\nabla_X,\slashed{\mathcal L}_Z]\slashed\nabla\Theta)_{XXX}&=\f32{\slashed\nabla}_X\big(\textrm{tr}\leftidx{^{(Z)}}{\slashed\pi}\big)\slashed\nabla_X\Theta_{XX},
\end{align*}
\end{enumerate}
where ${\leftidx{^{(Z)}}{\slashed\pi}_{XX}}={\leftidx{^{(Z)}}{\pi}_{XX}}$.
\end{lemma}

\section{Bootstrap assumptions on $\p\phi$  near $C_0$ and some related estimates }\label{BA}

We will derive the global estimates of a solution to \eqref{1.9} near $C_0$ by a bootstrap argument. To this end, for any given smooth solution $\phi$ to \eqref{1.9}, we assume that in $D^{{\mathfrak t},u}\subset A_{2\dl}$:
\begin{equation*}
\begin{split}
&\delta^l\|(\mathfrak t\mathring LZ^\al \varphi_\gamma, \mathfrak t\slashed dZ^\al\varphi_\gamma, \dl\mathring{\underline L}Z^\al\varphi_\gamma)\|_{L^\infty(\Sigma_{\mathfrak t}^{u})}\leq M\delta^{1-\varepsilon_0} \mathfrak t^{-1/2},\\
&\|(\mathfrak t^2\slashed\nabla^2\varphi_\gamma, \varphi_\gamma)\|_{L^\infty(\Sigma_{\mathfrak t}^{u})}\leq M\delta^{1-\varepsilon_0} \mathfrak t^{-1/2},
\end{split}\tag{$\star$}
\end{equation*}
where $|\al|\leq N$,  $N$ is a fixed large positive integer, $M$ is a positive constant to be chosen later
(at least double bounds of the corresponding quantities on time $t_0$),
$Z\in\{\varrho\mathring L,T,R\}$, and $l$ is the number of $T$ included in $Z^\al$.

We start with a rough estimate of $\mu$ under assumptions ($\star$). Note that
$1=g_{ij}\tilde T^i\tilde T^j=\big(1+O(M\delta^{1-\varepsilon_0} \mathfrak t^{-1/2})\big)\ds\sum_{i=1}^2|\tilde T^i|^2$
by \eqref{LTTT}.  This means
\begin{equation}\label{L}
|\tilde T^i|, |\mathring L^i|\leq 1+O(M\delta^{1-\varepsilon_0} \mathfrak t^{-1/2})
\end{equation}
due to $\mathring L^i=\varphi_i-\tilde T^i$. This, together with $(\star)$ and \eqref{lmu},  implies $|\mathring L\mu|\lesssim M\delta^{1-\varepsilon_0}\mathfrak t^{-3/2}\mu$. When $\delta>0$ is small, by
integrating $\mathring L\mu$ along integral curves of $\mathring L$ and noting $\mu=\f{1}{\sqrt{c}}
=1+O(\delta^{1-\varepsilon_0})$ on $\Sigma_{t_0}$,
one can get directly that
\begin{equation}\label{phimu}
\mu=1+O(M\delta^{1-\varepsilon_0}).
\end{equation}

To improve the estimates and close the assumptions ($\star$), one may rewrite \eqref{ge} in the frame
$\{\mathring{\underline L}, \mathring L, X\}$ as
\begin{equation}\label{fequation}
\mathring L\mathring{\underline L}\varphi_\gamma+\f{1}{2\varrho}\mathring{\underline L}\varphi_\gamma=H_\gamma,
\end{equation}
due to the fact that $\mu\Box_g\varphi_\g=-\mathscr D_{\mathring L\mathring {\underline L}}^2\varphi_\g+\mu\slashed g^{XX}\mathscr D_{X}^2\varphi_\g=\mathring L\mathring{\underline L}\varphi_\g-2\mu\zeta^X\slashed d_X\varphi_\g+\mu\slashed\triangle\varphi_\g-\mu(\textrm{tr}\tilde\theta+\textrm{tr}\chi)\mathring L\varphi_\g-\textrm{tr}\chi T\varphi_\g$
by \eqref{cdf} and \eqref{LuL}. In addition, by \eqref{zeta}-\eqref{theta},
\begin{equation}\label{H}
\begin{split}
H_\gamma=&\mu\slashed\triangle\varphi_\gamma+\big\{-\f12c^{-1}Tc+3c^{-1}\tilde{T}^aT\varphi_a+\f32c^{-1}\mu \tilde{T}^a\mathring L\varphi_a+\f12c^{-1}\mu(\slashed d^Xx^a)\slashed d_X\varphi_a\\
&+\f12\mu\textrm{tr}\check{\chi}+\f{\mu}{2\varrho}\big\}\mathring L\varphi_\gamma+\big\{-\f12\textrm{tr}\check{\chi}-\f12c^{-1}\mathring Lc-\f32c^{-1}\tilde{T}^a\mathring L\varphi_a+\f32c^{-1}(\slashed d^Xx^a)\slashed d_X\varphi_a\big\}\mathring{\underline L}\varphi_\gamma\\
&-c^{-1}\mu\big\{3\tilde{T}^a
+\mathring L^\al\big\}(\slashed d^X\varphi_\al)\slashed d_X\varphi_\gamma.
\end{split}
\end{equation}
Note from the expression of $H_\g$ that the terms containing
factors $T\varphi_\al$ or ${\mathring{\underline L}}\varphi_\al$,
which are not small ``enough" and have slow decay rate in time, appear always with some accompanying factors $\mathring L\varphi_\g$ or $\textrm{tr}\check\chi$ with the ``good"
smallness and fast time decay rate (see $-\f12c^{-1}(Tc)(\mathring L\varphi_\g)$ and
$-\f12\textrm{tr}\check\chi(\mathring {\underline L}\varphi_\g)$ in $H_\gamma$). This implies that $H_\g$ may possess
some desired properties helpful for our analysis later.

Unless stated otherwise, from now on until Section \ref{ert}, the pointwise estimates for the corresponding quantities
are all carried out inside the domain $D^{{\mathfrak t},u}$.

To estimate $\mathring{\underline L}\varphi_\gamma$ by integrating  \eqref{fequation}
along integral curves of $\mathring L$, one needs to estimate $\slashed dx^i$, $\check{\chi}$ and so on in
$H_{\g}$ first.

Note that $|\slashed dx^i|^2=\slashed g^{ab}\slashed d_ax^i\slashed d_bx^i=c-(\tilde T^i)^2$.
Then $(\star)$ and \eqref{L} imply that
\begin{equation}\label{dx}
|\slashed dx^i|\lesssim 1.
\end{equation}

To estimate $\check\chi$, as in \cite{Ding4}, one can derive the governing equations for $\check\chi$ first. Exactly by the
same procedure of proof in \cite[Lemma 5.2]{Ding4}, we have

\begin{lemma}\label{LTchi}
The second fundamental form $\chi$ and its ``error" form $\check\chi$, defined in \eqref{chith} and \eqref{errorv}
respectively, satisfy the following structure equations:
\begin{equation}\label{Lchi}
\begin{split}
\mathring L\chi_{XX}&=c^{-1}\slashed d_Xx^a(\slashed d_X\mathring L\varphi_a)
-c^{-1}\slashed g_{XX}\big(\mathring L^2\varphi_0+\varphi_a\mathring L^2\varphi_a\big)-c^{-1}\mathring L^\gamma(\slashed\nabla_{X}^2\varphi_\gamma)+(\textrm{tr}\chi)\chi_{XX}\\
&-c^{-1}\textrm{tr}\chi(\slashed d_Xx^a)\slashed d_X\varphi_a-\f12c^{-1}(\mathring Lc)\chi_{XX}
+c^{-2}f(\mathring L^1,\mathring L^2,\varphi)\left(
\begin{array}{ccc}
(\slashed d_X\varphi)(\slashed d_X\varphi)\\
\slashed g^{XX}(\slashed d_Xx^a\slashed d_X\varphi_a)^2\\
\slashed g^{XX}(\mathring L\varphi)(\mathring L\varphi)\\
(\slashed d_Xx^a\slashed d_X\varphi_a)\mathring L\varphi
\end{array}
\right),
\end{split}
\end{equation}
\begin{equation}\label{Tchi}
\begin{split}
\slashed{\mathcal L}_T\chi&_{XX}=\slashed\nabla_{X}^2\mu-\mu(\textrm{tr}\chi)\chi_{XX}-\f12c^{-1}(\tilde{T}^aT\varphi_a+Tc)\chi_{XX}+\f12c^{-1}\mu(\slashed d_Xx^a\slashed d_X\varphi_a)\textrm{tr}\chi\\
&+c^{-1}\Big\{\f12\slashed d_Xx^a(\slashed d_XT\varphi_a)-\slashed g_{XX}(\mathring LT\varphi_0+\varphi_a\mathring LT\varphi_a)+\mu\slashed\nabla_{X}^2\varphi_0+\mu(\varphi_a+\f12\tilde T^a)\slashed\nabla_{X}^2\varphi_a\Big\}\\
&+\f32c^{-1}(\tilde{T}^a\slashed d_X\varphi_a)\slashed d_X\mu+c^{-2}f_1(\mathring L^1,\mathring L^2,\varphi)\left(
\begin{array}{ccc}
\mu \mathring L\varphi\\
T{\varphi}\\
\end{array}
\right)
\left(
\begin{array}{ccc}
\slashed g_{XX}\mathring L\vec\varphi\\
\slashed d_Xx^a\cdot\slashed d\varphi_a\\
\end{array}
\right)\\
&
+c^{-2}f_2(\mathring L^1,\mathring L^2,\varphi)\left(
\begin{array}{ccc}
\mu(\slashed d_X\varphi)(\slashed d_X\varphi)\\
\slashed g^{XX}\mu(\slashed d_Xx^a\slashed d\varphi_a)^2\\
\end{array}
\right),\\
\end{split}
\end{equation}

and then
\begin{equation}\label{Lchi'}
\begin{split}
\mathring L\check\chi_{XX}=&c^{-1}\slashed d_Xx^a(\slashed d_X\mathring L\varphi_a)-c^{-1}\slashed g_{XX}\big(\mathring L^2\varphi_0+\varphi_a\mathring L^2\varphi_a\big)-c^{-1}\mathring L^\gamma(\slashed\nabla_{X}^2\varphi_\gamma)-\f12c^{-1}(\mathring Lc)\check\chi_{XX}\\
&-c^{-1}(\textrm{tr}\check\chi)(\slashed d_Xx^a)\slashed d_X\varphi_a-\f{1}{2}\varrho^{-1}c^{-1}(\mathring Lc)\slashed g_{XX}-\varrho^{-1}c^{-1}(\slashed d_Xx^a)\slashed d_X\varphi_a+(\textrm{tr}\check\chi)\check\chi_{XX}\\
&+c^{-2}f(\mathring L^1,\mathring L^2,\varphi)\left(
\begin{array}{ccc}
(\slashed d_X\varphi)(\slashed d_X\varphi)\\
\slashed g^{XX}(\slashed d_Xx^a\slashed d_X\varphi_a)^2\\
\slashed g^{XX}(\mathring L\varphi)(\mathring L\varphi)\\
(\slashed d_Xx^a\slashed d_X\varphi_a)\mathring L\varphi
\end{array}
\right),
\end{split}
\end{equation}
\begin{equation}\label{Tchi'}
\begin{split}
\slashed{\mathcal L}_T&\check{\chi}_{XX}=\slashed\nabla_{X}^2\mu-\mu(\textrm{tr}\check\chi)\check\chi_{XX}-\f12c^{-1}(\tilde{T}^aT\varphi_a+Tc)\chi_{XX}
+\f12c^{-1}\mu(\slashed d_Xx^a\slashed d_X\varphi_a)\textrm{tr}\chi\\
&
+c^{-1}\Big\{\f12\slashed d_Xx^a(\slashed d_XT\varphi_a)-\slashed g_{XX}(\mathring LT\varphi_0+\varphi_a\mathring LT\varphi_a)+\mu\slashed\nabla_{X}^2\varphi_0+\mu(\varphi_a+\f12\tilde T^a)\slashed\nabla_{X}^2\varphi_a\Big\}\\
&
+\f32c^{-1}(\tilde{T}^a\slashed d_X\varphi_a)\slashed d_X\mu+\f{c^{-1}}{\varrho}(Tc)\slashed g_{XX}+\f{c^{-1}\mu}{\varrho}(\mathring Lc)\slashed g_{XX}
-\f{2c^{-1}\mu}{\varrho}\slashed d_Xx^a\slashed d_X\varphi_a+\f{\mu-1}{\varrho^2}\slashed g_{XX}\\
&+c^{-2}f_1(\mathring L^1,\mathring L^2,\varphi)\left(
\begin{array}{ccc}
\mu \mathring L\varphi\\
T{\varphi}\\
\end{array}
\right)
\left(
\begin{array}{ccc}
\slashed g_{XX}\mathring L\varphi\\
\slashed d_Xx^a\cdot\slashed d\varphi_a\\
\end{array}
\right)
+c^{-2}f_2(\mathring L^1,\mathring L^2,\varphi)\left(
\begin{array}{ccc}
\mu(\slashed d_X\varphi)(\slashed d_X\varphi)\\
\slashed g^{XX}\mu(\slashed d_Xx^a\slashed d\varphi_a)^2
\end{array}
\right).
\end{split}
\end{equation}
\end{lemma}

Based on $\mathring L\check\chi$ in \eqref{Lchi'}, the estimate of $\check\chi$ could be achieved
by integrating along integral curves of $\mathring L$.

\begin{proposition}\label{chi'}
Under the assumptions $(\star)$ with $\delta>0$ suitably small, it holds that
\begin{equation}\label{echi'}
|\check{\chi}|=|\textrm{tr}\check\chi|\lesssim M\delta^{1-\varepsilon_0} \mathfrak t^{-3/2}
\end{equation}
and
\begin{equation}\label{chi}
|\chi|=\f{1}{\varrho}+O(M\delta^{1-\varepsilon_0} \mathfrak t^{-3/2}).
\end{equation}
\end{proposition}

\begin{proof}
It follows from
$\textrm{tr}\check\chi=\slashed g^{XX}\check{\chi}_{XX}$ that
\begin{equation}\label{Ltrhchi}
\mathring L\big(\textrm{tr}\check\chi\big)=-2(\textrm{tr}\check\chi)^2-\f{2}{\rho}\textrm{tr}\check\chi
+\slashed g^{XX}\mathring L\check\chi_{XX}.
\end{equation}
Substituting \eqref{Lchi'} into \eqref{Ltrhchi} and using $(\star)$, \eqref{L} and \eqref{dx} to estimate
the right hand side of \eqref{Ltrhchi} except $\check\chi$ itself, one can get
\[
|\mathring L\big(\varrho^2\textrm{tr}\check\chi\big)|\lesssim
M\delta^{1-\varepsilon_0}\mathfrak t^{-5/2}\varrho^2+M\delta^{1-\varepsilon_0}\mathfrak t^{-3/2}|\varrho^2\textrm{tr}\check\chi|
+\varrho^{-2}|\varrho^2\textrm{tr}\check\chi|^2.
\]
Thus, for small $\delta>0$, \eqref{echi'} follows from integrating along integral curves of $\mathring L$, which
also yields \eqref{chi} directly due to \eqref{errorv}.
\end{proof}

It follows from Proposition \ref{chi'}, \eqref{dL} and \eqref{deL} that for small $\delta>0$,
\begin{equation}\label{edL}
|\slashed d\mathring L^i|\lesssim\varrho^{-1},\quad |\slashed d\check L^i|\lesssim M\delta^{1-\varepsilon_0} \mathfrak t^{-3/2}.
\end{equation}
Note that
\begin{equation}\label{Y-7}
\mathring L\big(\varrho^2|\slashed d\mu|^2\big)=2\varrho^2
\big\{-\textrm{tr}\check\chi|\slashed d\mu|^2+(\slashed d_X\mathring L\mu)(\slashed d^X\mu)\big\}.
\end{equation}
Substituting \eqref{lmu} into \eqref{Y-7} and applying $(\star)$, \eqref{chi'} and \eqref{edL} give
\[
|\mathring L\big(\varrho|\slashed d\mu|\big)|\lesssim
M\delta^{1-\varepsilon_0} \mathfrak t^{-3/2}\big(\varrho|\slashed d\mu|\big)+M\delta^{1-\varepsilon_0} \mathfrak t^{-3/2}.
\]
This implies immediately that
\begin{equation}\label{dmu}
|\slashed d\mu|\lesssim M\delta^{1-\varepsilon_0} \mathfrak t^{-1}.
\end{equation}

Now we are ready to improve the estimate on $\mathring {\underline L}\varphi_\g$, which will
lead to  better estimates for $\mathring L\varphi_\gamma$
and some other related quantities independent of $M$.
\begin{proposition}\label{TVarphi}
Under the assumptions $(\star)$ with $\delta>0$ suitably small, it holds that
\begin{equation}\label{Tvarphi}
|\mathring{\underline L}\varphi_\gamma(\mathfrak t, u, \vartheta)|+|T\varphi_\gamma(\mathfrak t, u, \vartheta)|\lesssim\delta^{-\varepsilon_0} \mathfrak t^{-1/2}.
\end{equation}
\end{proposition}

\begin{proof}
	By $(\star)$, \eqref{L}, \eqref{dx} and \eqref{echi'}, it follows from \eqref{fequation} that
	\[
	|\mathring L\mathring{\underline L}\varphi_\gamma+\f{1}{2\varrho}\mathring{\underline L}\varphi_\gamma|\lesssim
	M^2\delta^{1-2\varepsilon_0} \mathfrak t^{-2}.
	\]
	This yields
	\begin{equation}\label{rLphi}
	|\mathring L\big(\varrho^{1/2} \mathring{\underline L}\varphi_\gamma\big)(\mathfrak t, u, \vartheta)|\lesssim
	M^2\delta^{1-2\varepsilon_0} \mathfrak t^{-3/2}.
	\end{equation}
	Integrating \eqref{rLphi} along integral curves of $\mathring L$ gives
	\begin{equation}\label{Y-8}
	|\varrho^{1/2} \mathring{\underline L}\varphi_\gamma(\mathfrak t, u, \vartheta)-\varrho_0^{1/2}\mathring{\underline L}\varphi_\gamma(t_0, u, \vartheta)|\lesssim  M^2\delta^{1-2\varepsilon_0},
	\end{equation}
	where $\varrho_0=t_0-u$.  Then the estimate on $|\mathring{\underline L}{\varphi_\gamma}|$ in \eqref{Tvarphi} follows
	from \eqref{Y-8} for small $\delta$. This, together with $\mathring{\underline L}=\mu\mathring L+2T$, yields the estimate
	for $|T\varphi_\gamma|$ in \eqref{Tvarphi} by $(\star)$.
\end{proof}

To improve the estimate on $\varphi_\gamma$ further, one treats $\eta$ first since   $T=\f{\p}{\p u}-\eta^XX$.

\begin{lemma}\label{eXi}
	Under the assumptions $(\star)$, it holds that for $\delta>0$ small,
	\begin{equation}\label{Xi}
	|\eta|=\sqrt{g_{XX}\eta^X\eta^X}\lesssim  M\delta^{1-\varepsilon_0}\mathfrak t.
	\end{equation}
\end{lemma}

\begin{proof}
	Note that $\eta$ is a vector field on $S_{\mathfrak t, u}$ and one has
	\begin{equation}\label{LXi}
	\mathring L\eta^X=[T,\mathring L]^X=\big(\mathscr{D}_T\mathring L-\mathscr{D}_{\mathring L}T\big)^X=\slashed d^X\mu+2\mu\zeta^X.
	\end{equation}
	Then, it follows from \eqref{zeta} that
	\[
	\mathring L(\varrho^{-2}|\eta|^2)=\varrho^{-2}\big\{2\textrm{tr}\check\chi|\eta|^2+\slashed d_X\mu\eta^X
	+2c^{-1}\mu\tilde{T}^a\slashed d_X\varphi_a\eta^X\big\}.
	\]
	This, together with \eqref{dmu}, \eqref{echi'} and $(\star)$, yields
	\begin{equation}\label{Y1}
	|\mathring L(\varrho^{-1}|\eta|)|\lesssim M\delta^{1-\varepsilon_0} \mathfrak t^{-3/2}(\varrho^{-1}|\eta|)
	+ M\delta^{1-\varepsilon_0} \mathfrak t^{-2}.
	\end{equation}
	Thus \eqref{Xi} follows immediately by integrating \eqref{Y1} along integral curves of $\mathring L$.
\end{proof}

Based on Proposition \ref{TVarphi} and Lemma \ref{eXi}, we can now improve the estimates
on $\varphi_\gamma$ and $\mathring L\varphi_\gamma$,
which are independent of $M$.

\begin{proposition}\label{Phi}
	Under the assumptions $(\star)$ with $\delta>0$ suitably small, it holds that
	\begin{equation}\label{phi}
	|\varphi_\gamma(\mathfrak t, u, \vartheta)|\lesssim\delta^{1-\varepsilon_0} \mathfrak t^{-1/2},
\qquad|\mathring L\varphi_\gamma(\mathfrak t, u, \vartheta)|\lesssim\delta^{1-\varepsilon_0} \mathfrak t^{-3/2}.
	\end{equation}
\end{proposition}

\begin{proof}
	First, substituting $T=\f{\p}{\p u}-\eta^XX$ into \eqref{Tvarphi} and using $(\star)$ and \eqref{Xi} lead to
	\begin{align}\label{Z-1}
	|\f{\p}{\p u}\varphi_\gamma(\mathfrak t, u,\vartheta)|\lesssim\delta^{-\varepsilon_0} \mathfrak t^{-1/2}.
	\end{align}
	Integrating \eqref{Z-1} from $0$ to $u$ yields the desired estimate on $\varphi_\gamma$ in \eqref{phi} since $\varphi_\g$ vanishes on $C_0$.
	
	Next, note that $\mathring L\varphi_\gamma$ satisfies the equation
	\begin{equation*}
	\begin{split}
	\mathring{\underline L}\mathring L\varphi_\gamma=-\f{1}{2\varrho}\mathring{\underline L}\varphi_\gamma+H_\gamma-\mathring L\mu\mathring L\varphi_\gamma+2\eta^A\slashed d_A\varphi_\gamma+2\mu\zeta^A\slashed d_A\varphi_\gamma
	\end{split}
	\end{equation*}
	due to \eqref{LuL} and \eqref{fequation}.
	Then $|\mathring{\underline L}\mathring L\varphi_\gamma|\lesssim\delta^{-\varepsilon_0} \mathfrak t^{-3/2}$ follows from \eqref{H}, \eqref{lmu} and \eqref{zeta}. Similarly as for \eqref{Z-1}, one can get
	\begin{equation}\label{puL}
	|\f{\p}{\p u}\mathring L\varphi_\gamma(\mathfrak t, u, \vartheta)|\lesssim\delta^{-\varepsilon_0} \mathfrak t^{-3/2}.
	\end{equation}
	Then the estimate for $\mathring L\varphi_\gamma$ in \eqref{phi} follows by integrating \eqref{puL} from $0$ to $u$.
\end{proof}

One can now use Propositions \ref{TVarphi} and \ref{Phi} to measure the errors of the components of the frames.

\begin{lemma}\label{gL}
	Under the assumptions $(\star)$ with $\delta>0$ suitably small, it holds that
	\begin{equation}\label{eL}
	|\check{L}^i|\lesssim \delta^{1-\varepsilon_0}\mathfrak t^{-1/2},\ \ |\check{T}^i|\lesssim \delta^{1-\varepsilon_0} \mathfrak t^{-1/2},
	\end{equation}
	\begin{equation}\label{r}
	|\f{r}{\varrho}-1|\lesssim\delta^{1-\varepsilon_0} \mathfrak t^{-1/2},\ \
	|{R}^k-{\Omega}^k|\lesssim\delta^{1-\varepsilon_0}\mathfrak t^{1/2} .
	\end{equation}
\end{lemma}

\begin{proof}
	First, it follows from \eqref{L}, \eqref{phi}, and \eqref{LeL} that
	\begin{equation}\label{4-0}
	|\mathring L\big(\varrho\check{L}^i\big)|\lesssim\delta^{1-\varepsilon_0} \mathfrak t^{-1/2}.
	\end{equation}
	Integrating \eqref{4-0} along integral curves of $\mathring L$ yields the desired estimate on $\check L^i$ in \eqref{eL},
	which implies the estimate on $\check T^i$ in \eqref{eL} by
	$\check T^i=\varphi_i-\check L^i$.
	
	Next, due to $\ds g_{ij}(\check T^i-\f{x^i}{\varrho})(\check T^j-\f{x^j}{\varrho})=1$, then
	$
	\f{r^2}{\varrho^2}=c-\delta_{ij}\check T^i\check T^j+2\delta_{ij}\check T^i\f{x^j}{\varrho}.
	$
	Thus,
	\begin{equation}\label{rrho}
	\f{r}{\varrho}-1=\f{2(\varphi_0+\varphi_i\omega^i)+2\check L^i\varphi_i-\delta_{ij}\check L^i\check L^j-2\delta_{ij}\check{L}^i\omega^j}{\sqrt{c+(\delta_{ij}\check T^i\omega^j)^2-\delta_{ij}\check T^i\check T^j}+1
		-\delta_{ij}\check T^i\omega^j},
	\end{equation}
	which implies the first inequality in \eqref{r} by $(\star)$ and \eqref{eL}, while the second one follows from this, \eqref{eL} and \eqref{r}.
\end{proof}

Note that the rotation vector field $R$ behaves just as the scaling operator $r\slashed\nabla$ under the assumptions $(\star)$
and the estimate \eqref{echi'} as stated in the following lemma, which is similar to \cite[Lemma 12.22]{J}.

\begin{lemma}\label{12form}
 Under the assumptions $(\star)$ with $\delta>0$ suitably small, it holds that

  (i) if $\kappa$ is a 1-form on $S_{\mathfrak t, u}$, then
  \begin{align}
  &(\kappa\cdot R)^2=\big(1+O(\delta^{1-\varepsilon_0} \mathfrak t^{-1/2})\big)r^2|\kappa|^2\label{1-f},\\
  &|\slashed{\mathcal L}_{R}\kappa|^2=\big(1+O(\delta^{1-\varepsilon_0} \mathfrak t^{-1/2} )\big)r^2|\slashed\nabla\kappa|^2+O(\delta^{1-\varepsilon_0}\mathfrak t^{-1/2}M^2)|\kappa|^2;\label{1f}
  \end{align}

  (ii) if $\Theta$ is a 2-form on $S_{\mathfrak t, u}$, then
  \begin{equation}\label{2-f}
  |\slashed{\mathcal L}_{R}\Theta|^2=r^2|\slashed\nabla\Th|^2\big(1+O(\delta^{1-\varepsilon_0} \mathfrak t^{-1/2})\big)
  +|\Theta|^2O(\delta^{1-\varepsilon_0} \mathfrak t^{-1/2}M^2).
  \end{equation}

\end{lemma}

Note that although the $L^\infty$ estimates for lower order derivatives of some geometrical quantities (such as \eqref{echi'}, \eqref{edL} and \eqref{dmu}) and the $M$-independent estimates for the first order derivatives of $\varphi_\gamma$ (see \eqref{Tvarphi} and \eqref{phi}) obtained
in the above are not enough to close the assumptions $(\star)$, yet the method  proving \eqref{phi} indicates that it is possible to improve the estimates on $\varphi_\gamma$ by making use of the equation \eqref{fequation}. Thus, under the assumptions $(\star)$, one can continue to improve the estimates of $\varphi_\g$ up to the $(N-1)^{th}$ order derivatives as given in the following proposition, whose proof is exactly same as those in \cite[Section 6]{Ding4} and thus omitted.

\begin{proposition}\label{LTRh}
Under the assumptions $(\star)$ with $\delta>0$ small, it holds that for any operator $Z\in\{\varrho\mathring L, T, R\}$ and $k\leq N-1$,
\begin{equation}\label{z}
\begin{split}
&|\slashed{\mathcal{L}}_{Z}^{k+1;i,l}\slashed dx^j|\lesssim \delta^{-i},\quad|\slashed{\mathcal{L}}_{ Z}^{k;i,l}\check{\chi}|\lesssim M\delta^{1-i-\varepsilon_0}\mathfrak t^{-3/2},\quad |\slashed{\mathcal{L}}_{Z}^{k+1;i,l}R|\lesssim M\delta^{1-i-\varepsilon_0}\mathfrak t^{1/2},\\
&|\slashed{\mathcal{L}}_{Z}^{k;i,l}\leftidx{^{(R)}}{\slashed\pi}|+|\slashed{\mathcal{L}}_{Z}^{k;i,l}\leftidx{^{(R)}}{\slashed\pi}_{\mathring L}|+|\slashed{\mathcal{L}}_{Z}^{k;i,l}\leftidx{^{(R)}}{\slashed\pi}_{T}|\lesssim M\delta^{1-i-\varepsilon_0}\mathfrak t^{-1/2},\\
&|Z^{k+1;i,l}\check{L}^j|\lesssim M\delta^{1-i-\varepsilon_0}\mathfrak t^{-1/2},\quad| Z^{k+1;i,l}\upsilon|\lesssim M\delta^{1-i-\varepsilon_0}\mathfrak t^{1/2},\\
&|Z^{k+1;i,l}\mu|\lesssim M\delta^{1-i-\varepsilon_0},\quad |\slashed{\mathcal{L}}_{Z}^{k;i,l}\leftidx{^{(T)}}{\slashed\pi}|\lesssim M\delta^{-i-\varepsilon_0}\mathfrak t^{-1/2},
\end{split}
\end{equation}
where $(k;i,l)$ means that the number of $Z$ is $k$, the number of $T$ is $i$ , and the number of $\varrho \mathring L$ is $l$.
\end{proposition}

\begin{proposition}\label{Z}
Under the same assumptions as in Proposition \ref{LTRh}, it holds that for $k\leq N-1$,
\begin{align}
&|Z^{k+2;i+1,l+1}\varphi_\gamma(\mathfrak t, u, \vartheta)|\lesssim\delta^{-i-\varepsilon_0}\mathfrak t^{-1/2},\label{ztl}\\
&|Z^{k+1;i,l+1}\varphi_\gamma(\mathfrak t, u, \vartheta)|\lesssim\delta^{1-i-\varepsilon_0}\mathfrak t^{-1/2}.\label{zl}
\end{align}
\end{proposition}

Finally, it is noted that under the assumptions $(\star)$,
we have obtained not only the estimates with $M$ dependent bounds in Proposition \ref{LTRh}, but also more refined estimates independent of $M$
in Proposition \ref{Z}.
On the other hand, if starting with Proposition \ref{Z} and repeating
the above analysis, one can improve the conclusions in \ref{LTRh}
such that all the related constants are independent of $M$ when $k\leq N-3$. Therefore,
from now on, we may apply these estimates without the constant $M$ since $N$ could be chosen large enough.

\section{Energy estimates for the linearized equation and some higher order $L^2$ estimates}\label{EE}

To close the bootstrap assumptions $(\star)$, one needs further refined estimates than those derived in Section \ref{BA}. To this end, we plan to construct some suitable energies for higher order derivatives of $\varphi_\g$. Note that $\varphi_\gamma$ satisfies the nonlinear equation \eqref{ge}, and each
derivative of $\varphi_\g$ also fulfills a similar equation with the same metric. Thus, in the section, we first focus on
the energy estimates for any smooth
function $\Psi$ solving the following linear equation
\begin{equation}\label{gel}
\mu\Box_g\Psi=\Phi
\end{equation}
for a given function $\Phi$, where $\Psi$ and its derivatives vanish on $C_0^{\mathfrak t}$. The following divergence theorem in $D^{{\mathfrak t},u}$ will be used, whose proof follows that of \cite[Lemma 10.12]{J}.
\begin{lemma}\label{divergence}
For any vector field $J$, it holds that
\begin{equation}\label{div}
\begin{split}
\int_{D^{{\mathfrak t}, u}}\mu\mathscr D_\al J^\al d\nu_{\slashed g}du'd\tau=&\int_{\Sigma_{\mathfrak t}^u}\big(-J_T-\mu J_{\mathring L}\big) d\nu_{\slashed g}du'-\int_{\Sigma_{t_0}^u}\big(-J_T-\mu J_{\mathring L}\big) d\nu_{\slashed g}du'\\
&-\int_{C_{u}^{\mathfrak t}}J_{\mathring L} d\nu_{\slashed g}d\tau.
\end{split}
\end{equation}
\end{lemma}

As explained in Section \ref{6-2}, the vectorfields $J'$s used in \cite{Ding4} are not suitable for \eqref{1.9}
due to the slow time decay of the solutions to the 2D wave equations. We will choose
some new  vectorfields $J'$s  in \eqref{div}, which  are expressed as
\begin{align}
&J_1:=-\varrho^{2m}g^{\al\kappa}Q_{\kappa\beta}\mathring L^\beta\p_\al,\quad J_2:=-g^{\al\kappa}Q_{\kappa\beta}\mathring{\underline L}^\beta\p_\al,\label{J1}\\
&J_3:=\big(\f12\varrho^{2m-1}\Psi\mathscr D^\al\Psi-\f14\Psi^2\mathscr D^\al(\varrho^{2m-1})\big)\p_\al,\label{J3}
\end{align}
where $m\in (\f12, \f34)$ is a fixed constant and $Q$ is the {\it{energy-momentum tensor field}} of $\Psi$ defined as
\begin{equation*}
\begin{split}
Q_{\al\beta}&=Q_{\al\beta}[\Psi]:=(\p_\al\Psi)(\p_\beta\Psi)-\f12 g_{\al\beta}g^{\nu\chi}(\p_\nu\Psi)(\p_\chi\Psi)\\
&=(\p_\al\Psi)(\p_\beta\Psi)-\f12 g_{\al\beta}\big\{|\slashed d\Psi|^2-\mu^{-1}(\mathring{\underline L}\Psi)(\mathring L\Psi)\big\}.
\end{split}
\end{equation*}

For each $J_i$ ($1\le i\le 3$), the corresponding terms in \eqref{div} will be analyzed respectively.

We start with the estimates of the right hand side of \eqref{div}. Note that the components of $Q_{\al\beta}$ relative to $\{\mathring L, \mathring{\underline L}, X\}$ are
\begin{equation*}
\begin{split}
&Q_{\mathring L\mathring L}=(\mathring L\Psi)^2,\quad Q_{\mathring{\underline L}\mathring{\underline L}}=(\mathring{\underline L}\Psi)^2,\quad Q_{\mathring L\mathring{\underline L}}=\mu|\slashed d\Psi|^2,\\
&Q_{\mathring LX}=(\mathring L\Psi)(\slashed d_X\Psi),\quad Q_{\mathring{\underline L}X}=(\mathring{\underline L}\Psi)(\slashed d_X\Psi),\\
&Q_{XX}=\f12(\slashed d_X\Psi)(\slashed d_X\Psi)+\f12\slashed g_{XX}\mu^{-1}(\mathring{\underline L}\Psi)(\mathring L\Psi).
\end{split}
\end{equation*}

Then it follows from \eqref{div}-\eqref{J1} that
\begin{equation}\label{divJ1}
\begin{split}
\int_{D^{{\mathfrak t}, u}}\mu\mathscr D_\al {J_1}^\al=&\int_{\Sigma_{\mathfrak t}^u}\f12\mu\varrho^{2m}\big((\mathring L\Psi)^2+|\slashed d\Psi|^2\big)-\int_{\Sigma_{t_0}^u}\f12\mu\varrho^{2m}\big((\mathring L\Psi)^2+|\slashed d\Psi|^2\big)\\
&+\int_{C_{u}^{\mathfrak t}}\varrho^{2m}(\mathring L\Psi)^2
\end{split}
\end{equation}
and
\begin{equation}\label{divJ2}
\begin{split}
\int_{D^{{\mathfrak t}, u}}\mu\mathscr D_\al {J_2}^\al=&\int_{\Sigma_{\mathfrak t}^u}\f12\big((\mathring {\underline L}\Psi)^2+\mu^2|\slashed d\Psi|^2\big)-\int_{\Sigma_{t_0}^u}\f12\big((\mathring {\underline L}\Psi)^2+\mu^2|\slashed d\Psi|^2\big)\\
&+\int_{C_{u}^{\mathfrak t}}\mu|\slashed d\Psi|^2.
\end{split}
\end{equation}
While \eqref{div} and \eqref{J3} yield
\begin{lemma}
It holds that
\begin{equation}\label{divJ3}
\begin{split}
\int_{D^{{\mathfrak t}, u}}\mu\mathscr D_\al {J_3}^\al=&\int_{\Sigma_{\mathfrak t}^u}\bigg(-\f12\varrho^{2m-1}\Psi(\mu\mathring L\Psi)+\f14\varrho^{2m-2}\Psi^2((2m-1)(\mu-2)+\varrho\mu\textrm{tr}\tilde\theta)\bigg)\\
&-\int_{\Sigma_{t_0}^u}\bigg(-\f12\varrho^{2m-1}\Psi(\mu\mathring L\Psi)+\f14\varrho^{2m-2}\Psi^2((2m-1)(\mu-2)+\varrho\mu\textrm{tr}\tilde\theta)\bigg)\\
&+\int_{C_{u}^{\mathfrak t}}\bigg(-\f14\textrm{tr}\chi\cdot\varrho^{2m-1}\Psi^2-\varrho^{2m-1}\Psi\mathring L\Psi\bigg).
\end{split}
\end{equation}
\end{lemma}

\begin{proof}
Applying \eqref{div} to $J_3$ yields
\begin{equation}\label{divJ_3}
\begin{split}
\int_{D^{{\mathfrak t}, u}}\mu\mathscr D_\al {J_3}^\al=&\int_{\Sigma_{\mathfrak t}^u}\Big(-\f12\varrho^{2m-1}\Psi(T\Psi)-\f12\varrho^{2m-1}\Psi(\mu\mathring L\Psi)+\f{2m-1}{4}\varrho^{2m-2}\Psi^2(\mu-1)\Big)\\
&-\int_{\Sigma_{t_0}^u}\Big(-\f12\varrho^{2m-1}\Psi(T\Psi)-\f12\varrho^{2m-1}\Psi(\mu\mathring L\Psi)
+\f{2m-1}{4}\varrho^{2m-2}\Psi^2(\mu-1)\Big)\\
&+\int_{C_{u}^{\mathfrak t}}\big(-\f12\varrho^{2m-1}\Psi(\mathring L\Psi)+\f{2m-1}{4}\varrho^{2m-2}\Psi^2\big).
\end{split}
\end{equation}
Note that $\Psi$ vanishes on $u=0$. Then
\begin{equation*}
\begin{split}
\int_{S_{\mathfrak t, u}}\varrho^{2m-1}\Psi^2d\nu_{\slashed g}&=\int_0^u\Big(\f{\p}{\p u'}\int_{S_{\mathfrak t, u'}}\varrho^{2m-1}\Psi^2d\nu_{\slashed g}\Big)du'\\
&=\int_{\Sigma_{\mathfrak t}^u}\Big(-(2m-1)\varrho^{2m-2}\Psi^2+2\varrho^{2m-1}\Psi T\psi+\varrho^{2m-1}\mu\textrm{tr}\tilde\theta\Psi^2\Big),
\end{split}
\end{equation*}
which implies
\begin{equation}\label{psiTpsi}
\int_{\Sigma_{\mathfrak t}^u}-\f12\varrho^{2m-1}\Psi(T\Psi)=-\f14\int_{S_{\mathfrak t, u}}\varrho^{2m-1}\Psi^2+\f14\int_{\Sigma_{\mathfrak t}^u}\big(-(2m-1)\varrho^{2m-2}\Psi^2+\varrho^{2m-1}\mu\textrm{tr}\tilde\theta\Psi^2\big).
\end{equation}
Taking $t=t_0$ in \eqref{psiTpsi} gives
\begin{equation}\label{psitpsi}
\int_{\Sigma_{t_0}^u}-\f12\varrho^{2m-1}\Psi(T\Psi)=-\f14\int_{S_{t_0, u}}\varrho^{2m-1}\Psi^2+\f14\int_{\Sigma_{t_0}^u}\big(-(2m-1)\varrho^{2m-2}\Psi^2+\varrho^{2m-1}\mu\textrm{tr}\tilde\theta\Psi^2\big).
\end{equation}
Thanks to the following identity
\begin{equation*}
\begin{split}
\int_{S_{\mathfrak t, u}}\varrho^{2m-1}\Psi^2d\nu_{\slashed g}&=\int_{S_{t_0, u}}\varrho^{2m-1}\Psi^2d\nu_{\slashed g}+\int_{t_0}^{\mathfrak t}\f{\p}{\p\tau}\Big(\int_{S_{\tau, u}}\varrho^{2m-1}\Psi^2d\nu_{\slashed g}\Big)d\tau\\
&=\int_{S_{t_0, u}}\varrho^{2m-1}\Psi^2d\nu_{\slashed g}+\int_{ C^{\mathfrak t}_{ u}}\Big(\mathring L(\varrho^{2m-1}\Psi^2)+\textrm{tr}\chi\cdot\varrho^{2m-1}\Psi^2\Big)d\nu_{\slashed g}d\tau,
\end{split}
\end{equation*}
 then   \eqref{psiTpsi} becomes
\begin{equation}\label{pTp}
\begin{split}
\int_{\Sigma_{\mathfrak t}^u}-\f12\varrho^{2m-1}\Psi(T\Psi)=&-\f14\int_{S_{t_0, u}}\varrho^{2m-1}\Psi^2-\f14\int_{ C^{\mathfrak t}_{ u}}\Big(\mathring L(\varrho^{2m-1}\Psi^2)+\textrm{tr}\chi\cdot\varrho^{2m-1}\Psi^2\Big)\\
&+\f14\int_{\Sigma_{\mathfrak t}^u}\big(-(2m-1)\varrho^{2m-2}\Psi^2+\varrho^{2m-1}\mu\textrm{tr}\tilde\theta\Psi^2\big).
\end{split}
\end{equation}
Substituting  \eqref{pTp} and \eqref{psitpsi} into \eqref{divJ_3} yields \eqref{divJ3}.
\end{proof}

It follows from \eqref{divJ3} and \eqref{divJ1} together with \eqref{phimu} that
\begin{equation}\label{j1j3}
\begin{split}
&\int_{D^{{\mathfrak t},u}}\mu(\mathscr D_\al {J_1}^\al-\mathscr D_\al{J_3}^\al)\\
=&\f12\int_{\Sigma_{\mathfrak t}^u}\Big(\mu\varrho^{2m}|\slashed d\Psi|^2+\mu\varrho^{2m}(\mathring L\Psi)^2+\mu\varrho^{2m-1}\Psi(\mathring L\Psi)+m\varrho^{2m-2}\Psi^2+O(\delta^{-\varepsilon_0}\varrho^{2m-3/2}\Psi^2)\Big)\\
&-\f12\int_{\Sigma_{t_0}^u}\Big(\mu\varrho^{2m}|\slashed d\Psi|^2+\mu\varrho^{2m}(\mathring L\Psi)^2+\mu\varrho^{2m-1}\Psi(\mathring L\Psi)+m\varrho^{2m-2}\Psi^2+O(\delta^{-\varepsilon_0}\varrho^{2m-3/2}\Psi^2)\Big)\\
&+\int_{C_{u}^{\mathfrak t}}\big((\varrho^m\mathring L\Psi+\f12\varrho^{m-1}\Psi)^2+O(\delta^{1-\varepsilon_0}\varrho^{2m-5/2}\Psi^2)\big),
\end{split}
\end{equation}
here one has used the estimate $|\mu\textrm{tr}\tilde\theta|\lesssim\delta^{-\varepsilon_0}\mathfrak t^{-1/2}$ due to \eqref{theta} and the estimates in Section \ref{BA}.

In order to estimate $\delta^{-\varepsilon_0}\int_{\Sigma_{\mathfrak t}^u}\varrho^{2m-3/2}\Psi^2$ and
$\delta^{1-\varepsilon_0}\int_{ C^{\mathfrak t}_{ u}}\varrho^{2m-5/2}\Psi^2$ in \eqref{j1j3}, one needs
the following inequality.
\begin{lemma}
Under the assumptions $(\star)$, it holds that for any $f\in C^1(D^{{\mathfrak t},u})$ vanishing on $C_0^t$,
\begin{equation}\label{f2}
\int_{\Sigma_{\mathfrak t}^u}\mu f^2+\delta\int_{C_{u}^{\mathfrak t}}\varrho^{-1}f^2\lesssim\int_{\Sigma_{t_0}^u}\mu f^2
+\int_{D^{{\mathfrak t}, u}}\big(|\mathring L(f^2)|+\delta|T(\varrho^{-1}f^2)|\big).
\end{equation}
\end{lemma}
\begin{proof}
   It follows from Lemma 3.4 in \cite{LS} that for any vectorfield $J$,
   \begin{equation}\label{Div}
   \mu\mathscr D_\al J^\al=-\mathring L(\mu J_{\mathring L}+J_T)-T(J_{\mathring L})+\slashed{div}(\mu\slashed J)-\mu(\textrm{tr}\tilde\theta+\textrm{tr}\chi)J_{\mathring L}-\textrm{tr}\chi J_T,
   \end{equation}
   where $\slashed J$ is the projection of $J$ on $S_{{\mathfrak t},u}$.
   Set $J=f^2\mathring L$ and $J=\delta\varrho^{-1}\mu^{-1}f^2 T$ in \eqref{Div} respectively, and then
   use \eqref{div} to obtain
    \begin{align}
    \int_{\Sigma_{\mathfrak t}^u}\mu f^2&=\int_{\Sigma_{t_0}^u}\mu f^2+\int_{D^{{\mathfrak t},u}}\big((\mathring L\mu+\mu\textrm{tr}\chi)f^2
    +\mu\mathring L(f^2)\big),\label{muf}\\
    \delta\int_{C_{u}^{\mathfrak t}}\varrho^{-1}f^2&=\delta\int_{D^{{\mathfrak t},  u}}\big(\mu\varrho^{-1}\textrm{tr}\tilde\theta f^2
    +T(\varrho^{-1}f^2)\big)\label{rf}.
    \end{align}
  By adding \eqref{muf} and \eqref{rf} together with the estimates on  the coefficients in the right hand sides
  of \eqref{muf}-\eqref{rf}, one arrives at
    \[
    \int_{\Sigma_{\mathfrak t}^u}\mu f^2+\delta\int_{C_{u}^{\mathfrak t}}\varrho^{-1}f^2\lesssim\int_{\Sigma_{t_0}^u}\mu f^2+\int_{D^{{\mathfrak t}, u}}\big(\varrho^{-1}f^2+\mu|\mathring L(f^2)|+\delta|T(\varrho^{-1}f^2)|\big).
    \]
    Hence \eqref{f2} follows directly from Gronwall's inequality.
\end{proof}

Let $f=\delta^{-\varepsilon_0/2}\varrho^{m-3/4}\Psi$ in \eqref{f2}. Using the facts that
\[
|\mathring L(f^2)|\lesssim\varrho^{-1}f^2+\delta^{-\varepsilon_0}\varrho^{2m-1/2}|\mathring L\Psi|^2
\]
and
\[
|T(\varrho^{-1}f^2)|\lesssim\delta^{-1}\varrho^{-1}f^2+\delta^{1-\varepsilon_0}\varrho^{2m-5/2}|T\Psi|^2,
\]
one then gets
\begin{equation}\label{SC}
\begin{split}
&\int_{\Sigma_{\mathfrak t}^u}\mu\delta^{-\varepsilon_0}\varrho^{2m-3/2}\Psi^2+\int_{C_{u}^{\mathfrak t}}\delta^{1-\varepsilon_0}\varrho^{2m-5/2}\Psi^2\\
\lesssim&\int_{\Sigma_{t_0}^u}\mu\delta^{-\varepsilon_0}\varrho^{2m-3/2}\Psi^2+\int_{D^{{\mathfrak t}, u}}\bigg(\delta^{-\varepsilon_0}\varrho^{-1/2}(\varrho^m\mathring L\Psi+\f12\varrho^{m-1}\Psi)^2+\delta^{2-\varepsilon_0}\varrho^{2m-5/2}|T\Psi|^2\bigg).
\end{split}
\end{equation}
Substituting \eqref{SC} into \eqref{j1j3} yields
\begin{equation}\label{gtr}
\begin{split}
&\int_{D^{{\mathfrak t},u}}\mu\big(\mathscr D_\al{J_1}^\al-\mathscr D_\al{J_3}^\al\big)\\
\gtrsim&\int_{\Sigma_{\mathfrak t}^u}\bigg(\varrho^{2m}|\slashed d\Psi|^2+\varrho^{2m}(\mathring L\Psi)^2+\varrho^{2m-2}\Psi^2\bigg)-\int_{\Sigma_{t_0}^u}\bigg(|\slashed d\Psi|^2+(\mathring L\Psi)^2+\delta^{-\varepsilon_0}\Psi^2\bigg)\\
&+\int_{C_{u}^{\mathfrak t}}(\varrho^{m}\mathring L\Psi+\f12\varrho^{m-1}\Psi)^2
-\int_{D^{{\mathfrak t},  u}}\delta^{-\varepsilon_0}\bigg(\varrho^{-1/2}(\varrho^m\mathring L\Psi+\f12\varrho^{m-1}\Psi)^2+\delta^{2}\varrho^{2m-5/2}(\mathring{\underline L}\Psi)^2\bigg).
\end{split}
\end{equation}
By \eqref{gtr} and the identity \eqref{divJ2}, it is natural to define the following energies and fluxes
\begin{align}
&E_1[\Psi](\mathfrak t, u):=\ds\int_{\Sigma_{\mathfrak t}^{u}}\bigg(\varrho^{2m}(\mathring L\Psi)^2+\varrho^{2m}|\slashed d\Psi|^2
+\varrho^{2m-2}\Psi^2\bigg),\label{E1}\\
&E_2[\Psi](\mathfrak t, u):=\ds\int_{\Sigma_{\mathfrak t}^{u}}((\mathring{\underline L}\Psi)^2+|\slashed d\Psi|^2),\label{E2}\\
&F_1[\Psi](\mathfrak t, u):=\ds\int_{C_{u}^{\mathfrak t}}(\varrho^m\mathring L\Psi+\f12\varrho^{m-1}\Psi)^2,\label{F1}\\
&F_2[\Psi](\mathfrak t,u):=\ds\int_{C_{u}^{\mathfrak t}}|\slashed d\Psi|^2.\label{F2}
\end{align}

Next, we treat the left hand side of \eqref{div}.
Set
\[
V_1=\varrho^{2m}\mathring L,\qquad V_2=\mathring {\underline L}.
\]
Direct computations give that
\begin{align}
&\mu\mathscr D_\al {J_1}^\al=-\Phi\varrho^{2m}(\mathring L\Psi)
-\f12\mu Q^{\al\beta}[\Psi]\leftidx{^{(V_1)}}\pi_{\al\beta},\label{lJ1}\\
&\mu\mathscr D_\al {J_2}^\al=-\Phi(\mathring {\underline L}\Psi)
-\f12\mu Q^{\al\beta}[\Psi]\leftidx{^{(V_2)}}\pi_{\al\beta},\label{lJ2}\\
&\mu\mathscr D_\al {J_3}^\al=\f12\varrho^{2m-1}\Psi\Phi
+\Big(-\f12\varrho^{2m-1}(\mathring L\Psi)(\mathring{\underline L\Psi})
+\f12\mu\varrho^{2m-1}|\slashed d\Psi|^2\Big)\nonumber\\
&\qquad\qquad+\f14\Psi^2(2m-1)\varrho^{2m-2}\Big((2m-2)\varrho^{-1}(\mu-2)+\mathring L\mu+\mu\textrm{tr}\t\theta+(\mu-1)\textrm{tr}\chi\Big).\label{lJ3}
\end{align}
Recall that the components of $\leftidx{^{(V_i)}}\pi_{\al\beta}$ in the
frame $\{\mathring L,\mathring {\underline L},X\}$ have been given in \eqref{Lpi}-\eqref{Rpi}.
To compute $Q^{\al\beta}\leftidx{^{(V_i)}}\pi_{\al\beta}$ in \eqref{lJ1}-\eqref{lJ2}, one can use the
components of the metric in the frame $\{\mathring L,\underline{\mathring L}, X\}$, given in \eqref{gab}, to derive directly that
\begin{equation}\label{v1}
\begin{split}
&-\f12 \mu Q^{\al\beta}[\Psi]\leftidx{^{(V_1)}}\pi_{\al\beta}\\
=&-\f14\mu^{-1}\leftidx{^{(V_1)}}\pi_{\mathring L\mathring{\underline L}} Q_{\mathring L\mathring{\underline L}}-\f18\mu^{-1}\leftidx{^{(V_1)}}\pi_{\mathring{\underline L}\mathring{\underline L}}Q_{\mathring L\mathring L}+\f12\leftidx{^{(V_1)}}\pi_{\mathring{\underline L}}^XQ_{\mathring LX}-\f12\mu\leftidx{^{(V_1)}}{\slashed\pi}^{XX}Q_{XX}\\
=&\bigg(\f12\varrho^{2m}\mathring L\mu+\mu (m-\f12)\varrho^{2m-1}\bigg)|\slashed d\Psi|^2+\bigg(m(\mu-2)\varrho^{2m-1}-\f12\varrho^{2m}\mathring L\mu\bigg)(\mathring L\Psi)^2+\varrho^{2m}\bigg(\slashed d^X\mu\\
&+2c^{-1}\mu\tilde T^a\slashed d^X\varphi_a\bigg)(\mathring L\Psi)(\slashed d_X\Psi)-\f12\mu\varrho^{2m}\textrm{tr}\check\chi|\slashed d\Psi|^2-\f12\varrho^{2m}\textrm{tr}\chi(\mathring L\Psi)(\mathring{\underline L}\Psi)
\end{split}
\end{equation}
and
\begin{equation}\label{v2}
\begin{split}
&-\f12 \mu Q^{\al\beta}[\Psi]\leftidx{^{(V_2)}}\pi_{\al\beta}\\
=&\f12\bigg(\mathring{\underline L}\mu+\mu\mathring L\mu+c^{-1}\mu Tc+c^{-1}\mu^2\mathring Lc+\mu^2\textrm{tr}\chi-2c^{-1}\mu^2\slashed d_Xx^a\cdot\slashed d^X\varphi_a\bigg)|\slashed d\Psi|^2\\
&-\bigg(2\mu c^{-1}\tilde T^a\slashed d^X\varphi_a+\slashed d^X\mu\bigg)(\slashed d_X\Psi)(\mathring{\underline L}\Psi)-\mu\slashed d^X\mu(\mathring L\Psi)(\slashed d_X\Psi)\\
&+\bigg(\f12c^{-1}Tc+\f12\mu c^{-1}\mathring Lc-c^{-1}\mu\slashed d^Xx^a\cdot\slashed d_X\varphi_a+\f12\mu\textrm{tr}\chi\bigg)(\mathring L\Psi)(\mathring{\underline L}\Psi).
\end{split}
\end{equation}

We first treat the terms involving $J_1$ and $J_3$.

Combining \eqref{lJ1} and \eqref{lJ3} and using \eqref{v1}, $\f12<m<\f34$ and $\mu<2$, one can get by the
estimates in Section \ref{BA} that
\begin{equation}\label{lJ13}
\begin{split}
&\int_{D^{{\mathfrak t},u}}\big(\mu\mathscr D_\al {J_1}^\al-\mu\mathscr D_\al{J_3}^\al\big)\\
=&-\int_{D^{{\mathfrak t},u}}\Phi(\varrho^{2m}\mathring L\Psi+\f12\varrho^{2m-1}\Psi)+\int_{D^{{\mathfrak t}, u}}\big(\mu(m-1)\varrho^{2m-1}+\f12\varrho^{2m}\mathring L\mu-\f12\mu\varrho^{2m}\textrm{tr}\check\chi\big)|\slashed d\Psi|^2\\
&{+\int_{D^{{\mathfrak t}, u}}\varrho^{2m}\big(m\varrho^{-1}(\mu-2)-\f12\mathring L\mu\big)(\mathring L\Psi)^2+\int_{D^{{\mathfrak t},u}}\varrho^{2m}(\slashed d_X\mu+2c^{-1}\mu\tilde{T}^a\slashed d_X\varphi_a)(\mathring L\Psi)(\slashed d^X\Psi)}\\
&-\int_{D^{{\mathfrak t}, u}}\f14\Psi^2(2m-1)\varrho^{2m-2}\Big((2m-2)\varrho^{-1}(\mu-2)+\mathring L\mu+\mu\textrm{tr}\t\theta+(\mu-1)\textrm{tr}\chi\Big)\\
&-\int_{D^{{\mathfrak t}, u}}\f12\varrho^{2m}\textrm{tr}\check\chi(\mathring L\Psi)(\mathring{\underline L\Psi})\\
\leq&-\int_{D^{{\mathfrak t},u}}\Phi(\varrho^{2m}\mathring L\Psi+\f12\varrho^{2m-1}\Psi)+\int_{D^{{\mathfrak t}, u}}\f12\varrho^{2m}\big(\mathring L\mu-\mu\textrm{tr}\check\chi\big)|\slashed d\Psi|^2-\int_{D^{{\mathfrak t}, u}}\f12\varrho^{2m}\mathring L\mu(\mathring L\Psi)^2\\
&{+\int_{D^{{\mathfrak t},u}}\varrho^{2m}(\slashed d_X\mu+2c^{-1}\mu\tilde{T}^a\slashed d_X\varphi_a)(\mathring L\Psi)(\slashed d^X\Psi)-\int_{D^{{\mathfrak t}, u}}\f12\varrho^{2m}\textrm{tr}\check\chi(\mathring L\Psi)(\mathring{\underline L\Psi})}\\
&-\int_{D^{{\mathfrak t}, u}}\f14(2m-1)\varrho^{2m-2}\Psi^2(\mathring L\mu+\mu\textrm{tr}\t\theta+(\mu-1)\textrm{tr}\chi)\\
\lesssim&|\int_{D^{{\mathfrak t},u}}\Phi(\varrho^{2m}\mathring L\Psi+\f12\varrho^{2m-1}\Psi)|+\int_{D^{{\mathfrak t},u}}\big(\delta^{1-\varepsilon_0}\varrho^{2m-3/2}|\slashed d\Psi|^2+\delta^{2-\varepsilon_0}\varrho^{2m-5/2}|\mathring{\underline L}\Psi|^2\big)\\
&+\int_{D^{{\mathfrak t},u}}\delta^{-\varepsilon_0}\varrho^{-1/2}(\varrho^m\mathring L\Psi+\f12\varrho^{m-1}\Psi)^2+\int_{D^{{\mathfrak t},u}}\delta^{-\varepsilon_0}\varrho^{2m-5/2}\Psi^2.
\end{split}
\end{equation}
The last integral in \eqref{lJ13} can be estimated by setting $f=\varrho^{m-3/4}\Psi$ in \eqref{f2} as
\begin{equation*}
\begin{split}
\delta\int_{C_u^t}\varrho^{2m-5/2}\Psi^2\lesssim&\int_{\Sigma_{t_0}^u}\Psi^2+\int_{D^{{\mathfrak t},u}}\big(\varrho^{2m-1/2}|\mathring L\Psi|^2+\delta^2\varrho^{2m-5/2}|\mathring{\underline L}\Psi|^2+{\varrho^{2m-5/2}\Psi^2}\big)\\
\lesssim&\int_{\Sigma_{t_0}^u}\Psi^2+\int_{D^{{\mathfrak t},u}}\big(\varrho^{-1/2}(\varrho^m\mathring L\Psi+\f12\varrho^{m-1}\Psi)^2+\delta^2\varrho^{2m-5/2}|\mathring{\underline L}\Psi|^2\big)\\
&+\int_{D^{{\mathfrak t},u}}\varrho^{2m-5/2}\Psi^2.
\end{split}
\end{equation*}
This, together with Gronwall's inequality, yields
\[
\delta\int_{C_u^t}\varrho^{2m-5/2}\Psi^2\lesssim\int_{\Sigma_{t_0}^u}\Psi^2+\int_{D^{{\mathfrak t},u}}\big(\varrho^{-1/2}(\varrho^m\mathring L\Psi+\f12\varrho^{m-1}\Psi)^2+\delta^2\varrho^{2m-5/2}|\mathring{\underline L}\Psi|^2\big).
\]
Thus
\begin{equation}\label{rpsi}
\begin{split}
\int_{D^{{\mathfrak t},u}}\varrho^{2m-5/2}\Psi^2\lesssim&\int_{\Sigma_{t_0}^u}\Psi^2+\int_{D^{{\mathfrak t},u}}\big(\varrho^{-1/2}(\varrho^m\mathring L\Psi+\f12\varrho^{m-1}\Psi)^2+\delta^2\varrho^{2m-5/2}|\mathring{\underline L}\Psi|^2\big).
\end{split}
\end{equation}
Substituting \eqref{rpsi} into \eqref{lJ13} yields
\begin{equation}\label{les}
\begin{split}
&\int_{D^{{\mathfrak t},u}}\big(\mu\mathscr D_\al {J_1}^\al-\mu\mathscr D_\al{J_3}^\al\big)\\
\lesssim&|\int_{D^{{\mathfrak t},u}}\Phi(\varrho^{2m}\mathring L\Psi+\f12\varrho^{2m-1}\Psi)|+\int_{D^{{\mathfrak t},u}}\big(\delta^{1-\varepsilon_0}\varrho^{2m-3/2}|\slashed d\Psi|^2+\delta^{2-\varepsilon_0}\varrho^{2m-5/2}|\mathring{\underline L}\Psi|^2\big)\\
&+\int_{D^{{\mathfrak t},u}}\delta^{-\varepsilon_0}\varrho^{-1/2}(\varrho^m\mathring L\Psi+\f12\varrho^{m-1}\Psi)^2+\int_{\Sigma_{t_0}^u}\delta^{-\varepsilon_0}\Psi^2.
\end{split}
\end{equation}

It follows from \eqref{gtr}, \eqref{les} and Gronwall's inequality that
\begin{equation}\label{EF1}
\begin{split}
E_1[\Psi](\mathfrak t,u)+F_1[\Psi](\mathfrak t,u)\lesssim &E_1[\Psi](t_0, u)+\int_{\Sigma_{t_0}}\delta^{-\varepsilon_0}\Psi^2+\delta^{2-\varepsilon_0}\int_{t_0}^\mathfrak t\tau^{2m-5/2}E_2[\Psi](\tau,u)d\tau\\
&+|\int_{D^{{\mathfrak t},u}}\Phi(\varrho^{2m}\mathring L\Psi+\f12\varrho^{2m-1}\Psi)|.
\end{split}
\end{equation}

It remains to treat \eqref{lJ2}. Recalling \eqref{v2} and estimating each coefficient in it by Proposition \ref{LTRh}, one can get
\begin{equation}\label{Qv2}
\begin{split}
&-\int_{D^{{\mathfrak t},u}}\f12\mu Q^{\al\beta}[\Psi]\leftidx{^{(V_2)}}\pi_{\al\beta}\\
\lesssim&\delta^{-\varepsilon_0}\int_{D^{{\mathfrak t},u}}\big(|\slashed d\Psi|^2+\delta\tau^{-1}|\mathring {\underline L}\Psi|\cdot|\slashed d\Psi|+\delta\tau^{-1}|\mathring L\Psi|\cdot|\slashed d\Psi|+\tau^{-1/2}|\mathring L\Psi|\cdot|\mathring{\underline L\Psi|}\big)\\
\lesssim&\delta^{-\varepsilon_0}\int_0^uF_2[\Psi](\mathfrak t,u')du'+\int_{t_0}^\mathfrak t\tau^{-m-\f12}E_2[\Psi](\tau,u)d\tau
+\delta^{-2\varepsilon_0}\int_{t_0}^\mathfrak t\tau^{-m-\f12}E_1[\Psi](\tau,u)d\tau.
\end{split}
\end{equation}

Then, it follows from \eqref{divJ2}, \eqref{E2}, \eqref{F2}, \eqref{lJ2}, \eqref{Qv2} and Gronwall's inequality that
\begin{equation}\label{EF2}
E_2[\Psi](t,u)+F_2[\Psi](\mathfrak t,u)\lesssim E_2[\Psi](t_0,u)+\int_{D^{{\mathfrak t},u}}|\Phi|\cdot|\mathring {\underline L}\Psi|
+\delta^{-2\varepsilon_0}\int_{t_0}^\mathfrak t \tau^{-m-\f12}E_1[\Psi](\tau,u)d\tau.
\end{equation}

Combining \eqref{EF1} with \eqref{EF2} and using Gronwall's inequality again,
due to $m\in (\f12,\f34)$, one concludes finally that
\begin{equation}\label{e}
\begin{split}
&\delta E_2[\Psi](\mathfrak t, u)+\delta F_2[\Psi](\mathfrak t, u)+E_1[\Psi](\mathfrak t, u)+F_1[\Psi](\mathfrak t, u)\\
\lesssim& \delta E_2[\Psi](t_0, u)+E_1[\Psi](t_0, u)+\int_{\Sigma_{t_0}}\delta^{-\varepsilon_0}\Psi^2+\delta\int_{D^{{\mathfrak t}, u}}|\Phi\cdot \mathring{\underline L}\Psi|\\
&+|\int_{D^{{\mathfrak t},u}}\Phi(\varrho^{2m}\mathring L\Psi+\f12\varrho^{2m-1}\Psi)|.
\end{split}
\end{equation}

\eqref{e} will be used to derive the energy estimates for $\varphi_\gamma$ and its derivatives. To this end, one can choose $\Psi=\Psi_\gamma^{k+1}:=Z^{k+1}\varphi_\gamma$ and then  $\Phi=\Phi_\gamma^{k+1}:=\mu\Box_g\Psi_\gamma^{k+1}$ ({$k\leq 2N-6$})
in \eqref{e}. Note that
\begin{equation}\label{Psi}
\begin{split}
\Phi_\gamma^{k+1}&=\mu\Box_g\Psi_\gamma^{k+1}=\mu[\Box_g,Z]\Psi_\gamma^{k}+Z\big(\mu\Box_g\Psi_\gamma^{k}\big)-(Z\mu)\Box_g\Psi_\gamma^{k}\\
&=\mu\textrm{div}\leftidx{^{(Z)}}C_\gamma^{k}+(Z+\leftidx{^{(Z)}}\chi)\Phi_\gamma^{k},
\end{split}
\end{equation}
where
\begin{equation}\label{ClP}
\begin{split}  &\leftidx{^{(Z)}}C_\gamma^{k}=\big(\leftidx{^{(Z)}}\pi^{\nu\beta}
-\f12g^{\nu\beta}\textrm{tr}_g\leftidx{^{(Z)}}\pi\big)\p_\nu\Psi_\gamma^{k}\p_\beta,\\
&\leftidx{^{(Z)}}\chi=\f12\textrm{tr}_g\leftidx{^{(Z)}}\pi-\mu^{-1}Z\mu,\\
&\Psi_\gamma^0=\varphi_\gamma,\quad\Phi_\g^0=\mu\Box_g\varphi_\g
\end{split}
\end{equation}
with $\textrm{tr}_g\leftidx{^{(Z)}}{\pi}=g^{\al\beta}\leftidx{^{(Z)}}\pi_{\al\beta}$ and $\Phi_\g^0$ being
the right hand side of \eqref{ge}.
Consequently, for $\Psi_\g^{k+1}=Z_{k+1}Z_{k}\cdots Z_{1}\varphi_\gamma$ with $Z_i\in\{\varrho\mathring L, T, R\}$, one can derive
by \eqref{Psi} and an induction argument that
\begin{equation}\label{Phik}
\begin{split}
\Phi_\gamma^{k+1}=&\sum_{j=1}^{k}\big(Z_{k+1}+\leftidx{^{(Z_{k+1})}}\chi\big)\dots\big(Z_{k+2-j}
+\leftidx{^{(Z_{k+2-j})}}\chi\big)\big(\mu\textrm{div}\leftidx{^{(Z_{k+1-j})}}C_\gamma^{k-j}\big)\\
&+\mu\textrm{div}\leftidx{^{(Z_{k+1})}}C_\gamma^{k}+\big(Z_{k+1}
+\leftidx{^{(Z_{k+1})}}\chi\big)\dots\big(Z_{1}+\leftidx{^{(Z_1)}}\chi\big)\Phi_\gamma^0,\qquad k\geq 1,\\
\Phi_\gamma^{1}=&\big(Z_{1}+\leftidx{^{(Z_1)}}\chi\big)\Phi_\gamma^0+\mu\textrm{div}\leftidx{^{(Z_{1})}}C_\gamma^{0}.
\end{split}
\end{equation}
To estimate $\leftidx{^{(Z)}}\chi$ and $\mu\textrm{div}\leftidx{^{(Z)}}C_\gamma^{k}$
in \eqref{Phik}, by
\begin{equation*}
\textrm{tr}_g\leftidx{^{(Z)}}{\pi}=-2\mu^{-1}\leftidx{^{(Z)}}\pi_{\mathring LT}+\textrm{tr}_{\slashed g}\leftidx{^{(Z)}}{\slashed{\pi}}
\end{equation*}
and \eqref{Lpi}-\eqref{Rpi}, we have that
\begin{equation}\label{lamda}
\begin{split}
&\leftidx{^{(T)}}\chi=-\f12c^{-1}Tc-\f12c^{-1}\mu\mathring Lc+c^{-1}\mu\slashed dx^a\cdot\slashed d\varphi_a
-\mu\textrm{tr}_{\slashed g}\chi,\\
&\leftidx{^{(\varrho\mathring L)}}\chi=\varrho\textrm{tr}\check\chi+2,\\
&\leftidx{^{(R)}}\chi=\upsilon\textrm{tr}{\chi}
+\f12c^{-1}\upsilon(\mathring Lc)-c^{-1}\upsilon\slashed dx^a\cdot\slashed d\varphi_a-\f12c^{-1}Rc.
\end{split}
\end{equation}
In addition, in the null frame $\{\mathring{\underline L},\mathring L,X\}$, the
term $\mu\textrm{div}\leftidx{^{(Z)}}C_\g^{k}$ can be written as
\begin{equation}\label{muC}
\begin{split}
\mu\textrm{div}\leftidx{^{(Z)}}C_\gamma^{k}=&-\f12\mathring L\big({\leftidx{^{(Z)}}C_\gamma^{k}}_{,\mathring{\underline L}}\big)-\f12\mathring{\underline L}\big({\leftidx{^{(Z)}}C_\gamma^{k}}_{,\mathring L}\big)
+\slashed{\textrm{div}}\big(\mu \leftidx{^{(Z)}}{\slashed C}_\g^{k}\big)\\
&-\f12\big(\mathring L\mu+\mu \textrm{tr}{\chi}+2\mu\textrm{tr}\tilde\theta\big){\leftidx{^{(Z)}}C_\g^{k}}_{,\mathring L}-\f12\textrm{tr}\chi{\leftidx{^{(Z)}}C_\g^{k}}_{,\mathring{\underline L}},
\end{split}
\end{equation}
where
\begin{equation}\label{C}
\begin{split}
&{\leftidx{^{(Z)}}C_\g^{k}}_{,\mathring L}=\leftidx{^{(Z)}}{\slashed\pi}_{\mathring LX}(\slashed d^X\Psi_\g^{k})-\f12(\textrm{tr}\leftidx{^{(Z)}}{\slashed\pi})\mathring L\Psi_\g^{k},\\
&{\leftidx{^{(Z)}}C_\gamma^{k}}_{,\mathring{\underline L}}=-2(\leftidx{^{(Z)}}\pi_{LT}+\mu^{-1}\leftidx{^{(Z)}}\pi_{TT})(\mathring L\Psi_\gamma^{k})+\leftidx{^{(Z)}}{\slashed\pi}_{\mathring{\underline L}X}(\slashed d^X\Psi_\gamma^{k})-\f12(\textrm{tr}_\leftidx{^{(Z)}}{\slashed\pi})\mathring{\underline L}\Psi_\gamma^{k},\\
&\mu{\leftidx{^{(Z)}}{\slashed C}_\gamma^{k}}_{,X}=-\f12\leftidx{^{(Z)}}{\slashed\pi}_{\mathring{\underline L}X}(\mathring L\Psi_\gamma^{k})-\f12\leftidx{^{(Z)}}{\slashed\pi}_{\mathring LX}(\mathring{\underline L}\Psi_\gamma^{k})+\leftidx{^{(Z)}}{{\pi}}_{\mathring LT}(\slashed d_X\Psi_\gamma^{k})+\f12\mu\textrm{tr}\leftidx{^{(Z)}}{\slashed\pi}\slashed d_X\Psi_\gamma^{k}.
\end{split}
\end{equation}
A direct substitution of \eqref{C} into \eqref{muC} would result in a lengthy and complicated equation for $\mu\textrm{div}\leftidx{^{(Z)}}C_\gamma^{k}$. To overcome this difficulty and to get the desired estimates efficiently, we follow the ideas in \cite{MY} to decompose $\mu\textrm{div}\leftidx{^{(Z)}}C_\gamma^{k}$ as
\begin{equation}\label{muZC}
\mu\textrm{div}\leftidx{^{(Z)}}C_\gamma^{k}=\leftidx{^{(Z)}}D_{\gamma,1}^k+\leftidx{^{(Z)}}D_{\gamma,2}^k
+\leftidx{^{(Z)}}D_{\gamma,3}^k,
\end{equation}
where
\begin{equation}\label{D1}
\begin{split}
\leftidx{^{(Z)}}D_{\gamma,1}^k=&\f12\textrm{tr}\leftidx{^{(Z)}}{\slashed{\pi}}\big(\mathring L\mathring{\underline L}\Psi_\gamma^k+\f12\textrm{tr}{\chi}\mathring{\underline L}\Psi_\gamma^k\big)
-\leftidx{^{(Z)}}{\slashed\pi}_{\mathring{\underline L}X}(\slashed d^X\mathring L\Psi_\gamma^{k})-\leftidx{^{(Z)}}{\slashed\pi}_{\mathring LX}(\slashed d^X \mathring{\underline L}\Psi_\gamma^{k})\\
&+(\leftidx{^{(Z)}}\pi_{\mathring LT}+\leftidx{^{(Z)}}\pi_{\tilde TT})(\mathring L^2\Psi_\gamma^k)
+\f12\mu\textrm{tr}\leftidx{^{(Z)}}{{\slashed\pi}}\slashed\triangle\Psi_\gamma^k
+\leftidx{^{(Z)}}\pi_{\mathring LT}\slashed\triangle\Psi_\gamma^k,
\end{split}
\end{equation}
\begin{equation}\label{D2}
\begin{split}
\leftidx{^{(Z)}}D_{\gamma,2}^k=&\mathring L(\leftidx{^{(Z)}}\pi_{\mathring LT}+\leftidx{^{(Z)}}\pi_{\tilde TT})\mathring L\Psi_\gamma^k-\big(\f12\slashed\nabla^X\leftidx{^{(Z)}}{\slashed\pi}_{\mathring{\underline L}X}
-\f14\mathring{\underline L}(\textrm{tr}\leftidx{^{(Z)}}{\slashed\pi})\big)\mathring L\Psi_\g^k\\
&-(\f12\slashed{\mathcal L}_{\mathring{\underline L}}\leftidx{^{(Z)}}{\slashed\pi}_{\mathring LX}
-\slashed d_X\leftidx{^{(Z)}}{\pi}_{\mathring LT})\slashed d^X\Psi_\gamma^k+\f12\slashed d\big(\mu\textrm{tr}\leftidx{^{(Z)}}{\slashed\pi}\big)\cdot\slashed d\Psi_\gamma^k\\
&-\f12(\slashed\nabla^X\leftidx{^{(Z)}}{\slashed\pi}_{\mathring LX})\mathring{\underline L}\Psi_\gamma^k
-\f12(\slashed{\mathcal L}_{\mathring L}\leftidx{^{(Z)}}{\slashed\pi}_{\mathring{\underline L}X})\slashed d^X\Psi_\gamma^k
+\f14\mathring L(\textrm{tr}\leftidx{^{(Z)}}{\slashed{\pi}})\mathring{\underline L}\Psi_\gamma^k,\qquad
\end{split}
\end{equation}
\begin{equation}\label{D3}
\begin{split}
\leftidx{^{(Z)}}D_{\gamma,3}^k=&\Big\{\textrm{tr}{\chi}(\leftidx{^{(Z)}}\pi_{\mathring LT}+\leftidx{^{(Z)}}\pi_{\tilde TT})+(\f14\mu\textrm{tr}{\chi}+\f12\mu\textrm{tr}\tilde\theta)\textrm{tr}\leftidx{^{(Z)}}{\slashed{\pi}}-\f12\slashed d^X\mu\leftidx{^{(Z)}}{\slashed\pi}_{\mathring LX}\Big\}\mathring L\Psi_\gamma^k\\
&+\f12\Big\{(\textrm{tr}\leftidx{^{(Z)}}{\slashed{\pi}})({\xi}_X+\mu\zeta_X)+(\mu\textrm{tr}\chi-\mathring L\mu)\leftidx{^{(Z)}}{\slashed\pi}_{\mathring LX}+\textrm{tr}{\chi}\leftidx{^{(Z)}}{\slashed\pi}_{\mathring{\underline L}X}\Big\}\slashed d^X\Psi_\gamma^k.
\end{split}
\end{equation}
Note that all the terms in $\leftidx{^{(Z)}}D_{\gamma,1}^k$ are the products of the deformation tensor and the second
order derivatives of $\Psi_\g^k$, except the first term containing the factor of the form $\mathring L\mathring{\underline L}\Psi_\gamma^k+\f12\textrm{tr}{\chi}\mathring{\underline L}\Psi_\gamma^k$ (see \eqref{D1}). It should be emphasized here that such a structure is crucial in our analysis since $\Psi_\gamma^k$ is the derivative of $\varphi_\gamma$ and by \eqref{fequation}, $\mathring L\mathring{\underline L}\varphi_\gamma+\f12\textrm{tr}{\chi}\mathring{\underline L}\varphi_\gamma=H_\g+\f12\textrm{tr}\check\chi\mathring{\underline L}\varphi_\g$ admits the better smallness and the faster time decay than those for $\mathring L\mathring{\underline L}\varphi_\gamma$ and $\f12\textrm{tr}{\chi}\mathring{\underline L}\varphi_\gamma$ separately. In addition, $\leftidx{^{(Z)}}D_{\gamma,2}^k$
collects all the products of the first order derivatives of the deformation tensor and the first order derivatives of
$\Psi_\g^k$, while $\leftidx{^{(Z)}}D_{\gamma,3}^k$ denotes all the other terms.

The explicit expression for $\Phi_\gamma^{k+1}$ obtained from \eqref{Phik}-\eqref{lamda} and \eqref{muZC}-\eqref{D3} will be used to estimate the corresponding last two integrals in \eqref{e}. Due to the structure of \eqref{e} for $\Psi=\Psi_\gamma^{k+1}$, it is natural to define the corresponding weighted energy and flux as in \cite{MY}:
\begin{align}
E_{i,p+1}(\mathfrak t, u)&=\sum_{\gamma=0}^2\sum_{|\al|=p}\delta^{2l}E_i[Z^{\al}\varphi_\gamma](\mathfrak t, u),\quad i=1,2,\label{ei}\\
F_{i,p+1}(\mathfrak t, u)&=\sum_{\gamma=0}^2\sum_{|\al|=p}\delta^{2l}F_i[Z^{\al}\varphi_\gamma](\mathfrak t, u),\quad i=1,2,\label{fi}\\
E_{i,\leq p+1}(\mathfrak t, u)&=\sum_{0\leq n\leq p}E_{i,n+1}(\mathfrak t, u),\quad i=1,2,\label{eil}\\
F_{i,\leq p+1}(\mathfrak t, u)&=\sum_{0\leq n\leq p}F_{i,n+1}(\mathfrak t, u),\quad i=1,2,\label{fil}
\end{align}
where $l$ is the number of $T$ in $Z^\al$.  We will treat these weighted energies in subsequent sections.

Next, we need to carry out the higher order $L^2$ estimates for some related quantities
so that the last two terms of \eqref{e} can be absorbed by the left hand side, and hence the higher order energy
estimates on \eqref{ge} can be done.
To this end, we now state two elementary lemmas, whose analogous results
can be found in \cite[Lemma 7.3]{MY}, \cite[Lemma 12.57]{J} and \cite[Lemma 8.1, 8.2]{Ding4}.

\begin{lemma}\label{L2T}
For any function $\psi\in C^1(D^{{\mathfrak t}, u})$ vanishing on $C_0$, it holds that  for small $\delta>0$,
\begin{align}
\int_{S_{\mathfrak t, u}} \psi^2&\lesssim\delta\int_{\Sigma_{\mathfrak t}^{u}}\big(|\mathring{\underline L}\psi|^2
+\mu^2|\mathring L\psi|^2\big)\label{SSi},\\
\int_{\Sigma_{\mathfrak t}^{u}} \psi^2&\lesssim\delta^2\int_{\Sigma_{\mathfrak t}^{u}}\big(|\mathring{\underline L}\psi|^2
+\mu^2|\mathring L\psi|^2\big).\label{SiSi}
\end{align}
Therefore,
\begin{align}
\int_{S_{\mathfrak t, u}} \psi^2&\lesssim\delta\big(E_2[\psi](\mathfrak t, u)+\varrho^{-2m}E_1[\psi](\mathfrak t, u)\big),\label{SE}\\
\int_{\Sigma_{\mathfrak t}^{u}} \psi^2&\lesssim\delta^2\big(E_2[\psi](\mathfrak t, u)+\varrho^{-2m}E_1[\psi](\mathfrak t, u)\big).\label{SiE}
\end{align}
Furthermore,
\begin{equation}\label{Em}
E_{1,\leq k+1}(\mathfrak t,u)\lesssim\delta^2\varrho^{2m-2}E_{2,\leq k+2}(\mathfrak t,u)+\delta^2\varrho^{-2}E_{1,\leq k+2}(\mathfrak t,u).
\end{equation}
\end{lemma}

\begin{lemma}\label{L2L}
Assume that $(\star)$ holds for small $\delta>0$. Then for any $f\in C(D^{{\mathfrak t}, u})$,
$
F(\mathfrak t,u,\vartheta):=\int_{t_0}^\mathfrak tf(\tau,u,\vartheta)d\tau
$ admits the following estimate:
\begin{equation}\label{Ff}
\|F\|_{L^2(\Sigma_{\mathfrak t}^u)}\leq(1+C\delta^{1-\varepsilon_0})\sqrt{\varrho({\mathfrak t},u)}
\int_{t_0}^\mathfrak t\f{1}{\sqrt{\varrho(\tau,u)}}
\|f\|_{L^2(\Sigma_{\tau}^u)}d\tau.
\end{equation}
\end{lemma}

Similarly to \cite[Section 8]{Ding4}, one can use Lemma \ref{L2T} and \ref{L2L} to derive $L^2$ estimates for higher order derivatives of
${\check\chi}$, $\slashed\pi_{\mathring LX}$, ${\slashed\pi}_{XX}$, $\check{L}$, $\upsilon$ and $x^j$. The proof are
analogous to those in \cite{Ding4} except different time decay rates, so will be omitted here.

\begin{proposition}\label{Y5}
Under the assumptions $(\star)$ with small $\delta>0$, it holds that for $k\leq 2N-6$,
\begin{align}
&\delta^l\|Z^{k+1}\check L^i\|_{L^2(\Sigma_{\mathfrak t}^{u})}\lesssim\delta^{3/2-\varepsilon_0}
+\mathfrak t^{1-m}\sqrt{\tilde E_{1,\leq k+2}(\mathfrak t,u)}+\delta\sqrt{\tilde E_{2,\leq k+2}(\mathfrak t,u)},\\
&\delta^l\|(\mathfrak t^{-1}Z^{k+2}x^i, \slashed{\mathcal L}^{k+1}_Z\slashed g)\|_{L^2(\Sigma_{\mathfrak t}^{u})}\lesssim\delta^{1/2}\mathfrak t^{1/2}
+\mathfrak t^{1-m}\sqrt{\tilde E_{1,\leq k+2}(\mathfrak t,u)}+\delta \sqrt{\tilde E_{2,\leq k+2}(\mathfrak t,u)},\\
&\delta^l\|(\slashed{\mathcal L}_Z^{k}\check\chi, \mathfrak t^{-2}Z^{k+1}\upsilon)\|_{L^2(\Sigma_{\mathfrak t}^{u})}\lesssim\delta^{3/2-\varepsilon_0}\mathfrak t^{-1}
+\mathfrak t^{-m}\sqrt{\tilde E_{1,\leq k+2}(\mathfrak t,u)}+\delta \mathfrak t^{-1}\sqrt{\tilde E_{2,\leq k+2}(\mathfrak t,u)},\label{Zkchi}\\
&\delta^l\|(Z^{k+1}\mu, \mathfrak t\slashed{\mathcal L}_{Z}^k\leftidx{^{(T)}}{\slashed\pi}_{\mathring L})\|_{L^2(\Sigma_{\mathfrak t}^u)}\lesssim\delta^{3/2-\varepsilon_0}\mathfrak t^{1/2}
+\mathfrak t^{1/2}\sqrt{\tilde E_{1,\leq k+2}(\mathfrak t,u)}+\delta \mathfrak t^{1/2}\sqrt{\tilde E_{2,\leq k+2}(\mathfrak t,u)},\\
&\delta^l\|\slashed{\mathcal L}_{Z}^k\big(\leftidx{^{(R)}}{\slashed\pi},\leftidx{^{(R)}}{\slashed\pi}_{\mathring L},\leftidx{^{(R)}}{\slashed\pi}_{T}\big)\|_{L^2(\Sigma_{\mathfrak t}^u)}
\lesssim\delta^{3/2-\varepsilon_0}+{\mathfrak t}^{1-m}\sqrt{\tilde E_{1,\leq k+2}(\mathfrak t,u)}
+\delta \sqrt{\tilde E_{2,\leq k+2}(\mathfrak t,u)},\\
&\delta^l\|\slashed{\mathcal L}_{Z}^k\leftidx{^{(T)}}{\slashed\pi}\|_{L^2(\Sigma_{\mathfrak t}^u)}\lesssim\delta^{1/2-\varepsilon_0}
+\delta^{-\varepsilon_0}\mathfrak t^{1/2-m}\sqrt{\tilde E_{1,\leq k+2}(\mathfrak t,u)}+\sqrt{\tilde E_{2,\leq k+2}(\mathfrak t,u)},
\end{align}
where $l$ is the number of $T$ in the corresponding derivatives and $\tilde E_{i,\leq k+2}(\mathfrak t,u)=\sup_{t_0\leq\tau\leq {\mathfrak t}}E_{i,\leq k+2}(\tau,u)$ $(i=1,2)$.
\end{proposition}

\section{$L^2$ estimates for the highest order derivatives of tr{$\chi$} and $\slashed\triangle\mu$}\label{L2chimu}

Analogously to \cite[Section 9]{Ding4}, due to \eqref{Psi} and \eqref{D2}, the top orders of derivatives of $\varphi$, $\chi$ and $\mu$ for the energy estimates in \eqref{e} are $2N-4$, $2N-5$ and $2N-4$ respectively. However, as shown in Proposition \ref{Y5}, the $L^2$ estimates for the $(2N-5)^{\text{th}}$ order derivatives of $\chi$ and $(2N-4)^{\text{th}}$ order derivatives of $\mu$ can be controlled by the $(2N-3)$ order energy of $\varphi$. So there is a mismatch here.
To overcome this difficulty, we need to deal with $\textrm{tr}{\chi}$ and $\slashed\triangle\mu$ with the corresponding top order derivatives as in \cite{Ding4}. Since the time decays of these derivatives are crucial, we will give all the  details of analysis for the 2D case here. Furthermore, different from the 4D case in \cite{Ding4},
we also need to improve the $L^2$ estimate of $\slashed{\mathcal L}_Z^k\check\chi$ in \eqref{Zkchi}
so that the last term of \eqref{e} can  be handled despite the slow time decay rate.

\subsection{Estimates for the derivatives of tr{$\chi$}}\label{trchi}

Due to \eqref{Lchi}, tr$\chi$ satisfies a transport equation as
\begin{equation}\label{Ltrchi}
\begin{split}
\mathring L(\textrm{tr}{\chi})=&c^{-1}\slashed dx^a\cdot\slashed d\mathring L\varphi_a
-c^{-1}(\mathring L^2\varphi_0+\varphi_a\mathring L^2\varphi_a)
-c^{-1}\mathring L^\g\slashed\triangle\varphi_\g-\f12c^{-1}(\mathring Lc)\textrm{tr}\chi\\
&-c^{-1}\textrm{tr}\chi(\slashed dx^a\cdot\slashed d\varphi_a)-|{{\chi}}|^2
+c^{-2}f(\slashed dx,\mathring L^1,\mathring L^2,\varphi,\slashed g)
\left(
\begin{array}{ccc}
\slashed d\varphi\cdot\slashed d\varphi\\
(\slashed dx^a\cdot\slashed d\varphi_a)^2\\
(\mathring L\varphi)(\mathring L\varphi)\\
(\slashed dx^a\cdot\slashed d\varphi_a)\mathring L\varphi
\end{array}
\right).
\end{split}
\end{equation}
In addition, \eqref{fequation} is rewritten as
\begin{equation}\label{triphi}
\begin{split}
\mu\slashed\triangle\varphi_\g=\mathring L\mathring{\underline L}\varphi_\g
+\f{1}{2\varrho}\mathring {\underline L}\varphi_\g-\tilde{H}_\g,
\end{split}
\end{equation}
where $\tilde{H}_\g=H_\g-\mu\slashed\triangle\varphi_\g$.
Then \eqref{Ltrchi} can be written as
\begin{equation}\label{Ltr}
\begin{split}
\mathring L\big(\textrm{tr}\chi-E)=\big(-\f{2}{\varrho}-c^{-1}\mathring Lc+c^{-1}\mathring L^\beta\mathring L\varphi_\beta\big)\textrm{tr}\chi+\f{1}{\varrho^2}-|\check\chi|^2+e,
\end{split}
\end{equation}
which contains only the first or zeroth order derivatives of $\varphi_\gamma$ on the right hand side, where
\begin{align}
E=&c^{-1}\slashed dx^a\cdot\slashed d\varphi_a-\f32c^{-1}\mathring Lc
+c^{-1}\mathring L^\beta\mathring L\varphi_\beta,\label{phiE}\\
e=&c^{-2}f(\slashed dx,\mathring L^1,\mathring L^2,\varphi,\slashed g)
\left(
\begin{array}{ccc}
\slashed d\varphi\cdot\slashed d\varphi\\
(\slashed dx^a\cdot\slashed d\varphi_a)^2\\
(\mathring L\varphi)(\mathring L\varphi)\\
(\slashed dx^a\cdot\slashed d\varphi_a)\mathring L\varphi
\end{array}
\right).\label{phie}
\end{align}
Set $F^k=\slashed dZ^k\textrm{tr}\chi-\slashed dZ^k E$ with $Z\in\{\varrho\mathring L,T,R\}$. It then
follows from \eqref{Ltr} inductively that
\begin{equation}\label{LFal}
\begin{split}
\slashed{\mathcal L}_{\mathring L}F^k=&\big(-\f{2}{\varrho}-c^{-1}\mathring Lc
+c^{-1}\mathring L^\beta\mathring L\varphi_\beta\big)F^k\\
&+\big(-\f{2}{\varrho}-c^{-1}\mathring Lc+c^{-1}\mathring L^\beta\mathring L\varphi_\beta\big)\slashed dZ^k E
-\slashed dZ^{k}(|\check\chi|^2)+e^k,
\end{split}
\end{equation}
where for $k\geq 1$,
\begin{equation}\label{eal}
\begin{split}
e^k=&\slashed{\mathcal L}_{Z}^k e^0+\sum_{k_1+k_2=k-1}\slashed{\mathcal L}_{Z}^{k_1}\slashed{\mathcal L}_{[\mathring L,Z]}F^{k_2}\\
&+\sum_{k_1\leq k-1}Z^{k-k_1}\big(-\f{2}{\varrho}-c^{-1}\mathring Lc
+c^{-1}\mathring L^\beta\mathring L\varphi_\beta\big)\slashed dZ^{k_1}\textrm{tr}\chi
\end{split}
\end{equation}
and
\begin{equation}\label{e0}
\begin{split}
e^0=&\slashed de+\slashed d\big(-c^{-1}\mathring Lc
+c^{-1}\mathring L^\beta\mathring L\varphi_\beta\big)\textrm{tr}\chi.
\end{split}
\end{equation}
Note that for any one-form $\kappa$ on $S_{{\mathfrak t},u}$, it holds that
\begin{equation}\label{Lxi}
\mathring L(\varrho^2|\kappa|^2)=-2\varrho^2\textrm{tr}\check{\chi}|\kappa|^2
+2\varrho^2(\textrm{tr}\slashed{\mathcal L}_{\mathring L}\kappa)|\kappa|.
\end{equation}
Then choosing $\kappa=\varrho^2F^k$ in \eqref{Lxi} and using \eqref{LFal} lead to
\begin{equation*}
\begin{split}
\mathring L\big(\varrho^6|F^k|^2\big)=&\varrho^6\big\{-2\textrm{tr}\check\chi|F^k|^2
+2\big(-c^{-1}\mathring Lc+c^{-1}\mathring L^\beta\mathring L\varphi_\beta\big)|F^k|^2+2e^k\cdot F^k\\
&+2\big(-\f{2}{\varrho}-c^{-1}\mathring Lc
+c^{-1}\mathring L^\beta\mathring L\varphi_\beta\big)\slashed d Z^k E\cdot F^k
-2\slashed dZ^k(|\check\chi|^2)\cdot F^k\big\}.
\end{split}
\end{equation*}
This yields,
\begin{equation}\label{LrhoF}
\begin{split}
|\mathring L(\varrho^3|F^k|)|\lesssim&\varrho^3|\check\chi|\cdot|F^k|+\varrho^3|
-c^{-1}\mathring Lc+c^{-1}\mathring L^\g\mathring L\varphi_\g|\cdot|F^k|+\varrho^3|
-\f2\varrho-c^{-1}\mathring Lc\\
&+c^{-1}\mathring L^\g\mathring L\varphi_\g|\cdot|\slashed dZ^k E|
+\varrho^3|\slashed dZ^k(|\check\chi|^2)|+\varrho^3|e^k|.
\end{split}
\end{equation}
Then it follows from this and \eqref{Ff} that
\begin{equation*}
\begin{split}
\delta^l\varrho^3\|F^k\|_{L^2(\Sigma_{\mathfrak t}^{u})}\lesssim&\delta^l\|F^k(t_0,\cdot,\cdot)\|_{L^2(\Sigma_{\mathfrak t}^{u})}
+\delta^l\varrho^{1/2}\int_{t_0}^\mathfrak t\Big\{\tau^{5/2}\|\check\chi\|_{L^\infty(\Sigma_\tau^u)}\cdot\|F^k\|_{L^2(\Sigma_{\tau}^{u})}\\
&+\tau^{5/2}\|-c^{-1}\mathring Lc+c^{-1}\mathring L^\gamma\mathring L\varphi_\gamma\|_{L^\infty(\Sigma_\tau^u)}\cdot\|F^k\|_{L^2(\Sigma_{\tau}^{u})}\\
&+\tau^{5/2}\|-\f{2}{\varrho}-c^{-1}\mathring Lc
+c^{-1}\mathring L^\gamma\mathring L\varphi_\gamma\|_{L^\infty(\Sigma_\tau^u)}\cdot\|\slashed dZ^k E\|_{L^2(\Sigma_{\tau}^{u})}\\
&+\tau^{5/2}\|\slashed dZ^k(|\check\chi|^2)\|_{L^2(\Sigma_{\tau}^u)}+\tau^{5/2}\|e^k\|_{L^2(\Sigma_{\tau}^u)}\Big\}d\tau.
\end{split}
\end{equation*}
Thus, the Gronwall's inequality yields
\begin{equation}\label{Eal}
\begin{split}
\delta^l\varrho^{5/2}\|F
^k\|_{L^2(\Sigma_{\mathfrak t}^{u})}\lesssim\delta^{3/2-\varepsilon_0}+\delta^l\int_{t_0}^\mathfrak t\Big\{&\tau^{3/2}\|\slashed dZ^k E\|_{L^2(\Sigma_{\tau}^{u})}+\tau^{5/2}\|\slashed dZ^k(|\check\chi|^{2})\|_{L^2(\Sigma_{\tau}^u)}\\
&+\tau^{5/2}\|e^k\|_{L^2(\Sigma_{\tau}^u)}\Big\}{d}\tau.
\end{split}
\end{equation}

Each term in the integrand of \eqref{Eal} will be estimated as follows.

First, due to $E=c^{-1}\slashed dx^a\cdot\slashed d\varphi_a-\f32c^{-1}\mathring Lc
+c^{-1}\mathring L^\beta\mathring L\varphi_\beta$, then
\begin{equation}\label{EL2}
\begin{split}
\delta^l\|\slashed dZ^k E\|_{L^2(\Sigma_{\mathfrak t}^{u})}\lesssim\delta^{3/2-\varepsilon_0}\mathfrak t^{-2}
+\mathfrak t^{-1-m}\sqrt{\tilde E_{1,\leq k+2}(\mathfrak t,u)}
+\delta^{2-\varepsilon_0}\mathfrak t^{-5/2}\sqrt{\tilde E_{2,\leq k+2}(\mathfrak t,u)},
\end{split}
\end{equation}
where one has used $L^\infty$ estimates in Section \ref{BA}, Proposition \ref{Y5} and Lemma \ref{L2T}.

Next, direct computations give
\begin{equation}\label{dZk}
\begin{split}
\slashed dZ^k(|\check{\chi}|^2)=&2\textrm{tr}\check\chi(\slashed dZ^k\textrm{tr}\chi)+2\sum_{k_1+k_2=k,\ k_1\leq k-1}Z^{k_2}\textrm{tr}\check\chi(\slashed dZ^{k_1}\textrm{tr}\check\chi)\\
=&2\textrm{tr}\check\chi(F^k+\slashed dZ^kE)+2\sum_{k_1+k_2=k,\ k_1\leq k-1}Z^{k_2}\textrm{tr}\check\chi(\slashed dZ^{k_1}\textrm{tr}\check\chi).
\end{split}
\end{equation}
Taking $L^2$ norm of \eqref{dZk} on the surfaces $\Sigma_{\mathfrak t}^u$ directly, and
applying the estimates in Section \ref{BA} and Section \ref{EE} to handle the lower and higher order derivatives respectively,
one can deduce
\begin{equation}\label{Y-19}
\begin{split}
\delta^l\|\slashed dZ^k(|\check{\chi}|^2)\|_{L^2(\Sigma_{\mathfrak t}^u)}\lesssim&\delta^{1-\varepsilon_0}\mathfrak t^{-3/2}\delta^l\|F^k\|_{L^2(\Sigma_{\mathfrak t}^u)}
+\delta^{5/2-2\varepsilon_0}\mathfrak t^{-7/2}\\
&+\delta^{1-\varepsilon_0}\mathfrak t^{-5/2-m}\sqrt{\tilde E_{1,\leq k+2}(\mathfrak t,u)}
+\delta^{2-\varepsilon_0}\mathfrak t^{-7/2}\sqrt{\tilde E_{2,\leq k+2}(\mathfrak t,u)}.
\end{split}
\end{equation}

Finally, it remains to treat $e^k$.
One starts with the term $\slashed{\mathcal L}_Z^{k_1}\slashed{\mathcal L}_{[\mathring L,Z]}F^{k_2}$
in \eqref{eal} with $k_1+k_2=k-1$. If $Z\in\{R,T\}$, then $[\mathring L, Z]=\leftidx^{{(Z)}}\slashed\pi_{\mathring L}^XX$
by \eqref{c}, and therefore, $\slashed{\mathcal L}_{[\mathring L,Z]}F^{k_2}=\leftidx^{{(Z)}}{\slashed\pi_{\mathring L}}\cdot\slashed\nabla F^{k_2}+F^{k_2}\cdot\slashed\nabla\leftidx^{{(Z)}}{\slashed\pi_{\mathring L}}$. This leads to that by Proposition \ref{Y5},
\begin{equation}\label{LF}
\begin{split}
\delta^l\|\slashed{\mathcal L}_Z^{k_1}\slashed{\mathcal L}_{[\mathring L,Z]}F^{k_2}
&\|_{L^2(\Sigma_{\mathfrak t}^u)}\lesssim\delta^{1-\varepsilon_0} \mathfrak t^{-3/2}\delta^l\|F^k\|_{L^2(\Sigma_{\mathfrak t}^u)}
+\delta^{5/2-2\varepsilon_0}\mathfrak t^{-7/2}\\
&+\delta^{1-\varepsilon_0} \mathfrak t^{-5/2-m}\sqrt{\tilde E_{1,\leq k+2}(\mathfrak t,u)}
+\delta^{2-\varepsilon_0}\mathfrak t^{-7/2}\sqrt{\tilde E_{2,\leq k+2}(\mathfrak t,u)}.
\end{split}
\end{equation}
If $Z=\varrho\mathring L$, then by $[\mathring L,Z]=\mathring L$ and \eqref{LFal},
\begin{equation}\label{Y-20}
\slashed{\mathcal L}_Z^{k_1}\slashed{\mathcal L}_{[\mathring L,Z]}F^{k_2}
=\slashed{\mathcal L}_Z^{k_1}\big\{(-\f2\varrho-c^{-1}\mathring Lc+c^{-1}\mathring L^\g\mathring L\varphi_\g)\slashed dZ^{k_2}\textrm{tr}\chi-\slashed dZ^{k_2}(|\check\chi|^2)+e^{k_2}\big\}.
\end{equation}
Note that the estimate of $\slashed{\mathcal L}_Z^{k_1}\slashed dZ^{k_2}(|\check\chi|^2)$
in \eqref{Y-20} follows from Proposition \ref{Y5} immediately as,
\begin{equation}\label{dchi2}
\begin{split}
\delta^l\|\slashed dZ^{k-1}(|\check{\chi}|^2)\|_{L^2(\Sigma_{\mathfrak t}^u)}\lesssim&\delta^{5/2-2\varepsilon_0}\mathfrak t^{-7/2}
+\delta^{1-\varepsilon_0}\mathfrak t^{-5/2-m}\sqrt{\tilde E_{1,\leq k+2}(\mathfrak t,u)}\\
&+\delta^{2-\varepsilon_0}\mathfrak t^{-7/2}\sqrt{\tilde E_{2,\leq k+2}(\mathfrak t,u)}.
\end{split}
\end{equation}
The other two terms in \eqref{Y-20} can be estimated similarly as for the remaining terms in \eqref{eal} as below.

Indeed, for the term $Z^{k-k_1}(-\f2\varrho-c^{-1}\mathring Lc
+c^{-1}\mathring L^\beta\mathring L\varphi_\beta)\slashed dZ^{k_1}\textrm{tr}\chi$
in \eqref{eal} with $k_1\leq k-1$, Proposition \ref{Y5} implies
\begin{equation}\label{Zrho}
\begin{split}
&\delta^l\|Z^{k-k_1}(-\f2\varrho-c^{-1}\mathring Lc+c^{-1}\mathring L^\beta\mathring L\varphi_\beta)\slashed dZ^{k_1}\textrm{tr}\chi\|_{L^2(\Sigma_{\mathfrak t}^u)}\\
\lesssim&\delta^{3/2-\varepsilon_0}\mathfrak t^{-3}+\mathfrak t^{-2-m}\sqrt{\tilde{E}_{1,\leq k+2}(\mathfrak t,u)}
+\delta \mathfrak t^{-3}\sqrt{\tilde E_{2,\leq k+2}(\mathfrak t,u)}.
\end{split}
\end{equation}

To estimate $\slashed{\mathcal L}_Z^k e^0$ in \eqref{eal},
by \eqref{e0}, \eqref{phie} and Proposition \ref{Y5}, one has
\begin{equation}\label{Le0}
\begin{split}
\delta^l\|\slashed{\mathcal L}_Z^k e^0\|_{L^2(\Sigma_{\mathfrak t}^u)}\lesssim&\delta^{3/2-\varepsilon_0}\mathfrak t^{-3}
+\mathfrak t^{-2-m}\sqrt{\tilde E_{1,\leq k+2}(\mathfrak t,u)}s+\delta^{2-\varepsilon_0}\mathfrak t^{-7/2}\sqrt{\tilde E_{2,\leq k+2}(\mathfrak t,u)}.
\end{split}
\end{equation}

Combining the estimates \eqref{LF}-\eqref{Le0} and \eqref{eal} yields
\begin{equation}\label{Y-21}
\begin{split}
\delta^l\|e^k\|_{L^2(\Sigma_{\mathfrak t}^u)}\lesssim&\delta^{1-\varepsilon_0}\mathfrak t^{-3/2}\delta^l\|F^k\||_{L^2(\Sigma_{\mathfrak t}^u)}d\tau
+\delta^{3/2-\varepsilon_0}\mathfrak t^{-3}+\mathfrak t^{-2-m}\sqrt{\tilde E_{1,\leq k+2}(\mathfrak t,u)}\\
&+\delta \mathfrak t^{-3}\sqrt{\tilde E_{2,\leq k+2}(\mathfrak t,u)}.
\end{split}
\end{equation}

By inserting \eqref{EL2}, \eqref{Y-19} and \eqref{Y-21} into \eqref{Eal}, one obtains from \eqref{Ltr} and \eqref{phiE} that
\begin{equation}\label{Fal}
\begin{split}
\delta^l\|F^k\|_{L^2(\Sigma_{\mathfrak t}^u)}\lesssim&\delta^{3/2-\varepsilon_0}\mathfrak t^{-2}+\mathfrak t^{-1-m}\sqrt{\tilde E_{1,\leq k+2}(\mathfrak t,u)}
+\delta \mathfrak t^{-2}\sqrt{\tilde E_{2,\leq k+2}(\mathfrak t,u)}.
\end{split}\end{equation}
In addition, due to the definition of $F^k$, $\slashed dZ^k\textrm{tr}\chi=F^k+\slashed d Z^k E$, and hence it holds that
\begin{equation}\label{d}
\begin{split}
&\delta^l\|\slashed dZ^k\textrm{tr}\chi\|_{L^2(\Sigma_{\mathfrak t}^u)}+\delta^l\|\slashed{\textrm{div}}{\slashed{\mathcal L}}_{Z}^k\check{\chi}\|_{L^2(\Sigma_{\mathfrak t}^u)}\\
\lesssim&\delta^{3/2-\varepsilon_0}\mathfrak t^{-2}+\mathfrak t^{-1-m}\sqrt{\tilde E_{1,\leq k+2}(\mathfrak t,u)}
+\delta \mathfrak t^{-2}\sqrt{\tilde E_{2,\leq k+2}(\mathfrak t,u)},
\end{split}
\end{equation}
where one has used \eqref{Fal} and \eqref{EL2}.

\begin{remark}
Since $\slashed dZ^k\textrm{tr}\check\chi=\slashed dZ^k\textrm{tr}\chi$
by \eqref{errorv}, \eqref{d} gives also the $L^2$ estimate for $\slashed dZ^k\textrm{tr}\check\chi$.
\end{remark}

\subsection{Estimates for the derivatives of $\slashed\triangle\mu$}\label{trimu}
As in Subsection \ref{trchi}, one can use \eqref{lmu} to estimate $\slashed\triangle\mu$.
Indeed, it follows from $[\mathring   L,\slashed\triangle]\mu=-2(\textrm{tr}\check\chi)\slashed\triangle\mu
-2\varrho^{-1}\slashed\triangle\mu-\slashed d\textrm{tr}\check\chi\cdot\slashed d\mu$ due to Lemma \ref{commute}
and \eqref{lmu} that
\begin{equation}\label{Ltrimu}
\begin{split}
\mathring L\slashed\triangle\mu=&[\mathring L,\slashed\triangle]\mu+\slashed\triangle\mathring L\mu\\
=&-2(\textrm{tr}\check\chi)\slashed\triangle\mu-\f2\varrho\slashed\triangle\mu
-{\slashed d\textrm{tr}\check\chi}\cdot\slashed d\mu+(\slashed\triangle\mu)c^{-1}(\mathring L^\al\mathring L\varphi_\al
-\mathring Lc)+2\slashed d\mu\cdot\slashed d(c^{-1}\mathring L^\al\mathring L\varphi_\al\\
&-c^{-1}\mathring Lc)+\mu\big\{\uwave{\slashed\triangle(c^{-1}\mathring L^\al)}\mathring L\varphi_\al
+2\slashed d(c^{-1}\mathring L^\al)\cdot\slashed d\mathring L\varphi_\al\big\}
+\uline{\mathring L(c^{-1}\mu\mathring L^\al\slashed\triangle\varphi_\al)}\\
&-\mathring L(c^{-1}\mu\mathring L^\al)\slashed\triangle\varphi_\al
-c^{-1}\mu\mathring L^\al[\mathring L,\slashed\triangle]\varphi_\al+\mu[\mathring L,\slashed\triangle]\ln c
-\underline{\mathring L(\mu\slashed\triangle\ln c)}+(\mathring L\mu)\slashed\triangle\ln c.
\end{split}
\end{equation}
To estimate the term with wavy line in \eqref{Ltrimu}, one notes that \eqref{deL} implies
\begin{equation}\label{Y-22}
\begin{split}
\slashed\triangle\check L^a=&\slashed d\textrm{tr}\chi\cdot\slashed dx^a
+\textrm{tr}\check\chi\slashed\triangle x^a+\slashed\nabla^X\big\{\f12c^{-1}(\mathring Lc)\slashed d_Xx^a
-\f12c^{-1}(\slashed d_Xc)\tilde T^a+c^{-1}\tilde T^b(\slashed d_X\varphi_b)\tilde T^a\big\},
\end{split}
\end{equation}
and the $L^2$ norm of each term in \eqref{Y-22}
can be estimated by \eqref{d}.

Observe that the two terms with the underline in \eqref{Ltrimu} are both derivatives with respect
to $\mathring L$, and then can be moved to the left hand side of \eqref{Ltrimu}. Let $\tilde E=c^{-1}\mu\mathring L^\al\slashed\triangle\varphi_\al-\mu\slashed\triangle\ln c$ and $\tilde F=\slashed\triangle\mu-\tilde E$.
Then \eqref{Ltrimu} can be rewritten as
\begin{equation*}
\begin{split}
\mathring L\tilde F=&(-2\textrm{tr}\check\chi-2\varrho^{-1}+c^{-1}\mathring L^\al\mathring L\varphi_\al
-c^{-1}\mathring Lc)\tilde F\\
&-\slashed d\textrm{tr}\chi\cdot(\slashed d\mu-c^{-1}\mu\mathring L\varphi_a\slashed dx^a
-c^{-1}\mu\mathring L^\al\slashed d\varphi_\al+\mu\slashed d\ln c)+\tilde e.
\end{split}
\end{equation*}
Let $\bar Z$ be any vector field in $\{T,R\}$ and set $\bar F^k=\bar Z^k\slashed\triangle\mu-\bar Z^k\tilde E$.
Then one can get by induction as for \eqref{LFal} that
\begin{equation}\label{LbarF}
\begin{split}
\mathring L\bar F^k=&(-2\textrm{tr}\check\chi-2\varrho^{-1}+c^{-1}\mathring L^\gamma\mathring L\varphi_\gamma
-c^{-1}\mathring Lc)\bar F^k\\
&-\slashed d\bar Z^k\textrm{tr}\chi\cdot(\slashed d\mu-c^{-1}\mu\mathring L\varphi_a\slashed dx^a
-c^{-1}\mu\mathring L^\al\slashed d\varphi_\al+\mu\slashed d\ln c)+[\mathring L,\bar Z]^X\slashed d_X\bar F^{k-1}
+\bar e^k
\end{split}
\end{equation}
with
\begin{equation*}
\begin{split}
\bar e^k=&\sum_{{\mbox{\tiny$\begin{array}{c}k_1+k_2=k-1\\|k_1|\geq 1\end{array}$}}}\slashed{\mathcal L}_{\bar Z}^{k_1}[\mathring L,\bar Z]^X\slashed d_X\bar F^{k_2}+\bar Z^{k}\tilde e+\sum_{{\mbox{\tiny$\begin{array}{c}k_1+k_2=k\\|k_1|\geq 1\end{array}$}}}\Big\{\bar Z^{k_1}(-2\textrm{tr}\check\chi-2\varrho^{-1}+c^{-1}\mathring L^\gamma\mathring L\varphi_\gamma
-c^{-1}\mathring Lc)\bar F^{k_2}\\
&-(\slashed d\bar Z^{k_2}\textrm{tr}\chi)\cdot\slashed{\mathcal L}_{\bar Z}^{k_1}(\slashed d\mu
-c^{-1}\mu\mathring L\varphi_a\slashed dx^a-c^{-1}\mu\mathring L^\al\slashed d\varphi_\al+\mu\slashed d\ln c)\Big\},
\end{split}
\end{equation*}
where the first sum on the right hand side above vanishes when $k=1$. Thus, applying \eqref{Ff} to $F(\mathfrak t,u,\vartheta)=\varrho^2\bar F^k(\mathfrak t,u,\vartheta)-\varrho_0^2\bar F^k(t_0,u,\vartheta)$
and using \eqref{LbarF} lead to
\begin{equation}\label{tildeF}
\begin{split}
\delta^l\varrho^{3/2}\|\bar F^k\|_{L^2(\Sigma_{\mathfrak t}^u)}
\lesssim&\delta^{3/2-\varepsilon_0}+\int_{t_0}^\mathfrak t\delta^{1-\varepsilon_0}\delta^l\|\bar F^k\|_{L^2(\Sigma_\tau^u)}d\tau\\
&+\int_{t_0}^\mathfrak t\delta^{1+l-\varepsilon_0}\tau^{1/2}\|\slashed d\bar Z^k\textrm{tr}\chi\|_{L^2(\Sigma_\tau^u)}d\tau+\int_{t_0}^\mathfrak t\tau^{3/2}\delta^l\|\bar e^k\|_{L^2(\Sigma_\tau^u)}d\tau.
\end{split}
\end{equation}
It follows from  \eqref{d} and Proposition \ref{Y5} that
\begin{equation}\label{Y-23}
\begin{split}
&\int_{t_0}^\mathfrak t\delta^{1+l-\varepsilon_0}\tau^{1/2}\|\slashed d\bar Z^k\textrm{tr}\chi\|_{L^2(\Sigma_\tau^u)}d\tau+\int_{t_0}^\mathfrak t\tau^{3/2}\delta^l\|\bar e^k\|_{L^2(\Sigma_\tau^u)}d\tau\\
\lesssim&\delta^{3/2-\varepsilon_0}+\int_{t_0}^\mathfrak t\biggl(\tau^{-1/2-m}\sqrt{\tilde E_{1,\leq k+2}(\tau,u)}
+\delta\tau^{-3/2}\sqrt{\tilde E_{2,\leq k+2}(\tau,u)}\biggr)d\tau\\
\lesssim&\delta^{3/2-\varepsilon_0}+\sqrt{\tilde E_{1,\leq k+2}(\mathfrak t,u)}
+\delta\sqrt{\tilde E_{2,\leq k+2}(\mathfrak t,u)}.
\end{split}
\end{equation}
Inserting \eqref{Y-23} into \eqref{tildeF} and appying the Gronwall's inequality yield
\begin{equation*}
\begin{split}
\delta^l\varrho^{3/2}\|\bar F^k\|_{L^2(\Sigma_{\mathfrak t}^u)}\lesssim&\delta^{3/2-\varepsilon_0}
+\sqrt{\tilde E_{1,\leq k+2}(\mathfrak t,u)}+\delta\sqrt{\tilde E_{2,\leq k+2}(\mathfrak t,u)},
\end{split}
\end{equation*}
and hence,
\begin{equation*}
\begin{split}
\delta^l\|\bar Z^k\slashed\triangle\mu\|_{L^2(\Sigma_{\mathfrak t}^u)}\lesssim&\delta^{3/2-\varepsilon_0}\mathfrak t^{-3/2}
+\mathfrak t^{-3/2}\sqrt{\tilde E_{1,\leq k+2}(\mathfrak t,u)}+\delta \mathfrak t^{-3/2}\sqrt{\tilde E_{2,\leq k+2}(\mathfrak t,u)}.
\end{split}
\end{equation*}

The other cases, which contain at least one $\varrho\mathring L$ in $Z^k$, can be treated by using \eqref{lmu} and commutators $[\varrho\mathring L,\bar Z]$ and $[\varrho\mathring L,\slashed\triangle]$.
Therefore, we eventually arrive at
\begin{equation}\label{Zmu}
\begin{split}
&\delta^l\|Z^k\slashed\triangle\mu\|_{L^2(\Sigma_{\mathfrak t}^u)}
\lesssim\delta^{3/2-\varepsilon_0}\mathfrak t^{-3/2}+\mathfrak t^{-3/2}\sqrt{\tilde E_{1,\leq k+2}(\mathfrak t,u)}
+\delta \mathfrak t^{-3/2}\sqrt{\tilde E_{2,\leq k+2}(\mathfrak t,u)}.
\end{split}
\end{equation}

At the end of this section, we are going to improve the $L^2$ estimate of $\slashed{\mathcal L}_Z^k\check\chi$
in \eqref{Zkchi}, which
was obtained by integrating $\slashed{\mathcal L}_{\mathring L}\slashed{\mathcal L}_Z^k\check\chi$ along integral curves
of $\mathring L$. Such an approach caused losses of time decay of some related terms. To avoid such a difficulty, we
now use \eqref{d} and \eqref{Zmu} and carry out $L^2$-estimate directly to the equation of $\check\chi$ under actions of different
vector fields.

\begin{corollary}
Under the assumptions $(\star)$ with $\delta>0$ small, it then holds that for $k\leq 2N-6$,
\begin{equation}\label{uZkchi}
\delta^{l}\|Z^k\textrm{tr}\check\chi\|_{L^2(\Sigma_{\mathfrak t}^{u})}\lesssim\delta^{3/2-\varepsilon_0}\mathfrak t^{-1}
+\delta \mathfrak t^{-1-m}\sqrt{\tilde E_{1,\leq k+2}(\mathfrak t,u)}+\delta \mathfrak t^{-1}\sqrt{\tilde E_{2,\leq k+2}(\mathfrak t,u)},
\end{equation}
where $l$ is the number of $T$ in the corresponding derivatives. Furthermore,
\begin{equation}\label{ZTchi}
\delta^{l}\|(\mathfrak t Z^k\mathring L\textrm{tr}\check\chi, \dl Z^{k}T\textrm{tr}\check\chi)\|_{L^2(\Sigma_{\mathfrak t}^{u})}\lesssim\delta^{3/2-\varepsilon_0}\mathfrak t^{-1}
+\mathfrak t^{-m}\sqrt{\tilde E_{1,\leq k+2}(\mathfrak t,u)}+\delta \mathfrak t^{-1}\sqrt{\tilde E_{2,\leq k+2}(\mathfrak t,u)}.
\end{equation}
\end{corollary}

\begin{proof} We first derive \eqref{uZkchi}.
Without loss of generality, $k\geq 1$ is assumed.
$Z^{k-1}(\varrho\mathring L)\textrm{tr}\check\chi$, $Z^{k-1}T\textrm{tr}\check\chi$
and $Z^{k-1}R\textrm{tr}\check\chi$ will be treated separately as follows.

First, it follows from \eqref{Lchi'} that
\begin{equation}\label{rLchi}
\begin{split}
\delta^l|Z^{k-1}(\varrho\mathring L)\textrm{tr}\check\chi|\lesssim&\delta^{l_1}|Z^{n_1}\textrm{tr}\check\chi|
+\delta^{1-\varepsilon_0}\mathfrak t^{-5/2}\delta^{l_2}|Z^{n_2}x|+\delta^{l_2}|\mathring LZ^{n_2}\varphi|\\
&+\delta^{l_2}|\slashed dZ^{n_2}\varphi|+\mathfrak t^{-1}\delta^{l_2}|Z^{n_2}\varphi|
+\delta^{1-\varepsilon_0}\mathfrak t^{-3/2}\delta^{l_1}|Z^{n_1}\check L^i|\\
&+\delta^{1-\varepsilon_0}\mathfrak t^{-3/2}\delta^{l_0}|\slashed{\mathcal L}_Z^{n_0}\leftidx{^{(R)}}{\slashed\pi_{\mathring L}}|
+\delta^{2-\varepsilon_0}\mathfrak t^{-3/2}\delta^{l_0}|\slashed{\mathcal L}_Z^{n_0}\leftidx{^{(T)}}{\slashed\pi_{\mathring L}}|,
\end{split}
\end{equation}
where $l_i$ is the number of $T$ in $Z^{n_i}$ $(i=0,1,2)$, and $n_i\leq k-2+i$. \eqref{rLchi} and Proposition \ref{Y5} imply
\begin{equation}\label{ZrLch}
\begin{split}
\delta^l\|Z^{k-1}(\varrho\mathring L)\textrm{tr}\check\chi\|_{L^2(\Sigma_{\mathfrak t}^{u})}
\lesssim\delta^{3/2-\varepsilon_0}\mathfrak t^{-1}+\mathfrak t^{-m}\sqrt{\tilde E_{1,\leq k+1}(\mathfrak t,u)}
+\delta \mathfrak t^{-1}\sqrt{\tilde E_{2,\leq k+1}(\mathfrak t,u)},
\end{split}
\end{equation}
which, together with \eqref{Em}, yields
\begin{equation}\label{Zk-1L}
\delta^l\|Z^{k-1}(\varrho\mathring L)\textrm{tr}\check\chi\|_{L^2(\Sigma_{\mathfrak t}^{u})}\lesssim\delta^{3/2-\varepsilon_0}\mathfrak t^{-1}
+\delta \mathfrak t^{-1-m}\sqrt{\tilde E_{1,\leq k+2}(\mathfrak t,u)}+\delta \mathfrak t^{-1}\sqrt{\tilde E_{2,\leq k+2}(\mathfrak t,u)}.
\end{equation}

Next, thanks to \eqref{Tchi'}, one has
\begin{equation}\label{8.37}
\begin{split}
\delta^{l+1}|Z^{k-1}T\textrm{tr}\check\chi|\lesssim&
\delta^{2-\varepsilon_0}\mathfrak t^{-3/2}\delta^{l_1}|Z^{n_1}\textrm{tr}\leftidx{^{(T)}}{\slashed\pi}|
+\delta^{1-\varepsilon_0}\mathfrak t^{-1/2}\delta^{1_1}|Z^{n_1}\textrm{tr}\check\chi|
+\delta^{l_1+1}|Z^{n_1}\slashed\triangle\mu|\\
&+\delta^{1-\varepsilon_0}\mathfrak t^{-5/2}\delta^{l_2}|Z^{n_2}x|
+\delta^{l_2}|\slashed dZ^{n_2}\varphi|+\delta^{l_2}|\mathring LZ^{n_2}\varphi|
+\mathfrak t^{-1}\delta^{l_2}|Z^{n_2}\varphi|\\
&+\delta^{1-\varepsilon_0}\mathfrak t^{-3/2}\delta^{l_1}|Z^{n_1}\check L^i|+\delta t^{-2}\delta^{l_2}|Z^{n_2}\mu|
+\delta^{1-\varepsilon_0}\mathfrak t^{-3/2}\delta^{l_0}|\slashed{\mathcal L}_Z^{n_0}\leftidx{^{(R)}}{\slashed\pi_{\mathring L}}|\\
&+\delta^{2-\varepsilon_0}\mathfrak t^{-3/2}\delta^{l_0}|\slashed{\mathcal L}_Z^{n_0}\leftidx{^{(T)}}{\slashed\pi_{\mathring L}}|,
\end{split}
\end{equation}
where the number of $T$ in $Z^{k-1}$ is $l$, $l_i$ and $n_i$ $(i=0,1,2)$ are given as in \eqref{rLchi}.
One now applies \eqref{Zmu} to estimate $Z^{n_1}\slashed\triangle\mu$ and uses Proposition \ref{Y5} to
handle the other terms in \eqref{8.37}. This leads to
\begin{equation}\label{ZTch}
\delta^{l+1}\|Z^{k-1}T\textrm{tr}\check\chi\|_{L^2(\Sigma_{\mathfrak t}^{u})}\lesssim\delta^{3/2-\varepsilon_0}\mathfrak t^{-1}
+\mathfrak t^{-m}\sqrt{\tilde E_{1,\leq k+1}(\mathfrak t,u)}+\delta \mathfrak t^{-1}\sqrt{E_{2,\leq k+1}(\mathfrak t,u)}.
\end{equation}
Thus
\begin{equation}\label{Zk-1T}
\delta^{l+1}\|Z^{k-1}T\textrm{tr}\check\chi\|_{L^2(\Sigma_{\mathfrak t}^{u})}
\lesssim\delta^{3/2-\varepsilon_0}\mathfrak t^{-1}+\delta \mathfrak t^{-1-m}\sqrt{\tilde E_{1,\leq k+2}(\mathfrak t,u)}
+\delta \mathfrak t^{-1}\sqrt{E_{2,\leq k+2}(\mathfrak t,u)}.
\end{equation}

Finally, due to $Z^{k-1}R\textrm{tr}\check\chi=[Z^{k-1}, R]\textrm{tr}\check\chi+RZ^{k-1}\textrm{tr}\check\chi$, then it holds that $$\delta^l|Z^{k-1}R\textrm{tr}\check\chi|\lesssim\delta^{1-\varepsilon_0}\mathfrak t^{-3/2}\delta^{l_0}|\slashed{\mathcal L}_Z^{n_0}\leftidx{^{(R)}}{\slashed\pi_{\mathring L}}|+\delta^{2-\varepsilon_0}\mathfrak t^{-5/2}\delta^{l_0}|\slashed{\mathcal L}_Z^{n_0}\leftidx{^{(R)}}{\slashed\pi_{T}}|+\mathfrak t\delta^{l_1}|\slashed dZ^{n_1}\textrm{tr}\check\chi|,$$
which, together with Proposition \ref{Y5}, \eqref{d} and \eqref{Em}, implies
\begin{equation}\label{Zk-1R}
\begin{split}
\delta^l\|Z^{k-1}R\textrm{tr}\check\chi\|_{L^2(\Sigma_{\mathfrak t}^{u})}&\lesssim\delta^{3/2-\varepsilon_0}\mathfrak t^{-1}
+\mathfrak t^{-m}\sqrt{\tilde E_{1,\leq k+1}(\mathfrak t,u)}+\delta \mathfrak t^{-1}\sqrt{E_{2,\leq k+1}(\mathfrak t,u)}\\
&\lesssim\delta^{3/2-\varepsilon_0}\mathfrak t^{-1}+\delta \mathfrak  t^{-1-m}\sqrt{\tilde E_{1,\leq k+2}(\mathfrak t,u)}
+\delta \mathfrak t^{-1}\sqrt{E_{2,\leq k+2}(\mathfrak t,u)}.
\end{split}
\end{equation}

Collecting \eqref{Zk-1L}, \eqref{Zk-1T} and \eqref{Zk-1R} leads to the desired estimate \eqref{uZkchi}.
Moreover, \eqref{ZTchi} follows from \eqref{ZrLch} and \eqref{ZTch} directly.
\end{proof}

\section{Estimates for the error terms}\label{ert}

With the $L^2$ estimates in Sections \ref{EE} and \ref{L2chimu},
we are ready to handle the error terms $\delta\int_{D^{{\mathfrak t}, u}}|\Phi\cdot\mathring{\underline L}\Psi|$
and $|\int_{D^{{\mathfrak t}, u}}\Phi(\varrho^{2m}\mathring L\Psi+\f12\varrho^{2m-1}\Psi)|$ in \eqref{e}, and then
get the final energy estimates for $\vp$.
For the 4D case in \cite{Ding4}, the time decay rate of the solution is $\varrho^{-3/2}$,
 which leads to the direct treatments for the term $|\int_{D^{{\mathfrak t}, u}}\Phi(\varrho^{2m}\mathring L\Psi+\f12\varrho^{2m-1}\Psi)|$
 with $m\in(0,1/2)$. However, as explained in Subsection \ref{6-2}, here the 2D case is rather different due to the slower time decay.
To estimate $\int_{D^{{\mathfrak t}, u}}\Phi(\varrho^{2m}\mathring L\Psi+\f12\varrho^{2m-1}\Psi)$
and its  top order derivatives,
one needs the following two lemmas.

\begin{lemma}
For any smooth functions $f$ and $h$, it holds that
\begin{align}
\int_{S_{\mathfrak t, u}}(fh)=&\int_{C_u^\mathfrak t}\{\mathring L(fh)+\textrm{tr}\chi(fh)\}+\int_{S_{t_0, u}}(fh)\no\\
=&\int_{C_u^\mathfrak t}f(\mathring Lh+\f{1}{2\varrho}h)+\int_{C_u^\mathfrak t}(\mathring Lf+\f{1}{2\varrho}f)h+\int_{C_u^\mathfrak t}\textrm{tr}\check\chi(fh)+\int_{S_{t_0, u}}(fh),\label{fh}\\
\int_{S_{\mathfrak t, u}}(Rf)h=&-\int_{S_{\mathfrak t, u}}f(Rh)-\f12\int_{S_{\mathfrak t, u}}\textrm{tr}\leftidx{^{(R)}}{\slashed\pi}(fh).\label{Rfh}
\end{align}
\end{lemma}

\begin{proof}
This follows from similar arguments in \cite{J}.
\end{proof}

\begin{lemma}\label{f1f2f3}
If $f_i$ $(i=1,2,3)$ are smooth functions, then it holds that
\begin{equation}\label{fLfRf}
\begin{split}
-\int_{D^{{\mathfrak t},u}}\{f_1\cdot(\mathring Lf_2+\f{1}{2\varrho}f_2)\cdot Rf_3\}=&-\int_{D^{{\mathfrak t},u}}\{f_1\cdot Rf_2\cdot\mathring Lf_3\}+\int_{\Sigma_{\mathfrak t}^u}\{f_1\cdot Rf_2\cdot f_3\}\\
&+\int_{D^{{\mathfrak t},u}}Er_1+\int_{\Sigma_{\mathfrak t}^u}Er_2+\int_{\Sigma_{t_0}^u}Er_3
\end{split}
\end{equation}
with
\begin{align*}
Er_1=&(\textrm{tr}\check\chi)f_1\cdot f_2\cdot Rf_3-Rf_1\cdot f_2\cdot\mathring Lf_3
-\f12(\textrm{tr}\leftidx{^{(R)}}{\slashed\pi})f_1\cdot f_2\cdot\mathring Lf_3
-f_1\cdot(\leftidx{^{(R)}}{\slashed\pi}_{\mathring LX}\slashed d^Xf_2)\cdot f_3\\
&-(\leftidx{^{(R)}}{\slashed\pi}_{\mathring LX}\slashed d^Xf_1)\cdot f_2\cdot f_3
-(\slashed{\textrm{div}}\leftidx{^{(R)}}{\slashed\pi}_{\mathring L})f_1\cdot f_2\cdot f_3
-(\mathring L+\f{1}{2\varrho})f_1\cdot Rf_2\cdot f_3\\
&-R(\mathring L+\f{1}{2\varrho})f_1\cdot f_2\cdot f_3-\f12(\textrm{tr}\leftidx{^{(R)}}{\slashed\pi})(\mathring L
+\f{1}{2\varrho})f_1\cdot f_2\cdot f_3,\\
Er_2=&Rf_1\cdot f_2\cdot f_3+\f12(\textrm{tr}\leftidx{^{(R)}}{\slashed\pi})f_1\cdot f_2\cdot f_3,\\
Er_3=&-f_1\cdot Rf_2\cdot f_3-Rf_1\cdot f_2\cdot f_3-\f12(\textrm{tr}\leftidx{^{(R)}}{\slashed\pi})f_1\cdot f_2\cdot f_3.
\end{align*}
\end{lemma}

\begin{proof}
Let $f=f_1\cdot Rf_3$ and $h=f_2$ in \eqref{fh} and then integrate over $[0,u]$ to get
\begin{equation}\label{f123}
\begin{split}
&-\int_{D^{{\mathfrak t},u}}\{f_1\cdot(\mathring Lf_2+\f{1}{2\varrho}f_2)\cdot Rf_3\}\\
=&\int_{D^{{\mathfrak t},u}}\textrm{tr}\check\chi\{f_1\cdot f_2\cdot Rf_3\}+\int_{D^{{\mathfrak t},u}}\{(\mathring L+\f{1}{2\varrho})(f_1\cdot Rf_3)\cdot f_2\}\\
&-\int_{\Sigma_{\mathfrak t}^u}\{f_1\cdot f_2\cdot Rf_3\}+\int_{\Sigma_{t_0}^u}\{f_1\cdot f_2\cdot Rf_3\}\\
:=&\int_{D^{{\mathfrak t},u}}\textrm{tr}\check\chi\{f_1\cdot f_2\cdot Rf_3\}+I+II+III.
\end{split}
\end{equation}
Choosing $f=f_3$ and $h=f_2(\mathring Lf_1+\f{1}{2\varrho}f_1)$ in \eqref{Rfh} yields
\begin{equation}\label{I}
\begin{split}
I=&-\int_{D^{{\mathfrak t},u}}R\{f_2(\mathring Lf_1+\f{1}{2\varrho} f_1)\}f_3
-\f12\int_{D^{{\mathfrak t},u}}\textrm{tr}\leftidx{^{(R)}}{\slashed\pi}\{f_2(\mathring Lf_1
+\f{1}{2\varrho}f_1)f_3\}\\
&+\int_{D^{{\mathfrak t},u}}\{f_1\cdot f_2\cdot R\mathring Lf_3\}
+\int_{D^{{\mathfrak t},u}}\{f_1\cdot f_2\cdot(\leftidx{^{(R)}}{\slashed\pi}_{\mathring LX}\slashed d^Xf_3)\}.
\end{split}
\end{equation}
Treating $\int_{D^{{\mathfrak t},u}}\{f_1\cdot f_2\cdot R\mathring Lf_3\}$ by using \eqref{Rfh}, one can get from \eqref{I} that
\begin{equation}\label{I2}
\begin{split}
I=&-\int_{D^{{\mathfrak t},u}}Rf_2(\mathring Lf_1+\f{1}{2\varrho}f_1)f_3-\int_{D^{{\mathfrak t},u}}f_2\{R(\mathring L+\f{1}{2\varrho})f_1\}f_3-\int_{D^{{\mathfrak t},u}}(\slashed{\textrm{div}}\leftidx{^{(R)}}{\slashed\pi}_{\mathring L})f_1\cdot f_2\cdot f_3\\
&-\f12\int_{D^{{\mathfrak t},u}}\textrm{tr}\leftidx{^{(R)}}{\slashed\pi}\{f_2(\mathring Lf_1+\f{1}{2\varrho}f_1)f_3\}
-\int_{D^{{\mathfrak t},u}}f_1\cdot Rf_2\cdot\mathring Lf_3-\int_{D^{{\mathfrak t},u}}Rf_1\cdot f_2\cdot\mathring Lf_3\\
&-\f12\int_{D^{{\mathfrak t},u}}\textrm{tr}\leftidx{^{(R)}}{\slashed\pi}\{f_1\cdot f_2\cdot\mathring Lf_3\}-\int_{D^{{\mathfrak t},u}}f_1(\leftidx{^{(R)}}{\slashed{\pi}}_{\mathring LX}\slashed d^Xf_2)f_3
-\int_{D^{{\mathfrak t},u}}(\leftidx{^{(R)}}{\slashed{\pi}}_{\mathring LX}\slashed d^Xf_1)f_2\cdot f_3.
\end{split}
\end{equation}
Similarly,
\begin{align}
II&=\int_{\Sigma_{\mathfrak t}^u}f_1(Rf_2)f_3+\int_{\Sigma_{\mathfrak t}^u}Rf_1\cdot f_2\cdot f_3
+\f12\int_{\Sigma_{\mathfrak t}^u}(\textrm{tr}\leftidx{^{(R)}}{\slashed\pi})f_1\cdot f_2\cdot f_3,\label{II}\\
III&=-\int_{\Sigma_{t_0}^u}f_1(Rf_2)f_3-\int_{\Sigma_{t_0}^u}Rf_1\cdot f_2\cdot f_3-\f12\int_{\Sigma_{t_0}^u}(\textrm{tr}\leftidx{^{(R)}}{\slashed\pi})f_1\cdot f_2\cdot f_3.\label{III}
\end{align}
Thus \eqref{fLfRf} follows from substituting \eqref{I2}-\eqref{III} into \eqref{f123}.
\end{proof}

Recall the notations in Section \ref{EE} that for $\Psi=Z^{k+1}\varphi_\gamma=Z_{k+1}Z_k\cdots Z_1\varphi_\gamma$, $\Phi=\Phi_\gamma^{k+1}\equiv J_1^k+J_2^k$ is given explicitly in \eqref{Phik} with $J_1^k$ and $J_2^k$ being the summation and the rest in \eqref{Phik} respectively. We will also use the notation that $\varphi_\gamma^n=Z_n\cdots Z_1\varphi_\gamma$ for $n\geq 1$, and $\varphi_\gamma^0=\varphi_\gamma$. Then the main task is to estimate $J_1^k$ and $J_2^k$.

\subsection{Estimates for $J_1^k$}\label{tot}

It follows from the explicit form of $J_1^k$  that the key is to estimate the derivatives of $\mu\textrm{div}\leftidx{^{(Z)}}C_\gamma^{n}$ $(0\leq n \leq k)$. Due to \eqref{muZC} and \eqref{D1}-\eqref{D3}, the treatment involving $\leftidx{^{(Z)}}D_{\gamma,2}^{n}$ will be given separately from those for $\leftidx{^{(Z)}}D_{\gamma,1}^{n}$ and $\leftidx{^{(Z)}}D_{\gamma,3}^{n}$, since the latter do not contain the top order derivatives of $\varphi_\gamma$.
\begin{enumerate}[(1)]
\item We start with the estimates involving $\leftidx{^{(Z)}}D_{\gamma,1}^{n}$ and $\leftidx{^{(Z)}}D_{\gamma,3}^{n}$. Substituting \eqref{Lpi}-\eqref{Rpi} into \eqref{D1} and \eqref{D3} yields
\begin{equation}\label{T1}
\begin{split}
\leftidx{^{(T)}}D_{\gamma,1}^{n}=&T\mu\mathring L^2\varphi_\gamma^{n}+\mu(\slashed d\mu
+2c^{-1}\mu\tilde T^a\slashed d\varphi_a)\cdot\slashed d\mathring L\varphi_\gamma^{n}+\f12\textrm{tr}\leftidx{^{(T)}}{\slashed\pi}(\mathring L\mathring{\underline L}\varphi_\gamma^{n}+\f12\textrm{tr}\chi\mathring{\underline L}\varphi_\gamma^{n})\\
&+(\slashed d\mu+2c^{-1}\mu\tilde T^a\slashed d\varphi_a)\cdot\slashed d\mathring{\underline L}\varphi_\gamma^{n}-T\mu\slashed\triangle\varphi_\gamma^{n}+\f12\big(-c^{-1}Tc-c^{-1}\mu\mathring Lc\\
&+2c^{-1}\mu\slashed dx^a\cdot\slashed d\varphi_a-2\mu\textrm{tr}\chi\big)\slashed\triangle \varphi_\gamma^{n},
\end{split}
\end{equation}
\begin{equation}\label{T3}
\begin{split}
\leftidx{^{(T)}}D_{\gamma,3}^{n}=&\big\{\textrm{tr}\chi T\mu+(\f14\mu\textrm{tr}\chi
+\f12\mu\textrm{tr}\tilde\theta)\textrm{tr}_{\slashed g}\leftidx{^{(T)}}{\slashed \pi}
-\f12|\slashed d\mu|^2-c^{-1}\mu\tilde T^a\slashed d\mu\cdot\slashed d\varphi_a\big\}\mathring L\varphi_\gamma^{n}\qquad\qquad\\
&+(\f12\textrm{tr}\leftidx{^{(T)}}{\slashed\pi}+\f12\mathring L\mu
-\mu\textrm{tr}\chi)(\slashed d\mu+2c^{-1}\mu\tilde T^a\slashed d\varphi_a)\cdot\slashed d\varphi_\gamma^{n},
\end{split}
\end{equation}
\begin{equation}\label{rL1}
\begin{split}
\leftidx{^{(\varrho\mathring L)}}D_{\gamma,1}^{n}=&(2-\mu+\varrho\mathring L\mu)\mathring L^2\varphi_\gamma^{n}-2\varrho(\slashed d\mu+2c^{-1}\mu\tilde T^a\slashed d\varphi_a)\cdot\slashed d\mathring L\varphi_\gamma^{n}\\
&+\varrho\textrm{tr}\chi(\mathring L\mathring{\underline L}\varphi_\gamma^{n}+\f12\textrm{tr}\chi\mathring{\underline L}\varphi_\gamma^{n})+\varrho(\mu\textrm{tr}\check\chi
-\mathring L\mu )\slashed\triangle\varphi_\gamma^{n},\qquad\qquad\qquad\qquad\quad
\end{split}
\end{equation}
\begin{equation}\label{rL3}
\begin{split}
\leftidx{^{(\varrho\mathring L)}}D_{\gamma,3}^{n}=&\textrm{tr}\chi\big\{2-\mu+\varrho\mathring L\mu+\varrho\mu\textrm{tr}\tilde\theta+\f12\varrho\mu\textrm{tr}\chi\big\}\mathring L\varphi_\gamma^{n}\qquad\qquad\qquad\qquad\quad\qquad\quad\\
&+2\varrho\textrm{tr}\chi(\slashed d\mu+2\mu c^{-1}\tilde T^a\slashed d\varphi_a)\cdot\slashed d\varphi_\g^n,
\end{split}
\end{equation}
\begin{equation}\label{R1}
\begin{split}
\leftidx{^{(R)}}D_{\gamma,1}^{n}=
&R\mu\mathring L^2\varphi_\gamma^{n}-\leftidx{^{(R)}}{\slashed\pi}_{\mathring{\underline L}}\cdot\slashed d\mathring L\varphi_\gamma^{n}+\f12\textrm{tr}\leftidx{^{(R)}}{\slashed\pi}(\mathring L\mathring{\underline L}\varphi_\gamma^{n}+\f12\textrm{tr}\chi\mathring{\underline L}\varphi_\gamma^{n})\qquad\qquad\qquad\\
&-\leftidx{^{(R)}}{\slashed\pi}_{\mathring L}\cdot\slashed d\mathring{\underline L}\varphi_\gamma^{n}
-R\mu\slashed\triangle\varphi_\gamma^{n}
+\f12\mu(\textrm{tr}\leftidx{^{(R)}}{\slashed\pi})\slashed\triangle\varphi_\g^n,
\end{split}
\end{equation}
\begin{equation}\label{R3}
\begin{split}
\leftidx{^{(R)}}D_{\gamma,3}^{n}=&\big\{\textrm{tr}\chi R\mu+\f12\mu(\textrm{tr}\tilde\theta
+\f12\textrm{tr}\chi)\textrm{tr}\leftidx{^{(R)}}{\slashed\pi}
+\f12\slashed d\mu\cdot\leftidx{^{(R)}}{\slashed\pi}_{\mathring L}\big\}\mathring L\varphi_\gamma^{n}
+\big\{\f12\textrm{tr}\leftidx{^{(R)}}{\slashed\pi}(\slashed d\mu\\
&+2\mu c^{-1}\tilde T^a\slashed d\varphi_a)-\f12\mathring L\mu\leftidx{^{(R)}}{\slashed\pi}_{\mathring L}+\f12\textrm{tr}\chi\leftidx{^{(R)}}{\slashed\pi}_{\mathring{\underline L}}
+\f12\mu\textrm{tr}\chi\leftidx{^{(R)}}{\slashed\pi}_{\mathring L}\big\}\cdot\slashed d\varphi_\gamma^{n}.
\end{split}
\end{equation}

Note that $\mathring L\underline{\mathring L}\varphi_\gamma^n+\f12\textrm{tr}\chi\underline{\mathring L}\varphi_\gamma^n$ appears in $\leftidx{^{(Z)}}D_{\gamma,1}^{n}$ and $J_1^k$ contains at most the $(k-n)^{\text{th}}$ order derivatives of $\leftidx{^{(Z)}}D_{\gamma,i}^{n}$ $(i=1,3)$. Then it can be checked that the $L^2$ norms of all the terms involving $\leftidx{^{(Z)}}D_{\gamma,i}^{n}$ $(i=1,3)$ in $J_1^k$ can be treated by using the $L^\infty$-estimates in Section \ref{BA} and the related $L^2$ estimates of Proposition \ref{Y5}.
Therefore, it holds that
\begin{equation}\label{D11}
\begin{split}
&\delta^{2l+1}\mid\int_{D^{{\mathfrak t}, u}}\sum_{j=1}^{k}\big(Z_{k+1}
+\leftidx{^{(Z_{k+1})}}\chi\big)\dots\big(Z_{k+2-j}+\leftidx{^{(Z_{k+2-j})}}\chi\big)\leftidx{^{(Z_{k+1-j})}}D_{\gamma,1}^{k-j}\\
&\qquad\qquad\qquad\cdot\mathring{\underline L}\varphi_\gamma^{k+1}\mid\\
\lesssim&\delta^{2l+1}\int_{t_0}^\mathfrak t\sum_{j=1}^{k}\|\big(Z_{k+1}+\leftidx{^{(Z_{k+1})}}\chi\big)\dots\big(Z_{k+2-j}
+\leftidx{^{(Z_{k+2-j})}}\chi\big)\leftidx{^{(Z_{k+1-j})}}D_{\gamma,1}^{k-j}\|_{L^2(\Sigma_{\tau}^{u})}\\
&\qquad\cdot\|\mathring{\underline L}Z^{k+1}\varphi_\gamma\|_{L^2(\Sigma_{\tau}^{u})}d\tau\\
\lesssim&\delta^{4-4\varepsilon_0}+\int_{t_0}^\mathfrak t\tau^{-3/2}\tilde E_{1,\leq k+2}(\tau,u)d\tau
+\delta\int_{t_0}^\mathfrak t\tau^{-3/2}\tilde E_{2,\leq k+2}(\tau,u)d\tau.
\end{split}
\end{equation}
Similarly, one can get that
\begin{equation}\label{D33}
\begin{split}
&\delta^{2l+1}\mid\int_{D^{{\mathfrak t}, u}}\sum_{j=1}^{k}\big(Z_{k+1}+\leftidx{^{(Z_{k+1})}}\chi\big)\dots\big(Z_{k+2-j}
+\leftidx{^{(Z_{k+2-j})}}\chi\big)\leftidx{^{(Z_{k+1-j})}}D_{\gamma,3}^{k-j}\\
&\qquad\qquad\qquad\cdot\mathring{\underline L}Z^{k+1}\varphi_\gamma\mid\\
\lesssim&\delta^{3-3\varepsilon_0}+\delta^{2-3\varepsilon_0}\int_{t_0}^\mathfrak t\tau^{-3/2}\tilde E_{1,\leq k+2}(\tau,u)d\tau+\delta^{2-\varepsilon_0}\int_{t_0}^\mathfrak t\tau^{-3/2}\tilde E_{2,\leq k+2}(\tau,u)d\tau.
\end{split}
\end{equation}

The corresponding terms in $\delta^{2l}|\int_{D^{{\mathfrak t}, u}}\Phi(\varrho^{2m}\mathring L\Psi+\f12\varrho^{2m-1}\Psi)|$
can also be estimated as
\begin{equation}\label{D13}
\begin{split}
&\delta^{2l}\int_{D^{{\mathfrak t}, u}}\mid\sum_{j=1}^{k}\big(Z_{k+1}+\leftidx{^{(Z_{k+1})}}\chi\big)
\dots\big(Z_{k+2-j}+\leftidx{^{(Z_{k+2-j})}}\chi\big)\big(\leftidx{^{(Z_{k+1-j})}}D_{\gamma,1}^{k-j}\\
&\qquad\qquad\qquad+\leftidx{^{(Z_{k+1-j})}}D_{\gamma,3}^{k-j}\big)(\varrho^{2m}\mathring L\varphi_\gamma^{k+1}+\f12\varrho^{2m-1}\varphi_\g^{k+1})\mid\\
\lesssim&\delta^{2l+1}\int_{D^{{\mathfrak t}, u}}\varrho^{2m}\Big\{\sum_{j=1}^{k}\big(Z_{k+1}
+\leftidx{^{(Z_{k+1})}}\chi\big)\dots\big(Z_{k+2-j}
+\leftidx{^{(Z_{k+2-j})}}\chi\big)\big(\leftidx{^{(Z_{k+1-j})}}D_{\gamma,1}^{k-j}\\
&\qquad+\leftidx{^{(Z_{k+1-j})}}D_{\gamma,3}^{k-j}\big)\Big\}^2+\delta^{2l-1}\int_{D^{{\mathfrak t}, u}}|(\varrho^{m}\mathring L\varphi_\gamma^{k+1}+\f12\varrho^{m-1}\varphi_\g^{k+1})|^2\\
\lesssim&\delta^{4-4\varepsilon_0}+\delta\int_{t_0}^\mathfrak t\tau^{-2}\tilde E_{1,\leq k+2}(\tau,u)d\tau
+\delta\int_{t_0}^\mathfrak t\tau^{-3+2m}\tilde E_{2,\leq k+2}(\tau,u)d\tau\\
&+\delta^{-1}\int_0^uF_{1,k+2}(\mathfrak t,u')du'.
\end{split}
\end{equation}

\item We now estimate the terms involving $\leftidx{^{(Z)}}D_{\gamma,2}^{n}$ $(0\leq n\leq k)$ in $J_1^k$. Note that in the special case $n=0$, $j$ equals $k$ in $J_1^k$, and the order of the top derivatives in $\leftidx{^{(Z)}}D_{\gamma,2}^{0}$ is $k$, so that $\leftidx{^{(Z)}}D_{\gamma,2}^{0}$ contains terms involving the $(k+1)^{\text{th}}$ order derivatives of the deformation tensor.
This prevents one from using Proposition \ref{Y5} to estimate the $L^2$ norm of $\leftidx{^{(Z)}}D_{\gamma,2}^0$
directly since the $L^2$ norm of the $(k+1)^{\text{th}}$ order derivatives of the deformation tensor can be controlled only by $\tilde E_{1,\leq k+3}(\mathfrak t,u)$ and $\tilde E_{2,\leq k+3}(\mathfrak t,u)$ in the energy estimates. Thus, we will
examine carefully the expression of $\leftidx{^{(Z)}}D_{\gamma,2}^n$ and apply the estimates in Section \ref{L2chimu}
to deal with the top order derivatives of $\textrm{tr}\chi$ and $\mu$.
Indeed, it follows from direct computations that
\begin{equation}\label{T2}
\begin{split}
\leftidx{^{(T)}}D_{\gamma,2}^{n}=&\mathring L T\mu\cdot\mathring L\varphi_\gamma^{n}-\f12\big(\slashed{\mathcal L}_{\mathring L}\leftidx{^{(T)}}{\slashed\pi}_{\mathring{\underline L}}\big)\cdot\slashed d\varphi_\gamma^n+\f12\boxed{\slashed{\textrm{div}}(\slashed d\mu}+2c^{-1}\mu\tilde T^a\slashed d\varphi_a)\underline L\varphi_\g^n\\
&+\big\{\f14\mathring{\underline L}\big(-c^{-1}Tc-c^{-1}\mu\mathring Lc+2\mu c^{-1}\slashed dx^a\cdot\slashed d\varphi_a\uwave{-2\mu\textrm{tr}\chi}\big)+\f12\boxed{\slashed\nabla_X\big(\mu\slashed d^X\mu}\\
&+2c^{-1}\mu^2\tilde T^a\slashed d\varphi_a\big)\big\}\mathring L\varphi_\gamma^n-\big\{\uline{\slashed dT\mu-\f12\slashed{\mathcal L}_{\mathring{\underline L}}(\slashed d\mu}+2c^{-1}\mu\tilde T^a\slashed d\varphi_a)\big\}\cdot\slashed d\varphi_\gamma^n\\
&-\f12\underbrace{\slashed d\big(2\mu^2\textrm{tr}\check\chi}+2\varrho^{-1}\mu^2+c^{-1}\mu Tc+c^{-1}\mu^2\mathring Lc-2c^{-1}\mu^2(\slashed dx^a)\cdot\slashed d\varphi_a\big)\cdot\slashed d\varphi_\gamma^n\\
&+\f14\big(\mathring L\textrm{tr}\leftidx{^{(T)}}{\slashed\pi}\big)\mathring{\underline L}\varphi_\gamma^n,
\end{split}
\end{equation}
\begin{equation}\label{rL2}
\begin{split}
\leftidx{^{(\varrho\mathring L)}}D_{\gamma,2}^{n}=&\big\{\mathring L\big(-\mu+\varrho\mathring L\mu\big)+\uwave{\f12\mathring{\underline L}\big(\varrho\textrm{tr}\check\chi\big)}-\boxed{\slashed\nabla_X\big(\varrho\slashed d^X\mu}+2c^{-1}\varrho\mu\tilde T^a\slashed d^X\varphi_a\big)\big\}\mathring L\varphi_\gamma^n\\
&+\f12\mathring L\big(\varrho\textrm{tr}\check\chi\big)\mathring{\underline L}\varphi_\gamma^n-\big\{\slashed{\mathcal L}_{\mathring L}\big(\varrho\slashed d\mu+2c^{-1}\varrho\mu\tilde T^a\slashed d\varphi_a\big)+\slashed d\big(\mu+\varrho\mathring L\mu\big)\big\}\cdot\slashed d\varphi_\gamma^n\\
&+\underbrace{\slashed d\big(\varrho\mu\textrm{tr}\chi\big)}\cdot\slashed d\varphi_\gamma^n,
\end{split}
\end{equation}
\begin{equation}\label{R2}
\begin{split}
\leftidx{^{(R)}}D_{\gamma,2}^{n}=&(\mathring LR\mu)\mathring L\varphi_\gamma^n-\big\{\f12\slashed{\mathcal L}_{\mathring L}\big(\leftidx{^{(R)}}{\slashed\pi}_{\mathring{\underline L}}\big)+\boxed{\slashed dR\mu}-\f12\uwave{\slashed{\mathcal L}_{\mathring{\underline L}}\big(R^X\textrm{tr}\check\chi\slashed g}-g_{ab}\epsilon_{i}^a\check L^i\slashed dx^b\\
&-\upsilon c^{-1}\tilde T^a\slashed d\varphi_a+\f12c^{-1}R^X\mathring Lc\slashed g+\f12c^{-1}\upsilon\slashed dc\big)\big\}\cdot\slashed d\varphi_\gamma^n+\f14\mathring L\big(\textrm{tr}\leftidx{^{(R)}}{\slashed\pi}\big)\mathring{\underline L}\varphi_\gamma^n\\
&+\f14\uwave{\mathring{\underline L}\big(2\upsilon\textrm{tr}\chi}+c^{-1}\upsilon\mathring Lc-2c^{-1}\upsilon\slashed dx^a\cdot\slashed d\varphi_a-c^{-1}Rc\big)\mathring L\varphi_\gamma^n-\f12\underbrace{\slashed\nabla_X\big(\mu R^X\textrm{tr}\check\chi}\\
&+g_{ab}\mu\epsilon_{i}^b \check L^i\slashed d^Xx^a+2g_{ab}\mu{\epsilon_{i}^a}\check T^i\slashed dx^b+\mu\upsilon c^{-1}\tilde T^a\slashed d\varphi_a+\f12c^{-1}\mu R^X\mathring Lc\\
&+\f12c^{-1}\mu\upsilon\slashed dc+2\upsilon\slashed d\mu-2c^{-1}\mu R\varphi_a\slashed d^Xx^a\big)\mathring L\varphi_\gamma^n+\f12\underbrace{\slashed\nabla_X\big(R^X\textrm{tr}\check\chi}\\
&-g_{ab}\epsilon_{i}^a\check L^i\slashed d^Xx^b-\upsilon c^{-1}\tilde T^a\slashed d^X\varphi_a+\f12c^{-1}R^X\mathring Lc+\f12c^{-1}\upsilon\slashed d_Xc\big)\mathring{\underline L}\varphi_\gamma^n\\
&+\f12\underbrace{\slashed d\big(2\upsilon\mu\textrm{tr}\chi}+c^{-1}\mu\upsilon\mathring Lc-2c^{-1}\mu\upsilon\slashed dx^a\cdot\slashed d\varphi_a
-c^{-1}\mu Rc\big)\cdot\slashed d\varphi_\gamma^n.
\end{split}
\end{equation}

It is emphasized that special attentions are needed for terms with underlines, wavy lines, boxes, or braces in \eqref{T2}-\eqref{R2}. In \eqref{T2},
due to $\f12\slashed{\mathcal L}_{\mathring{\underline L}}\slashed d\mu=\slashed dT\mu+\f12\mu\slashed d\mathring L\mu$, the
corresponding underline part is
\begin{equation}\label{Y-25}
\slashed dT\mu-\f12\slashed{\mathcal L}_{\mathring{\underline L}}\slashed d\mu=-\f12\mu\slashed d\mathring L\mu,
\end{equation}
which can be estimated by using \eqref{lmu}. For the terms with wavy lines in \eqref{T2}-\eqref{R2}, one can use \eqref{Tchi'} and \eqref{Lchi'} to replace $\f12\slashed{\mathcal L}_{\mathring{\underline L}}\check\chi$ by $\slashed\nabla^2\mu+\cdots$ which can be handled by \eqref{Zmu}.
We also apply \eqref{Zmu} and \eqref{d} to estimate those terms with boxes and braces respectively.
Meanwhile, one notes that in \eqref{T2}-\eqref{R2}, there are some terms containing derivatives of the deformation tensors with respect to $\mathring L$, for example, $\f12\big(\slashed{\mathcal L}_{\mathring L}\leftidx{^{(T)}}{\slashed\pi}_{\mathring{\underline L}}\big)\cdot\slashed d\varphi_\gamma^n$ appears in \eqref{T2}. However, these terms can be estimated by taking into account of \eqref{lmu}, \eqref{Lchi'}, \eqref{Lpi} and \eqref{Rpi}.

In summary, we can arrive at
\begin{equation}\label{D22}
\begin{split}
&\delta^{2l+1}\mid\int_{D^{{\mathfrak t}, u}}\sum_{j=1}^{k}\big(Z_{k+1}+\leftidx{^{(Z_{k+1})}}\chi\big)\dots\big(Z_{k+2-j}
+\leftidx{^{(Z_{k+2-j})}}\chi\big)\leftidx{^{(Z_{k+1-j})}}D_{\gamma,2}^{k-j}\\
&\qquad\qquad\cdot\mathring{\underline L}\varphi_\gamma^{k+1}\mid\\
\lesssim&\delta^{4-4\varepsilon_0}+\delta^{1-2\varepsilon_0}\int_{t_0}^\mathfrak t\tau^{-1/2-m}\tilde E_{1,\leq k+2}(\tau,u)d\tau+\delta\int_{t_0}^\mathfrak t\tau^{-1/2-m}\tilde E_{2,\leq k+2}(\tau,u)d\tau.
\end{split}
\end{equation}

It remains to deal with
\[
\delta^{2l}\mid \int_{D^{{\mathfrak t}, u}}\big(Z_{k+1}+\leftidx{^{(Z_{k+1})}}\chi\big)\dots\big(Z_{k+2-j}
+\leftidx{^{(Z_{k+2-j})}}\chi\big)\leftidx{^{(Z_{k+1-j})}}D_{\gamma,2}^{k-j}(\varrho^{2m}\mathring L\varphi_\gamma^{k+1}+\f12\varrho^{2m-1}\varphi_\g^{k+1})\mid,
\]
    where $j=1,2,\cdots k$, which will be estimated as follows.
    \begin{enumerate}
     \item For $Z_{k+1-j}=T$, according \eqref{T2}, one can get from  \eqref{d}, \eqref{Zmu} and Proposition \ref{Y5} that
     \begin{equation}\label{TD}
     \begin{split}
     &\delta^{2l}\mid\int_{D^{{\mathfrak t}, u}}\sum_{j=1}^k\big(Z_{k+1}+\leftidx{^{(Z_{k+1})}}\chi\big)\dots\big(Z_{k+2-j}
     +\leftidx{^{(Z_{k+2-j})}}\chi\big)\leftidx{^{(T)}}D_{\gamma,2}^{k-j}\cdot(\varrho^{2m}\mathring L\varphi_\gamma^{k+1}\\
     &\qquad\qquad\qquad+\f12\varrho^{2m-1}\varphi_\g^{k+1})\mid\\
     \lesssim&\delta^{2l+1}\int_{D^{{\mathfrak t},u}}\varrho^{2m}\mid\sum_{j=1}^k\big(Z_{k+1}
     +\leftidx{^{(Z_{k+1})}}\chi\big)\dots\big(Z_{k+2-j}
     +\leftidx{^{(Z_{k+2-j})}}\chi\big)\leftidx{^{(T)}}D_{\gamma,2}^{k-j}\mid^2\\
     &+\delta^{-1}\int_0^uF_{1,k+2}(\mathfrak t,u')du'\\
     \lesssim&\delta^{6-6\varepsilon_0}+\delta^{3-4\varepsilon_0}\int_{t_0}^\mathfrak t\tau^{-2}\tilde E_{1,\leq k+2}(\tau,u)d\tau+\delta^{3-2\varepsilon_0}\int_{t_0}^\mathfrak t\tau^{2m-3}\tilde E_{2,\leq k+2}(\tau,u)d\tau\\
     &+\delta^{-1}\int_0^uF_{1,k+2}(\mathfrak t,u')du'.
     \end{split}
     \end{equation}

     \item For $Z_{k+1-j}=\varrho\mathring L$, $Z^k\leftidx{^{(\varrho\mathring L)}}D_{\gamma,2}^0$
     contains a term $\f12\big(Z^k\mathring L(\varrho\textrm{tr}\check{\chi})\big)\underline{\mathring L}\varphi_\gamma$.
      Then \eqref{Lchi'} implies that
     \begin{equation}\label{Lrchi}
     \begin{split}
     \delta^l|Z^{k}\mathring L(\varrho\textrm{tr}\check\chi)|&\lesssim\delta^{l_2}|Z^{n_2}\textrm{tr}\check\chi|
     +\delta^{l_2}|Z^{n_2}\big(\varrho(\slashed\triangle\varphi_0+\mathring L^a\slashed\triangle\varphi_a)\big)|
     +\delta^{l_3}|\mathring LZ^{n_3}\varphi|\\
     &+\delta^{1-\varepsilon_0}\mathfrak t^{-5/2}\delta^{l_3}|Z^{n_3}x|+\mathfrak t^{-1}\delta^{l_3}|Z^{n_3}\varphi|
     +\delta^{1-\varepsilon_0}\mathfrak t^{-3/2}\sum_{i=1}^2\delta^{l_2}|Z^{n_2}\check L^i|\\
     &+\delta^{1-\varepsilon_0}\mathfrak t^{-3/2}\delta^{l_1}|\slashed{\mathcal L}_Z^{n_1}\leftidx{^{(R)}}{\slashed\pi_{\mathring L}}|
     +\delta^{2-\varepsilon_0}\mathfrak t^{-3/2}\delta^{l_1}|\slashed{\mathcal L}_Z^{n_1}\leftidx{^{(T)}}{\slashed\pi_{\mathring L}}|,
     \end{split}
     \end{equation}
     where $l_i$ is the number of $T$ in $Z^{n_i}$ and $n_i\leq k-2+i$ $(i=1,2,3)$. Due to \eqref{triphi},
     \begin{equation}\label{mutri}
     \begin{split}
     \mu(\slashed\triangle\varphi_0+\mathring L^a\slashed\triangle\varphi_a)
     =&\mathring L(\mathring L^\al\underline{\mathring L}\varphi^\al)-(\mathring L\mathring L^a)\underline{\mathring L}\varphi_a+\f1{2\varrho}\underline{\mathring L}^\al\mathring L\varphi_\al-\mathring L^\al\tilde H_\al.
     \end{split}
     \end{equation}
     Substituting \eqref{LL} and \eqref{H} (note that $\tilde H_\al=H_\al-\mu\slashed\triangle\varphi_\al)$
     into \eqref{mutri} yields
     \begin{equation}\label{tri}
     \begin{split}
     &\slashed\triangle\varphi_0+\mathring L^a\slashed\triangle\varphi_a\\
     =&(\mathring L+\f1{2\varrho})\mathring L\varphi_0+\sum_{a=1}^2\big\{(2\varphi_a-\mathring L^a)(\mathring L+\f1{2\varrho})\mathring L\varphi_a+\mathring L(2\varphi_a-\mathring L^a)\mathring L\varphi_a\big\}\\
     &+c^{-1}(\mathring L^\al\varphi_\al-\mathring Lc)\big\{\mathring L\varphi_0+\sum_{a=1}^2(2\varphi_a-\mathring L^a)\mathring L\varphi_a\big\}+\f12c^{-1}(\tilde T^a\mathring L\varphi_a\\
     &-\slashed dx^a\cdot\slashed d\varphi_a-c\textrm{tr}\check\chi-c\varrho^{-1})\mathring L^\al\mathring L\varphi_\al+c^{-1}(\mathring L^\al\slashed d\varphi_\al+3\tilde T^a\slashed d\varphi_a)\cdot\slashed dx^b(\mathring L\varphi_b)\\
     &+\f12c^{-1}(\mathring Lc+3\tilde T^a\mathring L\varphi_a-3\slashed dx^a\cdot\slashed d\varphi_a+c\textrm{tr}\check\chi)(\mathring L^\al+2\tilde T^\al)\mathring L\varphi_\al.
     \end{split}
     \end{equation}
     It then follows from \eqref{tri} and Proposition \ref{Y5} that
     for any fixed constant $\backepsilon\geq\f12$,
     \begin{equation}\label{Ztri}
     \begin{split}
     &\delta^{2l}\int_{D^{{\mathfrak t},u}}\varrho^{2m-\backepsilon}|Z^k\big(\varrho(\slashed\triangle\varphi_0+\mathring L^a\slashed\triangle\varphi_a)\big)|^2\\
     \lesssim&\delta^{2l_3}\int_{D^{{\mathfrak t},u}}\tau^{2m-\backepsilon}|(\mathring L+\f1{2\varrho})Z^{n_3}\varphi|^2+\delta^{2l_3}\int_{D^{{\mathfrak t},u}}\tau^{2m-2-\backepsilon}|Z^{n_3}\varphi|^2\\
     &+\delta^{2-2\varepsilon_0+2l_2}\int_{D^{{\mathfrak t},u}}\tau^{2m-3-\backepsilon}\sum_{i=1}^2|Z^{n_2}\check L^i|^2+\delta^{2-2\varepsilon_0+l_3}\int_{D^{{\mathfrak t},u}}\tau^{2m-5-\backepsilon}|Z^{n_3}x|^2\\
     &+\delta^{2-2\varepsilon_0+l_2}\int_{D^{{\mathfrak t},u}}\tau^{2m-1-\backepsilon}|Z^{n_2}\textrm{tr}\check\chi|^2
     +\delta^{2-2\varepsilon_0+2l_1}\int_{D^{{\mathfrak t},u}}\tau^{2m-3-\backepsilon}|\slashed{\mathcal L}_Z^{n_1}\leftidx{^{(R)}}{\slashed\pi}_{\mathring L}|^2\\
     &+\delta^{4-2\varepsilon_0+2l_1}\int_{D^{{\mathfrak t},u}}\tau^{2m-3-\backepsilon}|\slashed{\mathcal L}_Z^{n_1}\leftidx{^{(T)}}{\slashed\pi}_{\mathring L}|^2\\
     \lesssim&\int_0^uF_{1,\leq k+2}(\mathfrak t,u')du'+\delta^{3-2\varepsilon_0}+\delta^{2-2\varepsilon_0}\int_{t_0}^\mathfrak t\tau^{-1-\backepsilon}\tilde E_{1,\leq k+2}(\tau,u)d\tau\\
     &+\delta^{2}\int_{t_0}^\mathfrak t\tau^{2m-2-\backepsilon}\tilde E_{2,\leq k+2}(\tau,u)d\tau.
     \end{split}
     \end{equation}
   In addition, \eqref{SiE} implies that for the constant $\backepsilon\geq\f12$,
     \begin{equation}\label{LZphi}
     \begin{split}
     &\int_{D^{{\mathfrak t},u}}\varrho^{2m-\backepsilon}\delta^{2l_3}|\mathring LZ^{n_3}\varphi|^2\\
     \lesssim&\int_{D^{{\mathfrak t},u}}\tau^{2m-\backepsilon}\delta^{2l_3}|(\mathring L+\f1{2\varrho})Z^{n_3}\varphi|^2+\int_{D^{{\mathfrak t},u}}\tau^{2m-2-\backepsilon}\delta^{2l_3}|Z^{n_3}\varphi|^2\\
     \lesssim&\int_0^uF_{1,n_3+1}(\mathfrak t,u')du'+\delta^2\int_{t_0}^\mathfrak t\tau^{-2-\backepsilon}E_{1,n_3+1}(\tau,u)d\tau
     +\delta^2\int_{D^{{\mathfrak t},u}}\tau^{2m-2-\backepsilon}E_{2,n_3+1}.
     \end{split}
     \end{equation}
 On the other hand, applying \eqref{uZkchi}, \eqref{Ztri} and \eqref{LZphi} to estimate the first line at the right hand side of \eqref{Lrchi},
 and utilizing \eqref{SiE} and Proposition \ref{Y5} to handle the other terms, one then can obtain by
 choosing $\backepsilon=1$ in \eqref{Ztri} and \eqref{LZphi} that
     \begin{equation}\label{LrphiuL}
     \begin{split}
     &\delta^{2l}\int_{D^{{\mathfrak t},u}}\varrho^{2m}|Z^k\mathring L(\varrho\textrm{tr}\check\chi)|^2|\underline{\mathring L}\varphi_\gamma|^2\lesssim\delta^{2l-2\varepsilon_0}\int_{D^{{\mathfrak t},u}}\varrho^{2m-1}|Z^k\mathring L(\varrho\textrm{tr}\check\chi)|^2\\
     \lesssim&\delta^{3-4\varepsilon_0}+\delta^{2-4\varepsilon_0}\int_{t_0}^\mathfrak t\tau^{-2}\tilde E_{1,\leq k+2}(\tau,u)d\tau+\delta^{2-2\varepsilon_0}\int_{t_0}^\mathfrak t\tau^{2m-3}\tilde E_{2,\leq k+2}(\tau,u)d\tau\\
     &+\delta^{-2\varepsilon_0}\int_0^uF_{1,\leq k+2}(\mathfrak t,u')du'.
     \end{split}
     \end{equation}
     Therefore, thanks to \eqref{rL2}, Proposition \ref{Y5} and \eqref{LrphiuL}, one has
     \begin{equation*}
     \begin{split}
     &\delta^{2l}\mid\int_{D^{{\mathfrak t}, u}}\sum_{j=1}^k\big(Z_{k+1}+\leftidx{^{(Z_{k+1})}}\chi\big)\dots\big(Z_{k+2-j}
     +\leftidx{^{(Z_{k+2-j})}}\chi\big)\leftidx{^{(\varrho\mathring L)}}D_{\gamma,2}^{k-j}\cdot(\varrho^{2m}\mathring L\varphi_\gamma^{k+1}\\
     &\qquad\qquad\qquad+\f12\varrho^{2m-1}\varphi_\g^{k+1})\mid\\
     \lesssim&\delta^{2l+1}\int_{D^{{\mathfrak t},u}}\varrho^{2m}\mid\sum_{j=1}^k\big(Z_{k+1}
     +\leftidx{^{(Z_{k+1})}}\chi\big)\dots\big(Z_{k+2-j}+\leftidx{^{(Z_{k+2-j})}}\chi\big)\leftidx{^{(\varrho\mathring L)}}D_{\gamma,2}^{k-j}\mid^2\\
     &+\delta^{-1}\int_0^uF_{1,k+2}(\mathfrak t,u')du'
          \end{split}
     \end{equation*}
        \begin{equation}\label{LD}
     \begin{split}
     \lesssim&\delta^{4-4\varepsilon_0}+\delta^{3-4\varepsilon_0}\int_{t_0}^\mathfrak t\tau^{-2}\tilde E_{1,\leq k+2}(\tau,u)d\tau+\delta^{3-2\varepsilon_0}\int_{t_0}^\mathfrak t\tau^{2m-3}\tilde E_{2,\leq k+2}(\tau,u)d\tau\\
     &+\delta^{-1}\int_0^uF_{1,k+2}(\mathfrak t,u')du'.
     \end{split}
     \end{equation}

     \item Finally, we deal with the most difficult case, $Z_{k+1-j}=R$. In this case, it follows from \eqref{R2} that $Z^k\leftidx{^{(R)}}D_{\gamma,2}^0$ contains the term $\f12(RZ^k\text{tr}\check\chi)\underline{\mathring L}\varphi_\gamma$ whose treatment is more subtle and will be given later in Proposition \ref{9.1}. The rest can be estimated as follows
     \begin{equation}\label{D2-R}
     \begin{split}
     &\delta^l|Z^k\leftidx{^{(R)}}D_{\gamma,2}^0-\f12(RZ^k\textrm{tr}\check\chi)\underline{\mathring L}\varphi_\gamma|\\
     \lesssim&\delta^{-\varepsilon_0}\mathfrak t^{-1/2}\underbrace{\delta^{l_3}|\mathring LZ^{n_3}\varphi|}_{\eqref{LZphi}}+\delta^{1-2\varepsilon_0}\underbrace{\delta^{l_2}|Z^{n_2}\mathring L\textrm{tr}\check\chi|}_{\eqref{ZLch}}
     +\delta^{-\varepsilon_0}\mathfrak t^{-1/2}\underbrace{\delta^{l_2}|Z^{n_2}\textrm{tr}\check\chi|}_{\eqref{uZkchi}}\\
     &+\delta^{1-\varepsilon_0}\mathfrak t^{-2}\delta^{l_2}|\slashed{\mathcal L}^{n_2}_Z\leftidx{^{(R)}}{\slashed\pi}_{\mathring L}|+\delta^{1-\varepsilon_0}\mathfrak t^{-5/2}\delta^{l_3}|Z^{n_3}\mu|+\delta^{2-2\varepsilon_0}\mathfrak t^{-3}\delta^{l_3}\sum_{i=1}^2|Z^{n_3}\check L^i|\\
     &+\delta^{-\varepsilon_0}\mathfrak t^{-3/2}\delta^{l_2}\sum_{i=1}^2|Z^{n_2}\check L^i|+\delta^{1-2\varepsilon_0}\mathfrak t^{-3}\delta^{l_4}|Z^{n_4}x|+\delta^{-\varepsilon_0}\mathfrak t^{-3/2}\delta^{l_3}|Z^{n_3}\varphi|\\
     &+\delta^{-\varepsilon_0}\mathfrak t^{-5/2}\delta^{l_2}|Z^{n_2}\upsilon|
     +\delta^{1-\varepsilon_0}\mathfrak t^{-3/2}\delta^{l_3}|\underline{\mathring L}Z^{n_3}\varphi|
     +\delta^{1-2\varepsilon_0}\mathfrak t^{-1}\delta^{l_3}|\slashed dZ^{n_3}\varphi|\\
     &+\delta^{1-2\varepsilon_0}\mathfrak t^{-2}\delta^{l_2}|Z^{n_2}\textrm{tr}\leftidx{^{(R)}}{\slashed\pi}|
     +\delta^{2-2\varepsilon_0}\mathfrak t^{-3}\delta^{l_2}|\slashed{\mathcal L}_Z^{n_2}\leftidx{^{(R)}}{\slashed\pi}_T|
     +\delta^{1-2\varepsilon_0}\mathfrak t^{-3}\delta^{l_3}|Z^{n_3}\upsilon|\\
     &+\delta^{2-2\varepsilon_0}\mathfrak t^{-1}\underbrace{\delta^{l_2}|Z^{n_2}T\textrm{tr}\check\chi|}_{\eqref{ZTchi}}
     +\delta^{1-\varepsilon_0}\mathfrak t^{-1/2}\underbrace{\delta^{l_2}|Z^{n_2}\slashed\triangle\mu|}_{\eqref{Zmu}}
     +\delta^{1-\varepsilon_0}\mathfrak t^{-1/2}\underbrace{\delta^{l_2}|\slashed dZ^{n_2}\textrm{tr}\check{\chi}|}_{\eqref{d}},
     \end{split}
     \end{equation}
     where $l_i$ is the number of $T$ in $Z^{n_i}$ and $n_i\leq k-2+i$ $(i=2,3,4)$.
     To estimate the $L^2$-norm of $\delta^l|Z^k\leftidx{^{(R)}}D_{\gamma,2}^0-\f12(RZ^k\textrm{tr}\check\chi)\underline{\mathring L}\varphi_\gamma|$ over $D^{{\mathfrak t},u}$, one can bound the $L^2$-norms of the terms underlined with braces in \eqref{D2-R} by the corresponding estimates indicated bellow the braces.
      While the other terms without braces can be treated by using Proposition \ref{Y5} and \eqref{SiE}. Then one can conclude that
     \begin{equation}\label{D-R}
     \begin{split}
     &\mid\delta^{2l}\int_{D^{{\mathfrak t},u}}\big(Z^k\leftidx{^{(R)}}D_{\gamma,2}^0-\f12(RZ^k\textrm{tr}\check\chi)\underline{\mathring L}\varphi_\gamma\big)\cdot(\varrho^{2m}\mathring L\varphi_\gamma^{k+1}+\f12\varrho^{2m-1}\varphi_\g^{k+1})\mid\\
     \lesssim&\delta^{-1}\int_0^uF_{1,\leq k+2}(\mathfrak t,u')du'+\delta^{4-4\varepsilon_0}
     +\delta^{3-4\varepsilon_0}\int_{t_0}^\mathfrak t\tau^{-2}\tilde E_{1,\leq k+2}(\tau,u)d\tau\\
     &+\delta^{3-2\varepsilon_0}\int_{t_0}^\mathfrak t\tau^{2m-3}\tilde E_{2,\leq k+2}(\tau,u)d\tau.
     \end{split}
     \end{equation}
     Finally, it remains to treat $\f12(RZ^k\textrm{tr}\check\chi)\underline{\mathring L}\varphi_\gamma$, whose $L^2$ norm cannot be estimated by \eqref{d} directly since the resulting time decay rate is not enough to close the energy estimate \eqref{e} (see also the beginning of Section \ref{YY}). Our strategy here is based on the structural equation \eqref{Ltr}. Indeed, \eqref{Ltr} implies that $\mathring L(\text{tr}\check\chi-E)$ admits better rate of decay in time, which, combined with \eqref{fLfRf} and \eqref{Rfh}, will enable us to obtain the desired estimates for corresponding terms. Meanwhile, the terms involving $E$ defined by \eqref{phiE} can be handled easily by using Proposition \ref{Y5} and \eqref{LZphi} directly. Thus we will show
     \begin{proposition}\label{9.1}
     For $\delta>0$ small, it holds that
     \begin{equation}\label{RchiuL}
     \begin{split}
     &\mid\delta^{2l}\int_{D^{{\mathfrak t},u}}(RZ^k\textrm{tr}\check\chi)\underline {\mathring L}\varphi_\g(\varrho^{2m}\mathring LZ^{k+1}\varphi_\g+\f12\varrho^{2m-1}Z^{k+1}\varphi_\g)\mid\\
     \lesssim&\delta^{3-3\varepsilon_0}+\delta^{-1}\int_0^u\tilde F_{1,\leq k+2}(\mathfrak t,u')du'
     +\delta^{-1}\int_0^u\delta F_{2,k+2}(\mathfrak t,u')du'\\
     &+\delta^{2-3\varepsilon_0}\int_{t_0}^\mathfrak t\tau^{m-2}\tilde E_{1,\leq k+2}(\tau,u)d\tau
     +\delta^{2-\varepsilon_0}\int_{t_0}^\mathfrak t\tau^{2m-5/2}\tilde E_{2,\leq k+2}(\tau,u)d\tau\\
     &+\delta^{2-3\varepsilon_0}\tilde E_{1,\leq k+2}(\mathfrak t,u)
     +\delta^{2-\varepsilon_0}\tilde E_{2,\leq k+2}(\mathfrak t,u),
     \end{split}
     \end{equation}
     where
     $\tilde F_{i,\leq k+2}(\mathfrak t,u)=\sup_{t_0\leq\tau\leq\mathfrak t}F_{i,\leq k+2}(\mathfrak t,u),\quad i=1,2.$
     \end{proposition}
    \begin{proof}
     Noting \eqref{Ltr} for $\textrm{tr}\chi$, one can bound the left hand side of \eqref{RchiuL} by $|\bar I|+|\overline{II}|$ with
     \begin{equation*}
     \begin{split}
     &\bar I=\delta^{2l}\int_{D^{{\mathfrak t},u}}RZ^k(\textrm{tr}\check\chi-E)\cdot\underline {\mathring L}\varphi_\g(\varrho^{2m}\mathring LZ^{k+1}\varphi_\g+\f12\varrho^{2m-1}Z^{k+1}\varphi_\g),\\
     &\overline{II}=\delta^{2l}\int_{D^{{\mathfrak t},u}}(RZ^kE)\underline {\mathring L}\varphi_\g(\varrho^{2m}\mathring LZ^{k+1}\varphi_\g+\f12\varrho^{2m-1}Z^{k+1}\varphi_\g),
     \end{split}
     \end{equation*}
     where $E$ is defined in \eqref{phiE}.

     We start with the easy term $\overline{II}$. It is noted that $RZ^kE$ contains $Z^k(\slashed dx^a\cdot\slashed dR\varphi_a)$ and
     \begin{equation*}
     \slashed dx^a\cdot\slashed dR\varphi_a=\f12\textrm{tr}\leftidx{^{(R)}}{\slashed\pi}(\slashed dx^a\cdot\slashed d\varphi_a)+(Rx^a)\slashed\triangle\varphi_a.
     \end{equation*}
     Replacing $\slashed\triangle\varphi_a$ with $\mu^{-1}(\mathring L\underline{\mathring L}\varphi_a
     +\f1{2\varrho}\underline{\mathring L}\varphi_a-\tilde H_a)$ and applying \eqref{H} yield
     \begin{equation}\label{dd}
     \begin{split}
     &\slashed dx^a\cdot\slashed dR\varphi_a\\
     =&(\mathring{L}+\f1{2\varrho})R\varphi_0+\sum_{a=1}^2(2\varphi_a-\mathring L^a)(\mathring{L}+\f1{2\varrho})R\varphi_a
     +\f12\textrm{tr}\leftidx{^{(R)}}{\slashed\pi}(\slashed dx^a\cdot\slashed d\varphi_a)\\
     &-\sum_{a=1}^2(2\varphi_a-\mathring L^a)\leftidx{^{(R)}}{\slashed\pi}_{\mathring L}\cdot\slashed d\varphi_a-\leftidx{^{(R)}}{\slashed\pi}_{\mathring L}\cdot\slashed d\varphi_0-\f3{2\varrho}R\varphi_0\\
     &-\f1{2\varrho}\sum_{a=1}^2(4\varphi_a-\mathring L^a)R\varphi_a+c^{-1}f(\mathring L^1,\mathring L^2,\varphi)
     \left(
     \begin{array}{ccc}
     \mathring L\varphi\\
     \slashed dx^b\cdot\slashed d\varphi_b\\
     \textrm{tr}\check\chi
     \end{array}
     \right)R\varphi.
     \end{split}
     \end{equation}
     It then follows from \eqref{phiE} and \eqref{dd} that
     \begin{equation*}
     \begin{split}
     |\overline{II}|\lesssim&\delta^{-1}\int_0^uF_{1,k+2}(\mathfrak t,u')du'
     +\delta^{1-2\varepsilon_0}\int_{D^{{\mathfrak t},u}}\varrho^{2m-3}\delta^{2l_3}|Z^{n_3}\varphi|^2\\
     &+\delta^{3-4\varepsilon_0}\int_{D^{{\mathfrak t},u}}\varrho^{2m-6}\delta^{2l_4}|Z^{n_4}x|^2
     +\delta^{2l+1-2\varepsilon_0}\int_{D^{{\mathfrak t},u}}\varrho^{2m-1}|Z^k(\slashed dx^a\cdot\slashed dR\varphi_a)|^2\\
     &+\delta^{3-4\varepsilon_0}\int_{D^{{\mathfrak t},u}}\varrho^{2m-4}\delta^{2l_2}|\slashed{\mathcal L}_Z^{n_2}\leftidx{^{(R)}}{\slashed\pi}_{\mathring L}|^2+\delta^{5-4\varepsilon_0}\int_{D^{{\mathfrak t},u}}\varrho^{2m-6}\delta^{2l_1}|\slashed{\mathcal L}_Z^{n_1}\leftidx{^{(R)}}{\slashed\pi}_{T}|^2\\
     &+\delta^{5-4\varepsilon_0}\int_{D^{{\mathfrak t},u}}\varrho^{2m-4}\delta^{2l_1}|\slashed{\mathcal L}_Z^{n_1}\leftidx{^{(T)}}{\slashed\pi}_{\mathring L}|^2+\delta^{3-4\varepsilon_0}\int_{D^{{\mathfrak t},u}}\varrho^{2m-4}\delta^{2l_3}\sum_{i=1}^2|Z^{n_3}\check L^i|^2\\
      &+\delta^{2l+1-2\varepsilon_0}\int_{D^{{\mathfrak t},u}}\varrho^{2m-1}|\mathring LZ^kR\varphi|^2
           \end{split}
      \end{equation*}
           \begin{equation}\label{IIi}
      \begin{split}
      \lesssim&\delta^{-1}\int_0^uF_{1,k+2}(\mathfrak t,u')du'+\delta^{1-2\varepsilon_0}\int_{D^{{\mathfrak t},u}}\varrho^{2m-1}\delta^{2l}|(\mathring L+\f1{2\varrho})Z^kR\varphi|^2\\
     &+\delta^{3-4\varepsilon_0}\int_{D^{{\mathfrak t},u}}\varrho^{2m-6}\delta^{2l_4}|Z^{n_4}x|^2
     +\delta^{3-4\varepsilon_0}\int_{D^{{\mathfrak t},u}}\varrho^{2m-4}\delta^{2l_3}\sum_{i=1}^2|Z^{n_3}\check L^i|^2\\
     &+\delta^{3-4\varepsilon_0}\int_{D^{{\mathfrak t},u}}\varrho^{2m-4}\delta^{2l_2}|\slashed{\mathcal L}_Z^{n_2}\leftidx{^{(R)}}{\slashed\pi}_{\mathring L}|^2+\delta^{5-4\varepsilon_0}\int_{D^{{\mathfrak t},u}}\varrho^{2m-6}\delta^{2l_1}|\slashed{\mathcal L}_Z^{n_1}\leftidx{^{(R)}}{\slashed\pi}_{T}|^2\\
     &+\delta^{5-4\varepsilon_0}\int_{D^{{\mathfrak t},u}}\varrho^{2m-4}\delta^{2l_1}|\slashed{\mathcal L}_Z^{n_1}\leftidx{^{(T)}}{\slashed\pi}_{\mathring L}|^2+\delta^{3-4\varepsilon_0}\int_{D^{{\mathfrak t},u}}\varrho^{2m-4}\delta^{2l_2}|Z^{n_2}\textrm{tr}\leftidx{^{(R)}}{\slashed\pi}|^2\\
     &+\delta^{1-2\varepsilon_0}\int_{D^{{\mathfrak t},u}}\varrho^{2m-3}\delta^{2l_3}|Z^{n_3}\varphi|^2
     +\delta^{3-4\varepsilon_0}\int_{D^{{\mathfrak t},u}}\varrho^{2m-2}\delta^{2l_2}|Z^{n_2}\textrm{tr}\check\chi|^2\\
     &+\delta^{2l+1-2\varepsilon_0}\int_{D^{{\mathfrak t},u}}\varrho^{2m-1}|\mathring LZ^kR\varphi|^2,
     \end{split}
     \end{equation}
     where $l_i$ is the number of $T$ in $Z^{n_i}$ and $n_i\leq k-2+i$ $(i=1,2,3,4)$.
     Applying \eqref{LZphi} with $\backepsilon=1$ to estimate the last term and using Proposition \ref{Y5} to
     deal with the other corresponding terms, one then obtains
     \begin{equation}\label{iII}
     \begin{split}
     |\overline{II}|\lesssim&\delta^{-1}\int_0^uF_{1,k+2}(\mathfrak t,u')du'+\delta^{4-4\varepsilon_0}
     +\delta^{3-4\varepsilon_0}\int_{t_0}^\mathfrak t\tau^{-2}\tilde E_{1,\leq k+2}(\tau,u)d\tau\\
     &+\delta^{3-2\varepsilon_0}\int_{t_0}^\mathfrak t\tau^{2m-3}\tilde E_{2,\leq k+2}(\tau,u)d\tau.
     \end{split}
     \end{equation}
     We now treat the difficult term $\bar I$. Choose $f_1=\varrho^{2m-2}\underline{\mathring L}\varphi_\g$, $f_2=Z^{k+1}\varphi_\g$ and $f_3=\varrho^2Z^k(\textrm{tr}\check\chi-E)$ in \eqref{fLfRf}, and define
     \begin{equation*}
     \begin{split}
     &\bar I_1=\delta^{2l}\int_{D^{{\mathfrak t},u}}\mathring L\big(\varrho^2Z^k(\textrm{tr}\check\chi-E)\big)(\varrho^{2m-2}\underline{\mathring L}\varphi_\g)(RZ^{k+1}\varphi_\g),\\
     &\bar I_2=\delta^{2l}\int_{\Sigma_{\mathfrak t}^u}(RZ^{k+1}\varphi_\g)(\varrho^{2m-2}\underline{\mathring L}\varphi_\g)\big(\varrho^2Z^k(\textrm{tr}\check\chi-E)\big).
     \end{split}
     \end{equation*}
     It then holds that
     \begin{equation}\label{-bI}
     -\bar I=-\bar I_1+\bar I_2+\delta^{2l}\int_{D^{{\mathfrak t},u}}Er_1+\delta^{2l}\int_{\Sigma_{\mathfrak t}^u}Er_2+\delta^{2l}\int_{\Sigma_{t_0}^u}Er_3.
     \end{equation}
    Note that \eqref{Ltr} implies
     \begin{equation}\label{LrchiE}
     \begin{split}
     \mathring L\big(\varrho^2(\textrm{tr}\check\chi-E)\big)
     =&-2\varrho E+c^{-1}\varrho^2(\mathring L^\al\mathring L\varphi_\al-\mathring Lc)\textrm{tr}\chi-\varrho^2(\textrm{tr}\check\chi)^2+\varrho^2e\\
     =&c^{-1}\varrho(-2\slashed dx^a\cdot\slashed d\varphi_a+2\mathring Lc-\mathring L^\al\mathring L\varphi_\al)
     +c^{-1}\varrho^2(\mathring L^\al\mathring L\varphi_\al-\mathring Lc)\textrm{tr}\check\chi\\&-\varrho^2(\textrm{tr}\check\chi)^2+\varrho^2e.
     \end{split}
     \end{equation}
Then $\bar I_1$ can be rewritten as
     \begin{equation*}
     \begin{split}
     \bar I_1=&\delta^{2l}\int_{D^{{\mathfrak t},u}}[\mathring L,Z^k]\big(\varrho^2(\textrm{tr}\check\chi-E)\big)(\varrho^{2m-2}\underline{\mathring L}\varphi_\g)(RZ^{k+1}\varphi_\g)\\
     &+\delta^{2l}\int_{D^{{\mathfrak t},u}}\mathring L\big([\varrho^2,Z^k](\textrm{tr}\check\chi-E)\big)(\varrho^{2m-2}\underline{\mathring L}\varphi_\g)(RZ^{k+1}\varphi_\g)\\
     &+\delta^{2l}\int_{D^{{\mathfrak t},u}}\varrho^{2m-2}Z^k\mathring L\big(\varrho^2(\textrm{tr}\check\chi-E)\big)\underline{\mathring L}\varphi_\g(RZ^{k+1}\varphi_\g)\\
     =:&\bar I_{11}+\bar I_{12}+\bar I_{13}.
     \end{split}
     \end{equation*}
     $\bar I_{1i}$ ($i=1,2,3$) will be treated separately as follows.
     \begin{enumerate}

     \item
     For $\bar I_{11}$, one can use the facts that $[\mathring L,Z^k]=\ds\sum_{k_1+k_2=k-1}Z^{k_1}[\mathring L, Z]Z^{k_2}$,
     $[\mathring L,\varrho\mathring L]=\mathring L$ and $[\mathring L,\bar Z]
     =\leftidx{^{(\bar Z)}}{\slashed\pi}_{\mathring L}^X\slashed d_X$ with $\bar Z\in\{T,R\}$,
     and Proposition \ref{Y5} to obtain
     \begin{equation}\label{LrL}
\begin{split}
     &\mid\delta^{2l}\int_{D^{{\mathfrak t},u}}Z^{k_1}[\mathring L, \varrho\mathring L]Z^{k_2}\big(\varrho^2(\textrm{tr}\check\chi-E)\big)(\varrho^{2m-2}\underline{\mathring L}\varphi_\g)(RZ^{k_1}(\varrho\mathring L)Z^{k_2}R\varphi_\g)\mid\\
     \lesssim&\mid\delta^{2l}\int_{D^{{\mathfrak t},u}}Z^{k_1}\mathring LZ^{k_2}\big(\varrho^2(\textrm{tr}\check\chi-E)\big)(\varrho^{2m-2}\underline{\mathring L}\varphi_\g)([RZ^{k_1},\varrho\mathring L]Z^{k_2}R\varphi_\g)\mid\\
     &+\mid\delta^{2l}\int_{D^{{\mathfrak t},u}}Z^{k_1}\mathring LZ^{k_2}
     \big(\varrho^2(\textrm{tr}\check\chi-E)\big)(\varrho^{2m-2}\underline{\mathring L}\varphi_\g)(\varrho\mathring LRZ^{k_1}Z^{k_2}R\varphi_\g)\mid\\
          \lesssim&\delta^{-\varepsilon_0}\int_{D^{{\mathfrak t},u}}\varrho^{2m-3/2}\{\underbrace{\delta^{l_2}|Z^{n_2}\textrm{tr}
          \check\chi|}+\delta^{l_2}|Z^{n_2}E|\}\big\{\tau^{-1/2}\delta^{l_3}|Z^{n_3}\varphi|\\
     &\qquad\quad+\delta^{1-\varepsilon_0}\tau^{-1/2}\delta^{l_1}(|\slashed{\mathcal L}_Z^{n_1}\leftidx{^{(R)}}{\slashed\pi}_{\mathring L}|+|\slashed{\mathcal L}_Z^{n_1}\leftidx{^{(T)}}{\slashed\pi}_{\mathring L}|)+\underbrace{\varrho\delta^{l}|\mathring LZ^{k+1}\varphi|}\big\}\\
     \lesssim&\delta^{3-3\varepsilon_0}+\delta^{2-\varepsilon_0}\int_{t_0}^\mathfrak t\tau^{-5/2}\tilde E_{1,\leq k+2}(\tau,u)d\tau\\
     &+\delta^{2-\varepsilon_0}\int_{t_0}^\mathfrak t\tau^{2m-5/2}\tilde E_{2,\leq k+2}(\tau,u)d\tau+\delta^{-\varepsilon_0}\int_0^uF_{1,k+2}(\mathfrak t,u')du',
     \end{split}
     \end{equation}
     where the terms underlined by braces have been estimated by using \eqref{uZkchi}, \eqref{LZphi} with $\backepsilon=\f12$ and \eqref{Em}. Similarly,
     \begin{equation}\label{LbZ}
     \begin{split}
     &\mid\delta^{2l}\int_{D^{{\mathfrak t},u}}Z^{k_1}[\mathring L, \bar Z]Z^{k_2}\big(\varrho^2(\textrm{tr}\check\chi-E)\big)(\varrho^{2m-2}\underline{\mathring L}\varphi_\g)(RZ^{k+1}\varphi_\g)\mid\\
     \lesssim&\delta^{2-4\varepsilon_0}\int_{D^{{\mathfrak t},u}}\varrho^{4m-4}\delta^{2l_1}|\slashed{\mathcal L }_Z^{n_1}\leftidx{^{(\bar Z)}}{\slashed\pi}_{\mathring L}|^2+\delta^{2-4\varepsilon_0}\int_{D^{{\mathfrak t},u}}\varrho^{4m-2}\delta^{2l_2}|Z^{n_2}\textrm{tr}\check\chi|^2\\
     &+\delta^{4-6\varepsilon_0}\int_{D^{{\mathfrak t},u}}\varrho^{4m-7}\delta^{2l_3}|Z^{n_3}x|^2
     +\delta^{2-4\varepsilon_0}\int_{D^{{\mathfrak t},u}}\varrho^{4m-4}\delta^{2l_3}|Z^{n_3}\varphi|^2\\
     &+\delta^{4-6\varepsilon_0}\int_{D^{{\mathfrak t},u}}\varrho^{4m-5}\delta^{2l_2}\sum_{i=1}^2|Z^{n_2}\check L^i|^2
     +\int_{D^{{\mathfrak t},u}}\delta^{2l}|\slashed dZ^{k+1}\varphi|^2\\
     \lesssim&\delta^{5-6\varepsilon_0}+\delta^{4-4\varepsilon_0}\int_{t_0}^\mathfrak t\tau^{2m-4}\tilde E_{1,\leq k+2}(\tau,u)d\tau\\
     &+\delta^{4-4\varepsilon_0}\int_{t_0}^\mathfrak t\tau^{4m-4}\tilde E_{2,\leq k+2}(\tau,u)d\tau
     +\delta^{-1}\int_0^u\delta F_{2,k+2}(\mathfrak t,u')du',
     \end{split}
     \end{equation}
    where $l_i$ is the number of $T$ in $Z^{n_i}$ with $n_i\leq k-2+i$ $(i=1,2,3)$.

     Combining \eqref{LrL} with \eqref{LbZ} yields
     \begin{equation}\label{I11}
     \begin{split}
     |\bar I_{11}|\lesssim&\delta^{3-3\varepsilon_0}+\delta^{2-\varepsilon_0}\int_{t_0}^\mathfrak t\tau^{-5/2}\tilde E_{1,\leq k+2}(\tau,u)d\tau+\delta^{-1}\int_0^u\delta F_{2,k+2}(\mathfrak t,u')du'\\
     &+\delta^{2-\varepsilon_0}\int_{t_0}^\mathfrak t\tau^{2m-5/2}\tilde E_{2,\leq k+2}(\tau,u)d\tau
     +\delta^{-\varepsilon_0}\int_0^uF_{1,k+2}(\mathfrak t,u')du'.
     \end{split}
     \end{equation}

     \item $|\bar I_{12}|$ can be handled similarly as in Case i. Indeed, due to $[\varrho^2,\varrho\mathring L]f=-2\varrho^2 f$ and $[\varrho^2,T]f=2\varrho f$, then
     as for \eqref{LrL}, one can get
     \begin{equation}\label{rrL}
     \begin{split}
     &\mid\delta^{2l}\int_{D^{{\mathfrak t},u}}\mathring L\big(Z^{k_1}[\varrho^2,\varrho\mathring L]Z^{k_2}(\textrm{tr}\check\chi-E)\big)(\varrho^{2m-2}\underline{\mathring L}\varphi_\g)(RZ^{k_1}(\varrho\mathring L)Z^{k_2}R\varphi_\g)\mid\\
     \lesssim&\delta^{3-3\varepsilon_0}+\delta^{2-\varepsilon_0}\int_{t_0}^\mathfrak t\tau^{-5/2}\tilde E_{1,\leq k+2}(\tau,u)d\tau\\
     &+\delta^{2-\varepsilon_0}\int_{t_0}^\mathfrak t\tau^{2m-5/2}\tilde E_{2,\leq k+2}(\tau,u)d\tau
     +\delta^{-\varepsilon_0}\int_0^uF_{1,k+2}(\mathfrak t,u')du'.
     \end{split}
     \end{equation}
    Meanwhile, it follows from \eqref{uZkchi} and Proposition \ref{Y5} that
     \begin{equation}\label{rT}
     \begin{split}
     &\mid\delta^{2l}\int_{D^{{\mathfrak t},u}}\mathring L\big(Z^{k_1}[\varrho^2,T]Z^{k_2}(\textrm{tr}\check\chi-E)\big)(\varrho^{2m-2}\underline{\mathring L}\varphi_\g)(RZ^{k+1}\varphi_\g)\mid\\
     \lesssim&\int_{D^{{\mathfrak t},u}}\delta^{2l}|\slashed dZ^{k+1}\varphi|^2+\delta^{2-2\varepsilon_0}\int_{D^{{\mathfrak t},u}}\tau^{4m-3}\delta^{2l_2}\big(|Z^{n_2}\textrm{tr}\check\chi|^2
     +|Z^{n_2}E|^2\big)\\
     \lesssim&\delta^{-1}\int_0^u\delta F_{2,k+2}(\mathfrak t,u')du'+\delta^{5-4\varepsilon_0}
     +\delta^{4-2\varepsilon_0}\int_{t_0}^\mathfrak t\tau^{2m-5}\tilde E_{1,\leq k+2}(\tau,u)d\tau\\
     &+\delta^{4-2\varepsilon_0}\int_{t_0}^\mathfrak t\tau^{4m-5}\tilde E_{2,\leq k+2}(\tau,u)d\tau.
     \end{split}
     \end{equation}

      Therefore,
     \begin{equation}\label{I12}
     \begin{split}
     |\bar I_{12}|\lesssim&\delta^{3-3\varepsilon_0}+\delta^{2-\varepsilon_0}\int_{t_0}^\mathfrak t\tau^{-5/2}\tilde E_{1,\leq k+2}(\tau,u)d\tau+\delta^{-1}\int_0^u\delta F_{2,k+2}(\mathfrak t,u')du'\\
     &+\delta^{2-\varepsilon_0}\int_{t_0}^\mathfrak t\tau^{2m-5/2}\tilde E_{2,\leq k+2}(\tau,u)d\tau
     +\delta^{-\varepsilon_0}\int_0^uF_{1,k+2}(\mathfrak t,u')du'.
     \end{split}
     \end{equation}

     \item One can use \eqref{LrchiE} to estimate $\bar I_{13}$.
    Indeed, it follows from \eqref{Rfh} and \eqref{dd} that
     \begin{equation*}
     \begin{split}
     &\delta^{2l}\mid\int_{D^{{\mathfrak t},u}}\varrho^{2m-2}Z^k\big(c^{-1}\varrho(2\slashed dx^a\cdot\slashed d\varphi_a-2\mathring Lc+\mathring L^\al\mathring L\varphi_\al)\big)\underline{\mathring L}\varphi_\g(RZ^{k+1}\varphi_\g)\mid\\
     \leq&\delta^{2l}\mid\int_{D^{{\mathfrak t},u}}\varrho^{2m-2}RZ^k\big(c^{-1}\varrho(2\slashed dx^a\cdot\slashed d\varphi_a
     -2\mathring Lc+\mathring L^\al\mathring L\varphi_\al)\big)\underline{\mathring L}\varphi_\g(Z^{k+1}\varphi_\g)\mid\\
     &+\delta^{2l}\mid\int_{D^{{\mathfrak t},u}}\varrho^{2m-2}Z^k\big(c^{-1}\varrho(2\slashed dx^a\cdot\slashed d\varphi_a-2\mathring Lc
     +\mathring L^\al\mathring L\varphi_\al)\big)R\underline{\mathring L}\varphi_\g(Z^{k+1}\varphi_\g)\mid\\
     &+\f12\delta^{2l}\mid\int_{D^{{\mathfrak t},u}}\varrho^{2m-2}Z^k\big(c^{-1}\varrho(2\slashed dx^a\cdot\slashed d\varphi_a
     -2\mathring Lc+\mathring L^\al\mathring L\varphi_\al)\big)\underline{\mathring L}\varphi_\g(Z^{k+1}\varphi_\g)\textrm{tr}\leftidx{^{(R)}}{\slashed\pi}\mid\\
     \lesssim&\int_{D^{{\mathfrak t},u}}\delta^{1-\varepsilon_0}\varrho^{2m-3/2}\big\{\delta^{1-\varepsilon_0+l_4}\tau^{-5/2}|Z^{n_4}x|
     +\tau^{-1}\delta^{l_3}|Z^{n_3}\varphi|+\delta^l|Z^k(\mathring L+\f1{2\varrho})R\varphi|\\
     &+\delta^{1-\varepsilon_0+l_2}\tau^{-3/2}(|Z^{n_2}\textrm{tr}\leftidx{^{(R)}}{\slashed\pi}|
     +|\slashed{\mathcal L}_Z^{n_2}\leftidx{^{(R)}}{\slashed\pi}_{\mathring L}|+\delta|\slashed{\mathcal L}_Z^{n_2}\leftidx{^{(T)}}{\slashed\pi}_{\mathring L}|+\tau|Z^{n_2}\textrm{tr}\check\chi|)\\
     &+\delta^{2-\varepsilon_0+l_1}\tau^{-5/2}|\slashed{\mathcal L}_Z^{n_1}\leftidx{^{(R)}}{\slashed\pi}_T|
     +\underbrace{\delta^l|\mathring LRZ^k\varphi|}+\delta^{1-\varepsilon_0+l_3}\tau^{-3/2}|Z^{n_3}\check L^i|\big\}\delta^l|Z^{k+1}\varphi|,
     \end{split}
     \end{equation*}
     where $l_i$ is the number of $T$ in $Z^{n_i}$ with $n_i\leq k-2+i$ $(i=1,2,3,4)$, and the term underline with brace can be estimated by using \eqref{LZphi} for $\backepsilon=\f12$.

     This, together with Lemma \ref{L2T}, Proposition \ref{Y5} and \eqref{LZphi}, yields
          \begin{equation}\label{dxuL}
     \begin{split}
     &\delta^{2l}\mid\int_{D^{{\mathfrak t},u}}\varrho^{2m-2}Z^k\big(c^{-1}\varrho(2\slashed dx^a\cdot\slashed d\varphi_a-2\mathring Lc+\mathring L^\al\mathring L\varphi_\al)\big)\underline{\mathring L}\varphi_\g(RZ^{k+1}\varphi_\g)\mid\\
     \lesssim&\delta^{4-2\varepsilon_0}+\delta^{3-2\varepsilon_0}\int_{t_0}^\mathfrak t\tau^{-3/2}\tilde E_{1,\leq k+2}(\tau,u)d\tau\\
     &+\delta^{3-2\varepsilon_0}\int_{t_0}^\mathfrak t\tau^{2m-5/2}\tilde E_{2,\leq k+2}(\tau,u)d\tau+\int_0^uF_{1,\leq k+2}(\mathfrak t,u')du'.
     \end{split}
     \end{equation}
     In addition, for $l_i$ and $n_i$ as in \eqref{LbZ}, direct computations yield
     \begin{equation}\label{r2e}
     \begin{split}
     &\delta^{2l}\mid\int_{D^{{\mathfrak t},u}}\varrho^{2m-2}Z^k\big(c^{-1}\varrho^2(\mathring L^\al\mathring L\varphi_\al-\mathring Lc)\textrm{tr}\check\chi-\varrho^2(\textrm{tr}\check\chi)^2+\varrho^2e\big)\underline{\mathring L}\varphi_\g(RZ^{k+1}\varphi_\g)\mid\\
     \lesssim&\int_{D^{{\mathfrak t},u}}\delta^{2l}|\slashed dZ^{k+1}\varphi|^2+\int_{D^{{\mathfrak t},u}}\delta^{2-4\varepsilon_0}\tau^{4m-4}\delta^{2l_3}|Z^{n_3}\varphi|^2+\int_{D^{{\mathfrak t},u}}\delta^{2-4\varepsilon_0}\tau^{4m-2}\delta^{2l_2}|Z^{n_2}\textrm{tr}\check\chi|^2
    \\
     &+\int_{D^{{\mathfrak t},u}}\delta^{4-6\varepsilon_0}\tau^{4m-7}\delta^{2l_3}|Z^{n_3}x|^2
     +\int_{D^{{\mathfrak t},u}}\delta^{4-6\varepsilon_0}\tau^{4m-5}\delta^{2l_2}\sum_{j=1}^2|Z^{n_2}L^j|^2\\
     \lesssim&\delta^{5-6\varepsilon_0}+\delta^{4-4\varepsilon_0}\int_{t_0}^\mathfrak t\tau^{2m-4}\tilde E_{1,\leq k+2}(\tau,u)d\tau+\delta^{4-4\varepsilon_0}\int_{t_0}^\mathfrak t\tau^{4m-4}\tilde E_{2,\leq k+2}(\tau,u)d\tau\\
     &+\delta^{-1}\int_0^u\delta F_{2,k+2}(\mathfrak t,u')du',
     \end{split}
     \end{equation}
     where the last integral in the second inequality has been estimated by \eqref{uZkchi}.
    Collecting \eqref{LrchiE}, \eqref{dxuL} and \eqref{r2e} yields
     \begin{equation}\label{I13}
     \begin{split}
     |\bar I_{13}| \lesssim&\delta^{4-2\varepsilon_0}+\delta^{3-2\varepsilon_0}\int_{t_0}^\mathfrak t\tau^{-3/2}\tilde E_{1,\leq k+2}(\tau,u)d\tau+\delta^{-1}\int_0^u\delta F_{2,k+2}(\mathfrak t,u')du'\\
     &+\delta^{3-2\varepsilon_0}\int_{t_0}^\mathfrak t\tau^{2m-5/2}\tilde E_{2,\leq k+2}(\tau,u)d\tau+\int_0^uF_{1,\leq k+2}(\mathfrak t,u')du'.
     \end{split}
     \end{equation}
     \end{enumerate}
    We conclude from \eqref{I11}, \eqref{I12} and \eqref{I13} that
    \begin{equation}\label{bI}
    \begin{split}
    |\bar I_1|\lesssim&\delta^{3-3\varepsilon_0}+\delta^{2-\varepsilon_0}\int_{t_0}^\mathfrak t\tau^{-3/2}\tilde E_{1,\leq k+2}(\tau,u)d\tau+\delta^{-1}\int_0^u\delta F_{2,k+2}(\mathfrak t,u')du'\\
    &+\delta^{2-\varepsilon_0}\int_{t_0}^\mathfrak t\tau^{2m-5/2}\tilde E_{2,\leq k+2}(\tau,u)d\tau
    +\delta^{-\varepsilon_0}\int_0^uF_{1,\leq k+2}(\mathfrak t,u')du'.
    \end{split}
    \end{equation}

    To deal with $\bar I_2$, one can apply \eqref{LrchiE} to any vectorfield $\bar Z\in\{R, T\}$ to get
    \begin{equation}\label{LRbZ}
    \begin{split}
    &\delta^l|\mathring L\big(\varrho^2R\bar Z^k(\textrm{tr}\check\chi-E)\big)|\\
    \leq&\delta^l|[\mathring L,R\bar Z^k]\big(\varrho^2(\textrm{tr}\check\chi-E)\big)|
    +\delta^l|\mathring LR[\varrho^2,\bar Z^k](\textrm{tr}\check\chi-E)|+\delta^l|R\bar Z^k\mathring L\big(\varrho^2(\textrm{tr}\check\chi-E)\big)|\\
    \lesssim&\delta^{1-\varepsilon_0}\varrho^{1/2}\delta^{l_2}|R\bar Z^{n_2}(\textrm{tr}\check\chi-E)|
    \delta^{1-\varepsilon_0}\mathfrak t^{-1/2}\delta^{l_2}|\slashed{\mathcal L}_{\bar Z}^{n_2}\leftidx{^{(\bar Z)}}{\slashed\pi}_{\mathring L}|+\delta^{1-\varepsilon_0}\mathfrak t^{-3/2}\delta^{l_4}|Z^{n_4}x|\\
    &+\delta^{1-\varepsilon_0}\varrho^{1/2}\delta^{l_3}|\slashed dZ^{n_3}\varphi|+\delta^{1-\varepsilon_0}\mathfrak t^{-1/2}\delta^{l_3}\sum_{i=1}^2|\bar Z^{n_3}\check L^i|+\varrho\delta^{l_2}|\bar Z^{n_2}(\slashed dx^a\cdot\slashed dR\varphi_a)|\\
    &+\delta^{l_3}|Z^{n_3}\varphi|+\mathfrak t\delta^{l_3}|\mathring LZ^{n_3}\varphi|
    +\delta^{2-\varepsilon_0}\mathfrak t^{-3/2}\delta^{l_1}|\slashed{\mathcal L}_Z^{n_1}\leftidx{^{(R)}}{\slashed\pi}_T|
    +\delta^{1-\varepsilon_0}\mathfrak t^{1/2}\delta^{l_2}|Z^{n_2}\textrm{tr}\check\chi|\\
    &+\delta\varrho\delta^{l_2}|\bar Z^{n_2}\mathring L\textrm{tr}\check\chi|.
    \end{split}
    \end{equation}
    This, together with \eqref{Y5} and \eqref{ZTchi}, yields
    \begin{equation}\label{LRbZL2}
    \begin{split}
    &\delta^l\|\mathring L\big(\varrho^2R\bar Z^k(\textrm{tr}\check\chi-E)\big)\|_{L^2(\Sigma_{\mathfrak t}^u)}\\
    \lesssim&\delta^{1-\varepsilon_0}\varrho^{1/2}\delta^{l}\|R\bar Z^{k}(\textrm{tr}\check\chi-E)\|_{L^2(\Sigma_{\mathfrak t}^u)}+\delta^{3/2-\varepsilon_0}
    +\delta^{1-\varepsilon_0}\mathfrak t^{1/2-m}\sqrt{\tilde E_{1,\leq k+2}(\mathfrak t,u)}\\
    &+\delta\sqrt{\tilde E_{2,\leq k+2}(\mathfrak t,u)}+\varrho\delta^{l_2}\|\bar Z^{n_2}(\slashed dx^a\cdot\slashed dR\varphi_a)\|_{L^2(\Sigma_{\mathfrak t}^u)}+t\delta^{l_3}\|\mathring LZ^{n_3}\varphi\|_{L^2(\Sigma_{\mathfrak t}^u)}.
    \end{split}
    \end{equation}

    On the other hand, applying \eqref{Ff} to $F=\varrho^2R\bar Z^k(\textrm{tr}\check\chi-E)(t,u,\vartheta)-\varrho_0^2R\bar Z^k(\textrm{tr}\check\chi-E)(t_0,u,\vartheta)$, one can get from \eqref{LRbZL2} and \eqref{dd} that
    \begin{equation}\label{rR}
    \begin{split}
    &\varrho^{3/2}\delta^l\|R\bar Z^k(\textrm{tr}\check\chi-E)\|_{L^2(\Sigma_{\mathfrak t}^u)}\\
    \lesssim&\delta^{3/2-\varepsilon_0}\mathfrak t^{1/2}+\delta^{1-\varepsilon_0}\mathfrak t^{1-m}\sqrt{\tilde E_{1,\leq k+2}(\mathfrak t,u)}
    +\delta \mathfrak t^{1/2}\sqrt{\tilde E_{2,\leq k+2}(\mathfrak t,u)}\\
    &+\int_{t_0}^\mathfrak t\tau^{1/2}\big\{\delta^{l_2}\|\bar Z^{n_2}(\slashed dx^a\cdot\slashed dR\varphi_a)\|_{L^2(\Sigma_\tau^u)}+\delta^{l_3}\|\mathring LZ^{n_3}\varphi\|_{L^2(\Sigma_\tau^u)}\big\}d\tau\\
    \lesssim&\int_{t_0}^\mathfrak t\tau^{1/2}\delta^{l_3}\|(\mathring L+\f 1{2\varrho})\bar Z^{n_3}\varphi\|_{L^2(\Sigma_\tau^u)}d\tau+\delta^{3/2-\varepsilon_0}\mathfrak t^{1/2}\\
    &+\delta^{1-\varepsilon_0}\mathfrak t^{1-m}\sqrt{\tilde E_{1,\leq k+2}(\mathfrak t,u)}+\delta \mathfrak t^{1/2}\sqrt{\tilde E_{2,\leq k+2}(\mathfrak t,u)}\\
    \lesssim& \mathfrak t^{1-m}\Big(\int_0^u F_{1,\leq k+2}(\mathfrak t,u')du'\Big)^{1/2}+\delta^{3/2-\varepsilon_0}\mathfrak t^{1/2}\\
    &+\delta^{1-\varepsilon_0}\mathfrak t^{1-m}\sqrt{\tilde E_{1,\leq k+2}(\mathfrak t,u)}+\delta \mathfrak t^{1/2}\sqrt{\tilde E_{2,\leq k+2}(\mathfrak t,u)}.\\
    \end{split}
    \end{equation}
    When there is at least one $\varrho\mathring L$ in $Z^k$, that is, $Z^k=Z^{k_1}(\varrho\mathring L)\bar Z^{k_2}$ for $k_1+k_2=k-1$,
    according to \eqref{Ltr}, Proposition \ref{Y5}, \eqref{d} and \eqref{Em}, one can deduce that
    \begin{equation}\label{RZL}
    \begin{split}
    &\delta^l\|RZ^{k_1}(\varrho\mathring L)\bar Z^{k_2}(\textrm{tr}\check\chi-E)\|_{L^2(\Sigma_{\mathfrak t}^u)}\\
    \leq&\delta^l\|RZ^{k_1}[\varrho\mathring L, \bar Z^{k_2}](\textrm{tr}\check\chi-E)\|_{L^2(\Sigma_{\mathfrak t}^u)}
    +\delta^l\|RZ^{k_1}\bar Z^{k_2}(\varrho\mathring L)(\textrm{tr}\check\chi-E)\|_{L^2(\Sigma_{\mathfrak t}^u)}\\
    \lesssim&\delta^{1-\varepsilon_0}\mathfrak t^{-3/2}\delta^{l_1}\|\slashed{\mathcal L}_Z^{n_1}\leftidx{^{(\bar Z)}}{\slashed\pi}_{\mathring L}\|_{L^2(\Sigma_{\mathfrak t}^u)}+\delta^{l_1}\|RZ^{n_1}\textrm{tr}\check\chi\|_{L^2(\Sigma_{\mathfrak t}^u)}
    +\mathfrak t^{-1}\delta^{l_3}\|Z^{n_3}\varphi\|_{L^2(\Sigma_{\mathfrak t}^u)}\\
    &+\delta^{1-\varepsilon_0}\mathfrak t^{-1/2}\delta^{l_1}\|Z^{n_1}\textrm{tr}\check\chi\|_{L^2(\Sigma_{\mathfrak t}^u)}
    +\delta^{1-\varepsilon_0}\mathfrak t^{-3/2}\delta^{l_2}\|Z^{n_2}\check L^i\|_{L^2(\Sigma_{\mathfrak t}^u)}\\
    &+\delta^{1-\varepsilon_0}\mathfrak t^{-5/2}\delta^{l_3}\|Z^{n_3}x^i\|_{L^2(\Sigma_{\mathfrak t}^u)}\\
    \lesssim&\delta^{3/2-\varepsilon_0}\mathfrak t^{-1}+\delta^{1-\varepsilon_0}\mathfrak t^{-1/2-m}\sqrt{\tilde E_{1,\leq k+2}(\mathfrak t,u)}
    +\delta \mathfrak t^{-1}\sqrt{\tilde E_{2,\leq k+2}(\mathfrak t,u)}.
    \end{split}
    \end{equation}

    Thus, it follows from Proposition \ref{LTRh} that
    \begin{equation}\label{RZ}
    \begin{split}
    \delta^l\|RZ^k(\textrm{tr}\check\chi-E)&\|_{L^2(\Sigma_{\mathfrak t}^u)}\lesssim \mathfrak t^{-1/2-m}\Big(\int_0^u F_{1,\leq k+2}(\mathfrak t,u')du'\Big)^{1/2}+\delta^{3/2-\varepsilon_0}\mathfrak t^{-1}\\
    &+\delta^{1-\varepsilon_0}\mathfrak t^{-1/2-m}\sqrt{\tilde E_{1,\leq k+2}(\mathfrak t,u)}+\delta \mathfrak t^{-1}\sqrt{\tilde E_{2,\leq k+2}(\mathfrak t,u)},
    \end{split}
    \end{equation}
    which, together with \eqref{Rfh} and \eqref{uZkchi}, yields
    \begin{equation}\label{bI2}
    \begin{split}
    |\bar I_2|=&\delta^{2l}|\int_{\Sigma_{\mathfrak t}^u}(Z^{k+1}\varphi_\g)\varrho^{2m}R\big\{\underline{\mathring L}\varphi_\g Z^k(\textrm{tr}\check\chi-E)\big\}dx\\
    &\qquad+\f12\int_{\Sigma_{\mathfrak t}^u}\varrho^{2m}(Z^{k+1}\varphi_\g)\underline{\mathring L}\varphi_\g Z^k(\textrm{tr}\check\chi-E)\cdot\textrm{tr}\leftidx{^{(R)}}{\slashed\pi}dx|\\
    \lesssim&\delta^{-\varepsilon_0}\varrho^{2m-1/2}\delta^l\|Z^{k+1}\varphi\|_{L^2(\Sigma_{\mathfrak t}^u)}
    \cdot\delta^l\|RZ^k(\textrm{tr}\check\chi-E)\|_{L^2(\Sigma_{\mathfrak t}^u)}\\
    &+\delta^{-\varepsilon_0}\varrho^{2m-1/2}\delta^l\|Z^{k+1}\varphi\|_{L^2(\Sigma_{\mathfrak t}^u)}
    \cdot\delta^l\|Z^k(\textrm{tr}\check\chi-E)\|_{L^2(\Sigma_{\mathfrak t}^u)}\\
    \lesssim&\delta^{3-3\varepsilon_0}\varrho^{2m-3/2}+\delta^{-\varepsilon_0}\varrho^{-1/2}
    \int_0^uF_{1,\leq k+2}(\mathfrak t,u')du'\\
    &+\delta^{2-3\varepsilon_0}\varrho^{-1/2}\tilde E_{1,\leq k+2}(\mathfrak t,u)
    +\delta^{2-\varepsilon_0}\varrho^{2m-3/2}\tilde E_{2,\leq k+2}(\mathfrak t,u).
    \end{split}
    \end{equation}
   Recall \eqref{-bI}. It remains to estimate terms involving $Er_i$ $(i=1,2,3)$ given in Lemma \ref{f1f2f3}. First, one has
    \begin{equation}\label{Er1}
    \begin{split}
    &\delta^{2l}\mid\int_{D^{{\mathfrak t},u}}Er_1\mid\\
    \lesssim&\delta^{-\varepsilon_0}\int_{t_0}^\mathfrak t\tau^{2m-1}\delta^{l}\|Z^{k+1}\varphi\|_{L^2(\Sigma_\tau^u)}
    \Big\{\delta^{1-\varepsilon_0+l}\|\slashed dZ^k(\textrm{tr}\check\chi-E)\|_{L^2(\Sigma_\tau^u)}+\tau^{-1/2}\delta^{l_2}\|Z^{n_2}\textrm{tr}\check\chi\|_{L^2(\Sigma_\tau^u)}\\
    &
    +\delta^{1-\varepsilon_0}\tau^{-2}\delta^{l_1}\|\slashed{\mathcal L}_Z^{n_1}\leftidx{^{(\bar Z)}}{\slashed\pi}_{\mathring L}\|_{L^2(\Sigma_\tau^u)}+\tau^{-3/2}\delta^{l_3}\|Z^{n_3}\varphi\|_{L^2(\Sigma_\tau^u)}
    +\delta^{1-\varepsilon_0}\tau^{-2}\delta^{l_2}\sum_{i=1}^2\|Z^{n_2}\check L^i\|_{L^2(\Sigma_\tau^u)}\\
    &+\delta^{1-\varepsilon_0}\tau^{-3}\delta^{l_3}\|Z^{n_3}x\|_{L^2(\Sigma_\tau^u)}\Big\}d\tau+\delta^{1-2\varepsilon_0}\int_{t_0}^\mathfrak t\tau^{2m-1}\delta^{l}\|\slashed dZ^{k+1}\varphi\|_{L^2(\Sigma_\tau^u)}\delta^l\|Z^k(\textrm{tr}\check\chi-E)\|_{L^2(\Sigma_\tau^u)}d\tau\\
    &+\int_{t_0}^\mathfrak t\mid\int_{\Sigma_{\mathfrak t}^u}\delta^{2l}(RZ^{k+1}\varphi_\g)\varrho^{2m-1}\underline{\mathring L}\varphi_\g Z^k(\textrm{tr}\check\chi-E)dx\mid d\tau,
    \end{split}
    \end{equation}
    here one has used the identity $(\mathring L+\f1{2\varrho})f_1=(2m-2)\varrho^{2m-3}\underline{\mathring L}\varphi_\g
    +\varrho^{2m-2}(\mathring L+\f1{2\varrho})\underline{\mathring L}\varphi_\g=(2m-2)\varrho^{2m-3}\underline{\mathring L}\varphi_\g+O(\delta^{1-2\varepsilon_0}\varrho^{2m-4})$ due to \eqref{fequation}. Notice that the last term on
    the right hand side of \eqref{Er1} is just $\int_{t_0}^\mathfrak t\varrho^{-1}|\bar I_2|d\tau$ which can be estimated
    by \eqref{bI2}, while the other terms can be estimated by using \eqref{d}, \eqref{uZkchi}, \eqref{SiE} and Proposition \ref{Y5}. Thus it holds that
    \begin{equation}\label{Er}
    \begin{split}
     \delta^{2l}\mid\int_{D^{{\mathfrak t},u}}&Er_1\mid\lesssim\delta^{3-3\varepsilon_0}
     +\delta^{2-3\varepsilon_0}\int_{t_0}^\mathfrak t\tau^{m-2}\tilde E_{1,\leq k+2}(\tau,u)d\tau\\
     &+\delta^{2-\varepsilon_0}\int_{t_0}^\mathfrak t\tau^{2m-5/2}\tilde E_{2,\leq k+2}(\tau,u)d\tau
     +\delta^{-1}\int_0^u\tilde F_{1,\leq k+2}(\mathfrak t,u')du'.
    \end{split}
    \end{equation}
    Next, it follows directly from \eqref{uZkchi} and Proposition \ref{Y5} that
    \begin{equation}\label{Er23}
    \delta^{2l}|\int_{\Sigma_{\mathfrak t}^u}Er_2|+\delta^{2l}|\int_{\Sigma_{t_0}^u}Er_3|\lesssim\delta^{3-3\varepsilon_0}
    +\delta^{2-3\varepsilon_0}\tilde E_{1,\leq k+2}(\mathfrak t,u)+\delta^{2-\varepsilon_0}\tilde E_{2,\leq k+2}(\mathfrak t,u).
    \end{equation}
    Then the estimate for $\bar I$ follows from \eqref{bI}, \eqref{bI2}, \eqref{Er}, \eqref{Er23} and \eqref{-bI}. This and \eqref{iII}
    complete the estimate \eqref{RchiuL} in Proposition \ref{9.1}.
    \end{proof}
Based on \eqref{D-R} and \eqref{RchiuL}, we can end this subsection with
\begin{equation}\label{RD}
\begin{split}
&\delta^{2l}\mid\int_{D^{{\mathfrak t}, u}}\sum_{j=1}^k\big(Z_{k+1}+\leftidx{^{(Z_{k+1})}}\chi\big)\dots\big(Z_{k+2-j}
+\leftidx{^{(Z_{k+2-j})}}\chi\big)\leftidx{^{(R)}}D_{\gamma,2}^{k-j}\cdot(\varrho^{2m}\mathring L\varphi_\gamma^{k+1}\\
&\qquad\qquad\qquad+\f12\varrho^{2m-1}\varphi_\g^{k+1})\mid\\
\lesssim&\mid\delta^{2l}\int_{D^{{\mathfrak t},u}}Z^k\ \leftidx{^{(R)}}D_{\g,2}^0(\varrho^{2m}\mathring L\varphi_\g^{k+1}+\f12\varrho^{2m-1}\varphi_\g^{k+1})\mid+\delta^{-1}\int_0^uF_{1,k+2}(\mathfrak t,u')du'\\
&+\sum_{p=0}^{k-1}\delta^{1+2l}\int_{D^{{\mathfrak t},u}}|Z^p\ \leftidx{^{(R)}}D_{\g,2}^{k-j}|^2\cdot|Z^{\leq k-1-p}\ \leftidx{^{(Z)}}\chi|^2\\
&+\delta^{1+2l}\sum_{j=1}^{k-1}\int_{D^{{\mathfrak t}, u}}\mid\sum_{j=1}^k\big(Z_{k+1}+\leftidx{^{(Z_{k+1})}}\chi\big)\dots\big(Z_{k+2-j}
+\leftidx{^{(Z_{k+2-j})}}\chi\big)\leftidx{^{(R)}}D_{\gamma,2}^{k-j}\mid^2\\
\lesssim&\delta^{3-3\varepsilon_0}+\delta^{-1}\int_0^u\tilde F_{1,\leq k+2}(\mathfrak t,u')du'
+\delta^{-1}\int_0^u\delta F_{2,k+2}(\mathfrak t,u')du'\\
&+\delta^{1-2\varepsilon_0}\int_{t_0}^\mathfrak t\tau^{m-2}\tilde E_{1,\leq k+2}(\tau,u)d\tau
+\delta^{2-\varepsilon_0}\int_{t_0}^\mathfrak t\tau^{2m-5/2}\tilde E_{2,\leq k+2}(\tau,u)d\tau\\
&+\delta^{2-3\varepsilon_0}\tilde E_{1,\leq k+2}(\mathfrak t,u)+\delta^{2-\varepsilon_0}\tilde E_{2,\leq k+2}(\mathfrak t,u).
\end{split}
\end{equation}
\end{enumerate}

\end{enumerate}

\subsection{Estimates for $J_2^k$}\label{l}
Recall that $J_2^k\equiv\mu\textrm{div}\leftidx{^{(Z_{k+1})}}C_\gamma^{k}+\big(Z_{k+1}
+\leftidx{^{(Z_{k+1})}}\chi\big)\dots\big(Z_{1}+\leftidx{^{(Z_1)}}\chi\big)\Phi_\gamma^0$ does not contain
 the top order derivatives of $\textrm{tr}{\chi}$ and $\slashed\triangle\mu$. Therefore, according to
 Proposition \ref{Y5} and the expressions of $\leftidx{^{(Z)}}D_{\g, j}^k$ in \eqref{T1}-\eqref{R3} and \eqref{T2}-\eqref{R2},
one can get
\begin{equation}\label{ZL}
\begin{split}
&\delta^{2l+1}\int_{D^{{\mathfrak t},u}}|\sum_{j=1}^3\leftidx{^{(Z_{k+1})}}D_{\gamma,j}^k\cdot \mathring{\underline L}\varphi_\gamma^{k+1}|\\
\lesssim&\delta^{3-2\varepsilon_0}\int_{t_0}^\mathfrak t\tau^{-2}\tilde E_{1,\leq k+2}(\tau,u)d\tau+\delta\int_{t_0}^\mathfrak t\tau^{-2m}\tilde E_{2,\leq k+2}(\tau,u)d\tau+\delta^{-1}\int_0^u F_{1,\leq k+2}(\mathfrak t,u')du'
\end{split}
\end{equation}
and
\begin{equation}\label{ZrL}
\begin{split}
&\delta^{2l}|\int_{D^{{\mathfrak t},u}}\sum_{j=1}^3\leftidx{^{(Z_{k+1})}}D_{\gamma,j}^k\cdot (\varrho^{2m}\mathring L\varphi_\gamma^{k+1}+\f12\varrho^{2m-1}\varphi_\g^{k+1})|\\
\lesssim&\delta^{3-2\varepsilon_0}\int_{t_0}^\mathfrak t\tau^{-2}\tilde E_{1,\leq k+2}(\tau,u)d\tau
+\delta^{3-2\varepsilon_0}\int_{t_0}^\mathfrak t\tau^{2m-3}\tilde E_{2,\leq k+2}(\tau,u)d\tau\\
&+\delta^{-1}\int_0^u F_{1,\leq k+2}(\mathfrak t,u')du'.
\end{split}
\end{equation}

Note that $\Phi_\g^0=\mu\Box_g\varphi_\g$ is given explicitly in \eqref{ge}.
Then, Proposition \ref{Y5} and \eqref{lamda} lead to
\begin{equation}\label{Phi0}
\begin{split}
&\delta^{2l+1}\int_{D^{{\mathfrak t},u}}|(Z_{k+1}+\leftidx{^{(Z_{k+1})}}\chi)\cdots(Z_{1}
+\leftidx{^{(Z_{1})}}\chi)\Phi_\gamma^0\cdot\mathring{\underline L}\varphi_\gamma^{k+1}|\\
\lesssim&\delta^{3-3\varepsilon_0}+\delta^{2-3\varepsilon_0}\int_{t_0}^\mathfrak t\tau^{-2m-1/2}\tilde E_{1,\leq k+2}(\tau,u)d\tau+\delta^{2-\varepsilon_0}\int_{t_0}^\mathfrak t\tau^{-3/2}\tilde E_{2,\leq k+2}(\tau,u)d\tau\\
&+\delta^{-1}\int_0^uF_{1,\leq k+2}(\mathfrak t,u')du'
\end{split}
\end{equation}
and
\begin{equation}\label{rPhi}
\begin{split}
&\delta^{2l}|\int_{D^{{\mathfrak t},u}}(Z_{k+1}+\leftidx{^{(Z_{k+1})}}\chi)\cdots(Z_{1}
+\leftidx{^{(Z_{1})}}\chi)\Phi_\gamma^0\cdot(\varrho^{2m}\mathring L\varphi_\gamma^{k+1}
+\f12\varrho^{2m-1}\varphi_\g^{k+1})|\\
\lesssim&\delta^{4-4\varepsilon_0}+\delta^{3-4\varepsilon_0}\int_{t_0}^\mathfrak t\tau^{-2}\tilde E_{1,\leq k+2}(\tau,u)d\tau+\delta^{3-2\varepsilon_0}\int_{t_0}^\mathfrak t\tau^{2m-3}\tilde E_{2,\leq k+2}(\tau,u)d\tau\\
&+\delta^{-1}\int_0^uF_{1,\leq n+1}(\mathfrak t,u')du'.
\end{split}
\end{equation}

\section{Global estimates in $A_{2\dl}$}\label{YY}

Based on the estimates in Section \ref{EE}-\ref{ert}, we are now ready to prove the global uniform estimates of the smooth solution $\phi$ to the equation \eqref{1.9} with
initial data \eqref{i0}  and \eqref{i1}-\eqref{Y-0} near $C_0$ by following the framework in \cite{Ding4}. Furthermore, in order to estimate the solution inside $C_0$, we also need to estimate $\phi$ on $\t C_{2\dl}$ as in \cite{Ding4}. But thanks to the special structure of \eqref{1.9}, it is easier to obtain $\mathring L\sim L$, $\mathring L+2T\sim \underline L$ and $R\sim\O$ than that in \cite{Ding4}
 (here the equivalence of two vectorfields means that there are same time decay rates and smallness orders when acting on the solution $\phi$).
 Indeed, substituting \eqref{D11}-\eqref{D13}, \eqref{D22}, \eqref{TD}, \eqref{LD} and \eqref{RD}-\eqref{rPhi} into \eqref{e}, and using the Gronwall's inequality, one can get that under the assumptions $(\star)$ with
small $\delta>0$ for $\f12<m<\f34$,
\begin{equation}\label{E}
\delta\tilde E_{2,\leq 2N-4}(\mathfrak t,u)+\delta\tilde F_{2,\leq 2N-4}(\mathfrak t,u)+\tilde E_{1,\leq 2N-4}(\mathfrak t,u)+\tilde F_{1,\leq 2N-4}(\mathfrak t,u)\lesssim\delta^{2-2\varepsilon_0}.
\end{equation}
Based on \eqref{E}, we can close the bootstrap
assumptions $(\star)$ in Section \ref{BA}. To this end, one needs the following Sobolev type embedding formula (see \cite[Proposition 18.10]{J}).
\begin{lemma}
For any function $f\in H^2(S_{{\mathfrak t},u})$, under the assumptions $(\star)$ for $\delta>0$ small, it holds that
\begin{equation}\label{et}
\|f\|_{L^\infty(S_{{\mathfrak t},u})}\lesssim\f{1}{\sqrt{\mathfrak t}}\sum_{a\leq 1}\|R^a f\|_{L^2(S_{\mathfrak t, u})}.
\end{equation}
\end{lemma}

It follows from \eqref{et}, \eqref{E} and \eqref{SSi} that for $k\leq 2N-6$,
\begin{equation}\label{im}
\delta^l|Z^k\varphi_\gamma|\lesssim\f{\delta^l}{\sqrt{\mathfrak t}} \sum_{a\leq 1}\|R^a Z^k\varphi_\gamma\|_{L^2(S_{\mathfrak t, u})}{\lesssim}\f{\delta^{1/2}}{\sqrt{\mathfrak t}}\big(\sqrt{E_{1,\leq 2N-4}}
+\sqrt{E_{2,\leq 2N-4}}\big)\lesssim\delta^{1-\varepsilon_0} \mathfrak t^{-1/2},
\end{equation}
which is independent of $M$.
This closes the bootstrap assumptions $(\star)$, and hence the uniform estimates and existence of the solution
$\phi$ to \eqref{1.9} with \eqref{i0}  and \eqref{i1}-\eqref{Y-0} in the domain $D^{{\mathfrak t},4\delta}$ can be proved  by
the standard continuity argument (see Figure \ref{pic:p3} in Subsection \ref{p}).

Finally, let $\Gamma\in\{(t+r)L,\underline L,\O\}$ be defined in the end of Section \ref{in}. For any $(t,x)\in\tilde C_{2\delta}$, we will refine the estimates on $\mid\Gamma^\al\phi\mid$, $\O^k\phi$ and $L\O^k\phi$ with better smallness as
\begin{equation}\label{bgp-0}
\begin{split}
|\Gamma^\al\phi(t,x)|\lesssim&\delta^{2-\varepsilon_0} t^{-1/2},\quad |\al|\leq 2N-9,\\
|(\O^k\phi, tL\O^k\phi)(t,x)|\lesssim&\dl^{3-2\ve_0}t^{-1/2},\quad k\leq 2N-9,
\end{split}
\end{equation}
which will be crucial to derive the global estimates and existence of the solution $\phi$ to $\eqref{1.9}$ in $B_{2\dl}$.

First, we improve the estimates on derivatives of $\mathring L^\al\mathring L\varphi_\al$,
$\mathring L^\al\varphi_\al$ and $u-(\mathfrak t-r)$.

Using \eqref{et} again, one can get by \eqref{SSi} that
\begin{equation*}
\begin{split}
&\|\varrho\delta^lZ^\beta(\mathring L^\al\varphi_\al)\|_{L^{\infty}(S_{{\mathfrak t},u})}\\
\lesssim&\delta^{l+1/2}\mathfrak t^{-1/2}\Big(\|\mathring{\underline L}(\varrho R^{\leq 1}Z^\beta(\mathring L^\al\varphi_\al))\|_{L^2(\Sigma_{\mathfrak t}^u)}
+\|\mathring{L}(\varrho R^{\leq 1}Z^\beta(\mathring L^\al\varphi_\al))\|_{L^2(\Sigma_{\mathfrak t}^u)}\Big)\\
\lesssim&\delta^{l+1/2}\mathfrak t^{-1/2}\Big\{\|R^{\leq1}Z^\beta(\mathring L^\al\varphi_\al)\|_{L^2(\Sigma_{\mathfrak t}^u)}
+\varrho\|[\mathring{\underline L},R]Z^\beta(\mathring L^\al\varphi_\al)\|_{L^2(\Sigma_{\mathfrak t}^u)}\\
&\qquad\qquad\quad+\varrho\|R^{\leq1}[\mathring{\underline L},Z^\beta](\mathring L^\al\varphi_\al)\|_{L^2(\Sigma_{\mathfrak t}^u)}
+\varrho\|R^{\leq1}Z^\beta\mathring{\underline L}(\mathring L^\al\varphi_\al)\|_{L^2(\Sigma_{\mathfrak t}^u)}\Big\}.
\end{split}
\end{equation*}
Since $\mathring{\underline L}(\mathring L^\al\varphi_\al)=(\mathring{\underline L}\mathring L^\al)\varphi_\al
+\mathring{\underline L}^\al \mathring L\varphi_\al$ and $\mathring{\underline L}\mathring L^\al$ is a
combination of \eqref{LL} and \eqref{TL}, for $|\beta|\leq 2N-7$,
one then has
\begin{equation}\label{mLphi}
\begin{split}
&\|\varrho\delta^lZ^\beta(\mathring L^\al\varphi_\al)\|_{L^{\infty}(S_{{\mathfrak t},u})}\\
\lesssim&\delta^{1/2}\mathfrak t^{-1/2}\big\{\delta^{l_2}\|Z^{n_2}\varphi\|_{L^2(\Sigma_{\mathfrak t}^u)}
+\delta^{1-\varepsilon_0+l_1}\mathfrak t^{-1/2}\sum_{i=1}^2\|Z^{n_1}\check{L}^i\|_{L^2(\Sigma_{\mathfrak t}^u)}\big\}\\
&+\delta^{3/2-\varepsilon_0+l_2}\mathfrak t^{-1}\big\{\|Z^{n_2}\mu\|_{L^2(\Sigma_{\mathfrak t}^u)}
+\mathfrak t^{-1}\|Z^{n_2}x\|_{L^2(\Sigma_{\mathfrak t}^u)}\big\}\\
&+\delta^{3/2-\varepsilon_0+l_0}\mathfrak t^{-1}\big\{\|\slashed{\mathcal L}_Z^{n_0}\leftidx{^{(\bar Z)}}{\slashed\pi}_{\mathring L}\|_{L^2(\Sigma_\tau^u)}+\|\slashed{\mathcal L}_Z^{n_0}\leftidx{^{(R)}}{\slashed\pi}_{T}\|_{L^2(\Sigma_\tau^u)}\big\}\\
\lesssim&\delta^{2-\varepsilon_0}\mathfrak t^{-1/2}+\delta^{3/2-\varepsilon_0}\mathfrak t^{-1/2}\sqrt{\tilde{E}_{1,\leq2N-4}}
+\delta^{3/2}\mathfrak t^{-1/2}\sqrt{\tilde{E}_{2,\leq2N-4}}\\
\lesssim&\delta^{2-\varepsilon_0}\mathfrak t^{-1/2},
\end{split}
\end{equation}
where $l_i$ is the number of $T$ in $Z^{n_i}$ and $n_i\leq |\beta|+i$ $(i=0,1,2)$. \eqref{mLphi} implies that
in the domain $D^{{\mathfrak t},u}$, $\mathring L^\al\varphi_\al$ can be estimated more precisely as
\begin{equation}\label{mZLphi}
\delta^l|Z^\beta(\mathring L^\al\varphi_\al)|\lesssim\delta^{2-\varepsilon_0}\mathfrak t^{-3/2},\quad|\beta|\leq 2N-7.
\end{equation}
Similarly, it holds that
\begin{equation}\label{mZL}
\delta^l|Z^\beta(\mathring L^\al\mathring L\varphi_\al)|\lesssim\delta^{2-\varepsilon_0}\mathfrak t^{-5/2},\quad|\beta|\leq 2N-8.
\end{equation}

In addition, it follows from \eqref{LeL} and \eqref{mZL} that for $|\beta|\leq 2N-8$,
\begin{equation}\label{ZdL}
\delta^l|Z^\beta\check L^a|\lesssim\delta^{2-\varepsilon_0} \mathfrak t^{-1},
\end{equation}
which leads to
\begin{equation}\label{rrh}
\mid Z^\beta\big( 1-\f{r}{\varrho}\big)\mid\lesssim\delta^{2-\varepsilon_0} \mathfrak t^{-1},\quad |\beta|\leq 2N-8
\end{equation}
by \eqref{rrho} and $\varphi_0+\varphi_i\omega^i=\mathring L^\al\varphi_\al-\varphi_i\check L^i+\varphi_i\omega^i(1-\f{r}{\varrho})=O(\delta^{2-\varepsilon_0}\mathfrak t^{-3/2})+\varphi_i\omega^i(1-\f{r}{\varrho})$.
It follows from \eqref{rrh} that the distance between $C_0$ and $C_{4\delta}$ on the hypersurface $\Sigma_{\mathfrak t}$ is $4\delta
+O(\delta^{2-\varepsilon_0})$ and the characteristic
surface $C_u$ $(0\leq u\leq 4\delta)$ is almost straight with the error $O(\delta^{2-\varepsilon_0})$ from
the corresponding  outgoing conic surface.

On the other hand, note that
\begin{equation}\label{sL}
\begin{split}
&L=\mathring L+\mu^{-1}\big\{(1-\f r\varrho)-g_{ij}\check L^i\tilde T^j+(\f{\varrho}{r}-1)g_{ij}\check T^i\tilde T^j\big\}T+\f{\varrho}{r}g_{ij}\check T^i(\slashed d^Xx^j)X,\\
&\underline L=\mathring L+\mu^{-1}\big\{1+\f{\varrho}{r}-c^{-1}\varphi_i\tilde T^i-\f{\varrho}{r}g_{ij}\check T^i\tilde T^j\big\}T-\f{\varrho}{r}g_{ij}\check T^i(\slashed d^Xx^j)X,\\
&\Omega=R-\mu^{-1}g_{ab}{\epsilon_{i}^a}x^i\check L^b T+\mu^{-1}c^{-1}\epsilon_i^ax^i\varphi_aT.
\end{split}
\end{equation}
When $|\beta|\leq 2N-7$,
\begin{equation*}
\begin{split}
TZ^\beta(\epsilon_i^ax^i\varphi_a)=&[T,Z^\beta](\epsilon_i^ax^i\varphi_a)+Z^\beta T(\epsilon_i^ax^i\varphi_a)\\
=&\sum_{|\beta_1|+|\beta_2|=|\beta|-1}Z^{\beta_1}[T,Z]Z^{\beta_2}(\epsilon_i^ax^i\varphi_a)+Z^\beta(\mu\epsilon_i^a\tilde{T}^i\varphi_a)\\
&+Z^\beta\big(T^\al R\varphi_\al+g_{mn}\epsilon_j^mx^j\check T^n\tilde T^a T\varphi_a\big),
\end{split}
\end{equation*}
and then $|TZ^\beta(\epsilon_i^ax^i\varphi_a)|\lesssim\delta^{1-2\varepsilon_0-l}$ ($l$ is the number of $T$ in $Z^\beta$),
which implies that
\begin{equation}\label{Lxphi}
|\underline LZ^\beta(\epsilon_i^ax^i\varphi_a)|\lesssim\delta^{1-2\varepsilon_0-l},\quad|\beta|\leq2N-7
\end{equation}
due to the second equation in \eqref{sL} and \eqref{im}. Integrate \eqref{Lxphi}
along integral curves of $\underline L$,  and use the
zero boundary value on $C_0$ to get
\begin{equation}\label{Zxphi}
|Z^\beta(\epsilon_i^ax^i\varphi_a)|\lesssim\delta^{2-2\varepsilon_0-l},\quad|\beta|\leq 2N-7.
\end{equation}
Therefore, in $D^{{\mathfrak t},u}$, collecting \eqref{sL}, \eqref{rrh}, \eqref{ZdL}, \eqref{Zxphi} and \eqref{im} yields
\begin{equation}\label{gv}
\mid\Gamma^\al\varphi_\gamma\mid\lesssim\delta^{1-l-\varepsilon_0}\mathfrak t^{-1/2},\quad |\al|\leq 2N-7,
\end{equation}
where $l$ is the number of $\underline L$ in $\Gamma^\al$.

Recall that $\phi$ is the solution of \eqref{1.9} and $\varphi_\gamma=\p_\gamma\phi$. Thus \eqref{gv}
implies that for $|\al|\leq 2N-7$,
\begin{equation}\label{Y-27}
\mid\underline L\Gamma^\al\phi\mid\lesssim\delta^{1-l-\varepsilon_0}\mathfrak t^{-1/2}.
\end{equation}
Hence, as for \eqref{Zxphi}, one can get that in $D^{{\mathfrak t},4\delta}$, for $|\al|\leq 2N-7$,
\begin{equation}\label{gp}
\mid\Gamma^\al\phi\mid\lesssim\delta^{2-l-\varepsilon_0}\mathfrak t^{-1/2}.
\end{equation}

For any point $P(t^0,x^0)\in\tilde C_{2\delta}$,
there is an integral line of
$L$ across this point and the initial point is denoted by $P_0(t_0,x_0)$
on $\Sigma_{t_0}$ with $|x_0|=1$.
It follows from \eqref{me} that
\begin{equation}\label{Y-28}
\mid L(r^{1/2}{\p^\al\underline L}\phi)\mid\lesssim\delta^{2-2\varepsilon_0-|\al|}t^{-3/2},\quad|\al|\leq 1.
\end{equation}
Integrating \eqref{Y-28} along integral curves of $L$ and applying \eqref{local3-3} to show
that on $\tilde C_{2\delta}$,
\begin{equation}\label{PuL}
\mid{\p^\al\underline L}\phi\mid\lesssim\delta^{2-2\varepsilon_0-|\al|} t^{-1/2},\quad|\al|\leq 1.
\end{equation}
Using \eqref{PuL} and \eqref{me} again gives $\mid L(r^{1/2}{\underline L}\phi)\mid\lesssim\delta^{2-\varepsilon_0}t^{-3/2}$,
which implies in turn that
\[
|\underline L\phi|\lesssim\delta^{2-\varepsilon_0}t^{-1/2},
\]
and hence, $\mid L{\underline L}\phi\mid\lesssim\delta^{2-\varepsilon_0} t^{-3/2}$ holds by \eqref{me}.
An induction argument and \eqref{me} show
\begin{equation*}\label{bgp}
\mid\Gamma^\al\phi\mid\lesssim\delta^{2-\varepsilon_0}t^{-1/2},\quad|\al|\leq 2N-9,\  \text{on}\ {\tilde C_{2\dl}}.
\end{equation*}

In addition, it follows from
	\begin{equation*}
	\begin{split}
	&\underline L(\f 1r\phi+2L\phi)=\f 1{r^2}\phi+\f1rL\phi+\f2{r^2}\O^2\phi+ 2(g^{\al\beta,\gamma}\p_\gamma\phi+h^{\al\beta,\nu\g}\p_\nu\phi\p_\g\phi)\p_{\al\beta}^2\phi
	\end{split}
	\end{equation*}
	and \eqref{gv} that on $D^0=\{(t,x): 0\leq t-|x|\leq 2\dl, t+|x|\geq 2+2\dl\}$,
	\begin{equation}\label{LLa}
	|\underline L(\f 1r
	\O^k\phi+2L\O^k\phi)|_{D^0}\lesssim\delta^{2-2\ve_0}t^{-2},\ k\leq 2N-9.
	\end{equation}
	Integrating \eqref{LLa} along integral curves of $\underline L$ in $D^0$ yields
	\begin{equation}\label{La}
	|\f 1r
	\O^k\phi+2L\O^k\phi|_{\tilde C_{2\dl}}\lesssim\delta^{3-2\ve_0}t^{-2},\ k\leq 2N-9,
	\end{equation}
	which means $|L(r^{1/2}\O^k\phi)|_{\tilde C_{2\dl}}\lesssim\delta^{3-2\ve_0}t^{-3/2}$, and then
	\begin{equation}\label{Ophi}
	|\O^k\phi|_{\tilde C_{2\dl}}\lesssim\delta^{3-2\ve_0}t^{-1/2}\ \text{for}\ k\leq 2N-9
	\end{equation}
	thanks to \eqref{local3-2}.
	It follows from \eqref{La}-\eqref{Ophi} directly that
		\begin{equation*}
	|L\O^k\phi|_{\tilde C_{2\dl}}\lesssim\delta^{3-2\ve_0}t^{-3/2}\ \text{for}\ k\leq 2N-9.
	\end{equation*}
Thus, \eqref{bgp-0} is proved.

\section{Global estimates inside $B_{2\dl}$ and the proof of Theorem \ref{main}}\label{inside}

In this section, we derive the global estimates and existence of the solution $\phi$ to \eqref{1.9} inside $B_{2\dl}$.
It is emphasized that different from the problem of 2D or 3D small value solutions inside the cone in \cite{Ding}
and \cite{Godin07}, here the solution $\phi$ to \eqref{1.9} may be large in  $B_{2\dl}$
due to the short pulse initial data on time $t_0$. Nonetheless,
we still intend to make suitable energy assumptions and use Sobolev embedding to estimate $L^\infty$ norm of the solution
in  $B_{2\dl}$. To close the energy assumptions, we will control the power of $\dl$ in $L^\infty$ norm of the solution, and hence
a modified Klainerman-Sobolev Lemma is introduced whose proof is similar to those in \cite{Ding4}.
However, the energies used in \cite{Ding4} are not suitable for \eqref{1.9} due to the slow time decay of the related quantities in 2D.
Making crucial use of the two null conditions for \eqref{1.9}, we can establish some classes of global weighted spacetime energy estimates here.

Define
$$
D_t:=\{(\bar t,x): \bar t-|x|\geq 2\delta, t_0\leq\bar t\leq t\}\subset B_{2\dl}
$$
to be the shaded part in Figure \ref{pic:p4} below. Note that for $\delta >0$ small, the $L^\infty$ norm of $\phi$ and
its first order derivatives are small on the boundary $\tilde C_{2\delta}$ of $B_{2\dl}$ (especially, $\Gamma^{\al}\phi$
admits the better smallness
$O(\dl^{2-\ve_0})$ on $\tilde C_{2\delta}$, see \eqref{bgp-0}).

	\begin{figure}[htbp]
	\centering
	\includegraphics[scale=0.45]{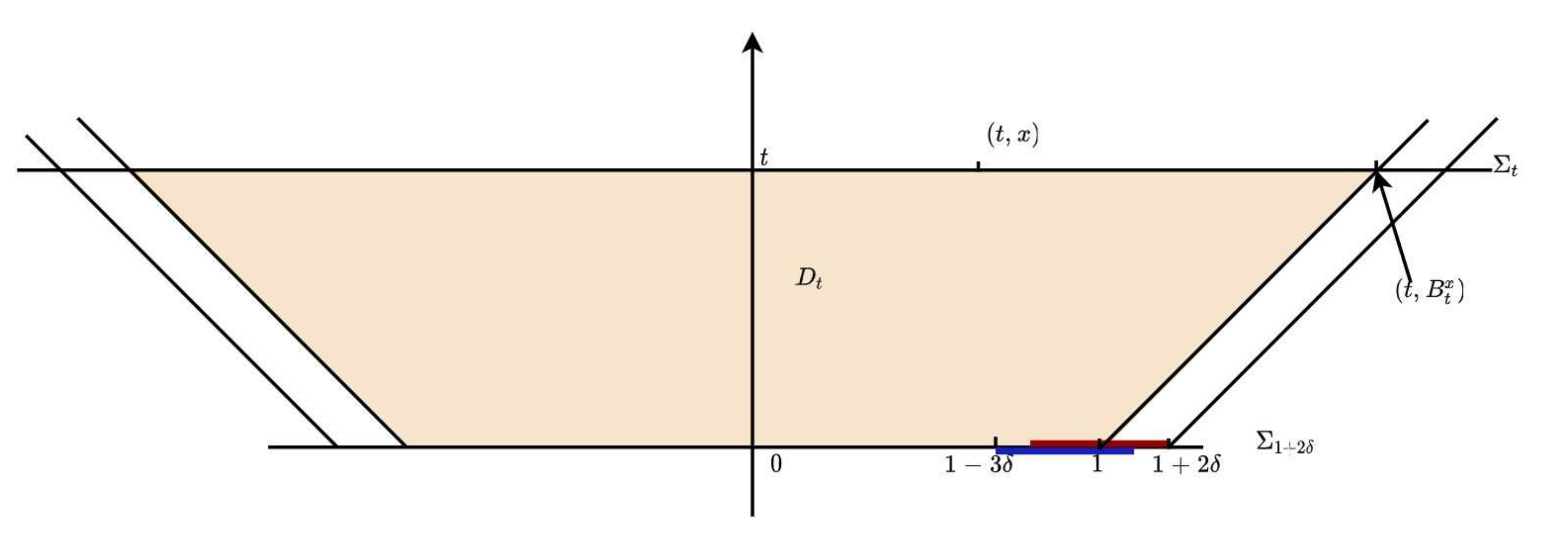}
	\caption{The domain $D_t$ inside $B_{2\dl}$}\label{pic:p4}
\end{figure}

As in \cite[Lemma 12.1]{Ding4} (compared to \cite[Proposition 3.1]{MPY},
the refined inner estimate is new due to the appearance of factor $\delta^{(i-1)s}$),
the following modified Klainerman-Sobolev inequalities hold true.

\begin{lemma}\label{KS}
For any function $f(t,x)\in C^\infty(\mathbb R^{1+2})$, $t\geq 1$,  $(t,x)\in D_{T}=
\{(t,x): t-|x|\geq 2\delta, t_0\leq t\leq T\}$, the following inequalities hold:
\begin{equation}\label{leq}
\mid f(t,x)\mid\lesssim\sum_{i=0}^2 t^{-1}\delta^{(i-1)s}\|\bar{\Gamma}^{i}f(t,\cdot)\|_{L^2(r\leq t/2)},\ |x|\leq\f14t,
\end{equation}
\begin{equation}\label{geq}
\mid f(t,x)\mid\lesssim \mid f(t,B_t^x)\mid+\sum_{a\leq 1,|\beta|\leq 1}t^{-1/2}
\|\Omega^a\p^\beta f(t,\cdot)\|_{L^2(t/4\leq r\leq t-2\delta)},\ |x|\geq\f14t,
\end{equation}
where $\bar{\Gamma}\in\{S,H_1,H_2,\Omega\}$, $(t,B_t^x)$ is the intersection point of
the boundary $\tilde C_{2\delta}$ and the ray crossing $(t,x)$ which emanates from $(t,0)$,
and $s$ is the any nonnegative constant in \eqref{leq}.
\end{lemma}

We also need the following inequality which is similar to Lemma 2.3 for 3D case in \cite{DY}.

\begin{lemma}\label{L2}
For $f(t,x)\in C^\infty(\mathbb R^{1+2})$ and $t\geq 1$, it holds that for $1\leq\bar t\leq t-2\delta$,
\begin{equation}\label{L2L2}
\|\f{f(t,\cdot)}{1+t-|\cdot|}\|_{L^2(\bar t\leq |x|\leq t-2\delta)}\lesssim t^{1/2}
\|f(t,B_t^{\cdot})\|_{L^\infty(\bar t\leq |x|\leq t-2\delta)}+\|\p f(t,\cdot)\|_{L^2(\bar t\leq |x|\leq t-2\delta)}.
\end{equation}
\end{lemma}

We now apply the energy method to derive the global estimates and the existence of solution $\phi$ to \eqref{1.9} in $B_{2\dl}$.
To this end, motivated by the works on global solutions with small data to 3D nonlinear wave equations
satisfying the first null condition in \cite{Lin2,Ding1}, one can define the energy as
\begin{equation}\label{Linfty}
E_{k,l}(t)=\|\p\tilde\Gamma^k\Omega^l\phi(t,\cdot)\|_{L^2(\Sigma_t\cap D_t)}^2+\iint_{D_t}\f{|\tilde Z\tilde\Gamma^k\Omega^l\phi|^2(t',x)}{1+(t'-|x|)^{3/2}}dxdt',
\end{equation}
where $\tilde Z\in\{\tilde Z_i=\omega^i\p_t+\p_i, i=1,2\}$, $\tilde\Gamma\in\{\p,S,H_1,H_2\}$. Based on the
estimate \eqref{local3-2} on $\Sigma_{t_0}$, one can assume that for $t\ge t_0$, there exists a uniform constant $M_0$ such that
\begin{equation}\label{EA}
\begin{split}
E_{k,l}(t)\leq {M_0}^2\delta^{2a_k},\quad k+l\leq 5
\end{split}
\end{equation}
with $a_k=\f52-\ve_0-k$.

According to Lemma \ref{KS}, \ref{L2}, and \eqref{EA}, we can obtain
the following $L^\infty$ estimates.

\begin{proposition}\label{P4.1}
Under the assumptions \eqref{EA}, when $\delta>0$ is small, it holds that in $B_{2\dl}$,
\begin{equation}\label{tZ}
\begin{split}
&|\tilde Z\tilde\G^k\Omega^{l}\phi|\lesssim M_0\delta^{3/2-\varepsilon_0-k}t^{-3/2}(1+t-r),\quad k+l\leq 3
\end{split}
\end{equation}
and
\begin{equation}\label{pphi}
\begin{split}
&|\p\tilde\G^k\Omega^{l}\phi|\lesssim M_0\delta^{3/2-\varepsilon_0-k}t^{-1/2},\quad k+l\leq 3.
\end{split}
\end{equation}
\end{proposition}
\begin{proof}
First, for $|x|\leq\f t4$, one gets from \eqref{leq} that
\begin{equation}\label{Y-30}
\begin{split}
&|\tilde Z\tilde\G^k\Omega^{l}\phi|+|\p\tilde\G^k\Omega^{l}\phi|\\\lesssim
&t^{-1}\Big\{\delta^{-s}\|\p\tilde\G^k\Omega^{l}\phi\|_{L^2(\Sigma_t\cap D_t)}
+\|\bar{\Gamma}\p\tilde\G^k\Omega^{l}\phi\|_{L^2(\Sigma_t\cap D_t)}+\delta^{s}\|\bar{\Gamma}^2\p\tilde\G^k\Omega^{l}\phi\|_{L^2(\Sigma_t\cap D_t)}\big\}.
\end{split}
\end{equation}
Choosing $s=1$ in \eqref{Y-30} and utilizing \eqref{EA} yield
\begin{equation}\label{1s}
|\tilde Z\tilde\G^k\Omega^{l}\phi|+|\p\tilde\G^k\Omega^{l}\phi|\lesssim M_0\delta^{3/2-\varepsilon_0-k} t^{-1},\ k+l\leq 3.
\end{equation}

 Next, for $\f t4\leq |x|\leq t-2\delta$, choosing $f(t,x)=(1+t-|x|)^{-1}|\tilde Z\tilde\G^k\Omega^{l}\phi(t,x)|$ in \eqref{geq},
and applying \eqref{L2L2} to $\O^{\leq1}\tilde Z\tilde\G^k\Omega^{l}\phi(t,x)$, one can get
\begin{equation}\label{Y-31}
\begin{split}
&\f{|\tilde Z\tilde\G^k\Omega^{l}\phi(t,x)|}{1+t-|x|}
\lesssim|\tilde Z\tilde\G^k\Omega^{l}\phi(t,B_t^x)|+t^{-1/2}\|\f{\Omega^{\leq 1}\p^{\leq 1}\tilde Z\tilde\G^k\Omega^{l}\phi}{1+t-r}
\|_{L^2(t/4\leq r\leq t-2\delta)}\\
\lesssim&\|\Omega^{\leq 1}\tilde Z\tilde\G^k\Omega^{l}\phi(t,B_t^{\cdot})\|_{L^\infty(t/4\leq r\leq t-2\delta)}
+t^{-1/2}\|\p\Omega^{\leq 1}\tilde Z\tilde\G^k\Omega^{l}\phi\|_{L^2(t/4\leq r\leq t-2\delta)}.
\end{split}
\end{equation}
Since $\o^i\p_t+\p_i=\o^iL+\f{1}{r}\o^i_\perp\O$, $|\Omega^{\leq 1}\tilde Z\tilde\G^k\Omega^{l}\phi(t,B_t^{x})|
\lesssim\delta^{2-\varepsilon_0} t^{-3/2}$ due to \eqref{bgp-0}, then it follows from \eqref{Y-31} and \eqref{EA} that
\begin{equation}\label{1l}
\f{|\tilde Z\tilde\G^k\Omega^{l}\phi(t,x)|}{1+t-|x|}\lesssim M_0\delta^{3/2-\varepsilon_0-k}t^{-3/2},\ k+l\leq 3.
\end{equation}
In addition, \eqref{geq} implies directly that
\begin{equation}\label{1lp}
|\p\tilde\G^k\Omega^{l}\phi(t,x)|\lesssim M_0\delta^{3/2-\varepsilon_0-k}t^{-1/2},\ k+l\leq 3.
\end{equation}
Thus it follows from \eqref{1s}, \eqref{1l} and \eqref{1lp} that
\begin{align*}
&|\tilde Z\t\G^k\Omega^{l}\phi|\lesssim M_0\delta^{3/2-\varepsilon_0-k}t^{-3/2}(1+t-r),\ k+l\leq 3,\\
&|\p\t\G^k\Omega^{l}\phi|\lesssim M_0\delta^{3/2-\varepsilon_0-k}t^{-1/2},\ k+l\leq 3.
\end{align*}
\end{proof}

\begin{corollary}\label{C4.1}
Under the same conditions in Proposition \ref{P4.1}, it holds that in $B_{2\dl}$,
\begin{equation}
|\p^2\tilde\G^k\Omega^{l}\phi|\lesssim M_0\delta^{1/2-\varepsilon_0-k}t^{-1/2}(1+t-r)^{-1},\ k+l\leq 2
\end{equation}
and
\begin{equation}
|\tilde Z\p\tilde\G^k\Omega^{l}\phi|\lesssim M_0\delta^{1/2-\varepsilon_0-k}t^{-3/2},\ k+l\leq 2.
\end{equation}
\end{corollary}
\begin{proof}
These results follow from
\begin{equation}\label{11.4}
\p_{i}=\f1{1+t-r}\p_i-\f{1}{1+t-r}\big(\f{x^i}{t+r}S-\f{t}{t+r}H_i+\f{x_\perp^i}{t+r}\O\big)\quad\text{with}\  x_\perp=(-x^2,x^1)
\end{equation}
and \eqref{tZ}-\eqref{pphi} directly.
\end{proof}

One can now carry out the energy estimates in $D_t$ by using the ghost weight in \cite{A}.
Indeed, choosing a multiplier $W\p_tv$ with $W=e^{2(1+t-r)^{-1/2}}$ and integrating
$W\p_tv g^{\al\beta}(\p\phi)\p_{\al\beta}^2v$ over $D_t$ yield
\begin{equation}\label{integrte}
\begin{split}
&\int_{\Sigma_t\cap D_t}\f12\Big\{(\p_tv)^2+(1+2\p_t\phi+|\nabla\phi|^2)|\nabla v|^2-\sum_{i,j}\p_i\phi\p_j\phi\p_iv\p_jv\Big\}W\\
&+\iint_{D_t}\f12W\sum_i\f{|\tilde Z_iv|^2}{(1+\tau-r)^{3/2}}\\
=&\int_{\Sigma_{t_0}\cap D_t}\f12\Big\{(\p_tv)^2+(1+2\p_t\phi+|\nabla\phi|^2)|\nabla v|^2-\sum_{i,j}\p_i\phi\p_j\phi\p_iv\p_jv\Big\}W\\
&+\f{\sqrt 2}{2}\int_{\tilde C_{2\delta}\cap D_t}\Big\{\f12(1+2\p_t\phi+|\nabla\phi|^2)\big((Lv)^2
+\f{(\O v)^2}{r^2}\big)-\big(L\phi+\f{(\O\phi)^2}{2r^2}\big)(\p_tv)^2\\
&\qquad-\f12\sum_{i,j}\p_i\phi\p_j\phi\tilde Z_iv\tilde Z_jv\Big\}W\\
&+\iint_{D_t}\f12\Big\{\sum_{i,j}(-2\p_t\phi\delta^{ij}+\p_i\phi\p_j\phi
-|\nabla\phi|^2\delta^{ij})\tilde Z_iv\tilde Z_jv-(\omega^i\omega^j\\
&\qquad\qquad-\delta^{ij})\tilde Z_i\phi\tilde Z_j\phi(\p_tv)^2+2\o^i(\tilde Z_i\phi)(\p_t v)^2\Big\}\f{W}{(1+\tau-r)^{3/2}}\\
&-\iint_{D_t}W\p_tvg^{\al\beta}\p_{\al\beta}^2 v
+\iint_{D_t}\Big\{O(\tilde Z\p\phi\cdot\p v\cdot\p v)+O(\p\phi\cdot\tilde Z\p\phi\cdot\p v\cdot\p v)\Big\}W\\
&+\iint_{D_t}\p_t^2\phi\Big\{(1-\omega^k\p_k\phi)\sum_i(\tilde Z_iv)^2+\sum_{i,j}\omega^i\p_j\phi\tilde Z_iv\tilde Z_jv\Big\}W.
\end{split}
\end{equation}
Note that due to \eqref{pphi}, the integrand of $\int_{\Sigma_t\cap D_t}$ in \eqref{integrte} is equivalent to
\begin{equation}\label{SD}
\big((\p_tv)^2+|\nabla v|^2\big)W,
\end{equation}
while the integrand of $\int_{\tilde C_{2\delta}\cap D_t}$ in \eqref{integrte} can be controlled by
\begin{equation}\label{C2}
\big\{(Lv)^2+\f{1}{r^2}(\O v)^2+\delta^{2-\varepsilon_0}\tau^{-3/2}(\p_tv)^2\big\}W
\end{equation}
with the help of \eqref{bgp-0} and $\tilde Z_i=\omega^iL+\f{\omega^i_\perp}{r}\O$, where the constant coefficients have been neglected.

Inserting \eqref{SD} and \eqref{C2} into \eqref{integrte}, and using
Proposition \ref{P4.1} and Corollary \ref{C4.1},
one then can get by  Gronwall's inequality that for small $\delta>0$,
\begin{equation}\label{W}
\begin{split}
&\int_{\Sigma_t\cap D_t}\big((\p_tv)^2+|\nabla v|^2\big)W+\iint_{D_t}\sum_i\f{|\tilde Z_iv|^2}{(1+\tau-r)^{3/2}}W\\
\lesssim&\int_{\Sigma_{t_0}\cap D_t}\big((\p_tv)^2+|\nabla v|^2\big)W+\iint_{D_t}\mid W\p_tvg^{\al\beta}\p_{\al\beta}^2v\mid\\
&+\int_{\tilde C_{2\delta}\cap D_t}\big\{(Lv)^2+\f{1}{r^2}(\O v)^2+\delta^{2-\varepsilon_0}\tau^{-3/2}(\p_tv)^2\big\}W.
\end{split}
\end{equation}

To close the bootstrap assumptions \eqref{EA}, one will apply \eqref{W} to $v=\tilde\Gamma^k\Omega^l\phi$ $(k+l\leq 6)$.
By \eqref{bgp-0},
\begin{equation*}
|(L\tilde\Gamma^k\Omega^l\phi)^2+\f{1}{r^2}(\O\tilde\Gamma^k\Omega^l\phi)^2
+\delta^{2-\varepsilon_0}\tau^{-3/2}(\p_t\tilde\Gamma^k\Omega^l\phi)^2|\lesssim
\left\{
\begin{aligned}
\dl^{6-4\ve_0}\tau^{-5/2},\ k=0,\\
\delta^{4-2\varepsilon_0}\tau^{-5/2},\ k\geq 1
\end{aligned}
\right.
\end{equation*}
hold on $\tilde C_{2\delta}$. Therefore,
\begin{equation}\label{phit0}
\int_{\tilde C_{2\delta}\cap D_t}\big\{(Lv)^2+\f{1}{r^2}(\Omega v)^2+\delta^{2-\varepsilon_0}\tau ^{-3/2}(\p_tv)^2\big\}W\lesssim\left\{
\begin{aligned}
\dl^{6-4\ve_0},&\ k=0,\ l\leq 6,\\
\delta^{4-2\varepsilon_0},&\ k\geq 1,\ k+l\leq 6.
\end{aligned}
\right.
\end{equation}
In addition, on the initial hypersurface $\Sigma_{t_0}\cap D_t$,
it holds that $|\p\tilde\Gamma^k\Omega^l\phi|\lesssim\delta^{2-k-\varepsilon_0}$
for $k+l\leq 6$ by \eqref{local3-2}. Hence, \eqref{W} gives that for $\ve_0\in(0,\f12)$,
\begin{align}
E_{k,l}(t)\lesssim\delta^{5-2\varepsilon_0-2k}+\iint_{D_t}\mid (\p_t\tilde\Gamma^k\Omega^l\phi)(g^{\al\beta}\p_{\al\beta}^2\tilde\Gamma^k\Omega^l\phi)\mid,\quad k+l\leq 6.\label{E2-6}
\end{align}
It remains to estimate  $\iint_{D_t}\mid(\p_t\tilde\Gamma^k\Omega^l\phi)(g^{\al\beta}\p_{\al\beta}^2\tilde\Gamma^k\Omega^l\phi)\mid$
in \eqref{E2-6}.

\begin{theorem}\label{T4.1}
Under the assumptions \eqref{EA} with $\delta>0$ small, it holds that
\begin{equation}\label{Eklt}
E_{k,l}(t)\lesssim\delta^{2a_k}t^{2\iota},\quad k+l\leq 6,
\end{equation}
where $a_k$ $(k=0,1,\cdots,5)$ are those constants defined in \eqref{EA}, $a_6=-\f72-\varepsilon_0$, and $\iota$
is some constant multiple of $\delta^{\varsigma}$ with $0<\varsigma<1-2\varepsilon_0$.
\end{theorem}
\begin{proof}
Acting the operator $\tilde\Gamma^k\Omega^l$ on \eqref{1.9} and commuting it with $g^{\al\beta}\p_{\al\beta}^2$ yield
\begin{equation*}
\begin{split}
g^{\al\beta}\p_{\al\beta}^2\tilde\Gamma^k\Omega^l\phi=&\sum_{\mbox{\tiny$\begin{array}{cc}k_1+k_2\leq k,l_1+l_2\leq l\\k_2+l_2<k+l\end{array}$}}G_1(\p\tilde\Gamma^{k_1}\Omega^{l_1}\phi,\p^2\tilde\Gamma^{k_2}\Omega^{l_2}\phi)\\
&+\sum_{\mbox{\tiny$\begin{array}{cc}k_1+k_2+k_3\leq k\\l_1+l_2+l_3\leq l\\k_3+l_3<k+l\end{array}$}}G_2(\p\tilde\Gamma^{k_1}\Omega^{l_1}\phi,
\p\tilde\Gamma^{k_2}\Omega^{l_2}\phi,\p^2\tilde\Gamma^{k_3}\Omega^{l_3}\phi),
\end{split}
\end{equation*}
where $G_1$ is a generic quadratic form satisfying the first null condition, and $G_2$ is a generic cubic form
satisfying the second null condition. Hence it follows from Lemma 2.2 in \cite{Hou-Yin-1} that
\begin{equation}\label{Y-33}
\begin{split}
&\iint_{D_t}\mid(\p_t\tilde\Gamma^k\Omega^l\phi)(g^{\al\beta}\p_{\al\beta}^2\tilde\Gamma^k\Omega^l\phi)\mid\\
\lesssim&\iint_{D_t}|\p\tilde\Gamma^k\Omega^l\phi|\Big\{\sum_{\mbox{\tiny$\begin{array}{cc}k_1+k_2\leq k,l_1+l_2\leq l\\k_2+l_2<k+l\end{array}$}}(|\tilde Z\tilde\Gamma^{k_1}\Omega^{l_1}\phi|\cdot|\p^2\tilde\Gamma^{k_2}\Omega^{l_2}\phi|+|\p\tilde\Gamma^{k_1}\Omega^{l_1}\phi|\cdot|\tilde Z\p\tilde\Gamma^{k_2}\Omega^{l_2}\phi|)\\
&+\sum_{\mbox{\tiny$\begin{array}{cc}k_1+k_2+k_3\leq k\\l_1+l_2+l_3\leq l\\k_3+l_3<k+l\end{array}$}}(|\tilde Z\tilde\Gamma^{k_1}\Omega^{l_1}\phi|\cdot|\p\tilde\Gamma^{k_2}\Omega^{l_2}\phi|\cdot|\p^2\tilde\Gamma^{k_3}
\Omega^{l_3}\phi|+|\p\tilde\Gamma^{k_1}\Omega^{l_1}\phi|\cdot|\p\tilde\Gamma^{k_2}\Omega^{l_2}\phi|\cdot|\tilde Z\p\tilde\Gamma^{k_3}\Omega^{l_3}\phi|)\Big\}.
\end{split}
\end{equation}
To estimate the right hand side of \eqref{Y-33}, one can check easily that in \eqref{Y-33}, the
last summation  has better smallness and time decay, so one has only to deal with
the first summation.

\vskip 0.1 true cm

If $k_1+l_1\leq k_2+l_2$, then $k_1+l_1\leq 3$. Thus Proposition \ref{P4.1} implies that
\[
|\tilde Z\tilde\Gamma^{k_1}\Omega^{l_1}\phi|\lesssim M_0\delta^{3/2-\varepsilon_0-k_1}t^{-3/2}(1+t-r),\quad|\p\tilde\Gamma^{k_1}\Omega^{l_1}\phi|\lesssim M_0\delta^{3/2-\varepsilon_0-k_1}t^{-1/2}.
\]
Then, \eqref{11.4} and \eqref{EA} imply that since $\ve_0\in(0,\f12)$ and $k_2+l_2\leq 5$, thus
\begin{equation}\label{k1}
\begin{split}
&\iint_{D_t}|\p\tilde\Gamma^k\Omega^l\phi|\big\{|\tilde Z\tilde\Gamma^{k_1}\Omega^{l_1}\phi|\cdot|\p^2\tilde\Gamma^{k_2}\Omega^{l_2}\phi|
+|\p\tilde\Gamma^{k_1}\Omega^{l_1}\phi|\cdot|\tilde Z\p\tilde\Gamma^{k_2}\Omega^{l_2}\phi|\big\}\\
\lesssim&\iint_{D_t}|\p\tilde\Gamma^k\Omega^l\phi|\dl^{3/2-\ve_0-k_1}M_0\big\{\tau^{-3/2}|(1+\tau-r)\p^2\t\G^{k_2}\O^{l_2}\phi|+\tau^{-1/2}|\tilde Z\p\tilde\Gamma^{k_2}\Omega^{l_2}\phi|\big\}\\
\lesssim&\sum_{a=0}^{6-k}\int_{t_0}^t\tau^{-3/2}E_{k,a}(\tau)d\tau
+\sum_{k_1+k_2\leq k}\delta^{3-2\varepsilon_0-2k_1}M_0^2\int_{t_0}^t\tau^{-3/2}\dl^{5-2\ve_0-2(k_2+1)}M_0^2d\tau\\
\lesssim&\dl^{5-2\ve_0-2k}+\sum_{a=0}^{6-k}\int_{t_0}^t\tau^{-3/2}E_{k,a}(\tau)d\tau.
\end{split}
\end{equation}

If $k_1+l_1>k_2+l_2$, then $k_2+l_2\leq 2$. It follows from Corollary \ref{C4.1} that
\[|\p^2\t\Gamma^{k_2}\Omega^{l_2}\phi|\lesssim M_0\delta^{1/2-\varepsilon_0-k_2}t^{-1/2}(1+t-r)^{-1},\quad|\tilde Z\p\t\Gamma^{k_2}\Omega^{l_2}\phi|\lesssim M_0\delta^{1/2-\varepsilon_0-k_2}t^{-3/2},\]
which imply that there exists a $\varsigma\in(0, 1-2\varepsilon_0)$ such that
\begin{equation}\label{k2}
\begin{split}
&\iint_{D_t}|\p\tilde\Gamma^k\Omega^l\phi|\big\{|\tilde Z\tilde\Gamma^{k_1}\Omega^{l_1}\phi|\cdot|\p^2\tilde\Gamma^{k_2}\Omega^{l_2}\phi|+|\p\tilde\Gamma^{k_1}\Omega^{l_1}\phi|\cdot|\tilde Z\p\tilde\Gamma^{k_2}\Omega^{l_2}\phi|\big\}\\
\lesssim&\iint_{D_t}|\p\tilde\Gamma^k\Omega^l\phi|\dl^{1/2-\ve_0-k_2}\tau^{-1/2}M_0\big\{\f{|\tilde Z\tilde\Gamma^{k_1}\Omega^{l_1}\phi|}{1+\tau-r}+\tau^{-1}|\p\tilde\Gamma^{k_1}\Omega^{l_1}\phi|\big\}\\
\lesssim&\delta^{\varsigma}\int_{t_0}^t\tau^{-1}E_{k,l}(\tau)d\tau
+\sum_{k_1+k_2\leq k}M_0^2\delta^{1-2\varepsilon_0-2k_2-\varsigma}E_{k_1,l_1}(t)+\int_{t_0}^t\tau^{-3/2}E_{k,l}(\tau)d\tau\\
&+\sum_{k_1+k_2\leq k}M_0^2\delta^{1-2\varepsilon_0-2k_2}\int_{t_0}^t\tau^{-3/2}E_{k_1,l_1}(\tau)d\tau\\
\lesssim&\dl^{5-2\ve_0-2k}+\sum_{a=0}^{6-k}\int_{t_0}^t\tau^{-3/2}E_{k,a}(\tau)d\tau+\delta^{\varsigma}\int_{t_0}^t\tau^{-1}E_{k,l}(\tau)d\tau.
\end{split}
\end{equation}

Substituting \eqref{k1} and \eqref{k2} into \eqref{E2-6} yields
\begin{equation}\label{Ekl}
\begin{split}
&\sum_{l=0}^{6-k}E_{k,l}(t)\lesssim\dl^{5-2\ve_0-2k}t^{2\iota},
\end{split}
\end{equation}
which proves \eqref{Eklt}.

\end{proof}

To prove the global estimates of the solution in $B_{2\dl}$, one needs to get $E_{k,l}(t)\lesssim\dl^{2a_k},\ k+l\leq 5$ which are independent of $M_0$ for all time. Since \eqref{Eklt} has been proved, then it holds that
\begin{equation}\label{st}
E_{k,l}(t)\lesssim\dl^{2a_k},\ k+l\leq 5\ \text{and}\ t_0\leq t\leq e^{1/\iota}.
\end{equation}
It remains to show that the inequalities in \eqref{st} for $t\geq e^{1/\iota}$.

As for the $M_0$-independent energy estimate in \eqref{Eklt}, one can also obtain the $M_0$ independent $L^\infty$ estimates
of $\phi$ and its derivatives corresponding to Proposition \ref{P4.1} and Corollary \ref{C4.1}.
\begin{corollary}\label{C11.2}
When $\delta>0$ is small, it holds that in the domain $B_{2\dl}$,
\begin{equation}\label{tZi}
|\tilde Z\tilde\G^k\Omega^{l}\phi|\lesssim\delta^{3/2-\varepsilon_0-k}t^{-3/2+\iota}(1+t-r),\ |\p\tilde\G^k\Omega^{l}\phi|\lesssim\delta^{3/2-\varepsilon_0-k}t^{-1/2+\iota}\quad\text{for}\ k+l\leq 3\\
\end{equation}
and
\begin{equation}\label{t}
\begin{split}
|\p^2\tilde\G^k\Omega^{l}\phi|\lesssim\delta^{1/2-\varepsilon_0-k}t^{-1/2+\iota}(1+t-r)^{-1},\ |\tilde Z\p\tilde\G^k\Omega^{l}\phi|\lesssim\delta^{1/2-\varepsilon_0-k}t^{-3/2+\iota},\quad\text{for}\  k+l\leq 2.
\end{split}
\end{equation}
\end{corollary}

Compared with \eqref{EA}, Proposition \ref{P4.1} and Corollary \ref{C4.1}, though the estimates in \eqref{Eklt}
and Corollary \ref{C11.2} do not depend on $M_0$, yet they contain increasing time factors $t^{2\iota}$ or $t^\iota$.
To overcome this difficulty and close the assumptions \eqref{EA}, we now study the equation on the difference
between $\phi$ and $\phi_a$, where $\phi_a$ satisfies
\begin{equation}\label{phia}
\left\{
\begin{aligned}
&-\p_t^2\phi_a+\triangle\phi_a=0,\\
&\phi_a(1,x)=\phi(1,x),\\
&\p_t\phi_a(1,x)=\p_t\phi(1,x).
\end{aligned}
\right.
\end{equation}

\begin{proposition}\label{as}
$\phi_a$ defined by \eqref{phia} satisfies the following estimate on the hypersurface $\Sigma_t\cap D_t$:
\begin{equation}\label{sd}
\begin{split}
\int_{\Sigma_t\cap D_t}|\p\tilde\Gamma^k\O^l\phi_a|^2dx&+\iint_{D_t}\f{|\tilde Z\tilde\Gamma^k\O^l\phi_a(\tau,x)|^2}{(1+\tau-|x|)^{3/2}}dxd\tau\\
&+\int_{\tilde C_{2\delta}\cap D_t}\big(|L\tilde\Gamma^k\O^l\phi_a|^2+\f1{r^2}|\O\tilde\Gamma^k\O^l\phi_a|^2\big)\lesssim\dl^{5-2\ve_0-2k}.
\end{split}
\end{equation}
\end{proposition}

\begin{proof}
	Exactly similar to Theorem \ref{Th2.1}, \eqref{phia} has a local smooth solution $\phi_a$ on $[1,t_0]$ satisfying
	\begin{align}
	&|L^s\p^q\O^k\phi_a(t_0,x)|\lesssim\delta^{2-|q|-\varepsilon_0},\quad |x| \in[1-2\delta, 1+2\delta],\label{Lphia}\\
	&|\p^q\O^k\phi_a(t_0,x)|\lesssim\dl^{3-\ve_0-|q|},
	\quad |x|\in [1-3\delta, 1+\delta].\label{pphia}
	\end{align}
Recall that
\[
D^0=\{(t,x)|0\leq t-|x|\leq 2\dl, t+|x|\geq 2+2\dl\}.
\]
When $t\geq t_0$, one can estimate $\phi_a(t,x)$ in $D^0$ by the standard energy method. Indeed, it holds that
\begin{equation}\label{energy}
\int_{\Sigma_t\cap D^0}|\p\hat Z^\al \phi_a|^2+\f{\sqrt 2} 2\int_{\tilde C_{2\delta}\cap D^0}\big(|L\hat Z^\al\phi_a|^2+\f1{r^2}|\O\hat Z^\al\phi_a|^2\big)=\int_{\Sigma_{t_0}\cap D^0}|\p\hat Z^\al\phi_a|^2,
\end{equation}
where $\hat Z\in\{\p,\O,S,H_1,H_2\}$. If follows from \eqref{Lphia} that $|\p\hat Z^\al\phi_a|_{\Sigma_{t_0}\cap D^0}\lesssim\delta^{1-\varepsilon_0-l}$ with $l$ being the number of $\p$ in $\hat Z^\al$. Therefore, \eqref{energy} implies
\begin{equation}\label{oen}
\int_{\Sigma_t\cap D^0}|\p\hat Z^\al\phi_a|^2+\f{\sqrt 2} 2\int_{\tilde C_{2\delta}\cap D^0}\big(|L\hat Z^\al\phi_a|^2+\f1{r^2}|\O\hat Z^\al\phi_a|^2\big)\lesssim\delta^{3-2\varepsilon_0-2l}.
\end{equation}
By the following Sobolev's imbedding theorem on the circle $\mathbb S^1_r$
(with center at the origin and radius $r$)
$$
|w(t,x)|\lesssim \f{1}{\sqrt{r}}\|\O^{\leq 1}w\|_{L^2(\mathbb S^1_r)},
$$
one can get that for any point $(t,x)$ in $D^0$,
\begin{equation}\label{hatz}
|\hat Z^\al\phi_a(t,x)|\lesssim t^{-1/2}\|\O^{\leq 1}\hat Z^\al\phi_a\|_{L^2(\mathbb S_r)}\lesssim
t^{-1/2}\delta^{1/2}\|\p\O^{\leq 1}\hat Z^\al\phi_a\|_{L^2(\Sigma_{\mathfrak t}\cap D^0)}\lesssim
\delta^{2-\varepsilon_0-l} t^{-1/2}.
\end{equation}
In addition, $\bar\Gamma^\al\phi_a$ solves
\[
L\big(r^{1/2}\underline L\bar\Gamma^\al\phi_a\big)=\f12r^{-1/2}L\bar\Gamma^\al\phi_a,
\]
which implies $|L\big(r^{1/2}\underline L\bar\Gamma^\al\phi_a\big)|\lesssim\delta^{2-\varepsilon_0}t^{-2}$ by \eqref{hatz} since $\bar\Gamma\in\{S,H_1,H_2,\O\}$. And hence, on the surface $\tilde C_{2\delta}$,
$|\underline L\bar\Gamma^\al\phi_a|\lesssim\delta^{2-\varepsilon_0}t^{-1/2}$ holds
after integrating $L\big(r^{1/2}\underline L\bar\Gamma^\al\phi_a\big)$ along integrate curves of $L$ on $\tilde C_{2\delta}$. Then,
\[
|L\tilde\Gamma\O^l\phi_a|+\f1r|\O\tilde\Gamma\O^l\phi_a|\lesssim\delta^{2-\varepsilon_0}t^{-3/2}\quad \text{on}\ \tilde C_{2\delta}.
\]
By an induction argument, one can get
\begin{equation*}
|L\tilde\Gamma^k\O^l\phi_a|+\f1r|\O\tilde\Gamma^k\O^l\phi_a|\lesssim\delta^{2-\varepsilon_0}t^{-3/2}\quad \text{on}\ \tilde C_{2\delta}.
\end{equation*}
Following the proof for \eqref{LLa}-\eqref{Ophi}, one has
\[
|\O^k\phi_a|_{\t C_{2\dl}}\lesssim\dl^{3-\ve_0}t^{-1/2},\ |L\O^k\phi_a|_{\t C_{2\dl}}\lesssim\dl^{3-\ve_0}t^{-3/2}.
\]
Therefore,
\begin{equation}\label{eLO}
\int_{\tilde C_{2\delta}\cap D_t}\big(|L\tilde\Gamma^k\O^l\phi_a|^2+\f1{r^2}|\O\tilde\Gamma^k\O^l\phi_a|^2\big)
\lesssim\left\{
\begin{aligned}
\dl^{6-2\ve_0},&\ k=0,\\
\delta^{4-2\varepsilon_0},&\ k\geq 1.
\end{aligned}
\right.
\end{equation}

It follows from the property of the weight function $W$ in \eqref{integrte} and \eqref{phia} that
\begin{equation}\label{sdt}
\begin{split}
&\int_{\Sigma_t\cap D_t}W|\p\tilde\Gamma^k\O^l\phi_a|^2+\iint_{D_t}W\f{|\tilde Z\tilde\Gamma^k\O^l\phi_a|^2}{(1+\tau-r)^{3/2}}\\
=&\int_{\Sigma_{t_0}\cap D_t}W|\p\tilde\Gamma^k\O^l\phi_a|^2+\f{\sqrt 2}2\int_{\tilde C_{2\delta}\cap D_t}W\big(|L\tilde\Gamma^k\O^l\phi_a|^2+\f1{r^2}|\O\tilde\Gamma^k\O^l\phi_a|^2\big).
\end{split}
\end{equation}
Then the estimate \eqref{sd} is a direct consequence of \eqref{pphia} and \eqref{eLO}-\eqref{sdt}.
\end{proof}

Similarly as for Proposition \ref{P4.1}, one can use Lemma \ref{KS}, \ref{L2} and Proposition \ref{as} to
get the $L^\infty$ estimate of
$\phi-\phi_a$ in $$\tilde D_t:=\{(\bar t,x): \bar t-|x|\geq 2\delta, T_0\leq\bar t\leq t\}$$ with $t\geq T_0$ and $T_0=e^{1/\iota}$.
Indeed, let $\dot\phi=\phi-\phi_a$. Then $\dot\phi$ solves
\begin{equation}\label{dphi}
\left\{
\begin{aligned}
&-\p_t^2\dot\phi+\triangle\dot\phi=2\sum_{i=1}^2\p_i\phi\p_t\p_i\phi-2\p_t\phi\triangle\phi
+\sum_{i,j=1}^2\p_i\phi\p_j\phi\p_i\p_j\phi-|\nabla\phi|^2\triangle\phi,\\
&\dot\phi(T_0,x)=\p_t\dot\phi(T_0,x)=0.
\end{aligned}
\right.
\end{equation}

\begin{proposition}\label{dpe}
If $\delta>0$ is small and $k+l\leq 5$, then for $t\geq T_0$,
\begin{equation}\label{dphiL2}
\|\p\tilde\Gamma^k\Omega^l\dot\phi\|_{L^2(\Sigma_t\cap \t D_t)}^2+\iint_{\t D_t}\f{|\tilde Z\tilde\Gamma^k\Omega^l\dot\phi|^2(t',x)}{1+(t'-|x|)^{3/2}}dxdt'\lesssim\delta^{2a_k}.
\end{equation}
\end{proposition}

\begin{proof}
By commuting the operator $\tilde\Gamma^k\O^l$ with $-\p_t^2+\triangle$,
and noting that  the right hand side of \eqref{dphi} satisfies the first and second null conditions,
one gets
from \eqref{tZi} and \eqref{t} that
\begin{equation}\label{bdphi}
\begin{split}
&|(-\p_t^2+\triangle)\tilde\Gamma^k\O^l\dot\phi|\\
\lesssim&\sum_{\mbox{\tiny$\begin{array}{cc}k_1+k_2\leq k\\l_1+l_2\leq l\end{array}$}}\big(|\tilde Z\tilde\Gamma^{k_1}\Omega^{l_1}\phi|\cdot|\p^2\tilde\Gamma^{k_2}\Omega^{l_2}\phi|+|\p\tilde\Gamma^{k_1}\Omega^{l_1}\phi|\cdot|\tilde Z\p\tilde\Gamma^{k_2}\Omega^{l_2}\phi|\big)\\
&+\sum_{\mbox{\tiny$\begin{array}{cc}k_1+k_2+k_3\leq k\\l_1+l_2+l_3\leq l\end{array}$}}\big(|\tilde Z\tilde\Gamma^{k_1}\Omega^{l_1}\phi|\cdot|\p\tilde\Gamma^{k_2}\Omega^{l_2}\phi|\cdot|\p^2\tilde\Gamma^{k_3}\Omega^{l_3}\phi|\\
&\qquad\qquad\qquad\qquad+|\p\tilde\Gamma^{k_1}\Omega^{l_1}\phi|\cdot|\p\tilde\Gamma^{k_2}\Omega^{l_2}\phi|\cdot|\tilde Z\p\tilde\Gamma^{k_3}\Omega^{l_3}\phi|\big)\\
\lesssim&\sum_{\mbox{\tiny$\begin{array}{cc}k_1+k_2\leq k\\l_1+l_2\leq l\\k_1+l_1\leq2\end{array}$}}\big\{\delta^{\varpi_{k_1}}t^{-3/2+\iota}(1+t-r)|\p^2\tilde\Gamma^{k_2}\O^{l_2}\phi|
+\delta^{\varpi_{k_1}}t^{-1/2+\iota}|\tilde Z\p\tilde\Gamma^{k_2}\O^{l_2}\phi|\big\}\\
&+\sum_{\mbox{\tiny$\begin{array}{cc}k_1+k_2\leq k\\l_1+l_2\leq l\\k_2+l_2\leq2\end{array}$}}\big\{\delta^{\varpi_{k_2+1}}t^{-1/2+\iota}(1+t-r)^{-1}|\tilde Z\Gamma^{k_1}\O^{l_1}\phi|+\delta^{\varpi_{k_2+1}}t^{-3/2+\iota}|\p\tilde\Gamma^{k_1}\O^{l_1}\phi|\big\},
\end{split}
\end{equation}
where $\varpi_k=\f{3}{2}-\varepsilon_0-k$, $k=0,1,2,3$. And hence,
\begin{equation}\label{bpL2}
\begin{split}
&\|(-\p_t^2+\triangle)\tilde\Gamma^k\O^l\dot\phi\|_{L^2(\Sigma_t\cap D_t)}\\
\lesssim&\sum_{s=0}^{\min\{2,k\}}\delta^{\varpi_s}t^{-3/2+\iota}\sum_{\mbox{\tiny$\begin{array}{cc}p+q\leq 6\\p\leq k-s+1\end{array}$}}\sqrt{E_{p,q}(t)}\\
&+\sum_{s=1}^{\min\{k+1,3\}}\delta^{\varpi_s}t^{-1/2+\iota}\sum_{\mbox{\tiny$\begin{array}{cc}p+q\leq 5\\p\leq k-s+1\end{array}$}}\Big\{\|\f{\tilde Z\tilde\Gamma^p\O^q\phi}{1+t-r}\|_{L^2(\Sigma_t\cap D_t)}+t^{-1}\sqrt{E_{p,q}(t)}\Big\}.
\end{split}
\end{equation}
On the other hand, Lemma \ref{L2} and Theorem \ref{T4.1} imply that
\begin{equation}\label{Zp}
\|\f{\tilde Z\tilde\Gamma^p\O^q\phi}{1+t-r}\|_{L^2(\Sigma_t\cap \t D_t)}\lesssim\delta^{2-\varepsilon_0}t^{-1}+\delta^{a_{p+1}}t^{-1+\iota}.
\end{equation}
Substituting  \eqref{Zp} and \eqref{Eklt} into \eqref{bpL2} yields
\begin{equation}\label{b}
\begin{split}
\|(-\p_t^2+\triangle)\tilde\Gamma^k\O^l\dot\phi\|_{L^2(\Sigma_t\cap \t D_t)}\lesssim&\sum_{s=0}^{\min\{2,k\}}\delta^{\varpi_s}\delta^{a_{k-s+1}}t^{-3/2+2\iota}
+\sum_{s=1}^{\min\{k+1,3\}}\delta^{\varpi_s}\delta^{a_{k-s+2}}t^{-3/2+2\iota}\\
\lesssim&\delta^{2-2\ve_0-k}t^{-3/2+2\iota}\lesssim \dl^{3-2\ve_0-k}t^{-7/6+2\iota}
\end{split}
\end{equation}
since $t\geq \dl^{-3}$ for $t\geq T_0$.
Integrate $W(\p_t\tilde\Gamma^k\O^l\dot\phi)\big((-\p_t^2+\triangle)\tilde\Gamma^k\O^l\dot\phi\big)$ over
domain $\t D_t$ with $W=e^{2(1+t-r)^{-1/2}}$ as in \eqref{integrte}. Then it follows from \eqref{phit0}, \eqref{sd} and \eqref{b} that
for $\varepsilon_0<\f12$,
\begin{equation*}
\begin{split}
&\int_{\Sigma_t\cap\t D_t}W|\p\tilde\Gamma^k\O^l\dot\phi|^2+\iint_{\t D_t}\f W{(1+\tau-r)^{3/2}}|\tilde Z\tilde\Gamma^k\O^l\dot\phi|^2\\
\lesssim&\int_{\Sigma_{T_0}\cap\t D_t}|\p\tilde\Gamma^k\O^l\dot\phi|^2+\iint_{\t D_t}W|\p_t\tilde\Gamma^k\O^l\dot\phi|\cdot|(-\p_t^2+\triangle)\tilde\Gamma^k\O^l\dot\phi|
\end{split}
\end{equation*}
\begin{equation}\label{edphi}
\begin{split}
&+\int_{\tilde C_{2\delta}\cap\t D_t}\big(|L\tilde\Gamma^k\O^l\dot{\phi}|^2+\f{1}{r^2}|\O\tilde\Gamma^k\O^l\dot{\phi}|^2\big)\\
\lesssim&\int_{T_0}^t\tau^{-7/6+2\iota}\|\p\tilde\Gamma^k\O^l\dot\phi\|_{L^2(\Sigma_\tau\cap\t D_t)}^2d\tau+\delta^{5-2\varepsilon_0+2k}.
\end{split}
\end{equation}

Therefore, \eqref{dphiL2} follows from \eqref{edphi} directly since $\iota>0$ is small enough.
\end{proof}

\begin{theorem}\label{11.2}
When $\delta>0$ is small, there exists a smooth solution $\phi$ to \eqref{1.9} in $B_{2\dl}$ for $t\geq t_0$.
\end{theorem}

\begin{proof}
Since $\phi=\phi_a+\dot\phi$, then \eqref{st}, Proposition \ref{as} and \ref{dpe} imply that
\begin{equation}\label{C-0}
E_{k,l}(t)\lesssim\delta^{2a_k}\quad\text{for}\ k+l\leq 5,
\end{equation}
which is independent of $M_0$, then the assumptions \eqref{EA} are then improved.
\end{proof}

Finally, we prove Theorem \ref{main}.
\begin{proof}
Theorem \ref{Th2.1} gives the local existence of smooth solution $\phi$ to \eqref{1.9} with \eqref{i0}
and \eqref{i1}-\eqref{Y-0}.
On the other hand, the a priori global uniform estimates of the solution in $A_{2\dl}$
and in $B_{2\dl}$  have been established in Section \ref{YY} and Theorem \ref{11.2} respectively.
Then it follows from the existence of the smooth solution to \eqref{1.9} and the continuous induction
argument that the proof of $\phi\in C^\infty([1,+\infty)\times\mathbb R^2)$  is finished.
In addition, $|\na\phi|\lesssim\delta^{1-\varepsilon_0}t^{-1/2}$ follows from  \eqref{local1-2}, \eqref{local2-2},
\eqref{im}, and the first inequality in \eqref{pphi}. Furthermore,  $|\phi|\lesssim\delta^{3/2-\varepsilon_0}t^{1/2}$ follows from \eqref{gp}, \eqref{bgp-0}, \eqref{pphi} and the Newton-Leibnitz formula. Thus Theorem \ref{main}
is proved.

\end{proof}

\appendix
\setcounter{equation}{1}
\section{Derivative estimate for the 2D wave equation with short pulse data}

In this Appendix, we derive the derivative estimate for the solution to the 2D homogeneous wave equation $\square\phi=0$ 
with the short pulse data \eqref{i0} under the condition \eqref{i1}. Set
\begin{align*}
&{\tilde\Sigma}_t:=\{(t,x): t\geq 1, x\in\mathbb R^2\},\quad {\tilde\Sigma}_{t}^{u}:=\{(t,x): 0\leq\f{t-r}{2}\leq  u\},\ u\in [0, 2\delta],\\
&{\tilde C}_{u}^t:=\{(t',x): 1+2\dl\le t'\le t, \f{t-r}{2}=u\},\quad {\tilde S}_{t, u}:={\tilde\Sigma}_{t}\cap {\tilde C}_{u}^t,\\
&I:=\{(t,x):1\leq t\leq 1+2\dl, 2-\dl-t\leq r\leq t\},\quad II:=\{(t,x):t\geq 1+2\dl, t-4\dl\leq r\leq t\},\\
&III:=\{(t,x):t\geq 1+2\dl, r\leq t-4\dl\},\quad IV=\{(t,x):1\leq t\leq 1+2\dl, r\le 2-\dl-t\}.
\end{align*}

\begin{figure}[htbp]
	\centering
	\includegraphics[scale=0.8]{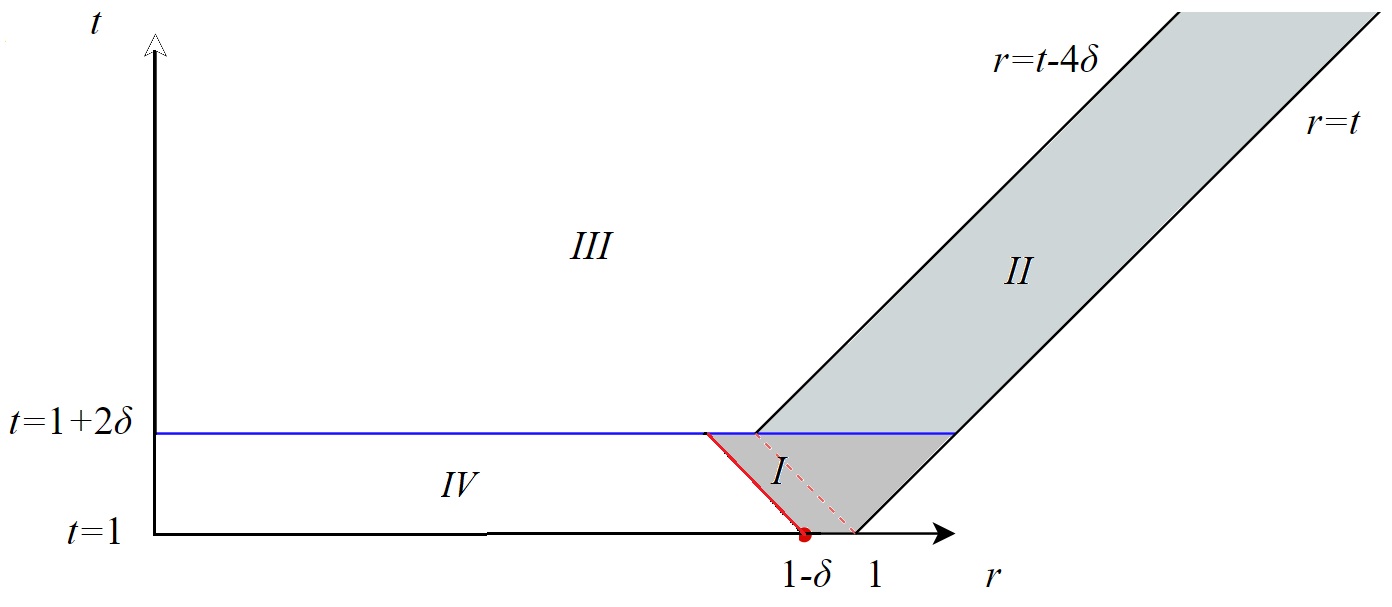}
	\caption{The Domains I, II, III, IV}\label{pic:p5}
\end{figure}

Let $\phi$ be the global smooth solution of $\square\phi=0$ with initial data \eqref{i0}. Then it holds
\begin{equation}\label{int}
\int_{{\tilde\Sigma}_t}|\p \G^\al\O^c Z^{\beta} \phi(t, x)|^2dx=\int_{{\tilde\Sigma}_1}|\p \G^\al\O^c Z^{\beta}\phi(1, x)|^2dx\lesssim\dl^{3-2\ve_0-2|\al|},\ |\beta|\leq 1,
\end{equation}
where $\G\in\{\p,S,H_1,H_2\}$, $Z\in\{S,H_1,H_2\}$, and condition \eqref{i1} has been used.



\vskip 0.2 true cm

${\bf \bullet}$
In domain $I$, it follows from the embedding theorem on $\mathbb S^1$ and \eqref{int} that when $|\beta|\leq 1$,
	\begin{equation}\label{lE}
	|\G^\al\O^c Z^\b \phi(t,x)|\lesssim \|\O^{\leq 1}\G^\al\O^c Z^\b \phi\|_{L^2({\tilde S}_{t,r})}\lesssim\delta^{1/2}\|\p\O^{\leq 1}\G^\al\O^c Z^\b \phi\|_{L^2({\tilde\Sigma}_{t})}\lesssim\delta^{2-|\al|-\ve_0}.
	\end{equation}
	This implies that for $a=0,1$, due to $({t+r})L={S+\o^iH_i}$, it holds that
	\begin{equation}\label{LLEa}
	|L^a\G^\al\O^c\phi(t,x)|\lesssim\sum_{|\b|\leq 1}|Z^\b\G^\al\O^c\phi(t,x)|\lesssim\delta^{2-\ve_0-|\al|}.
	\end{equation}

In addition, in the domain $\tilde D=\{(t,x): 1\le t\le 1+2\dl, 2-\dl-t\le r\le t-\dl\}$, for $(t,x)\in \tilde D$, by
	\begin{equation*}
	L(r^{1/2}\underline L\G^\al\O^c \phi)=\f12r^{-1/2}L\G^\al\O^c \phi+r^{-3/2}\O^2\G^\al\O^c \phi,
	\end{equation*}
	together with \eqref{LLEa}, one has $|L(r^{1/2}\underline L\G^\al\O^c \phi)|_{\tilde D}\lesssim\dl^{2-\ve_0-|\al|}$. It follows from this and integration along integral curves of $L$ that
	\begin{equation}\label{uL}
	|\underline L\G^\al\O^c \phi|_{\tilde D}\lesssim\dl^{3-\ve_0-|\al|}.
	\end{equation}

	Collecting \eqref{LLEa} and \eqref{uL} yields $|\p\G^\al\O^c \phi|_{\tilde D}\lesssim\dl^{2-\ve_0-|\al|}$ and
		\begin{equation}\label{P}
	\|\p\G^\al\O^c \phi\|_{L^2({\tilde\Sigma}_t\cap {\tilde D})}\lesssim\dl^{5/2-\ve_0-|\al|}.
	\end{equation}

	Similarly to \eqref{uL}, for $a=1$, integrating \eqref{LLEa} along integral curves of $L$ yields
	\begin{equation}\label{D}
	|\G^\al\O^c \phi|_{\tilde D}\lesssim\dl^{3-\ve_0-|\al|}.
	\end{equation}
	
${\bf \bullet}$	
In domain $II$, similar to \eqref{LLEa}, one can obtain
	\begin{equation}\label{iI}
	\begin{split}
	&|\p \phi(t,x)|\lesssim t^{-1/2}\|\p\O^{\leq 1} \phi\|_{L^2({\tilde S}_{t,r})}
	\lesssim\dl^{1/2}t^{-1/2}\|\p^2\O^{\leq 1} \phi\|_{L^2({\tilde \Sigma}_t^{2\dl})}\lesssim\dl^{1-\ve_0}t^{-1/2},\\
	&|L^a\G^\al\O^c\phi(t,x)|\lesssim\delta^{2-\ve_0-|\al|}t^{-1/2-a},\ a=0,1.
	\end{split}
	\end{equation}

	On the other hand, it follows from
	\begin{equation*}
	\begin{split}
	&\underline L(\f 1r\G^\al\O^c \phi+2L\G^\al\O^c \phi)=\f 1{r^2}\G^\al\O^c \phi+\f1rL\G^\al\O^c \phi+\f2{r^2}\O^2\G^\al\O^c \phi
	\end{split}
	\end{equation*}
	and \eqref{iI} that $$|\underline L(\f 1r\G^\al\O^c \phi+2L\G^\al\O^c \phi)|\lesssim\dl^{2-\ve_0-|\al|}t^{-5/2}.$$
This implies
	\begin{equation}\label{LA}
		|\f 1r\G^\al\O^c \phi+2L\G^\al\O^c \phi|\lesssim\dl^{3-\ve_0-|\al|}t^{-5/2}
	\end{equation}
and
	$$|L(r^{1/2}\G^\al\O^c \phi)|_{{\tilde C}_{2\dl}^t}\lesssim\delta^{3-\ve_0-|\al|}t^{-2}.$$
Therefore, thanks to \eqref{D}, one has
	\begin{equation}\label{Ov}
	|\G^\al\O^c \phi|_{{\tilde C}_{2\dl}^t}\lesssim\delta^{3-\ve_0-|\al|}t^{-1/2}.
	\end{equation}
	Combining \eqref{LA} with \eqref{Ov} shows that
	\begin{equation}\label{Lv}
	|L\G^\al\O^c \phi|_{{\tilde C}_{2\dl}^t}\lesssim\delta^{3-\ve_0-|\al|}t^{-3/2}.
	\end{equation}

${\bf \bullet}$		
In domain $III$,  one has that
\begin{equation}\label{e3}
\begin{split}
&\int_{{\tilde \Sigma}_t\cap III}|\p\G^\al\O^c \phi(t,x)|^2dx\\
=&\int_{{\tilde \Sigma}_{1+2\dl}\cap{\tilde  D}}|\p\G^\al\O^c \phi(1+2\dl,x)|^2dx+\int_{{\tilde C}_{2\dl}^t}\big(|L\G^\al\O^c \phi|^2+\f1{r^2}|\O\G^\al\O^c \phi|^2\big)dS\\
\lesssim&\dl^{5-2\ve_0-2|\al|},
\end{split}
\end{equation}
where \eqref{P}, \eqref{Ov} and \eqref{Lv} are used.

When $|x|\geq\f14t$, it follows from \eqref{geq} and \eqref{e3} that
\begin{equation}\label{Ge}
|\p \phi(t,x)|\leq\dl^{1-\ve_0}t^{-1/2}.
\end{equation}

When $|x|\leq \f14t$, choose $s=1$ in \eqref{leq} and use \eqref{e3} to obtain
\begin{equation}\label{14}
|\p \phi(t,x)|\leq\dl^{3/2-\ve_0}t^{-1}.
\end{equation}

${\bf \bullet}$		
In domain $IV$, by the finite propagation speed
of the wave equation
and the compact support of $(\phi_0,\phi_1)$, then for $(t,x)\in IV$, one has
\begin{equation}\label{14-00}
\phi(t,x)\equiv0.
\end{equation}

Therefore, for any $(t,x)\in[1,\infty)\times \mathbb R^2$, by \eqref{lE}, \eqref{iI}, \eqref{Ge}-\eqref{14} and \eqref{14-00}
we get
\begin{equation*}
|\p \phi(t,x)|\leq\dl^{1-\ve_0}t^{-1/2}.
\end{equation*}

\end{document}